\documentclass[a4paper,12pt]{amsart}

\usepackage{amssymb,xspace,amscd,yhmath,mathdots,txfonts,
mathrsfs, youngtab,stmaryrd}

\usepackage{imakeidx}
\makeindex[columns=3, title=Notation Index]

\usepackage[matrix,arrow,curve]{xy}\CompileMatrices

\usepackage{color}
\usepackage{hyperref}

\usepackage{yfonts}[1998/10/03]

\DeclareMathAlphabet{\mathpzc}{OT1}{pzc}{m}{it}

\numberwithin{equation}{section}

\theoremstyle{plain}

\newtheorem{thm}{Theorem}[section]

\newtheorem{lem}[thm]{Lemma}

\newtheorem{cor}[thm]{Corollary}

\newtheorem{prop}[thm]{Proposition}

\theoremstyle{definition}
\newtheorem{dfn}[thm]{Definition}
\newtheorem{ntz}[thm]{Notation}
\newtheorem{defn}[thm]{Definition}
\newtheorem{exam}[thm]{Example}

\newtheorem{rmk}[thm]{Remark}

\DeclareMathAlphabet{\mathpzc}{OT1}{pzc}{m}{it}

\newcommand\gr{\mathfrak{g}_{\mathbb{R}}}

\DeclareMathOperator{\R}{{\mathbb{R}}}

\DeclareMathOperator{\kt}{{\kappaup}}

\DeclareMathOperator{\supp}{\mathrm{supp}}
\DeclareMathOperator{\ad}{\mathrm{ad}}

\newcommand\Cd{\mathfrak{C}}

\DeclareMathOperator{\Z}{\mathbb{Z}}

\DeclareMathOperator{\pt}{\mathfrak{p}}

\DeclareMathOperator{\C}{\mathbb{C}}

\newcommand\Kf{\mathbf{K}}
\newcommand\sgn{\mathrm{sgn}}

\newcommand\Aut{\mathpzc{Aut}}
\newcommand\Inv{\mathpzc{I\!{n}{v}}}

\newcommand\Ib{\mathbf{I}}

\newcommand\Invs{\mathpzc{I\!{n}{v}}^{\tauup}_{\!\stt}(\gr,\hr)}

\newcommand\gt{\mathfrak{g}}

\newcommand\Rad{\mathpzc{R}}

\newcommand\Btt{\texttt{B}}
\newcommand\Mtt{\texttt{M}}

\newcommand\hg{\mathfrak{h}}

\newcommand\Fq{\mathpzc{F}}

\newcommand\spt{\mathfrak{sp}}

\newcommand\Wf{\mathbf{W}}

\newcommand\Sb{\mathbf{S}}

\newcommand\bt{\mathfrak{b}}

\newcommand\gl{\mathfrak{gl}}

\newcommand\slt{\mathfrak{sl}}
\newcommand\su{\mathfrak{su}}

\newcommand\Af{\mathbf{A}}

\newcommand{\Bz}{\mathpzc{B}}

\newcommand\kll{\texttt{k}}

\newcommand{\lt}{\mathfrak{l}}

\newcommand\id{\mathrm{id}}

\newcommand\vq{\mathpzc{v}}

\newcommand\bil{\texttt{b}}

\newcommand\so{\mathfrak{so}}

\newcommand\Hb{\mathbb{H}}

\newcommand\sfv{\textsf{v}}

\newcommand\Hom{\mathrm{Hom}}
\newcommand\Homz{\mathrm{Hom}(\Z[\Rad],\Z_{2}^{*})}
\newcommand\Homs{\mathrm{Hom}_{\stt}(\Z[\Rad],\Z_{2}^{*})}

\newcommand\epi{\varepsilon}

\newcommand\sq{\mathpzc{s}}

\newcommand\kq{\mathpzc{k}}

\newcommand\ttt{\texttt{t}}
\newcommand\vtt{\texttt{v}}

\newcommand\stt{\texttt{s}}
\newcommand\att{\texttt{a}}

\newcommand\hsi{\mathfrak{h}_{\stt}}
\newcommand\ftt{\texttt{f}}
\newcommand\gtt{\texttt{g}}

\newcommand\e{\texttt{e}}

\newcommand\hr{\mathfrak{h}_{\mathbb{R}}}
\newcommand\gu{\mathfrak{g}_{\tauup}}
\newcommand\gs{\mathfrak{g}_{\sigmaup}}

\newcommand\hu{i\mathfrak{h}_{\R}}
\newcommand\hs{\mathfrak{h}_{\sigmaup}}
\newcommand\ks{\kappaup_{\sigmaup}}
\newcommand\ps{\mathfrak{p}_{\sigmaup}}
\newcommand\htt{\texttt{h}}
\newcommand\ptt{\texttt{p}}
\newcommand\Stt{\texttt{S}}
\newcommand\hst{\mathfrak{h}_{\stt}}
\newcommand\go{\mathfrak{g}_{0}}
\newcommand\ho{\mathfrak{h}_{0}}
\newcommand\Rd[2]{\Rad_{\,\;#1}^{#2}}

\newcommand\Reg{\mathpzc{Reg}}

   \def\DHLhksqrt#1#2{\setbox0=\hbox{$#1\sqrt{#2\,}$}\dimen0=\ht0
     \advance\dimen0-0.2\ht0
     \setbox2=\hbox{\vrule height\ht0 depth -\dimen0}%
     {\box0\lower0.4pt\box2}}
\title{Root involutions, real forms and diagrams}

\author{S.Marini, C.Medori, M.Nacinovich}
\address{Stefano Marini: Dipartimento di Scienze Matematiche, Fisiche e Informatiche\\ Universit\`a di Parma\\ Parco Area delle Scienze 53/A (Campus), 43124 Parma
 (Italy)} \email{stefano.marini@unipr.it}

\address{Costantino Medori:
Dipartimento di Scienze Matematiche, Fisiche e Informatiche\\ Universit\`a di Parma\\ Parco Area delle Scienze 53/A (Campus), 43124 Parma
 (Italy)} \email{costantino.medori@unipr.it}

\address{Mauro Nacinovich:
Dipartimento di Matematica\\ II Universit\`a di Roma
``Tor Ver\-ga\-ta''\\ Via della Ricerca Scientifica\\ 00133 Roma
(Italy)}
\email{nacinovi@mat.uniroma2.it}


\date\today

\begin{document}
\begin{abstract}
We study the correspondence between 
 equivalence classes of pairs consisting of real semisimple Lie algebras
and their Cartan subalgebras and involutions of the corresponding root
system.
This can be graphically described by introducing 
\emph{$S\!${-} and $\Sigma$-diagrams}, 
generalizing those of Satake and Vogan.
\end{abstract}

\thanks{This work has been financially supported by the Programme “FIL-Quota 
Incentivante” of University of Parma, co-sponsored by Fondazione 
Cariparma, by the PRIN project ``Real and Complex Manifolds: Topology, 
Geometry and holomorphic dynamics'',  and by 
the group G.N.S.A.G.A. of I.N.d.A.M.}

\maketitle 
\tableofcontents

\section*{Introduction}
In his papers \cite{C14, C29} \'E.  Cartan
gave  
a complete classification 
of irreducible Riemannian globally symmetric spaces
and
simple real Lie algebras. In his treatment he employes
maximally compact and maximally
vector Cartan subalgebras. 
\par 

In this paper,  to settle the ground
for application to investigating the $CR$ geometry of 
orbits of real forms in complex flag manifolds (see e.g.  \cite{MMN21, Wolf69}
and references therein), we  
study equivalence classes 
of pairs consisting of 
real semisimple Lie algebras and their Cartan subalgebras.
A study of conjugacy classes of Cartan subalgebras in real semisimple Lie algebras can be found in 
\cite{Kost55,  Sug69}.
\par
In our approach, we reduce, modulo equivalence, to real forms described
by anti-involutions fixing given split and compact forms. This has the
advantage of allowing us to keep fixed a Cartan subalgebra of the complex form
and hence the corresponding root system, very much simplifying the relative combinatorics.
\par 
The real Cartan subalgebras we consider are 
this
fixed complex Cartan subalgebra. 
They correpsond
to 
involutions $\stt$
of its root system $\Rad$; 
each one may be contained in different semisimple real forms, 
as different anti-involutions $\sigmaup$ can induce the same involution $\stt$ on the roots.
Imaginary roots for $\stt$ are  \textit{compact} or \textit{not-compact},
according with the action of $\sigmaup$ on the corresponding $\slt_{2}(\C)$.
\par
Involutions $\stt$ of $\Rad$ can be graphically described by
\textit{$S\!$-diagrams}, which slightly generalize Satake's (see \cite{Ara62, sa60}).
Different real forms corresponding to a same $\stt$ are identified
by \textit{$\Sigma$-diagrams},
that are obtained  by changing colour to imaginary 
not compact roots.  $S\!$-diagrams characterize real Cartan subalgebras,
$\Sigma$-diagram real forms. As Vogan's  (see e.g. \cite[Ch.6,\S{8}]{Kn:2002}),
$\Sigma$-diagrams are not unique. They may be made unique by adding  extra requirements
on their simple positive imaginary roots.
General $\Sigma$-diagrams 
depend on the choice of  suitable Weyl chambers and no wonder that also the different
possible choices of \textit{admissible} Weyl chambers is relevant in the applications to 
$CR$ geometry (see e.g.~\cite{AMN06b}). 
In the spirit of \cite{MMN23}  the unique closed orbit of a real form in a
 complex flag manifold can be described by utilising 
 a maximally non compact Cartan subalgebra ( see \cite{Dietrich2013} for definition and explicit computation)
 contained in the in the Lie algebra of its
 isotropy group: one can use for them 
 Satake diagrams. The isotropy groups of open orbits contain maximally compact Cartan subalgebras
 and Vogan diagrams may be employed for their description. Consideration of intermediate orbits
 require utilising intermediate classes of Cartan subalgebras and more general diagrams.\par
  \medskip
\noindent

\section{Preliminaries}\label{s1}

\subsection{Split and compact forms,  Chevalley systems}\label{s1.1}\quad\par
Let $\gt$ be a complex Lie algebra. Its \emph{real forms} are its real subalgebras $\go$ 
such that
\begin{equation*}
 \gt\,{=}\,\go\oplus{i}\go.
\end{equation*}
Real forms are loci of fixed points of anti-$\C$-linear involutions (see \S\ref{s2}).
\par\smallskip
We will focus on the case where $\gt$ is a complex finite dimensional 
semisimple Lie algebra. \par
Any complex semi\-sim\-ple Lie algebra $\gt$ admits a \emph{split real form}
$\gr$ and a \emph{compact form} $\gu$, 
which are unique modulo conjugation (see \cite[p.296]{Kn:2002}).

 \par \index{$\gr$} \index{$\gu$}
Fix a complex Cartan subalgebra
$\hg$ 
of $\gt$, i.e.  a nilpotent selfnormalizing subalgebra.  The \emph{root system} $\Rad\,{=}\,\Rad(\gt,\hg)$
is the subset of the dual $\hg^{*}$ of $\hg$ 
consisting of the nonzero $\alphaup$'s  for which 
\begin{equation}
 \gt^{\alphaup}=\{Z\in\gt\mid [H,Z]=\alphaup(H)\,Z,\;\forall H\in\hg\}\neq\{0\}.
\end{equation}
\par
We recall (see e.g. \cite[Ch.VIII,\S{2}.{4}]{Bou82})
\begin{thm}\label{thmchev} Let $\gt,\,\hg,\,\Rad$ be as above. 
Then 
\begin{equation}\label{e1.2}
 \hr=\{H\in\hg\mid \alphaup(H)\in\R,\;\forall\alphaup\in\Rad\}
\end{equation}
is a real form of $\hg$.\par\index{$\hr$}
We can find a set 
\begin{equation}\label{e1.3}
\{(X_{\alphaup},H_{\alphaup})\}_{\alphaup\in\Rad}\subset\gt\times\hr
\end{equation}
 with the properties:
\begin{enumerate}
\item $\alphaup(H_{\betaup})\,{\in}\,\Z,$ for all $\alphaup,\betaup\,{\in}\,\Rad\,;$
 \item $X_{\alphaup}\in\gt^{\alphaup}$, i.e. $[H,X_{\alphaup}]\,{=}\,\alphaup(H)X_{\alphaup},$ 
 $\forall H\in\hg\,;$
 \item For every $\alphaup\,{\in}\,\Rad$\,, we have  
\begin{equation*}
 [X_{\alphaup},X_{-\alphaup}]=-H_{\alphaup},\;\; [H_{\alphaup},X_{\alphaup}]=2X_{\alphaup},
 \;\; [H_{\alphaup},X_{-\alphaup}]=-2X_{-\alphaup};
\end{equation*}
\item  The direct sum 
\begin{equation}\label{e1.4}
 \gr=\hr\oplus{\sum}_{\alphaup\in\Rad}\gr^{\alphaup},\;\;\text{where}\;\;
 \gr^{\alphaup}{=}\,\langle X_{\alphaup}\rangle_{\R}
\end{equation}
is a split real form of $\gt\,;$
\item the $\R$-linear self-map of $\gr$ defined by 
\begin{equation*}
 \tauup(H)={-}H, \;\forall H\,{\in}\,\hr,\quad \tauup(X_{\alphaup})=X_{-\alphaup},\;\;\forall\alphaup\in\Rad
\end{equation*}
is an involutive automorphism of $\gr$, defining a compact form of $\gt\! :$
\begin{equation}\label{e1.5}
\vspace{-20pt}
 \gu= \{X\mid X\in\gr, \; \tauup(X)\,{=}\,X\}\oplus\{iX\mid X\in\gr,\;\tauup(X)\,{=}\,{-}X\}.
 \end{equation}
\qed
\end{enumerate}
\end{thm}

\begin{defn}
A set $\{(X_{\alphaup},H_{\alphaup})\}_{\alphaup\in\Rad}$ satisfying Conditions 
$(1), (2), (3), (4), (5)$ of  Theorem\,\ref{thmchev} 
is called a  
\emph{Chevalley system} for $(\gt,\hg)$. 
\end{defn}
For every pair of roots $\alphaup,\betaup$ in $\Rad$, with $\betaup\,{\neq}{\pm}\alphaup$, 
there are nonnegative integers $p_{\alphaup,\betaup},q_{\alphaup,\betaup}$
such that 
\begin{equation*}
 \{j\in\Z\mid \betaup\,{+}j\,\alphaup\in\Rad\}=\{j\in\Z\mid -q_{\alphaup,\betaup}\leq{j}\leq{p}_{\alphaup,\betaup}\}.
\end{equation*}
Having fixed a Chevalley system $\{(X_{\alphaup},H_{\alphaup})\}_{\alphaup\in\Rad}$, define 
coefficients $\{N_{\alphaup,\betaup}\}_{\alphaup,\betaup\in\Rad}$ by 
\begin{equation} \label{e1.6}
\begin{cases}
[X_{\alphaup},X_{\betaup}]=N_{\alphaup,\betaup}X_{\alphaup+\betaup}, & \text{if 
$\alphaup,\betaup,\alphaup{+}\betaup\in\Rad$}\\
N_{\alphaup,\betaup}=0, &\text{if $\alphaup,\betaup\in\Rad,$ $\alphaup{+}\betaup\notin\Rad$.}
\end{cases}
\end{equation}

Then
(see e.g. \cite[Ch.VIII,\S{2.4}]{Bou82})
\begin{equation} \label{e1.5a} \begin{cases}
0\leq p_{\alphaup,\betaup}+q_{\alphaup,\betaup}\leq{3},\\
N_{\alphaup,\betaup}=N_{-\alphaup,-\betaup}={\pm}(q_{\alphaup,\betaup}+1), \\
N_{\alphaup,\betaup}\cdot N_{-\alphaup,\alphaup+\betaup}=-p_{\alphaup,\betaup}(q_{\alphaup,\betaup}+1).
 \end{cases}
\end{equation}
\begin{lem}\label{l1.2}
 For every root $\alphaup\,{\in}\,\Rad$ set 
\begin{equation}
 K_{\alphaup}=X_{\alphaup}+X_{-\alphaup},\quad T_{\alphaup}=X_{\alphaup}-X_{-\alphaup}.
\end{equation}
 Then $\ad(K_{\alphaup})$ 
 is semisimple, with eigenvalues contained in $\{0,\pm{i},\pm{2}i,{\pm}3i\}$. \par 
 We have 
\begin{equation} \label{eq1.7}
\begin{cases}
 \exp(\ttt \ad(K_{\alphaup}))(H_{\alphaup})=H_{\alphaup}\cos(2\ttt)-T_{\alphaup}\sin(2\ttt),\\
 \exp(\ttt \ad(K_{\alphaup}))(T_{\alphaup})=H_{\alphaup}\sin(2\ttt)+T_{\alphaup}\cos(2\ttt).
\end{cases}
\end{equation}
 
\end{lem} 
\begin{proof} Having fixed $\alphaup\,{\in}\,\Rad$,
we decompose $\gr$ into a direct sum of $\ad(K_{\alphaup})$-invariant
 subspaces, consisting of:
\begin{itemize}
\item the subspace $\{H\,{\in}\,\hr\mid \alphaup(H)\,{=}\,0\}$ of $\hr$;
\item the subspace $\langle{K}_{\alphaup}\rangle_{\R}$;
 \item the subspace $\langle{H}_{\alphaup},T_{\alphaup}\rangle_{\R}$;
 \item for each root $\betaup\,{\neq}\,{\pm}\alphaup$ and such that $\betaup{-}\alphaup\,{\notin}\,\Rad$,
 the subspace $$V_{\betaup}{=}\,\langle{X}_{\betaup+j\alphaup}\,{\mid}\,j\in\Z,\;\betaup\,{+}j\alphaup\in\Rad
 \rangle_{\R}\,{=}\,{\sum}_{j=0}^{p_{\alphaup,\betaup}}\gr^{\betaup+j\alphaup}.$$
\end{itemize}
\par
If $\alphaup(H)\,{=}\,0,$ then $[K_{\alphaup},H]\,{=}\,0$. Since $[K_{\alphaup},K_{\alphaup}]\,{=}\,0,$ 
the first two subspaces of the list
are contained in the kernel of $\ad(K_{\alphaup})$.\par
The restriction of $\ad(K_{\alphaup})$ to the 
invariant subspace $\langle{H}_{\alphaup},T_{\alphaup}\rangle_{\R}$
is represented, in the canonical basis, by the matrix 
\begin{equation*} 
\begin{pmatrix}
 0 & 2\\
 -2\; & 0
\end{pmatrix}
\end{equation*}
and therefore is semisimple with eigenvalues ${\pm}2i$. \par
The restriction of the exponential 
$\exp(\ttt\ad(K_{\alphaup}))$ 
is represented by the matrix 
\begin{equation*} 
\begin{pmatrix}
 \cos(2\ttt) & \sin(2\ttt)\\
 -\sin(2\ttt) & \cos(2\ttt)
\end{pmatrix},
\end{equation*}
yielding \eqref{eq1.7}.
\par\smallskip
Clearly the subspaces $V_{\betaup}$ are $\ad(K_{\alphaup})$-invariant. 
The matrix of the restriction of $\ad(K_{\alphaup})$ to $V_{\betaup}$,
in the canonical basis $(X_{\betaup+j\alphaup})_{0\leq{j}\leq{p}_{\alphaup,\betaup}}$ depends on its
dimension, which can be $1,2,3,4$. We have, in the different cases,  
\begin{equation*} \begin{cases}
 (0), & \text{if $\dim(V_{\betaup})=1$},\\[4pt]
\begin{pmatrix}
0 & N_{-\alphaup,\betaup+\alphaup}\\
  N_{\alphaup,\betaup} & 0
\end{pmatrix}, & \text{if $\dim(V_{\betaup})=2$},\\[12pt]
\begin{pmatrix}
 0 & N_{-\alphaup,\betaup+\alphaup} & 0\\
 N_{\alphaup,\betaup} & 0 & N_{-\alphaup,\betaup+2\alphaup}\\
 0 & N_{\alphaup,\betaup+\alphaup} & 0
\end{pmatrix}, & \text{if $\dim(V_{\betaup})=3$},\\[20pt]
\begin{pmatrix}
 0 & N_{-\alphaup,\betaup+\alphaup} & 0 & 0\\
 N_{\alphaup,\betaup} & 0 & N_{-\alphaup,\betaup+2\alphaup} & 0\\
 0 & N_{\alphaup,\betaup+\alphaup} & 0 & N_{-\alphaup,\betaup+3\alphaup}\\
 0 & 0 & N_{\alphaup,\betaup+2\alphaup}& 0
\end{pmatrix}, & \text{if $\dim(V_{\betaup})=4$}.
\end{cases}
\end{equation*}
By \eqref{e1.5a}
the characteristic polynomials $\ptt(\lambdaup)$ are: 
\begin{equation*} 
\ptt(\lambdaup)=\begin{cases}
 \lambdaup, & \text{if $\dim(V_{\betaup})=1$,}\\
 \lambdaup^{2}+1, & \text{if $\dim(V_{\betaup})=2$,}\\
 \lambdaup(\lambdaup^{2}+4),& \text{if $\dim(V_{\betaup})=3$,}\\ 
 (\lambdaup^{2}+1)(\lambdaup^{2}+9), & \text{if $\dim(V_{\betaup})=4$}.
\end{cases}
\end{equation*}
In all cases the eigenvalues are distinct. This yields the statement. 
\end{proof}
\subsection{Weyl chambers}
\label{s1.2.1} An $H$ of $\hr$ is \emph{regular} if $\alphaup(H)\,{\neq}\,0$
for all $\alphaup\,{\in}\,\Rad$. Regular elements form an open subset $\Reg(\hr)$ of $\hr$.
The \emph{Weyl chambers} are the connected components of $\mathpzc{Reg}(\hr)$.
They form a set $\Cd(\Rad)$ on which the dual action of the Weyl group, i.e.  a subgroup of the isometry group of the root system,  is simple and transitive.
If $H_{0}\,{\in}\,C\,{\in}\,\Cd(\Rad),$ then 
\begin{gather}\label{e1.10}
 \Rad=\Rad^{+}(C)\cup\Rad^{-}(C),
\intertext{where}
\label{e1.11}
\begin{cases}
 \Rad^{+}(C)\,{=}\,\{\alphaup\,{\in}\,\Rad\,{\mid}\,\alphaup(H_{0})\,{>}\,0\}\,{=}\,
 \{\alphaup\,{\in}\,\Rad\,{\mid}\,\alphaup(H)\,{>}\,0,\;\forall H\,{\in}\,C\},\\
  \Rad^{-}(C)\,{=}\,\{\alphaup\,{\in}\,\Rad\,{\mid}\,\alphaup(H_{0})\,{<}\,0\}\,{=}\,
 \{\alphaup\,{\in}\,\Rad\,{\mid}\,\alphaup(H)\,{<}\,0,\;\forall H\,{\in}\,C\},
\end{cases}
\end{gather}
are the cones of
\textit{positive and negative roots} for the lexicographic order defined by $C$. 
The set    
\begin{equation*}
 \Bz(C)\,{=}\,\{\alphaup\in\Rad^{+}(C)\mid \alphaup\,{-}\,\betaup\notin\Rad^{+}(C),\;\forall\betaup\,{\in}\,\Rad^{+}(C)\}
\end{equation*}
of minimal elements of $\Rad^{+}(C)$
is the \emph{base of simple  roots} in $\Rad^+(C)$. 
Each $\alphaup$ in $\Rad$ is a linear combination 
\begin{equation}\label{e1.15}
 \alphaup={\sum}_{\betaup\in\Bz(C)}k_{\alphaup,\betaup}\betaup,\;\;\text{with $k_{\alphaup,\betaup}\,{\in}\,\Z$,}\;
\begin{cases}
 k_{\alphaup,\betaup}{\geq}0\;\;\forall \betaup, &\text{if $\alphaup\,{\in}\,\Rad^{+}(C),$}\\
  k_{\alphaup,\betaup}{\leq}0\;\;\forall \betaup, &\text{if $\alphaup\,{\in}\,\Rad^{-}(C).$}
\end{cases}
\end{equation}
The choice of $C$ defines a 
$\Z$-grading $\kq_{\; C}$ \index{$\kq_{\; C}$} 
and a notion of $C$-\emph{support} 
on $\Z[\Rad]$: \index{$\supp_{C}$}

\begin{equation} \begin{cases}
 \kq_{\;C}(\alphaup)=
{\sum}_{\betaup\in\Bz(C)}k_{\alphaup,\betaup},\\[5pt]
 \supp_{C}(\alphaup)=\{\betaup\,{\in}\,\Bz(C)\,{\mid}\,k_{\alphaup,\betaup}\,{\neq}\,0\}.
 \end{cases}
\end{equation}
where  $ \alphaup={\sum}_{\betaup\in\Bz(C)}k_{\alphaup,\betaup}\betaup.$

\section{Involutions of $\Rad$}
\subsection{Involutions
}\label{s1.4} 
The restriction to $\hr$ of 
the Killing form 
of $\gt$ is a real valued 
positive definite scalar product, 
that will be denoted by
$(\;\cdot\;|\;\cdot\;)$. 
The \textit{duality map}
\begin{equation}
 \hr^{*}\ni\xiup\to{\htt}_{\xiup}\in\hr,\;\text{is characterized by}\;\, (H\mid\htt_{\xiup})\,{=}\,
 \xiup(H),\;\forall
 H\in\hr,
\end{equation}
and  the 
\textit{dual scalar product} on $\hr^{*}$ is 
\begin{equation}(\xiup\,|\,\etaup)= (\htt_{\xiup}\,|\, \htt_{\etaup}),\;\;\forall\,\xiup,\etaup\in\hr^{*}.
\end{equation} 
Since $\alphaup(H_{\alphaup})\,{=}\, 2$, we have 
\begin{equation}
 \htt_{\alphaup}=\dfrac{ 2H_{\alphaup}}{\;\,\|H_{\alphaup}\|^{2}}\,, \quad
 H_{\alphaup}=\dfrac{2\htt_{\alphaup}}{\; \,\|\htt_{\alphaup}\|^{2}}\,.
\end{equation}
 
We also use the notation 
\begin{equation}
 \langle\xiup \mid \etaup\rangle = \frac{2(\xiup \mid \etaup)}{\;(\etaup \mid \etaup)}
 = \frac{2(\xiup\mid \etaup)}{\;\|\etaup \|^{2}},\;
 \text{if $\xiup,\etaup\in\hr^{*},\; \etaup\neq{0}$.}
\end{equation}
\par
For $\alphaup\,{\in}\Rad$, the reflection 
\begin{equation*}
 s_{\alphaup}(\xiup)= \xiup-\langle\xiup\mid\alphaup\rangle\,\alphaup,\;\; 
 \xiup\in\hr^{*},
\end{equation*}
maps $\Rad$ onto itself. The  
reflections $\sq_{\alphaup}$, for $\alphaup\,{\in}\,\Rad$,
generate the
\textit{Weyl group}
$\Wf(\Rad)$ of $\Rad$. \index{$\Wf(\Rad)$}
It is a normal subgroup of the \textit{full automorphism group} 
$\Af(\Rad)$, consisting of all orthogonal transformations of $\hr^{*}$ 
leaving $\Rad$ invariant.
\par
\smallskip
We will denote by  
$\Ib(\Rad)$ the involutions in $\Af(\Rad)$ and by $\Ib_{\Wf}(\Rad)$ \index{$\Af(\Rad)$}
those belonging to $\Wf(\Rad)$. \index{$\Ib(\Rad)$} \index{$\Ib_{\Wf}(\Rad)$}
\par
An involution $\stt\,{\in}\,\Ib(\Rad)$ defines by duality an
involution $\stt^{*}$ on $\hr$, yielding a decomposition
\index{$\hst^{+}$}\index{$\hst^{-}$}
\begin{equation} \label{e2.7}
\begin{cases}
\hr\,{=}\,\hst^{+}\,{\oplus}\,\hst^{-},\;\;\text{with}\, \hst^{\pm}\,{=}\,\{H\,{\in}\,\hr\,{\mid}\,\stt^{*}(H)\,{=}\,{\pm}H\},\\
	H=H^{\stt}_{+}+H^{\stt}_{-},\;\;\text{with}\;\; H^{\stt}_{\pm}=\tfrac{1}{2}(H\pm\stt^{*}(H))\,{\in}\,\hst^{\pm}.
	\end{cases}
	\end{equation}
	\index{$H^{\stt}_{+}$} \index{$H^{\stt}_{-}$}
	\par
	
	We associate to $\stt$ the real form
\begin{equation}
	\hst=\hst^{+}\oplus {i}\hst^{-}
\end{equation}
	of $\hg$ and the
	$\stt$-invariant partition of $\Rad$
\begin{equation} \label{e1.13}
\Rad=\Rad_{\,\;\circ}^{\stt}\cup \Rad_{\,\;\bullet}^{\stt}\cup\Rad_{\,\;\star}^{\stt},\;\;
\text{with}\;\;
\begin{cases}
 \Rad_{\,\;\circ}^{\stt}=\{\alphaup\in\Rad\mid \stt(\alphaup)=\alphaup\},\\
 \Rad_{\,\;\bullet}^{\stt}=\{\alphaup\in\Rad\mid \stt(\alphaup)={-}\alphaup\},\\
  \Rad_{\,\;\star}^{\stt}=\{\alphaup\in\Rad\mid \stt(\alphaup)\neq{\pm}\alphaup\}.
  \end{cases}
\end{equation}
We have
\begin{equation*}
	\stt(\alphaup)(H)=\alphaup(\stt^{*}(H))=\alphaup(H_{+}^{\stt})-\alphaup(H_{-}^{\stt}),\;\;\forall\alphaup\,{\in}\,\Rad,\;
	\forall H\in\hr
	\end{equation*}
and therefore the roots of $\Rad^{\stt}_{\,\;\circ}$ (resp. $\Rad_{\,\;\bullet}^{\stt}$)
are those taking real (resp. purely imaginary) values on $\hst$.
For this reason we call the roots of $\Rad^{\stt}_{\,\;\circ}$, $\Rad^{\stt}_{\,\;\bullet}$, $\Rad^{\stt}_{\,\;\star}$
\textit{real}, \textit{imaginary} and \textit{complex} for $\stt$, respectively. \par
	If $H$ is any regular element of $\hr$, then
	\begin{equation*}
	\Rd{\circ}{\stt}\,{=}\,\{\alphaup\,{\mid}\,\alphaup(H^{\stt}_{-})\,{=}\,0\},\;\;
	\Rd{\bullet}{\stt}\,{=}\,\{\alphaup\,{\mid}\,\alphaup(H^{\stt}_{+})\,{=}\,0\},\;\;
	\Rd{\star}{\stt}\,{=}\,\{\alphaup\,{\mid}\,\alphaup(H^{\stt}_{+})\,{\cdot}\,\alphaup(H^{\stt}_{-})\,{\neq}\,0\}.
	\end{equation*}
\begin{rmk} 
 Real and  imaginary roots form  
 closed subsystems.
 Indeed 
 all roots which are contained in the linear span of $\Rad_{\,\;\circ}^{\stt}$
(resp. $\Rad_{\,\;\bullet}^{\stt}$) belong to $\Rad_{\,\;\circ}^{\stt}$
(resp. $\Rad_{\,\;\bullet}^{\stt}$)
and therefore $\Rad_{\,\;\circ}^{\stt}$ and $\Rad_{\,\;\bullet}^{\stt}$
are root systems
 (see e.g. \cite[Ch.VI,\S{1},Prop.4]{Bou68}).\par
 Since an involution is the restriction to $\Rad$ of an isometry of $\hr^{*}$, we have
\begin{lem} If $\stt$ is an involution of $\Rad$, then 
\begin{equation*}\vspace{-20pt}
 \Rd{\bullet}{\stt}\perp{\Rd{\circ}{\stt}}.
\end{equation*}\qed
\end{lem}
\end{rmk} 
\begin{dfn} 
The \emph{antipodal involution} of $\hr^{*}$ is the map
\begin{equation} \att:\hr^{*}\ni\xiup\longrightarrow
{-}\,\xiup\,{\in}\,\hr^{*}.
\end{equation}
 \index{$\att$}
\end{dfn}
A root system $\Rad$ is  \emph{irreducible} if it is not the union of two proper  orthogonal subsets.
\begin{lem} If $\Rad$ is irreducible,  then 
the center of $\Af(\Rad)$  is $\{\id,\att\}$.
\end{lem} 
\begin{proof} If $\ftt$ belongs to the center of $\Af(\Rad)$, 
 then $\ftt{\circ}\sq_{\alphaup}\,{=}\,\sq_{\alphaup}{\circ}\ftt$
 for every root $\alphaup$. 
 In particular, 
\begin{align*}
 {-}\ftt(\alphaup)=\ftt{\circ}\sq_{\alphaup}(\alphaup)=\sq_{\alphaup}\circ\ftt(\alphaup)=
 \ftt(\alphaup)-\langle\ftt(\alphaup)|\alphaup\rangle\alphaup \,\Rightarrow\,
 \ftt(\alpha)=\tfrac{1}{2}\langle\ftt(\alphaup)|\alphaup\rangle\alphaup.
\end{align*}
Since $\ftt$ is an isometry, we have $\tfrac{1}{2}\langle\ftt(\alphaup)|\alphaup\rangle\,{=}\,{\pm}1$
and, by the assumption that $\Rad$ is irreducible, this coefficients is always equal to either $1$
or  $({-}1)$. 
\end{proof}
\par
An involution $\stt\,{\in}\,\Ib(\Rad)$ defines by duality an 
involution $\stt^{*}$ on $\hr$, yielding a decomposition 
\index{$\hst^{+}$}\index{$\hst^{-}$}
\begin{equation} \label{e2.7}
\begin{cases}
\hr\,{=}\,\hst^{+}\,{\oplus}\,\hst^{-},\;\;\text{with}\, \hst^{\pm}\,{=}\,\{H\,{\in}\,\hr\,{\mid}\,\stt^{*}(H)\,{=}\,{\pm}H\},\\
H=H^{\stt}_{+}+H^{\stt}_{-},\;\;\text{with}\;\; H^{\stt}_{\pm}=\tfrac{1}{2}(H\pm\stt^{*}(H))\,{\in}\,\hst^{\pm}.
\end{cases}
\end{equation}
\index{$H^{\stt}_{+}$} \index{$H^{\stt}_{-}$} 
\par 
We have 
\begin{equation*}
 \stt(\alphaup)(H)=\alphaup(\stt^{*}(H))=\alphaup(H_{+}^{\stt})-\alphaup(H_{-}^{\stt}),\;\;\forall\alphaup\,{\in}\,\Rad,\;
 \forall H\in\hr.
\end{equation*}
If $H$ is a regular element of $\hr$, then
for every $H\,{\in}\,\Reg(\hr)$, we have 
\begin{equation*}
 \Rd{\circ}{\stt}\,{=}\,\{\alphaup\,{\mid}\,\alphaup(H^{\stt}_{-})\,{=}\,0\},\;\;
  \Rd{\bullet}{\stt}\,{=}\,\{\alphaup\,{\mid}\,\alphaup(H^{\stt}_{+})\,{=}\,0\},\;\;
   \Rd{\star}{\stt}\,{=}\,\{\alphaup\,{\mid}\,\alphaup(H^{\stt}_{+})\,{\cdot}\,\alphaup(H^{\stt}_{-})\,{\neq}\,0\}.
\end{equation*}
%
\subsection{Orthogonal and strongly orthogonal roots} Let $\Rad$ be a root system.
Two roots $\alphaup,\betaup$ are said to be
\emph{strongly orthogonal} when
neither $\alphaup\,{+}\,\betaup$ nor $\alphaup\,{-}\,\betaup$ is a root. Strongly orthogonal
roots are orthogonal, and also the vice versa is true when 
one of them is long, i.e. has maximal length within the irreducible component 
to which it belongs.  
 If $\Rad$ contains roots of different lengths and two short roots
$\alphaup,\betaup$ are orthogonal, but not strongly orthogonal,  
then  ${\pm}\alphaup{\pm}\betaup$ are long roots. 
\begin{lem}\label{l2.4}
If $\alphaup_{1},\hdots,\alphaup_{r}$ is an orthogonal set of roots, then there is a
strongly orthogonal
set of  roots $\betaup_{1},\hdots,\betaup_{r}$  with 
\begin{equation*}\tag{$*$}
 \langle\alphaup_{1},\hdots,\alphaup_{r}\rangle_{\R}=\langle\betaup_{1},\hdots,\betaup_{r}\rangle_{\R}.
\end{equation*} 
\end{lem}
\begin{proof} 
We
argue by recurrence on $r$. 
When $r\,{=}\,1,$ there is noting to prove. If $r\,{=}\,2$ and $\alphaup_{1},\alphaup_{2}$ are not
strongly orthogonal, then $\betaup_{1}{=}\,\alphaup_{1}{+}\,\alphaup_{2}$ and
$\betaup_{2}{=}\,\alphaup_{1}{-}\,\alphaup_{2}$ are strongly orthogonal roots and
$(*)$ holds. \par
If $r{>}2$ and we assume that the statement is true
for a lesser number of orthogonal roots, we first 
reduce to the case where $\alphaup_{2},\hdots,\alphaup_{r}$
are strongly orthogonal. If $\alphaup_{1},\alphaup_{2},\hdots,\alphaup_{r}$ are strongly
orthogonal, then 
there is nothing to prove. Otherwise, by reordering, we can assume that $\alphaup_{1}$ and
$\alphaup_{2}$ are not strongly orthogonal. Then  $\betaup_{1}{=}\,\alphaup_{1}{+}\,\alphaup_{2}$
and  $\betaup_{2}{=}\,\alphaup_{1}{-}\,\alphaup_{2}$ are roots 
forming with $\alphaup_{3},\hdots,\alphaup_{r}$ a 
strongly orthogonal system
because, being long, they are strongly orthogonal to their orthogonal roots. 
\end{proof}
\begin{ntz}\label{perp}
 We write $\alphaup\Perp\betaup$
 to indicate that the roots $\alpha$ and $\beta$ are \textit{strongly}
 orthogonal.

\end{ntz}
\subsection{Decompositions} \label{s2.3}
\begin{dfn} \label{d2.2} We call \emph{special} an involution $\epi$ of $\Rad$ for which
$\Rad^{\epi}_{\,\;\bullet}\,{=}\,\emptyset$ and
denote by $\Ib^{*}(\Rad)$ the set of special involutions of $\Rad$.
\index{$\Ib^{\star}(\Rad)$} \index{$\Ib^{\star}(\Rad)$}
\end{dfn}
\begin{lem}
If $\stt\,{\in}\,\Ib(\Rad)$ and $\betaup\,{\in}\,\Rad$, then $\stt$ and $\sq_{\,\betaup}$ commute if and
only if $\betaup\,{\in}\,\Rd{\circ}{\stt}\,{\cup}\,\Rd{\bullet}{\stt}$.  
\begin{proof}
If $\stt$ and $\sq_{\,\betaup}$ commute, we have 
\begin{align*} {-}\stt(\betaup)= \stt\circ\sq_{\,\betaup}(\betaup)=
 \sq_{\,\betaup}\circ\stt(\betaup)=\stt(\betaup)-\langle\stt(\betaup)|\betaup\rangle\betaup\,\Rightarrow
\stt( \betaup)\,{=}\,\tfrac{1}{2}\langle\stt(\betaup)|\betaup\rangle\betaup
\end{align*}
and this implies that $\stt(\betaup)\,{=}\,{\pm}\betaup$. If vice versa $\stt(\betaup)\,{=}\,\lambdaup\betaup$,
with $\lambdaup\,{=}\,{\pm}1$, we obtain, for all $\xiup\,{\in}\,\hr^{*}$, 
\begin{equation*} 
 \stt\circ\sq_{\,\betaup}(\xiup)=\stt(\xiup)-\lambdaup \langle\xiup|\betaup\rangle\betaup=
 \stt(\xiup)-\langle\xiup|\stt(\betaup)\rangle\betaup=\stt(\xiup)-\langle\stt(\xiup)|\betaup\rangle\betaup=
 \sq_{\,\betaup}\circ\stt(\xiup).
\end{equation*}
\end{proof}
\end{lem}
\begin{prop}\label{p2.7}

 For every  
$\stt\,{\in}\,\Ib(\Rad)$ we can find  an $\epi\,{\in}\,\Ib^{*}(\Rad)$ and a system $(\betaup_{1},\hdots,\betaup_{r})$ 
of strongly orthogonal roots in~$\Rd{\circ}{\epi}\,{\cap}\,\Rd{\bullet}{\stt}$
such that 
\begin{equation}\label{e2.8}
 \stt=\varepsilon\circ\sq_{\,\betaup_{1}}\circ\cdots\circ\sq_{\,\betaup_{r}}.
\end{equation}
\end{prop}
\begin{proof} If $\stt\,{\in}\,\Ib^{*}(\Rad)$, then there is nothing to prove. \par 
If $\stt\,{\notin}\,\Ib^{*}(\Rad)$, then we take a maximal system $\betaup_{1},\hdots,\betaup_{r}$ 
of orthogonal roots in $\Rd{\bullet}{\stt}$, which, by Lemma\,\ref{l2.4} we can take strongly orthogonal,  
and define 
\begin{equation*}\tag{$*$}
 \epi=\stt\circ\sq_{\,\betaup_{1}}\circ\cdots\circ\sq_{\,\betaup_{r}}.
\end{equation*}
We claim that $\Rd{\bullet}{\epi}\,{=}\,\emptyset$. Assume indeed, by contradiction, that there is
a root $\alphaup$ in $\Rd{\bullet}{\epi}$. Then 
\begin{equation*}
 {-}\alphaup=\epi(\alphaup)=\stt(\alphaup){+}{\sum}_{i=1}^{r}\langle\alphaup|\betaup_{i}\rangle\betaup_{i}\;\;\;
 \text{and}\;\; \; \alphaup{+}\stt(\alphaup)={-}{\sum}_{i=1}^{r}\langle\alphaup|\betaup_{i}\rangle\betaup_{i}
\end{equation*}
implies that $\stt(\alphaup)\,{=}\,{-}\alphaup$ and 
$\langle\alphaup|\betaup_{i}\rangle\,{=}\,0$ for all $i\,{=}\,1,\hdots,r$, because 
the left and right side of the last equality are eigenvectors of $\stt$, corresponding to the eigenvalues
$1$ and $({-}1)$, respectively.
Thus $\alphaup$ is a root of $\Rd{\bullet}{\stt}$ orthogonal to $\betaup_{1},\hdots,\betaup_{r},$
contradicting maximality.\par
Therefore $\epi\,{\in}\,\Ib^{*}(\Rad)$ and $(*)$ implies \eqref{e2.8}. Clearly
$\betaup_{1},\hdots,\betaup_{r}$ are in $\Rd{\circ}{\epi}$. The proof is complete. 
\end{proof}
\begin{rmk}
 The number $r$ 
 of reflections 
 in the decomposition \eqref{e2.8} equals the maximum number of pairwise orthogonal roots
 in $\Rad_{\,\;\bullet}^{\stt}$ and, in particular, 
 is an invariant of $\stt$, which is called its \emph{lenght} and will be indicated by 
 $r(\stt)$. \index{$r(\stt)$}
\end{rmk}
\subsection{Sets of strongly orthogonal roots}\label{s2.4}
\quad\par
We showed in \S\ref{s2.3}
that involutions of a root system 
are associated to subsets of strongly orthogonal roots. In this subsection we 
classify, modulo equivalence, systems of strongly
orthogonal roots. We can restrain our consideration to irreducible root systems.\par
Maximal systems of strongly orthogonal roots were classified in~\cite{AK2002}. 
When all roots have the same length, or in a system of type $\textsc{G}_{2}$, 
all maximal systems are  equivalent; orthogonal short roots of $\textsc{B}_{\ell}$ and $\textsc{F}_{4}$
are not strongly  
orthogonal and thus 
systems of strongly orthogonal roots
in $\textsc{B}_{\ell}$ and $\textsc{F}_{4}$  
may contain at most
one short root: in $\textsc{B}_{2n}$ and $\textsc{F}_{4}$ 
there are two classes, one containing and the other one not containing
a short root, while 
a maximal system of $\textsc{B}_{2n{+}1}$ always contains a short root and also
in this case we have a unique class.
In a system of type $\textsl{C}_{\ell}$ there 
are systems of up to $[\ell/2]$  strongly orthogonal short roots and each can be completed to a maximal
one: the equivalence classes are indexed by the number
of short roots they contain. The number of equivalence classes of maximal systems of strongly orthogonal
roots in irreducible root systems 
is given in the table: \vspace{5pt}
\begin{equation*} 
\begin{array}{| c | c | c |  c | c | c | c | c | c | c | c|}
\hline
 &\textsc{A}_{\ell} & \textsc{B}_{2n} & \textsc{B}_{2n{+}1} & \textsc{C}_{\ell} & \textsc{D}_{\ell} & \textsc{E}_{6} 
 & \textsc{E}_{7} 
 &\textsc{E}_{8} & \textsc{F}_{4} & \textsc{G}_{2} \\
 \hline
 \#classes&1 & 2 & 1 & [\ell/2] & 1 & 1 & 1 & 1 & 2 & 1\\
 \hline 
\#elements &[(\ell{+}1)/2]& 2n  & 2n{+}1 & \ell &  [\ell/2] & 4& 7 & 8 & 4,3 & 2\\
 \hline 
\end{array}
\end{equation*}\vspace{5pt}
\par 
 As usual, while explicitly describing below root systems, we will indicate by 
$\e_{1},\hdots,\e_{n}$ orthonormal vectors in suitable euclidean spaces. 
\subsubsection*{Strongly orthogonal roots in $\textsc{A}_{\ell}$} We take 
\begin{equation}\label{rA}
 \Rad(\textsc{A}_{\ell})\,{=}\,\{{\pm}(\e_{i}{-}\e_{j})\mid 1{\leq}i{<}j{\leq}\ell{+}1\}.
\end{equation}
Sets of  orthogonal roots contain at most $[(\ell{+}1)/2]$ elements and, 
for each positive integer $h$, with $2h{\leq}\ell{+}1$, all sets of $h$ strongly orthogonal roots  are equivalent to 
\begin{equation*}
 \texttt{S}_{h}(\textsc{A}_{\ell})\,{=}\,\{\e_{2i-1}{-}\e_{2i}\,{\mid}\,1{\leq}i{\leq}h\}.
\end{equation*}
The maximal ones are equivalent to $\Stt_{[(\ell+1)/2]}(\textsc{A}_{\ell})$.
\subsubsection*{Strongly orthogonal roots in $\textsc{B}_{\ell}$}
We take for $\textsc{B}_{\ell}$ the root system 
\begin{equation}\label{rB}
 \Rad(\textsc{B}_{\ell})=\{{\pm}\e_{i}\mid 1{\leq}i{\leq}\ell\}\cup\{{\pm}\e_{i}{\pm}\e_{j}\mid 1{\leq}i{<}j{\leq}\ell\}.
\end{equation}
Systems of strongly orthogonal roots of $\Rad(\textsc{B}_{\ell})$ 
are characterized by couples of nonnegative integers $r_{1},r_{2}$
with  $2r_{1}{+}r_{2}\,{\leq}\,\ell$
and are equivalent~to 
\begin{equation*}
 L_{r_{1},r_{2}}(\textsc{B}_{\ell})\,{=}
\begin{cases}
 \{\e_{2i-1}{-}\e_{2i}\,{\mid}\,1{\leq}i{\leq}r_{1}\}\,{\cup}\,\{\e_{\ell{-}2i}{\pm}\e_{\ell{-}2i{+}1}\,{\mid}\,
 \begin{smallmatrix}2{\leq}2i{<}r_{2}\end{smallmatrix}\}\,{\cup}\,\{\e_{\ell}\}, \;
 \text{if $r_{2}$ is odd,}\\
 \{\e_{2i-1}{-}\e_{2i}\,{\mid}\,1{\leq}i{\leq}r_{1}\}\cup\{\e_{\ell{-}2i{+}1}{\pm}\e_{\ell{-}2i{+}2}\,{\mid}\,
\begin{smallmatrix} 2{\leq}2i{\leq}r_{2}\end{smallmatrix}\}, \quad
 \text{if $r_{2}$ is even}.
\end{cases}
\end{equation*}
We have $\#(L_{r_{1},r_{2}}(\textsc{B}_{\ell}))\,{=}\,r_{1}{+}r_{2}$, with $2r_{1}{+}r_{2}\,{\leq}\,\ell$, and
those containing a short root have $r_{2}$ odd, while
those with $r_{2}$ even consist of long roots.
\begin{rmk}
	The intersection of $\Rad$ with the linear subspace orthogonal of $ L_{r_{1},r_{2}}(\textsc{B}_{\ell})$
	in $\hr^{*}$
	is a root system of type
$	
\underbrace{\textsc{A}_{1}\,{\times}\,\cdots\,{\times}\,\textsc{A}_{1}}_{r_{1}\;\text{times}}\,\times\;
	\textsc{B}_{\ell-r_{1}-r_{2}} .$
This clarifies the invariant meaning of the pair $(r_{1},r_{2})$. Analogoys consideration hold for the
pairs $(r_{1},r_{2})$ that will be introduced for types $\textsc{C}_{\ell}$, $\textsc{D}_{\ell}$.
\end{rmk}

\subsubsection*{Strongly orthogonal roots in $\textsc{C}_{\ell}$}
We take for $\textsc{C}_{\ell}$ the root system 
\begin{equation}\label{rC}
 \Rad(\textsc{C}_{\ell})=\{{\pm}2\e_{i}\mid 1{\leq}i{\leq}\ell\}\cup\{{\pm}\e_{i}{\pm}\e_{j}\mid 1{\leq}i{<}j{\leq}\ell\}.
\end{equation}
The equivalence classes of sets of
strongly orthogonal roots are indexed by two
integers $r_{1},r_{2}$ with $2r_{1}{+}r_{2}\,{\leq}\,\ell$, and are equivalent to 
\begin{align}
\texttt{S}_{r_{1},r_{2}}(\textsc{C}_{\ell}){=}\{\e_{2i-1}{-}\e_{2i}\,{\mid}\,1{\leq}i{\leq}r_{1}\}\,{\cup}\,
\{2\e_{\ell-i+1}\,{\mid}\,1{\leq}i{\leq}r_{2}\},\qquad\\
\notag \text{with}\; \#\texttt{S}_{r_{1},r_{2}}\,{=}\,r_{1}{+}r_{2}.
\end{align}

The maximal ones have $2r_{1}{+}r_{2}\,{=}\,\ell$.
\subsubsection*{Strongly orthogonal roots in $\textsc{D}_{\ell}$}
We take for $\textsc{D}_{\ell}$ the root system 
\begin{equation}\label{rD}
 \Rad(\textsc{D}_{\ell})=\{{\pm}\e_{i}{\pm}\e_{j}\mid 1{\leq}i{<}j{\leq}\ell\}.
\end{equation}
The equivalence classes of sets of 
 strongly orthogonal roots are indexed by two
integers $r_{1},r_{2}$ with $2(r_{1}{+}r_{2})\,{\leq}\,\ell$, and are equivalent to 
\begin{equation*}
\texttt{S}_{r_{1},r_{2}}{=}\{\e_{2i-1}{-}\e_{2i}\mid 1{\leq}i{\leq}r_{1}\}\cup
\{\e_{\ell-2i+2}{\pm}\e_{\ell-2i+1}\mid 1{\leq}i{\leq}r_{2}\},\;\;\text{with}\;\; \#\texttt{S}_{r_{1},r_{2}}\,{=}\,r_{1}{+}2r_{2}.
\end{equation*}
The maximal one has $r_{1}\,{=}\,0$, $r_{2}\,{=}\,[\ell/2]$.
\subsubsection*{Root systems of type $\textsc{E}$}
 In  the following three  subsections
we will consider subsets of orthogonal roots in 
exceptional root systems of type $\textsc{E}$. \par
\begin{ntz}
We found convenient to introduce the notation 
\begin{equation} \label{zE}
\begin{cases}
 \zetaup_{\emptyset}{=}{-}\tfrac{1}{2}(\e_{1}{+}\cdots{+}\e_{8}),\\
 \begin{aligned}
 & \zetaup_{\, i_{1},\hdots,i_{k}}\,{=}\,\tfrac{1}{2}(\e_{i_{1}}{+}\cdots{+}\e_{i_{k}}{-}\e_{i_{k+1}}{-}\cdots{-}\e_{i_{8}}),
  \\
 &\quad \text{if $1{\leq}k{\leq}8$ and 
$i\,{\in}\,\Sb_{8}$ is any permutation of $\{1,\hdots,8\}$.}
  \end{aligned} 
\end{cases}
\end{equation}
\end{ntz}
\par
We represent a root system of type $\textsc{E}_{8}$ by 
\begin{equation} \label{re8}
 \Rad(\textsc{E}_{8}){=}\{{\pm}\zetaup_{\emptyset}\}{\cup}\{{\pm}\e_{i}{\pm}\e_{j},\,
 {\pm}\zetaup_{i,j}\,{\mid}\,\begin{smallmatrix}
  1{\leq}i{<}j{\leq}8\end{smallmatrix}\}{\cup}\{\zetaup_{i_{1},i_{2},i_{3},i_{4}}\,{\mid}\,
 \begin{smallmatrix} 1{\leq}i_{1}{<}i_{2}{<}i_{3}{<}i_{4}{\leq}8\end{smallmatrix}\}
\end{equation}
an take as a standard basis of positive simple roots 
\begin{equation}\label{bE8}
 \Bz(\textsc{E}_{8})=\left\{\!\!\!\!\!\begin{array}{l l l l}
 & \alphaup_{1}{=}\zetaup_{\emptyset},\;&\alphaup_{2}{=}\e_{1}{+}\e_{2},\;
 \alphaup_{3}{=}\e_{2}{-}\e_{1},\; 
  \alphaup_{4}{=}\e_{3}{-}\e_{2},\\
 &  \alphaup_{5}{=}\e_{4}{-}\e_{3},\; &\alphaup_{6}{=}\e_{5}{-}\e_{4},\; 
  \alphaup_{7}{=}\e_{6}{-}\e_{5},\; \alphaup_{8}{=}\e_{7}{-}\e_{6}\,
  \end{array}\right \},
\end{equation}
corresponding to the
Dynkin diagram
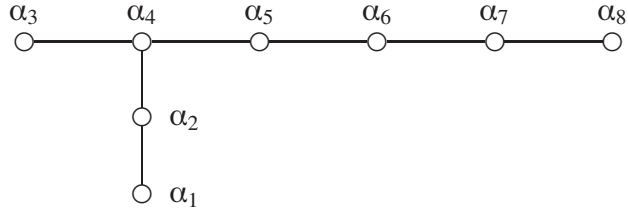
\begin{equation*}
  \xymatrix @M=0pt @R=4pt @!C=20pt{
  \alphaup_{3}&\alphaup_{4}&\alphaup_{5}&\alphaup_{6}&\alphaup_{7}&\alphaup_{8}\\
  \medcirc \ar@{-}[r] 
  &\medcirc\ar@{-}[r]\ar@{-}[dddd]
  &\medcirc\ar@{-}[r]
  &\medcirc\ar@{-}[r]
  &\medcirc\ar@{-}[r]
  &\medcirc \\
  \\
  \\
  \\
  &\medcirc\ar@{-}[dddd] &\!\!\!\!\!\!\!\!\!\!\!\!\!\!\!\!\!\!\!\!\!\!\!\!\!\!\!\!\alphaup_{2}
  \\
  \\
  \\
  \\
  &\medcirc&\!\!\!\!\!\!\!\!\!\!\!\!\!\!\!\!\!\!\!\!\!\!\!\!\!\!\!\!\alphaup_{1}
  }
\end{equation*}
\par\medskip
We fix  
 the maximal system of  strongly  orthogonal roots
\begin{equation}
 \Mtt'(\textsc{E}_{8})=\left\{ 
\begin{aligned}
& \betaup_{1}{=}\zetaup_{1,3,5,7},&\betaup_{2}{=}\zetaup_{1,2,7,8},&&\betaup_{3}{=}\zetaup_{1,3,6,8},
&&\betaup_{4}{=}\zetaup_{1,2,5,6},\\
&\betaup_{5}{=}\zetaup_{1,4,5,8},&\betaup_{6}{=}\zetaup_{1,2,3,4}, && \betaup_{7}{=}\zetaup_{1,4,6,7}, &&
\betaup_{8}{=}{-}\zetaup_{\emptyset}
\end{aligned}\right\}
\end{equation}
and  note that the involution $\texttt{c}\,{=}\,\sq_{\,\zetaup_{1,4,6,8}}{\circ}\sq_{\,\zetaup_{7,8}}{\circ}\sq_{\,\betaup_{4}}{\circ}
\sq_{\,\betaup_{6}}$ transforms $\Mtt'(\textsc{E}_{8})$ into  
\begin{equation}
  \Mtt(\textsc{E}_{8})=\left\{ 
\begin{aligned}
& \gammaup_{1}{=}\e_{1}{-}\e_{2},&\gammaup_{2}{=}\e_{1}{+}\e_{2},&&\gammaup_{3}{=}\e_{3}{-}\e_{4},
&&\gammaup_{4}{=}\e_{3}{+}\e_{4},\\
&\gammaup_{5}{=}\e_{5}{-}\e_{6},&\gammaup_{6}{=}\e_{5}{+}\e_{6}, && \gammaup_{7}{=}\e_{7}{-}\e_{8}, &&
\gammaup_{8}{=}\e_{7}{+}\e_{8}.
\end{aligned}\right\}
\end{equation}
[We set $\gammaup_{i}{=}\texttt{c}(\betaup_{i}),$ for $i{=}1,\hdots,8$.] \par\smallskip
\begin{rmk} The root system $\Rad(\textsc{E}_{8})$ contains the root system $\Rad(\textsc{D}_{8})$.
Its complement $\Theta$ consists of  roots $\zetaup_{I}$ with $I\,{\subset}\,\{1,\hdots,8\}$
having an even number of indices.  It is useful to recall their orthogonality relations: 
\begin{equation*} 
\begin{cases}
 \zetaup_{\emptyset}^{\perp}\,{\cap}\,\Theta=\{\zetaup_{i_{1},i_{2},i_{3},i_{4}}\,{\mid}\, 
\begin{smallmatrix}
 1{\leq}i_{1}{<}i_{2}{<}i_{3}{<}i_{4}{\leq}8
\end{smallmatrix}\!\},\\[15pt]
\zetaup_{j_{1},j_{2}}^{\perp}\,{\cap}\,\Theta=
\left\{{\pm}\zetaup_{i_{1},i_{2}}\left| 
\begin{gathered}
 \begin{smallmatrix}
 1{\leq}i_{1}{<}i_{2}{\leq}8,
\end{smallmatrix}\\[-6pt]
\begin{smallmatrix}
 \{i_{1},i_{2}\}\cap\{j_{1},j_{2}\}=\emptyset
\end{smallmatrix}
\end{gathered}\!\right\}\right.\cup
\left\{\zetaup_{i_{1},i_{2},i_{3},i_{4}}\left| 
\begin{gathered}
 \begin{smallmatrix}
 1{\leq}i_{1}{<}i_{2}{<}i_{3}{<}i_{4}{\leq}8,
\end{smallmatrix}\\[-6pt]
\begin{smallmatrix}
 \#(\{i_{1},i_{2},i_{3},i_{4}\}\cap\{j_{1},j_{2}\})=1
\end{smallmatrix}
\end{gathered}\!\right\}\right. ,\\[15pt] 
\begin{aligned}
\zetaup_{j_{1},j_{2},j_{3},j_{4}}^{\perp}\,{\cap}\,\Theta=\{{\pm}\zetaup_{\emptyset}\}\,{\cup}\,
\left\{{\pm}\zetaup_{i_{1},i_{2}}\left| 
\begin{gathered}
 \begin{smallmatrix}
 1{\leq}i_{1}{<}i_{2}{\leq}8,
\end{smallmatrix}\\[-6pt]
\begin{smallmatrix}
 \#(\{i_{1},i_{2}\}\cap\{j_{1},j_{2},j_{3},j_{4}\})=1
\end{smallmatrix}
\end{gathered}\!\right\}\right. \,\,\,\,\,\,\,\,\,\,\,\\
\cup
\left\{\zetaup_{i_{1},i_{2},i_{3},i_{4}}\left| 
\begin{gathered}
 \begin{smallmatrix}
 1{\leq}i_{1}{<}i_{2}{<}i_{3}{<}i_{4}{\leq}8,
\end{smallmatrix}\\[-6pt]
\begin{smallmatrix}
 \#(\{i_{1},i_{2},i_{3},i_{4}\}\cap\{j_{1},j_{2},j_{3},j_{4}\})=2
\end{smallmatrix}
\end{gathered}\!\right\}\right.
\end{aligned}
\end{cases}
\end{equation*}

 \end{rmk}
The choice \eqref{bE8} has 
the advantage that the Dynkin diagrams of $\textsc{E}_{7}$ and $\textsc{E}_{6}$
may be obtained by erasing the node $\alphaup_{8}$ in the first and the nodes $\alphaup_{7},\alphaup_{8}$
in the second case. Indeed we can define $\Rad(\textsc{E}_{7})$ to be the subsystem 
consisting of the roots of $\Rad(\textsc{E}_{8})$
orthogonal to $\e_{7}{-}\e_{8}$: 
\begin{equation}\label{rE7}
\Rad(\textsc{E}_{7})=\{{\pm}\zetaup_{\emptyset},\; {\pm}(\e_{7}{+}\e_{8}),\, {\pm}\zetaup_{7,8}\}\cup
 \{{\pm}\e_{i}{\pm}\e_{j},\,{\pm}\zetaup_{i,j}\, {\pm}\zetaup_{i,j,7,8}{\mid}\, 
\begin{smallmatrix}
 1{\leq}i{<}j{\leq}6
\end{smallmatrix}\},
\end{equation}
with basis of positive simple roots 
\begin{equation}\label{bE7}
 \Bz(\textsc{E}_{7})=\{\alphaup_{1},\alphaup_{2},\alphaup_{3},
 \alphaup_{4},\alphaup_{5},\alphaup_{6},\alphaup_{7}\},
\end{equation}
and $\Rad(\textsc{E}_{6})$ to be the subsystem of $\Rad(\textsc{E}_{8})$
orthogonal to $\e_{7}{-}\e_{8}$ and to $\e_{6}{-}\e_{7}$:
\begin{equation}\label{rE6}
 \Rad(\textsc{E}_{6})=\{{\pm}\zetaup_{\emptyset}\}\cup
 \{{\pm}\e_{i}{\pm}\e_{j},\;{\pm}\zetaup_{i,j}\,{\mid}\, 
\begin{smallmatrix}
 1{\leq}i{<}j{\leq}5
\end{smallmatrix}\}\cup\{{\pm}\zetaup_{i,6,7,8}\,{\mid}\, 
\begin{smallmatrix}
 1{\leq}i{\leq}5
\end{smallmatrix}\}
\end{equation}
with basis of positive simple roots 
\begin{equation}\label{bE6}
 \Bz(\textsc{E}_{6})=\{\alphaup_{1},\alphaup_{2},\alphaup_{3},
 \alphaup_{4},\alphaup_{5},\alphaup_{6}\}.
\end{equation}
 \par
 With these choices, it is convenient to fix the maximal systems of strongly  orthogonal roots 
\begin{align}\label{maxE7}
 & \Mtt(\textsc{E}_{7})\;=\{\gammaup_{1},\gammaup_{2},\gammaup_{3},\gammaup_{4},\gammaup_{5},
 \gammaup_{6},\gammaup_{8}\} , && \text{for $\Rad(\textsc{E}_{7})$},\\
\label{maxE6}
 &\Mtt(\textsc{E}_{6})\; =\{\gammaup_{1},\gammaup_{2},\gammaup_{3},\gammaup_{4}\},
 && \text{for $\Rad(\textsc{E}_{6})$}. 
\end{align}
Another convenient choice for a system of type $\textsc{E}_{7}$ is to take the orthogonal
of $\zetaup_{\emptyset}$ in $\Rad(\textsc{E}_{8})$, and for type $\textsc{E}_{6}$ that of
$\{\zetaup_{\emptyset},\;\e_{7}{+}\e_{8}\}$: 
\begin{align}
 \label{rE7a}
 &\Rad'(\textsc{E}_{7})\;{=}\{{\pm}(\e_{i}{-}\e_{j})\,{\mid}\, 
\begin{smallmatrix}
 1{\leq}i{<}{j}{\leq}8
\end{smallmatrix}\!\}\cup\{\zetaup_{i_{1},i_{2},i_{3},i_{4}}\,{\mid}\, 
\begin{smallmatrix}
 1{\leq}i_{1}{<}i_{2}{<}i_{3}{<}i_{4}{\leq}8
\end{smallmatrix}\!\},\\
\label{rE6a}
&\Rad'(\textsc{E}_{6})\;{=}\{{\pm}(\e_{7}{-}\e_{8})\}\cup
\{{\pm}(\e_{i}{-}\e_{j})\,{\mid}\, 
\begin{smallmatrix}
 1{\leq}i{<}{j}{\leq}6
\end{smallmatrix}\!\}\cup\{{\pm}\zetaup_{i_{1},i_{2},i_{3},7}\,{\mid}\, 
\begin{smallmatrix}
 1{\leq}i_{1}{<}i_{2}{<}i_{3}{\leq}6
\end{smallmatrix}\!\}
\end{align}
The corresponding Dynkin diagrams, canonical bases and maximal systems of  strongly  orthogonal roots
are 

\begin{align} 
&  \xymatrix @M=0pt @R=4pt @!C=20pt{
  \alphaup'_{1}&\alphaup'_{2}&\alphaup'_{3}&\alphaup'_{4}&\alphaup'_{5}&\alphaup'_{6}\\
  \medcirc \ar@{-}[r] 
  &\medcirc\ar@{-}[r]
  &\medcirc\ar@{-}[r]\ar@{-}[dddd]
  &\medcirc\ar@{-}[r]
  &\medcirc\ar@{-}[r]
  &\medcirc \\
  \\
  \\
  \\
  &&\medcirc &\!\!\!\!\!\!\!\!\!\!\!\!\!\!\!\!\!\!\!\!\!\!\!\!\!\!\!\!\alphaup'_{7}
  }\\[4pt]
  \label{bE7a}
 &\Bz'(\textsc{E}_{7})=\{\alphaup'_{i}{=}\e_{i}{-}\e_{i+1}\,{\mid}\, 
\begin{smallmatrix}
 1{\leq}i{\leq}6
\end{smallmatrix}\!\}\cup\{\alphaup'_{7}{=}\zetaup_{4,5,6,7}\},\\
\label{maxE7a}
& \Mtt'(\textsc{E}_{7})=\{\betaup_{1},\betaup_{2},\betaup_{3},\betaup_{4},\betaup_{5},\betaup_{6},\betaup_{7}\},
\end{align}
and 
\begin{align}
&  \xymatrix @M=0pt @R=4pt @!C=20pt{
  \alphaup'_{1}&\alphaup'_{2}&\alphaup'_{3}&\alphaup'_{4}&\alphaup'_{5}\\
  \medcirc \ar@{-}[r] 
  &\medcirc\ar@{-}[r]
  &\medcirc\ar@{-}[r]\ar@{-}[dddd]
  &\medcirc\ar@{-}[r]
  &\medcirc
\\
  \\
  \\
  \\
  &&\medcirc &\!\!\!\!\!\!\!\!\!\!\!\!\!\!\!\!\!\!\!\!\!\!\!\!\!\!\!\!\alphaup'_{6}
  }\\[4pt]
\label{bE6a}
& \Bz'(\textsc{E}_{6})=\{\alphaup'_{i}{=}\e_{i}{-}\e_{i+1}\,{\mid}\, 
\begin{smallmatrix}
 1{\leq}i{\leq}5
\end{smallmatrix}\!\}\cup\{\alphaup'_{6}{=}\zetaup_{4,5,6,7}\}\\
 &\Mtt'(\textsc{E}_{6})=\{\betaup_{1},\betaup_{3},\betaup_{5},\betaup_{7}\}
\end{align}
\begin{rmk} The presentation $\Rad'(\textsc{E}_{7})$ of type $\textsc{E}_{7}$ is nicely symmetrical.
A nice symmetrical root system of type $\textsc{E}_{6}$ is obtained 
by considering the vectors $\vtt_{i}\,{=}\,\e_{i}{-}\tfrac{1}{3}(\e_{1}{+}\e_{2}{+}\e_{3})$ ($i=1,2,3$)
of $\R^{3}$ and setting 
\begin{align}
  \Rad'''(\textsc{E}_{6})\,&{=}\,\{(\vtt_{i}{-}\vtt_{j},0,0),\, (0,\vtt_{i}{-}\vtt_{j},0),\, (0,0,\vtt_{i}{-}\vtt_{j})\,{\mid}\,1{\leq}i{\neq}j{\leq}3\}
  \qquad\\
  \notag
 \,&{\cup}\,\{{\pm}(\vtt_{i},\vtt_{j},\vtt_{h})\,{\mid}\,1{\leq}i,j,k{\leq}3\}.
\end{align}
\end{rmk}
\vfill
\pagebreak 

\subsubsection*{Strongly orthgonal roots in $\textsc{E}_{8}$}
On $\Rad(\textsc{E}_{8})$ there is a natural $\Z_{2}$-gradation, making its subalgebra
$\Rad(\textsc{D}_{8})$ (see \eqref{rD})
its even 
and $\Theta$ its odd subset. This makes easier to compute permutations
of $\Mtt'(\textsc{E}_{8})$ induced by the Weyl group of $\Rad(\textsc{E}_{8})$ 
by restricting to the elements of the
Weyl group of~$\Rad(\textsc{D}_{8})$.
\par
We represent in the Table I  below 
Weyl rotations of  $\Rad(\textsc{D}_{8})$ 
acting on $\Mtt'(\textsc{E}_{8})$. To 
the pair of roots $\alphaup_{1},\alphaup_{2}$ on the left corresponds 
the rotation $\sq_{\alphaup_{1}}{\circ}\,\sq_{\alphaup_{2}}$, whose effect on $\Mtt'(\textsc{E}_{8})$ is 
the product of transpositions $(\betaup_{i_{1}},\betaup_{i_{2}})
(\betaup_{i_{3}},\betaup_{i_{4}})$ on the right. Since $\Mtt'(\textsc{E}_{8})$ is a basis of
$\hr^{*}{\simeq}\,\R^{8}$,
this means that 
\begin{equation}
 \sq_{\,\alphaup_{1}}\circ\sq_{\,\alphaup_{2}}=\sq_{\,\betaup_{i_{1}}-\,\betaup_{i_{2}}}\circ
 \sq_{\,\betaup_{i_{3}}-\,\betaup_{i_{4}}}.
\end{equation}\par
\begin{equation*} \begin{gathered}
\begin{array}{| c |c || c | c |}
\hline 
(\e_{2}{-}\e_{3}),(\e_{5}{-}\e_{8}) & (\betaup_{1},\betaup_{2})(\betaup_{3},\betaup_{4}) &
 (\e_{5}{-}\e_{8}),(\e_{6}{-}\e_{7}) & (\betaup_{1},\betaup_{3})(\betaup_{2},\betaup_{4}) \\
 \hline
 (\e_{2}{-}\e_{3}),(\e_{6}{-}\e_{7}) & (\betaup_{1},\betaup_{4})(\betaup_{2},\betaup_{3}) & &\\
 \hline\hline
 (\e_{2}{-}\e_{5}),(\e_{3}{-}\e_{8}) & (\betaup_{1},\betaup_{2})(\betaup_{5},\betaup_{6}) & 
  (\e_{3}{-}\e_{8}),(\e_{4}{-}\e_{7}) & (\betaup_{1},\betaup_{5})(\betaup_{2},\betaup_{6}) \\
  \hline
 (\e_{2}{-}\e_{5}),(\e_{4}{-}\e_{7}) & (\betaup_{1},\betaup_{6})(\betaup_{2},\betaup_{5})&& \\
 \hline\hline
  (\e_{5}{-}\e_{6}),(\e_{7}{-}\e_{8}) & (\betaup_{1},\betaup_{3})(\betaup_{5},\betaup_{7}) &
 (\e_{3}{-}\e_{4}),(\e_{7}{-}\e_{8}) &
  (\betaup_{1},\betaup_{5})(\betaup_{3},\betaup_{7})\\
 \hline
(\e_{3}{-}\e_{4}),(\e_{5}{-}\e_{6}) & (\betaup_{1},\betaup_{7})(\betaup_{3},\betaup_{5}) && \\
 \hline\hline
   (\e_{2}{-}\e_{7}),(\e_{3}{-}\e_{6}) & (\betaup_{1},\betaup_{4})(\betaup_{6},\betaup_{7}) &
 (\e_{2}{-}\e_{7}),(\e_{4}{-}\e_{5}) & (\betaup_{1},\betaup_{6})(\betaup_{4},\betaup_{7}) \\ %
 \hline
  (\e_{3}{-}\e_{6}),(\e_{4}{-}\e_{5}) & (\betaup_{1},\betaup_{7})(\betaup_{4},\betaup_{6}) &&\\ %
   \hline
 \hline
  (\e_{2}{-}\e_{6}),(\e_{3}{-}\e_{7}) & (\betaup_{2},\betaup_{3})(\betaup_{6},\betaup_{7}) &
 (\e_{3}{-}\e_{7}),(\e_{4}{-}\e_{8}) & (\betaup_{2},\betaup_{6})(\betaup_{3},\betaup_{7}) \\
 \hline
  (\e_{2}{-}\e_{6}),(\e_{4}{-}\e_{8}) & (\betaup_{2},\betaup_{7})(\betaup_{3},\betaup_{6}) & & \\
  \hline\hline
 (\e_{5}{-}\e_{7}),(\e_{6}{-}\e_{8}) & (\betaup_{2},\betaup_{4})(\betaup_{5},\betaup_{7}) & %
  (\e_{2}{-}\e_{4}),(\e_{5}{-}\e_{7}) & (\betaup_{2},\betaup_{5})(\betaup_{4},\betaup_{7}) 
  \\
  \hline
 (\e_{2}{-}\e_{4}),(\e_{6}{-}\e_{8}) & (\betaup_{2},\betaup_{7})(\betaup_{4},\betaup_{5}) && \\ 
 \hline\hline 
  (\e_{2}{-}\e_{8}),(\e_{3}{-}\e_{5}) & (\betaup_{3},\betaup_{4})(\betaup_{5},\betaup_{6}) &
 (\e_{3}{-}\e_{5}),(\e_{4}{-}\e_{6}) & (\betaup_{3},\betaup_{5})(\betaup_{4},\betaup_{6}) \\
 \hline
 (\e_{2}{-}\e_{8}),(\e_{4}{-}\e_{6}) & (\betaup_{3},\betaup_{6})(\betaup_{4},\betaup_{5}) & & \\ 
 \hline
\hline
  (\e_{2}{+}\e_{8}),(\e_{3}{+}\e_{5}) & (\betaup_{1},\betaup_{2})(\betaup_{7},\betaup_{8})&
 (\e_{3}{+}\e_{5}),(\e_{4}{+}\e_{6}) & (\betaup_{1},\betaup_{7})(\betaup_{2},\betaup_{8}) \\
 \hline
   (\e_{2}{+}\e_{8}),(\e_{4}{+}\e_{6}) & (\betaup_{1},\betaup_{8})(\betaup_{2},\betaup_{7}) 
 & &\\
  \hline\hline
    (\e_{5}{+}\e_{7}),(\e_{6}{+}\e_{8}) & (\betaup_{1},\betaup_{3})(\betaup_{6},\betaup_{8}) &
    (\e_{2}{+}\e_{4}),(\e_{5}{+}\e_{7}) & (\betaup_{1},\betaup_{6})(\betaup_{3},\betaup_{8})
     \\
  \hline
 (\e_{2}{+}\e_{4}),(\e_{6}{+}\e_{8}) & (\betaup_{1},\betaup_{8})(\betaup_{3},\betaup_{6}) 
 &&\\ 
  \hline\hline
 (\e_{2}{+}\e_{6}),(\e_{3}{+}\e_{7}) & (\betaup_{1},\betaup_{4})(\betaup_{5},\betaup_{8}) 
&
  (\e_{3}{+}\e_{7}),(\e_{4}{+}\e_{8}) & (\betaup_{1},\betaup_{5})(\betaup_{4},\betaup_{8}) \\
  \hline
    (\e_{2}{+}\e_{6}),(\e_{4}{+}\e_{8}) & (\betaup_{1},\betaup_{8})(\betaup_{4},\betaup_{5}) &&\\ 
  \hline\hline
 (\e_{2}{+}\e_{7}),(\e_{3}{+}\e_{6}) & (\betaup_{2},\betaup_{3})(\betaup_{5},\betaup_{8}) &
   (\e_{2}{+}\e_{7}),(\e_{4}{+}\e_{5}) & (\betaup_{2},\betaup_{5})(\betaup_{3},\betaup_{8}) \\ \hline
   (\e_{3}{+}\e_{6}),(\e_{4}{+}\e_{5}) & (\betaup_{2},\betaup_{8})(\betaup_{3},\betaup_{5}) && 
 \\
 \hline\hline
 (\e_{5}{+}\e_{6}),(\e_{7}{+}\e_{8}) & (\betaup_{2},\betaup_{4})(\betaup_{6},\betaup_{8}) 
& 
  (\e_{3}{+}\e_{4}),(\e_{7}{+}\e_{8}) & (\betaup_{2},\betaup_{6})(\betaup_{4},\betaup_{8}) \\
  \hline
 (\e_{3}{+}\e_{4}),(\e_{5}{+}\e_{6}) & (\betaup_{2},\betaup_{8})(\betaup_{4},\betaup_{6}) && \\ 
  \hline\hline
  (\e_{2}{+}\e_{5}),(\e_{3}{+}\e_{8}) & (\betaup_{3},\betaup_{4})(\betaup_{7},\betaup_{8}) &
 (\e_{3}{+}\e_{8}),(\e_{4}{+}\e_{7}) & (\betaup_{3},\betaup_{7})(\betaup_{4},\betaup_{8}) \\
  \hline
  (\e_{2}{+}\e_{5}),(\e_{4}{+}\e_{7}) & (\betaup_{3},\betaup_{8})(\betaup_{4},\betaup_{7}) &&\\ 
  \hline\hline
 (\e_{2}{+}\e_{3}),(\e_{5}{+}\e_{8}) & (\betaup_{5},\betaup_{6})(\betaup_{7},\betaup_{8}) &
  (\e_{5}{+}\e_{8}),(\e_{6}{+}\e_{7}) & (\betaup_{5},\betaup_{7})(\betaup_{6},\betaup_{8}) \\
  \hline
 (\e_{2}{+}\e_{3}),(\e_{6}{+}\e_{7}) & (\betaup_{5},\betaup_{8})(\betaup_{6},\betaup_{7}) &&\\ 
  \hline
\end{array}\\
 [\text{Table I}]
 \end{gathered}
\end{equation*}
\begin{ntz}
Let $\sfv_{1},\sfv_{2},\sfv_{3},\sfv_{4}$ be four orthogonal vectors in
$\hr^{*}{=}\langle\Rad(\textsc{E}_{8})\rangle_{\R}$. 
We denote by $\Kf_{4}(\sfv_{1},\sfv_{2},\sfv_{3},\sfv_{4})$
 the Klein $4$-group of isometries of $\hr^{*}{\simeq}\R^{8}$ consisting of 
\begin{equation*}
\id,\; \sq_{\,\sfv_{1}-\,\sfv_{2}}\circ \sq_{\,\sfv_{3}-\,\sfv_{4}},\;
\sq_{\,\sfv_{1}-\,\sfv_{3}}\circ \sq_{\,\sfv_{2}-\,\sfv_{4}},\;
\sq_{\,\sfv_{1}-\,\sfv_{4}}\circ \sq_{\,\sfv_{2}-\,\sfv_{3}}.
\end{equation*}  
\end{ntz}
Lines in Table\,{I} are divided into couples, each one containing four $\betaup_{i}$'s
and representing a corresponding Klein $4$-group of involutions of $\Rad(\textsc{E}_{8})$. 
\par
This can be summarized in the following Lemma. 
\begin{lem} $\Wf(\Rad(\textsc{E}_{8}))$ contains the fourteen  
Klein $4$-groups 
\begin{align*}
& \Kf_{4}(\betaup_{1},\betaup_{2},\betaup_{3},\betaup_{4}),\;\;
 \Kf_{4}(\betaup_{1},\betaup_{2},\betaup_{5},\betaup_{6}),\;\;
 \Kf_{4}(\betaup_{1},\betaup_{2},\betaup_{7},\betaup_{8}),\;\;
  \Kf_{4}(\betaup_{1},\betaup_{3},\betaup_{5},\betaup_{7}),\\
& \Kf_{4}(\betaup_{1},\betaup_{3},\betaup_{6},\betaup_{8}),\;\; 
   \Kf_{4}(\betaup_{1},\betaup_{4},\betaup_{6},\betaup_{7}),\;\;
 \Kf_{4}(\betaup_{1},\betaup_{4},\betaup_{5},\betaup_{8}),\;\; %
 \Kf_{4}(\betaup_{2},\betaup_{3},\betaup_{6},\betaup_{7}),\\
& \Kf_{4}(\betaup_{2},\betaup_{3},\betaup_{5},\betaup_{8}),\;\;%
 \Kf_{4}(\betaup_{2},\betaup_{4},\betaup_{5},\betaup_{7}),\;\;
  \Kf_{4}(\betaup_{2},\betaup_{4},\betaup_{6},\betaup_{8}),\;\;
 \Kf_{4}(\betaup_{3},\betaup_{4},\betaup_{5},\betaup_{6}),\\
 &\Kf_{4}(\betaup_{3},\betaup_{4},\betaup_{7},\betaup_{8}),\;\;
 \Kf_{4}(\betaup_{5},\betaup_{6},\betaup_{7},\betaup_{8}).\;\;\qquad\qquad\qquad\quad
 \qquad\qquad\qquad\qed
\end{align*}
 \end{lem} 
 \begin{rmk} The indices of the $\betaup_{i}$'s involved in each of these Klein groups
 are all pairs on the same column of the diagram 
 \begin{equation*} \begin{array}{|c|c|c|c|c|c|c|} \hline
 12 & 13 & 14 & 15 & 16 & 17 & 18\\
 \hline
 34 & 24 & 23 & 26 & 25 & 28 & 27\\
 \hline
 56 & 57 & 58 & 48 & 47 & 35 & 36\\
 \hline
 78 & 68 & 67 & 37 & 38 & 46 & 45\\
 \hline
 \end{array}
 \end{equation*}
 In particular, if $i_{1},\hdots,i_{8}$ is a permutation of $\{1,\hdots,8\}$, then
 \begin{equation*}
 \Kf_{4}(\betaup_{i_{1}},\betaup_{i_{2}},\betaup_{i_{3}},\betaup_{i_{4}})\subset
 \Wf(\Rad(\textsc{E}_{8})) \;\Longleftrightarrow  \;
 \Kf_{4}(\betaup_{i_{5}},\betaup_{i_{6}},\betaup_{i_{7}},\betaup_{i_{8}})\subset
 \Wf(\Rad(\textsc{E}_{8})).
 \end{equation*}
\end{rmk}
\begin{cor}\label{c2.11}
 For any triple $(i,j,h)$ with 
 $1{\leq}i{<}j{<}h{\leq}8$ there is a unique $k,$ with $1{\leq}k{\leq}8$ and
 $k\,{\notin}\,\{i,j,h\}$ such that $\Kf_{4}(\betaup_{i},\betaup_{j},\betaup_{h},\betaup_{k})\,{\subset}\,
 \Wf(\Rad(\textsc{E}_{8}))$. \qed
\end{cor}
\begin{prop} All systems of strongly orthogonal roots of $\Rad(\textsc{E}_{8})$ containing 
$k$ elements, for $0{\leq}k{\leq}8$ and $k{\neq}4$,
 are equivalent. \par
 There are two equivalence classes of systems
 of $4$ strongly orthogonal roots $\alphaup_{1},\alphaup_{2},\alphaup_{3},\alphaup_{4}$
 in $\Rad(\textsc{E}_{8})$, depending on: 
\begin{itemize}
 \item[(I)\;] $\Kf_{4}(\alphaup_{1},\alphaup_{2},\alphaup_{3},\alphaup_{4})$ is contained in
 $\Wf(\Rad(\textsc{E}_{8}))$; 
 \item[(II)] $\Kf_{4}(\alphaup_{1},\alphaup_{2},\alphaup_{3},\alphaup_{4})$ is not contained in
 $\Wf(\Rad(\textsc{E}_{8}))$.
\end{itemize}
 \end{prop} 
\begin{proof}
 Since all maximal systems of strongly orthogonal roots are equivalent, we may restrict our consideration
 to subsets of $\Mtt'(\textsc{E}_{8})$. \par
Since all roots of the same length are conjugated under the Weyl group,  as can also seen by  checking the permutations of the $\betaup_{i}$' on Table\,{I},  we find that all sets containing
 a single root are equivalent. Thus, to prove that those containing two roots are equivalent, we may
 restrict to those containing $\betaup_{1}$. This is shown by the chain of equivalences 
  \begin{equation}\label{2-e8}
 \xymatrixcolsep{6pc}\xymatrix @M=4pt @R=35pt @!C=30pt{
 \{\betaup_{1},\betaup_{2}\}\ar@{<->}^{ (\betaup_{2},\betaup_{3})(\betaup_{6},\betaup_{7})}[r] &
 \{\betaup_{1},\betaup_{3}\}\ar@{<->}^{ (\betaup_{3},\betaup_{4})(\betaup_{5},\betaup_{6})}[r] &
 \{\betaup_{1},\betaup_{4}\}\ar@{<->}^{ (\betaup_{2},\betaup_{7})(\betaup_{4},\betaup_{5})}[r] & 
  \{\betaup_{1},\betaup_{5}\}\ar@{<->}_{ (\betaup_{3},\betaup_{4})(\betaup_{5},\betaup_{6})}[d]\\ 
  & \{\betaup_{1},\betaup_{8}\}\ar@{<->}^{ (\betaup_{5},\betaup_{6})(\betaup_{7},\betaup_{8})}[r]
 & \{\betaup_{1},\betaup_{7}\}
  \ar@{<->}^{ (\betaup_{2},\betaup_{3})(\betaup_{6},\betaup_{7})}[r]& 
  \{\betaup_{1},\betaup_{6}\}.}
 \end{equation}
 Hence, to check that all triples are equivalent, we may restrict to those containing $\betaup_{1}$
 and $\betaup_{2}$. We utilize the chain of equivalences:
 \begin{equation}\label{3-e8}
 \xymatrixcolsep{6pc}\xymatrix @M=4pt @R=35pt @!C=50pt{
 \{\betaup_{1},\betaup_{2},\betaup_{3}\}\ar@{<->}^{ (\betaup_{1},\betaup_{2})(\betaup_{3},\betaup_{4})}[r] &
 \{\betaup_{1},\betaup_{2},\betaup_{4}\}\ar@{<->}^{ (\betaup_{3},\betaup_{6})(\betaup_{4},\betaup_{5})}[r] &
 \{\betaup_{1},\betaup_{2},\betaup_{5}\}\ar@{<->}^{ (\betaup_{1},\betaup_{2})(\betaup_{5},\betaup_{6})}[d] \\
  \{\betaup_{1},\betaup_{2},\betaup_{7}\}\ar@{<->}^{ (\betaup_{5},\betaup_{6})(\betaup_{7},\betaup_{8})}[r]& 
  \{\betaup_{1},\betaup_{2},\betaup_{8}\}
  \ar@{<->}^{ (\betaup_{5},\betaup_{7})(\betaup_{6},\betaup_{8})}[r]& 
  \{\betaup_{1},\betaup_{2},\betaup_{6}\}.}
 \end{equation}
 By considering complements, we obtained that all sets of $k$ orthogonal roots of $\Rad(\textsc{E}_{8})$
 are equivalent when $0{\leq}k{\leq}8$ and $k\,{\neq}\,4$. \par
 While considering the equivalence of $4$-tuples, we may restrict to the case of those containing
 $\{\betaup_{1},\betaup_{2},\betaup_{3}\}$.  By Cor.\ref{c2.11} we know that 
 $\{\betaup_{1},\betaup_{2},\betaup_{3},\betaup_{4}\}$ is the single one yielding a Klein group.
 It will suffice to prove that the remaining four are all equivalent. We get the chain of equivalences:
  \begin{equation}\label{4-e8}
 \xymatrixcolsep{6pc}\xymatrix @M=4pt @R=35pt @!C=64pt{
 \{\betaup_{1},\betaup_{2},\betaup_{3},\betaup_{5}\}\ar@{<->}^{ (\betaup_{1},\betaup_{2})(\betaup_{5},\betaup_{6})}[r] &
 \{\betaup_{1},\betaup_{2},\betaup_{3},\betaup_{6}\}
 \ar@{<->}^{ (\betaup_{1},\betaup_{3})(\betaup_{6},\betaup_{8})}[d] \\
 \{\betaup_{1},\betaup_{2},\betaup_{3},\betaup_{7}\}\ar@{<->}^{ (\betaup_{5},\betaup_{6})(\betaup_{7},\betaup_{8})}[r] 
 & 
  \{\betaup_{1},\betaup_{2},\betaup_{3},\betaup_{8}\}.}
 \end{equation}
 The proof is complete. 
\end{proof}
\begin{rmk}
The condition that  
$\Kf_{4}(\alphaup_{1},\alphaup_{2},\alphaup_{3},\alphaup_{4})\,{\subset}\,\Wf(\Rad(\textsc{E}_{8}))$ is 
equivalent to the fact that $\alphaup_{1},\alphaup_{2},\alphaup_{3},\alphaup_{4}$ belong to 
a subsystem of type $\textsc{D}_{4}$ contained in $\Rad(\textsc{E}_{8})$.
 \end{rmk}

\subsubsection*{Strongly orthogonal roots in $\textsc{E}_{6}$} 
Let us take in $\Rad'(\textsc{E}_{6})$ the maximal system of strongly orthogonal roots 
\begin{equation*}
 \Mtt''(\textsc{E}_{6})=\{\gammaup_{1}=\e_{1}{-}\e_{2},\; 
 \gammaup_{3}=\e_{3}{-}\e_{4},\;
 \gammaup_{5}=\e_{5}{-}\e_{6},\;
 \gammaup_{7}=\e_{7}{-}\e_{8}\}.
\end{equation*}
The composition $\sq_{\,\e_{1}-\e_{3}}{\circ}\sq_{\,\e_{2}-\e_{4}}$ is an involution in $\Rad'(\textsc{E}_{6})$
which restricts on $\Mtt''(\textsc{E}_{6})$ to the permutation $(\gammaup_{1},\gammaup_{3})$.\par
The composition $\sq_{\,\e_{3}-\e_{5}}{\circ}\sq_{\,\e_{4}-\e_{6}}$ is an involution in $\Rad'(\textsc{E}_{6})$
which restricts on $\Mtt''(\textsc{E}_{6})$ to the permutation $(\gammaup_{3},\gammaup_{5})$.\par
The composition $\sq_{\,\zetaup_{1,2,5,8}}{\circ}\sq_{\,\zetaup_{3,4,5,8}}$ is an involution in $\Rad'(\textsc{E}_{6})$
which restricts on $\Mtt''(\textsc{E}_{6})$ to the permutation $(\gammaup_{5},\gammaup_{7})$.  Indeed 
$\gammaup_{1},\gammaup_{2}, \zetaup_{1,2,5,8}, \zetaup_{3,4,5,8}$ are orthogonal and
\begin{align*}
 \sq_{\,\zetaup_{1,2,5,8}}{\circ}\sq_{\,\zetaup_{3,4,5,8}}(\e_{5}{-}\e_{6})=
 \sq_{\,\zetaup_{1,2,5,8}}(\e_{5}{-}\e_{6}-\zetaup_{3,4,5,8})= \sq_{\,\zetaup_{1,2,5,8}}(\zetaup_{1,2,5,7})\,\\
 =\zetaup_{1,2,5,7}-\zetaup_{1,2,5,8}=\e_{7}{-}\e_{8}.
\end{align*}
The transpositions $(\gammaup_{1},\gammaup_{3})$,
$(\gammaup_{3},\gammaup_{5})$, $(\gammaup_{5},\gammaup_{7})$ on $\Mtt''(\textsc{E}_{6})$,
and thus all permutations of the elements of $\Mtt''(\textsc{E}_{6})$,  
are restrictions of Weyl rotations of $\Rad'(\textsc{E}_{6})$. 
Thus we obtain

\begin{prop}
 Systems of orthogonal roots of a root system of type 
 $\textsc{E}_{6}$ are equivalent if and only if they
 contain the same number of elements. \qed
\end{prop}

\subsubsection*{Strongly orthogonal roots in $\textsc{E}_{7}$} 
We take $\Rad'(\textsc{E}_{7})$ and $\Mtt'(\textsc{E}_{7})$ of \eqref{rE7a} and \eqref{maxE7a}.
Since two subsets $\Stt$ and $\Stt'$ of $\Mtt'(\textsc{E}_{7})$ are equivalent if and only if
their complements $\Mtt'(\textsc{E}_{7}){\setminus}\Stt$, $\Mtt'(\textsc{E}_{7}){\setminus}\Stt'$
are equivalent, the chains of equivalences \eqref{2-e8} show that all subsets of $k$
elements are equivalent if $0{\leq}k{\leq}7$ and $k\,{\neq}\,3,4$.\par
In the case $k\,{=}\,3$, by reducing to sets containing $\betaup_{1},\betaup_{2},$ 
the chain of equivalences at \eqref{3-e8} shows that 
all $\{\betaup_{1},\betaup_{2},\betaup_{i}\}$, with $i\,{=}\,3,4,5,6$ are equivalent. 
As we showed that there are two distinct classes of sets of four orthogonal roots in $\Rad(\textsc{E}_{8})$
if follows that in root systems of type $\textsc{E}_{7}$ 
there are two classes of equivalence for sets of either $3$ or $4$ elements. We obtained 
\begin{prop}
 All systems of $k$ orthogonal roots of $\Rad(\textsc{E}_{7})$ are equivalent if $k{\neq}3,4$. 
 The two classes of systems of $4$ orthogonal roots $\alphaup_{1},\alphaup_{2},\alphaup_{3},\alphaup_{4}$
are characterized by 
\begin{itemize}
 \item[(\textsc{I})\;] 
 $\Kf_{4}(\alphaup_{1},\alphaup_{2},\alphaup_{3},\alphaup_{4})\,{\subset}\,\Wf(\Rad(\textsc{E}_{7})$;
 \item[(\textsc{II})] 
 $\Kf_{4}(\alphaup_{1},\alphaup_{2},\alphaup_{3},\alphaup_{4})\,{\not\subset}\,\Wf(\Rad(\textsc{E}_{7})$.
\end{itemize}
The equivalence classe of
systems of orthogonal roots $\{\alphaup_{1},\alphaup_{2},\alphaup_{3}\}$ are characterized by 
\begin{itemize}
 \item[(\textsc{I})\;] $\exists\,\alphaup_{4}\,{\in}\,\{\alphaup_{1},\alphaup_{2},\alphaup_{3}\}^{\perp}$
 such that $\Kf_{4}(\alphaup_{1},\alphaup_{2},\alphaup_{3},\alphaup_{4})\,{\subset}\,\Wf(\Rad(\textsc{E}_{7})$;
 \item[(\textsc{II})] $\not\exists\,\alphaup_{4}\,{\in}\,\{\alphaup_{1},\alphaup_{2},\alphaup_{3}\}^{\perp}$
 such that $\Kf_{4}(\alphaup_{1},\alphaup_{2},\alphaup_{3},\alphaup_{4})\,{\subset}\,\Wf(\Rad(\textsc{E}_{7})$.\qed
\end{itemize}
\end{prop}

\subsubsection*{Strongly orthogonal roots in $\textsc{F}_{4}$} We take for $\textsc{F}_{4}$ the root
system 
\begin{equation}\label{rf4}
 \Rad(\textsc{F}_{4})\,{=}\,\{{\pm}\e_{i}\,{\mid}\,1{\leq}i{\leq}4\}\cup\{{\pm}\e_{i}{\pm}\e_{j}\,{\mid}\,1{\leq}i{<}j{\leq}4\}\cup
 \{\tfrac{1}{2}({\pm}\e_{1}{\pm}\e_{2}{\pm}\e_{3}{\pm}\e_{4})\}.
\end{equation}
Modulo equivalence, there are two classes of maximal sets of strongly orthogonal roots, with
representatives 
\begin{equation} 
\begin{cases}
 \Mtt(\textsc{F}_{4})=\{\e_{1}{\pm}\e_{2},\,\e_{3}{\pm}\e_{4}\},\\
 \Mtt'(\textsc{F}_{4})=\{\e_{1}{\pm}\e_{2},\, \e_{4}\}.
\end{cases}
\end{equation}
If $\betaup_{1},\betaup_{2}$
are long orthogonal roots, then $\tfrac{1}{2}(\betaup_{1}{\pm}\betaup_{2})$ are still roots. 
If
$\betaup_{1},\betaup_{2},\betaup_{3},\betaup_{4}$ are four distinct long orthogonal roots, then 
\begin{equation*}\begin{cases}
 \sq_{(\betaup_{1}-\betaup_{2})/2}\circ \sq_{(\betaup_{3}-\betaup_{4})/2}(
 \{\betaup_{1},\betaup_{2}\})=\{\betaup_{3},\betaup_{4}\},\\
 \sq_{(\betaup_{2}-\betaup_{3})/2}(
 \{\betaup_{1},\betaup_{2}\})=\{\betaup_{1},\betaup_{3}\}.
 \end{cases}
\end{equation*}
This shows that pairs of long orthogonal roots are equivalent and hence, since all long roots are
equivalent, that all  sets of orthogonal long roots containing the same number of terms are
equivalent. \par
Analogously, one easily checks that systems of strongly orthogonal roots containing
a short root are equivalent if and only if they have the same number of elements.\par 
\begin{prop}
 Modulo equivalence, the systems $\Stt$ of strongly orthogonal roots in $\Rad(\textsc{F}_{4})$ are 
\begin{equation*} 
\begin{array}{| c | l | l |} 
\hline
\#\Stt & \qquad\text{all long} &\qquad \text{one short} \\
\hline
1 & \Stt_{1,0}(\textsc{F}_{4})=\{\e_{1}{-}\e_{2}\} & \Stt_{0,1}(\textsc{F}_{4})=\{\e_{4}\} \\
\hline
2 & \Stt_{2,0}(\textsc{F}_{4})=\{\e_{1}{\pm}\e_{2}\} & \Stt_{1,1}(\textsc{F}_{4})=\{\e_{1}{-}\e_{2},\,\e_{4}\}\\
\hline
3 & \Stt_{3,0}(\textsc{F}_{4})=\{\e_{1}{\pm}\e_{2},\, \e_{3}{-}\e_{4}\} &\Stt_{2,1}(\textsc{F}_{4})= \{\e_{1}{\pm}\e_{2},\,\e_{4}\}\\
\hline
4 & \Stt_{4,0}(\textsc{F}_{4})=\{\e_{1}{\pm}\e_{2},\,\e_{3}{\pm}\e_{4}\} &\\
\hline 
\end{array}
 \end{equation*} \qed
\end{prop}
\subsubsection*{Strongly orthogonal roots in $\textsc{G}_{2}$} We take for $\textsc{G}_{2}$ the root
system 
\begin{equation}\label{rg2}
 \Rad(\textsc{G}_{2})
 =\{{\pm}(\e_{i}{-}\e_{j})\,{\mid}\,1{\leq}i{<}j{\leq}3\}\cup\{{\pm}(2\e_{i}{-}\e_{j}{-}\e_{h})\mid \{i,j,k\}{=}\{1,2,3\}\}. 
\end{equation}
There are no pairs of orthogonal short roots, nor of orthogonal long roots. Thus, we obtain 
\begin{prop}
Modulo equivalence the classes of strongly orthogonal root systems 
$\texttt{S}$  in $\Rad(\textsc{G}_{2})$ are
\begin{equation*}
\begin{array}{| c | l | l | }
\hline 
\#\texttt{S} & \qquad\text{all long} & \qquad\text{one short} \\
\hline 
1 & \Stt_{1,0}(\textsc{G}_{2})=\{2\e_{3}{-}\e_{1}{-}\e_{2}\} &  \Stt_{0,1}(\textsc{G}_{2})=\{\e_{1}{-}\e_{2}\}\\ 
\hline 
2 & &  \Stt_{1,1}(\textsc{G}_{2})=\{\e_{1}{-}\e_{2},\, 2\e_{3}{-}\e_{1}{-}\e_{2}\}\\
\hline
 \end{array}
\end{equation*}\qed
\end{prop}\par\smallskip
\subsection{Special involutions}\label{S2.5} In this subsection we characterize special involutions.
\begin{lem}\label{l2.18}
If $\epi$  is a special involution of $\Rad$, then we can find a Weyl chamber $C\,{\in}\,\Cd(\Rad)$ such that 
$\epi(\Rad^{+}(C))\,{=}\,\Rad^{+}(C)$. 
\end{lem} 
\begin{proof}
 If $\epi$ is a special involution and $H$ is a regular element of $\hr,$ then 
 $H^{\epi}_{+}\,{=}\,\tfrac{1}{2}(H\,{+}\,\epi^{*}(H))$
 is regular and, if $C$ is the connected component of $H^{\epi}_{+}$ in $\Reg(\hr)$,
then $\epi(\Rad^{+}(C))\,{=}\,\Rad^{+}(C)$.
\end{proof}
\begin{dfn} 
For  $C\,{\in}\,\Cd(\Rad)$, we set 
\begin{equation}
\begin{cases}
 \Af(\Rad,C)=\{\ftt\in\Af(\Rad)\mid\ftt(\Rad^{+}(C))=\Rad^{+}(C)\},\\
 \Ib(\Rad,C)=\Af(\Rad,C)\cap\Ib(\Rad).
 \end{cases}
\end{equation}
\end{dfn}
\begin{exam}[Complex type] Let $\Rad\,{=}\,\Rad'\,{\sqcup}\,\Rad''$ be the disjoint union of two
copies of an irreducible roots system. If $\ftt,\gtt:\Rad'\,{\to}\,\Rad''$ are two isometries, then
\begin{equation*}
(\ftt^{-1},\gtt): \Rad'\sqcup\Rad''\ni(\alphaup',\alphaup'')\longrightarrow(\ftt^{-1}(\alphaup''),\gtt(\alphaup'))
\in \Rad'\sqcup\Rad''
\end{equation*}
is an element of $\Af(\Rad){\setminus}\Wf(\Rad)$. When $\gtt\,{=}\,\ftt$, this is an involution 
belonging to $\Ib(\Rad,(C,\ftt^{*}(C))$ for all $C\,{\in}\,\Cd(\Rad')$.
All involutions $\stt$ of $\Rad$ with $\stt(\Rad')\,{\neq}\,\Rad'$, 
$\stt(\Rad'')\,{\neq}\,\Rad''$, are of this type. 
\end{exam}
Lemma\,\ref{l2.18} can be reformulated in the following:
\begin{lem}
If $\epi$  is a special involution, then we can find a chamber $C\,{\in}\,\Cd(\Rad)$ such that $\epi\,{\in}\,\Ib(\Rad,C)$.
\qed
\end{lem}

\begin{thm} \label{t3.16}
Let $\Rad$ be an irreducible root system. 
Then the set $\Ib^{*}(\Rad)$ 
\begin{itemize}
 \item contains only the identity if $\Rad$ is of type $\textsc{B}_{\ell}$, $\textsc{C}_{\ell}$, $\textsc{E}_{7}$,
$ \textsc{E}_{8}$, $\textsc{F}_{4}$, $\textsc{G}_{2}$;
\item contains non trivial involutions if $\Rad$ is of type $\textsc{A}_{\ell}$, $\textsc{D}_{\ell}$, $\textsc{E}_{6}$;
for each case the nontrivial involutions are all equivalent and equivalent to 
\begin{equation*} 
\begin{array}{| c | c |} \hline
\Rad(\textsc{A}_{\ell}) & \epi(\e_{i})={-}\e_{\ell{+}2-i},\; 1{\leq}i{\leq}\ell{+}1 \\
\hline
\Rad(\textsc{D}_{\ell}) & \epi(\e_{i})=\e_{i},\; 1{\leq}i{\leq}{\ell}{-}1,\; \epi(\e_{\ell})={-}\e_{\ell} \\
 \hline 
\Rad'( \textsc{E}_{6}) & \epi(\e_{i})={-}\e_{7-i},\; 1{\leq}i{\leq}6,\; \epi(\e_{7})={-}\e_{8}\\
 \hline
\end{array}
\end{equation*}
\end{itemize}
 \end{thm} 
\begin{proof}
For $\epi$ in $\Ib^{*}(\Rad)$, take a $C\,{\in}\,(\Rad)$
such that $\epi\,{\in}\,\Ib(\Rad,C)$. \par 
Since $\epi(\Rad^{+}(C))\,{=}\,\Rad^{+}(C)$,
the involution $\epi$ restricts to an involution of $\Bz(C)$.  The statement follows by checking the involutions of
the corresponding Dynkin diagrams.
\end{proof}
\begin{rmk} In $\Rad(\textsc{D}_{4})$ we take the Weyl chamber $C$ with 
\begin{equation*}
 \Bz(C)=\{\alphaup_{1}{=}\e_{1}{-}\e_{2},\,\alphaup_{2}{=}\,\e_{2}{-}\e_{3},\,\alphaup_{3}{=}\e_{3}{-}\e_{4},
 \alphaup_{4}{=}\e_{3}{+}\e_{4}\}. 
\end{equation*} Then $\sq_{\alphaup_{1}-\alphaup_{3}}$ is an involution in $\Ib(\Rad(\textsc{D}_{4}))$
which exchanges $\alphaup_{1}$ and $\alphaup_{3}$ and keeps $\alphaup_{2},\alphaup_{4}$ fixed,
while $\sq_{\e_{4}}$ is an involution exchanging $\alphaup_{3}$ and $\alphaup_{4}$ and keeping
$\alphaup_{1},\alphaup_{2}$ fixed. The composition $\sq_{\alphaup_{1}-\alphaup_{3}}{\circ}\sq_{\e_{4}}$
is the rotation $(\alphaup_{1},\alphaup_{3},\alphaup_{4})$, which keeps $\alphaup_{2}$ fixed.  
It is an element of $\Af(\Rad(\textsc{D}_{4}),C)$ which is not an involution.
\end{rmk}
\begin{prop}\label{prop.2.30}
Let $C\,{\in}\,\Cd(\Rad)$ and $\Af(\Rad,C)$ the subgroup of the elements of $\Af(\Rad)$ leaving
$\Rad^{+}(C)$ invariant. Then 
\begin{equation}
 \label{e2.35}
 \Af(\Rad)=\Wf(\Rad)\ltimes\Af(\Rad,C) \quad\text{(semidirect product)}.
\end{equation}
\end{prop} 
\begin{proof} 
 Decomposition \eqref{e2.35} is a consequence of 
 the fact that
 $\Wf(\Rad)$ is simply transitive on the set $\Cd(\Rad)$ of the Weyl chambers
 (see e.g. \cite[Ch.VI,\S{1}, n$^{\mathrm{o}}$5, Thm.2]{Bou68}).
\end{proof}
\subsection{Classification of involutions}

We use the classification of subsets of strongly orthogonal roots of \S\ref{s2.4}
and of special involutions of \S\ref{S2.5} 
to describe the involutions of irreducible root systems. 
For the notation, we refer to
the description of irreducible root systems given therein.
We have:
\begin{thm}\label{t2.26}
The involutions $\stt{\in}\Ib(\Rad)$ of  irreducible root systems are, modulo equivalence,

 \begin{equation*} 
\begin{gathered}
\begin{array}{| c |  l |}
\hline
 \text{Root system} & \qquad\qquad\qquad\quad \text{involution}  \\
 \hline
 \Rad(\textsc{A}_{\ell})
 & \begin{aligned}
 &\sq_{\e_{1}{-}\e_{2}}\circ\cdots\circ\sq_{\e_{2h-1}{-}\e_{2h}} ,\; \begin{smallmatrix}1{\leq}h{\leq}[(\ell+1)/2]
 \end{smallmatrix}\\[3pt]
 & \begin{aligned}\sq_{\e_{1}{+}\e_{\ell+1}}{\circ}\cdots{\circ}\sq_{\e_{p}{+}\e_{\ell+2-p}}\circ 
 \sq_{\e_{p+1}}{\circ}\cdots{\circ}\sq_{\e_{\ell+1-p}}, 
 \begin{smallmatrix}
  0{\leq}p{\leq}[(\ell+1)/2]
\end{smallmatrix}\end{aligned}\end{aligned}
\\
\hline
 \begin{aligned}
  &\Rad(\textsc{B}_{\ell})\\
  &\Rad(\textsc{C}_{\ell})\\ 
  &\Rad(\textsc{D}_{\ell})
  \end{aligned}
  & \begin{aligned}
  \sq_{\e_{1}{-}\e_{2}}\circ\cdots\circ\sq_{\e_{2r_{1}-1}{-}\e_{2r_{1}}}\circ
  \sq_{\e_{\ell-r_{2}+1}}\circ\sq_{\e_{\ell-r_{2}+2}}{\circ}\cdots\circ\sq_{\e_{\ell}}, 
\begin{smallmatrix}
 2r_{1}+r_{2}{\leq}\ell
\end{smallmatrix}
  \end{aligned}
\\
\hline
\Rad'(\textsc{E}_{6}) 
&  \begin{aligned}
& {\sq_{\betaup_1}, \sq_{\betaup_1}\circ\sq_{\betaup_3},\sq_{\betaup_1}\circ\sq_{\betaup_3}\circ\sq_{\betaup_5},\sq_{\betaup_1}\circ\sq_{\betaup_3}\circ\sq_{\betaup_5}\circ\sq_{\betaup_7}},\\
& \sq_{\e_{1}+\e_{6}}{\circ}\sq_{\e_{2}+\e_{5}}{\circ}\sq_{\e_{3}+\e_{4}},\, 
\sq_{\e_{1}+\e_{6}}{\circ}\sq_{\e_{2}+\e_{5}}{\circ}\sq_{\e_{3}}{\circ}\sq_{\e_{4}},\\
&\sq_{\e_{1}+\e_{6}}{\circ}\sq_{\e_{2}}{\circ}\sq_{\e_{3}}{\circ}\sq_{\e_{4}}{\circ}\sq_{\e_{5}},\, 
\sq_{\e_{1}}{\circ}\sq_{\e_{2}}{\circ}\sq_{\e_{3}}{\circ}\sq_{\e_{4}}{\circ}\sq_{\e_{5}}{\circ}\sq_{\e_{6}}
\end{aligned}
\\
\hline 
\Rad(\textsc{E}_{7}) & \begin{aligned}
&\sq_{\e_{1}-\e_{2}},\; \sq_{\e_{1}}{\circ}\sq_{\e_{2}},\; 
 \sq_{\e_{1}}{\circ}\sq_{\e_{2}}{\circ}\sq_{\e_{3}-\e_{4}},
 \sq_{\e_{1}-\e_{2}}{\circ}\sq_{\e_{3}-\e_{4}}{\circ}\sq_{\e_{5}-\e_{6}},\\
 &\sq_{\e_{1}}{\circ}\sq_{\e_{2}}{\circ}\sq_{\e_{3}}{\circ}\sq_{\e_{4}},\; \sq_{\e_{1}}{\circ}{\sq}_{\e_{2}}{\circ}
 \sq_{\e_{3}-\e_{4}}{\circ}\sq_{\e_{5}{-}\e_{6}},\; 
  \sq_{\e_{1}}{\circ}\sq_{\e_{2}}{\circ}\sq_{\e_{3}}{\circ}\sq_{\e_{4}}{\circ}\sq_{\e_{5}{-}\e_{6}},\\
 & \sq_{\e_{1}}{\circ}\sq_{\e_{2}}{\circ}\sq_{\e_{3}}{\circ}\sq_{\e_{4}}{\circ}\sq_{\e_{5}}{\circ}\sq_{\e_{6}},
\; \sq_{\e_{1}}{\circ}\sq_{\e_{2}}{\circ}\sq_{\e_{3}}{\circ}\sq_{\e_{4}}{\circ}\sq_{\e_{5}}{\circ}\sq_{\e_{6}}{\circ}
\sq_{\e_{7}+\e_{8}}
\end{aligned}
\\
\hline
\Rad(\textsc{E}_{8}) & 
\begin{aligned}
&\sq_{\e_{1}-\e_{2}},\; \sq_{\e_{1}}{\circ}\sq_{\e_{2}},\; \sq_{\e_{1}}{\circ}\sq_{\e_{2}}{\circ}\sq_{\e_{3}-\e_{4}},\;
\sq_{\e_{1}}{\circ}\sq_{\e_{2}}{\circ}\sq_{\e_{3}}\circ\sq_{\e_{4}},\\
& \sq_{\e_{1}}{\circ}\sq_{\e_{2}}{\circ}\sq_{\e_{3}-\e_{4}}{\circ}\sq_{\e_{5}-\e_{6}},\;
\sq_{\e_{1}}{\circ}\sq_{\e_{2}}{\circ}\sq_{\e_{3}}{\circ}\sq_{\e_{4}}{\circ}\sq_{\e_{5}-\e_{6}},\\
&\sq_{\e_{1}}{\circ}\sq_{\e_{2}}{\circ}\sq_{\e_{3}}{\circ}\sq_{\e_{4}}{\circ}\sq_{\e_{5}}{\circ}\sq_{\e_{6}},\;
\sq_{\e_{1}}{\circ}\sq_{\e_{2}}{\circ}\sq_{\e_{3}}{\circ}\sq_{\e_{4}}{\circ}\sq_{\e_{5}}{\circ}\sq_{\e_{6}}{\circ}
\sq_{\e_{7}-\e_{8}},\\
&\sq_{\e_{1}}{\circ}\sq_{\e_{2}}{\circ}\sq_{\e_{3}}{\circ}\sq_{\e_{4}}{\circ}\sq_{\e_{5}}{\circ}\sq_{\e_{6}}{\circ}
\sq_{\e_{7}}{\circ}\sq_{\e_{8}}
\end{aligned}
\\
\hline 
\Rad(\textsc{F}_{4})&  
\begin{aligned}
&\sq_{\e_{1}},\;\;\sq_{\e_{1}-\e_{2}},\;\; \sq_{\e_{1}}{\circ}\sq_{\e_{2}-\e_{3}}, \;\;
\sq_{\e_{1}}{\circ}\sq_{\e_{2}},\;\;
\sq_{\e_{1}}{\circ}\sq_{\e_{2}}{\circ}\sq_{\e_{3}},\\
& \sq_{\e_{2}}{\circ}\sq_{\e_{3}}{\circ}\sq_{\e_{1}+\e_{4}},\;\;
\sq_{\e_{1}}{\circ}\sq_{\e_{2}}{\circ}\sq_{\e_{3}}{\circ}\sq_{\e_{4}}
 \end{aligned}\\
\hline
\Rad(\textsc{G}_{2}) &  \sq_{\e_{1}-\e_{2}},\;\; \sq_{2\e_{1}-\e_{2}-\e_{3}},\;\;
\sq_{\e_{1}}{\circ}\sq_{\e_{2}}{\circ}\sq_{\e_{3}}\\
\hline
\end{array}\\[6pt]
\textsc{Table II}
\end{gathered}
\end{equation*}

{
We used Prop. \ref{p2.7} and Prop. \ref{prop.2.30},  together with  the classification of subsets of strongly orthogonal roots of \S\ref{s2.4} 
and of special involutions given in Thm.\ref{t3.16}.We observe that in cases where there are distinct equivalence classes for subsets of strongly orthogonal roots, the reflections must be taken within a single class.
The first line for $\Rad(\textsc{A}_{\ell})$ describes the involutions in the Weyl group,  the second one
those in $\Af{\setminus}\Wf$,  we notice that when $p{=}0$ and $\ell{>}1$ the second line is meant to be only $\sq_{\e_{1}}{\circ}\cdots{\circ}\sq_{\e_{\ell+1}}=-\id$.  
 The involutions of $\Rad(\textsc{D}_{\ell})$ belong to the Weyl group
iff $r_{2}{\in}\,2\Z$. We used the root system  $\Rad'(\textsc{E}_{6})$ (see \eqref{rE6a}) for the
involutions in the Weyl group and  for those in $\Af{\setminus}\Wf$, first line and second and thirds lines respectively}.\qed
\end{thm}
\section{$S\!$-diagrams} 
Weyl chambers define lexicographic orders on $\Rad$. 
Thus, when an involution $\stt$ is fixed,
the choice of a Weyl chamber $C$ yields  
a decomposition of the set of $\stt$-complex roots: 
\begin{equation} 
\begin{cases}
 \Rd{\star}{\stt}=\Rd{\oplus}{\stt}(C)\sqcup\Rd{\ominus}{\stt}(C).\quad\text{where}\\
 \Rd{\oplus}{\stt}(C)=\{\alphaup\,{\in}\,\Rad^{\pm}(C)\,{\mid}\,\stt(\alphaup)\,{\in}\,\Rad^{\pm}(C)\}\cap
 \Rd{\star}{\stt}
 \\
  \Rd{\ominus}{\stt}(C)=\{\alphaup\,{\in}\,\Rad^{\pm}(C)\,{\mid}\,\stt(\alphaup)\,{\in}\,\Rad^{\mp}(C)\}
\cap
 \Rd{\star}{\stt}.
\end{cases}
\end{equation}\par 
If $H$ is any element of $C$, 
then
\begin{equation*} \left\{ \begin{aligned}
&
 \Rd{\oplus}{\stt}(C)=\{\alphaup\,{\in}\Rd{\star}{\stt}\,{\mid}\, 
 \alphaup(H)\,{\cdot}\stt(\alphaup)(H)\,{>}\,0\},\\
 &
 \Rd{\ominus}{\stt}(C)=\{\alphaup\,{\in}\Rd{\star}{\stt}\,{\mid}\,  
 \alphaup(H)\,{\cdot}\stt(\alphaup)(H)\,{<}\,0\}.
 \end{aligned}
 \right.
\end{equation*}\par
With obvious notation, we get  a corresponding partition of the basis:
\begin{equation}
 \Bz(C)=\Bz^{\stt}_{\circ}(C)\sqcup\Bz^{\stt}_{\bullet}(C)\sqcup\Bz^{\stt}_{\,\star}(C),\;\;
 \Bz_{\star}^{\stt}(C)\,{=}\,\Bz_{\oplus}^{\stt}(C)\,{\sqcup}\,\Bz_{\ominus}^{\stt}(C).
\end{equation} \par \index{$\Bz_{\oplus}^{\stt}(C)$} \index{$\Bz_{\ominus}^{\stt}(C)$}
\index{$\Bz_{\circ}^{\stt}(C)$} \index{$\Bz_{\bullet}^{\stt}(C)$} \index{$\Bz_{\star}^{\stt}(C)$} 
Since $\Bz(C)$ is  basis of $\hr^*$,  
an involution $\stt$ is completely
determined by its action on its elements. We may take track of the involution by 
labelling the nodes of the Dynkin diagram corresponding to $C$ with \par\smallskip
\centerline{
\begin{tabular}{c l}
 $\medcirc$ & for a root in $\Bz_{\circ}^{\stt}(C)$;\\
$\medbullet$ & for a root in $\Bz_{\bullet}^{\stt}(C)$;\\
$\oplus$ & for a root in $\Bz_{\oplus}^{\stt}(C)$;\\
$\ominus$ & for  a root in $\Bz_{\ominus}^{\stt}(C)$.
\end{tabular}}
\par\smallskip
However, the information we obtain in this way is, in general, not sufficient to
characterize the involution.\par
\begin{exam} \label{e2.33}
The involutions $\stt'\,{=}\,\sq_{\e_{1}}{\circ}\sq_{\e_{3}}{\circ}\sq_{\e_{2}-\e_{4}}$,
$\stt''\,{=}\,\sq_{\e_{1}{-}\e_{4}}{\circ}\sq_{\e_{3}}$ \par \noindent
and $\stt'''\,{=}\,\sq_{\e_{1}+\e_{3}}{\circ}\sq_{\e_{2}-\e_{4}}$
yield the same diagram 
 \begin{equation*}
 \xymatrix @M=0pt @R=2pt @!C=15pt{ \alphaup_{1}&\alphaup_{2}&\alphaup_{3}&\alphaup_{4}\\
  \qquad\quad 
 \\
 \ominus\ar@{-}[r]
 &\oplus\ar@{-}[r]
 &\ominus\ar@{=>}[r]
 &\oplus}
 \end{equation*}
 for the canonical basis $\alphaup_{1}{=}\e_{1}{-}\e_{2},\,\alphaup_{2}{=}\e_{2}{-}\e_{3},\,\alphaup_{3}{=}\e_{3}{-}\e_{4},\,
 \alphaup_{4}{=}\e_{4}$. 
 Indeed:

\begin{align*} 
&\begin{cases}
 \stt'(\alphaup_{1})={-}(\alphaup_{1}{+}\alphaup_{2}{+}\alphaup_{3}),\\
 \stt'(\alphaup_{2})=\alphaup_{3}{+}2\alphaup_{4},\\
 \stt'(\alphaup_{3})={-}(\alphaup_{2}{+}2\alphaup_{3}{+}2\alphaup_{4}),\\
 \stt'(\alphaup_{4})=\alphaup_{2}{+}\alphaup_{3}{+}\alphaup_{4},
\end{cases}
\begin{cases}
 \stt''(\alphaup_{1})={-}(\alphaup_{2}{+}\alphaup_{3}),\\
 \stt''(\alphaup_{2})=\alphaup_{2}{+}2\alphaup_{3}{+}2\alphaup_{4},\\
 \stt''(\alphaup_{3})={-}(\alphaup_{1}{+}\alphaup_{2}{+}2\alphaup_{3}{+}2\alphaup_{4}),\\
 \stt''(\alphaup_{4})=\alphaup_{1}{+}\alphaup_{2}{+}\alphaup_{3}{+}\alphaup_{4},
\end{cases} \\
&\begin{cases}
 \stt'''(\alphaup_{1})\,{=}\, {-}\alphaup_{3},\\
 \stt'''(\alphaup_{2})\,{=}\,\alphaup_{1}{+}\alphaup_{2}{+}\alphaup_{3}{+}2\alphaup_{4},\\
 \stt'''(\alphaup_{3})\,{=}\,{-}\alphaup_{1},\\
 \stt'''(\alphaup_{4})\,{=}\,\alphaup_{2}{+}\alphaup_{3}{+}\alphaup_{4}.
\end{cases}
\end{align*}
\end{exam}\par\medskip
We will introduce below a nicer class of 
Weyl chambers,
characterized  by  condition \eqref{e2.39}
below (see e.g. \cite{AMN06b, Ara62,  sa60}),
that are suitable to graphically describe involutions.
\begin{lem}
 For $\stt\,{\in}\,\Ib(\Rad)$ and $C\,{\in}\,\Cd(\Rad)$ the following are equivalent 
\begin{gather}\label{e2.39}
  \Rd{\oplus}{\stt}(C)=\Rd{\star}{\stt},\\
  \label{e2.40}
  \Bz^{\stt}_{\star}(C)=\Bz^{\stt}_{\oplus}(C).
\end{gather}
\end{lem} 
\begin{proof}
 Clearly \eqref{e2.39}$\Rightarrow$\eqref{e2.40}. To prove the vice versa, we introduce on $\Rad$ 
 the $\Z$-valued function
\begin{equation*}
 \kq_{\;C}^{\stt,\star}(\alphaup)={\sum}_{\betaup\,{\in}\,\Bz_{\star}^{\stt}(C)}k_{\alphaup,\betaup},\;\;\;\text{for}\;\;\;
 \alphaup={\sum}_{\betaup\in\Bz(C)}k_{\alphaup,\betaup}\betaup.
\end{equation*}
This function is additive and,
when \eqref{e2.40} is valid, we have $\kq_{\;C}^{\stt,\star}(\betaup)\,{\geq}\,1$ for all $\betaup\,{\in}\,\Bz_{\star}^{\stt}(C)$.
From this observation we obtain \eqref{e2.39}.
\end{proof}
The following theorem extends Lemma\,\ref{l2.18} to general involutions.
\begin{thm}\label{t3.10}\label{thmSch}
 If $\stt\,{\in}\,\Ib(\Rad)$, then we can find a Weyl chamber $C$ for which 
the equivalent conditions \eqref{e2.39}, \eqref{e2.40} are satisfied.
\end{thm} 
\begin{proof} If $H\,{\in}\,\Reg(\hr)$, then (see \eqref{e2.7} for the definition of $H^{\stt}_{\pm}$) 
\begin{equation*}
 \alphaup(H)\,{\cdot}\,\stt(\alphaup)(H)=|\alphaup(H_{+}^{\stt})|^{2}-|\alphaup(H_{-}^{\stt})|^{2}.
\end{equation*}
Hence, for $t\,{>}\,(\sup_{\alpha\in\Rad}|\alphaup(H_{-}^{\stt})|/
(\min_{\alphaup\in\Rad{\setminus}\Rd{\bullet}{\stt}}|\alphaup(H)|)$, the Weyl chamber $C$ containing
$(t{H}_{+}^{\stt}{+}H_{-}^{\stt})$ satisfies \eqref{e2.39}.
 \end{proof}
We rember that a \emph{parabolic} set of roots is a closed subset of roots $\Rad'$ such that $\Rad' \cup (-\Rad') = \Rad$.  A
\emph{horocyclic} set is the complement of a parabolic set.  
 
\begin{prop}
For chambers $C$
 satisfying \eqref{e2.39}
\begin{equation*}
 \begin{array}{  l}
   \Rad^{+}(C){\setminus}\Rd{\bullet}{\stt} \;\;\text{is $\stt$-invariant and horocyclic},\\
   \Rad^{+}(C)\,{\cup}\,\Rd{\bullet}{\stt} \;\;\text{is $\stt$-invariant and parabolic}.
 \end{array}
\end{equation*}
\end{prop} 
\begin{proof} 
If $C$ satisfies \eqref{e2.39} and $H\,{\in}\,C$, then  
\begin{equation*}\vspace{-18pt}
 \Rad^{+}(C){\setminus}\Rad_{\,\;\bullet}^{\stt}\,{=}\,\{\alphaup\,{\in}\,\Rad\,{\mid}\,
 \alphaup(H_{+}^{\stt})\,{>}\,0\},\;\;
  \Rad^{+}(C)\,{\cup}\,\Rad_{\,\;\bullet}^{\stt}\,{=}\,\{\alphaup\,{\in}\,\Rad\,{\mid}\,
 \alphaup(H_{+}^{\stt})\,{\geq}\,0\}.
\end{equation*}
\end{proof}
\begin{defn}
 We call a 
 $C\,{\in}\,\Cd(\Rad)$ satisfying  
 \eqref{e2.39}  an \emph{$S\!$-chamber for $\stt$}.\par
Denote by $\Cd_{\stt}(\Rad)$ the family of $S$-chambers for~$\stt$. \index{$\Cd_{\stt}(\Rad)$}
\end{defn} 
\begin{prop}\label{p1.5} Let $\stt\,{\in}\,\Ib(\Rad)$ and 
$C\,{\in}\,\Cd_{\stt}(\Rad)$. Then 
\begin{equation}\label{e2.29}
 \forall\,\alphaup\,{\in}\,\Bz_{\star}^{\stt}(C)\;\exists ! \alphaup'\,{\in}\,\Bz(C)\;\;\text{such that}\;\;
 \stt(\alphaup)-\alphaup'\,{\in}\,\Z[\Bz_{\bullet}^{\stt}(C)],
\end{equation}

and $\alphaup''\,{=}\,\stt(\alphaup)\,{-}\,\alphaup'$ is the unique element of maximal $\kq_{\; C}$-degree in 
$\Z^{+}[\Bz_{\bullet}^{\stt}(C)]$ for which $\alphaup\,{+}\,\alphaup''$ is still a root. 
\end{prop} 
\begin{proof} For $\alphaup\,{=}\,{\sum}_{\betaup\,{\in}\,\Bz(C)}k_{\alphaup,\betaup}\betaup$, we set
$\kq^{\stt,+}_{\; C}(\alphaup)\,{=}\,{\sum}_{\betaup\in\Bz^{\stt}_{\circ}\cup\Bz_{\star}^{\stt}}k_{\alphaup,\betaup}$.
For every fixed $\betaup$ in $\Bz(C)$, the coefficient $k_{\alphaup,\betaup}$ is an additive function of $\alphaup$. 
Hence  also \begin{equation*}
\kq^{\stt,+}_{\; C}(\alphaup)\,{=}\,{\sum}_{\betaup\in\Bz_{\circ}^{\stt}\cup
\Bz_{\star}^{\stt}}k_{\alphaup,\betaup}\end{equation*}
is an additive function of $\alphaup$.
If $C$ is an $S$-chamber, then
$\kq^+_{\;C}(\stt(\alphaup)){\geq}1$ for  $\alphaup\,{\in}\,\Rad^{+}(C){\setminus}\Rd{\bullet}{\stt}$ 
and this implies that
 $\kq^+_{\;C}(\stt(\alphaup))\,{\geq}\,\kq^+_{\;C}(\alphaup)$ 
for all $\alphaup\,{\in}\,\Rad^{+}(C){\setminus}\Rad_{\,\;\bullet}^{\stt}$.
Since $\stt^{2}{=}\id,$  
also the opposite inequality holds and thus
\begin{equation}\label{eq1.21}
 \kq^+_{\;C}(\stt(\alphaup))=\kq^+_{\;C}(\alphaup),\;\;\forall \alphaup\in\Rad^{+}(C)\setminus\Rad_{\,\;\bullet}^{\stt}.
\end{equation}
In particular, for an $\alphaup\,{\in}\,\Bz_{\star}^{\stt}(C)$, in the linear combination 
\begin{equation*}
 \stt(\alphaup)={\sum}_{\betaup\in\Bz(C)}k_{\stt(\alphaup),\betaup}\betaup
 \end{equation*}
 there is only one $\alphaup'\,{\in}\,\Bz(C){\setminus}\Rad_{\,\;\bullet}^{\stt}$ with 
$k_{\stt(\alphaup),\alphaup'}\,{=}\,1,$
while all other $\betaup\,{\neq}\,\alphaup'$ in $\supp_{C}(\stt\alphaup))$
are imaginary for $\stt$.  
Then $\stt(\alphaup)\,{=}\,\alphaup'{+}\,\alphaup''$ with $\alphaup''\,{\in}\,\Z[\Bz_{\bullet}^{\stt}(C)]$.\par
Take now $\betaup$ with $\kq_{\;C}(\betaup)$ maximal among the $\betaup\,{\in}\,\Z[\Bz_{\bullet}^{\stt}(C)]$
with $\alphaup\,{+}\,\betaup\,{\in}\,\Rad.$ 
We have $k_{\alphaup{+}\betaup,\alphaup}{=}1$ and therefore $\alphaup\,{+}\,\betaup\,{\in}\,\Rad^{+}(C)$.
\par
Since $C\,{\in}\,\Cd_{\stt}(\Rad)$
and $\kq^{+}_{\;C}(\alphaup\,{+}\,\betaup)\,{=}\,\kq^{+}_{\;C}(\alphaup)\,{=}\,1$,
we obtain, by the first part of the proof, 
\begin{equation*}
 \stt(\alphaup\,{+}\,\betaup)=\alphaup'\,{+}\,\betaup',\;\;\;\text{with\; $\betaup'\,{\in}\,\Z^{+}[\Bz_{\bullet}^{\stt}(C)]$.}
\end{equation*}
By applying $\stt$ to both sides of the equality we get 
\begin{equation*}
 \alphaup\,{+}\,\betaup\,{+}\,\betaup'=\stt(\alphaup')\in\Rad.
\end{equation*}
Thus $\betaup'{=}\,0$ by the maximality assumption, and $\betaup\,{=}\,\alphaup''.$ 
\end{proof}
\begin{prop} \label{p2.32}
Let $\stt\,{\in}\,\Ib(\Rad)$. 
If $C\,{\in}\,\Cd_{\stt}(\Rad)$, then $\Bz_{\bullet}^{\stt}(C)$ is a basis of simple roots for  $\Rad_{\,\;\bullet}^{\stt}$
and for any $C'\,{\in}\,\Cd(\Rad_{\,\;\bullet}^{\stt})$ 
one can find a $C\,{\in}\,\Cd(\Rad)$ for which
$\Bz_{\bullet}^{\stt}(C)\,{=}\,\Bz(C)\,{\cap}\,\Rad_{\,\;\bullet}^{\stt}{=}\,\Bz(C')$.
\end{prop} 
\begin{proof}
Let $C$ be an $S$-chamber for $\stt$.  Every root $\alphaup$ 
 is a linear combination \eqref{e1.15},
with integral coefficients, 
 of elements of $\Bz(C)$. By Prop.\ref{p1.5} the root $\alphaup$ belongs to $\Rad_{\,\;\bullet}^{\stt}$
 if and only if $k_{\alphaup,\betaup}\,{=}\,0$ for all $\betaup\,{\notin}\,\Bz_{\bullet}^{\stt}(C).$ 
 Thus every root in $\Rad_{\,\;\bullet}^{\stt}$ is a linear combination with integral coefficients 
 of roots in $\Bz_{\bullet}^{\stt}(C)$ and this shows that this is a basis of simple roots for $\Rad_{\,\;\bullet}^{\stt}$.
 \par
 If $\betaup\,{\in}\,\Rad_{\,\;\bullet}^{\stt},$ then $\sq_{\,\betaup}(\Rd{\bullet}{\stt})\,{=}\,\Rd{\bullet}{\stt}$
 and hence, for $H\,{\in}\,C,$ 
\begin{equation*}
 \big| \alphaup(\sq_{\,\betaup}^{*}(H^{{\stt}}_{+}))\big|=\big|\sq_{\,\betaup}(\alphaup)(H^{{\stt}}_{+})\big|
 >\big| \sq_{\,\betaup}(\alphaup)(H^{{\stt}}_{-})\big|=\big| \alphaup(\sq_{\,\betaup}^{*}((H^{{\stt}}_{-}))\big|\;\;
 \forall\alphaup\,{\in}\,\Rad{\setminus}\Rad_{\,\;\bullet}^{\stt},
\end{equation*}
where the previous inequalities trivially holds when $\alphaup$ is real,  while when $\alphaup$ is complex we have  $\alphaup(H)\,{\cdot}\,\stt(\alphaup)(H)=|\alphaup(H_{+}^{\stt})|^{2}-|\alphaup(H_{-}^{\stt})|^{2}{>}0\,$ since $C{\in}\Cd_{\stt}(\Rad)$.
Since $\betaup(H^{{\stt}}_{+})\,{=}\,0$, we have $\big(\sq_{\,\betaup}^{*}(H)\big)^{\stt}_{-}\,{=}\,\sq_{\,\betaup}^{*}(H^{\stt}_{-})$
and therefore, by \eqref{e2.39} this shows that $\sq_{\,\betaup}^{*}(\Cd_{\stt}(\Rad))\,{=}\,\Cd_{\stt}(\Rad)$.

 \par
 Since restrictions of reflections with respect to imaginary roots generate 
the Weyl group of $\Rad_{\,\;\bullet}^{\stt},$ which is transitive on the basis of simple roots for 
$\Rad_{\,\;\bullet}^{\stt},$ 
this yields also the last statement of the Proposition. 
\end{proof} 
\begin{rmk}[$V$-chambers] If $\stt\,{\in}\,\Ib(\Rad)$, then also $\check{\stt}\,{=}\,\att\,{\circ}\,\stt$, defined by 
\begin{equation*}
 \check{\stt}(\alphaup)={-}\stt(\alphaup),\;\;\forall \alphaup\,{\in}\hr^{*}
\end{equation*}
 defines an element of $\Ib(\Rad),$ with 
\begin{equation*}
 \Rad_{\,\;\circ}^{\check{\stt}}=\Rad_{\,\;\bullet}^{\stt},\;\; \;
  \Rad_{\,\;\bullet}^{\check{\stt}}=\Rad_{\,\;\circ}^{\stt},\;\;\;
   \Rad_{\,\;\star}^{\check{\stt}}=\Rad_{\,\;\star}^{\stt}.
\end{equation*}
\par An $S$-chamber for $\check{\stt}$ has the property that 
\begin{equation}\label{e2.42}
\Bz_{\star}^{\stt}(C)=\Bz_{\ominus}^{\stt}(C).
\end{equation}
\par
A Weyl chamber $C$ satisfying \eqref{e2.42} is called a \emph{$V$-chamber for~$\stt$}
(cf.~\cite{AMN06b}).
\end{rmk}

\begin{prop} \label{p2.34}
If $\stt\,{\in}\,\Ib(\Rad)$, then we can find an $S$-chamber $C$ for $\stt$ and a
decomposition $\stt\,{=}\,\epi\,{\circ}\,\sq_{\,\betaup_{1}}\,{\circ}\,\cdots\,{\circ}\,\sq_{\,\betaup_{r}}$,
with $\epi\,{\in}\,\Ib^{*}(\Rad)$ and 
strongly orthogonal
$\betaup_{1},\hdots,\betaup_{r}$ in $\Rd{\circ}{\epi}\,{\cap}\,\Rd{\bullet}{\stt}$, 
such that $C$ is also an $S$-chamber for $\epi$.
 \end{prop} 
\begin{proof} Take a decomposition \eqref{e2.8} and 
$C\,{\in}\,\Cd_{\stt}(\Rad)$ with $\#(\Rad^{+}(C)\,{\cap}\,\epi(\Rad^{+}(C)))$ maximal. 
Assume by contradiction that $C$ is not an $S$-chamber for $\epi$. Then there is a $\betaup\,{\in}\,\Bz(C)$
such that $\epi(\betaup)\,{\notin}\,\Rad^{+}(C)$. We claim that
$\betaup\,{\in}\,\Bz_{\bullet}^{\stt}(C)$, because otherwise
$\betaup'\,{=}\,\sq_{\,\betaup_{1}}\,{\circ}\,\cdots\,{\circ}\,\sq_{\,\betaup_{r}}(\betaup)
\,{=}\,\betaup{-}{\sum}_{i=1}^{r}\langle\betaup|\betaup_{i}\rangle\betaup_{i}$ would be an element of
$\Rad^{+}(C)$ and $\stt(\betaup')\,{=}\,\epi(\betaup)\,{\in}\,\Rad^{+}(C)$. Being $\betaup\,{\in}\,\Bz_{\bullet}^{\stt}(C)$,
we have \begin{equation*}
\sq_{\,\betaup}(\Rad^{+}(C))\,{=}\,(\Rad^{+}(C){\setminus}\{\betaup\})\,{\cup}\,\{{-}\betaup\}\,{=}\,
\Rad^{+}(C')\end{equation*}
 for a $C'\,{\in}\,\Cd_{\stt}(\Rad)$ with $\epi(\Rad^{+}(C'))\,{=}\,\epi(\Rad^{+}(C))\,{\cup}\,\{{-}\betaup\}$,
 contradicting maximality.
 The proof is complete.
\end{proof}

Involutions of $\Rad$ can be described by using 
generalized  Satake  diagrams (see e.g. \cite{Ara62, sa60}),
which we call \mbox{\emph{$S\!${-}diagrams}},  where the Cartan subalgebra is not the maximally non compact
one.  \par
These  are obtained from the Dynkin diagram
relative to a $C\,{\in}\,\Cd_{\stt}(\Rad)$ by: 
\begin{itemize}
 \item indicating by ``$\medbullet$'' a node corresponding to a root in $\Bz_{\bullet}^{\stt}(C)$;
 \item if $\alphaup\,{\neq}\,\alphaup'\,{\in}\,\Bz_{\star}^{\stt}(C)$ and 
 $\alphaup'{-}\stt(\alphaup)\,{\in}\,\Z[\Rad_{\,\;\bullet}^{\stt}]$, then joining  the nodes corresponding
 to $\alphaup$ and $\alphaup'$ by a double arrow. 
 \end{itemize}
Nodes  not corresponding to roots in $\Bz_{\bullet}^{\stt}(C)$ are indicated by ``$\medcirc$''.
Those from $\Bz_{\star}^{\stt}(C)$ are either adjacent to one from $\Bz_{\bullet}^{\stt}(C)$ or 
connected to some other node by a double arrow, or both. \par
We can avoid to distinguish between roots in $\Bz_{\circ}^{\stt}(C)$ and in
$\Bz_{\star}^{\stt}(C)\,{=}\,\Bz_{\oplus}^{\stt}(C)$, as the latter are the ``white'' ones which
are adjacent to a ``black'' one or are joined to another ``white'' by a double arrow.\par
In \,Table\,$\textsc{II}$  we described typical involutions in such a way that the standard
Weyl chamber  is an $S\!$-chamber.\par 
The involutions are represented, on the subdiagrams consisting of the node of a complex root
and those which are either connected to it by an arrow or composed by adjacent \textit{black  nodes}, by
the following rules, some of which 
show extra features with respect to those appearing
in Satake's diagrams. We note that the black nodes and the lines connecting them form
the diagram of a \textit{compact} Lie algebra. 
\smallskip
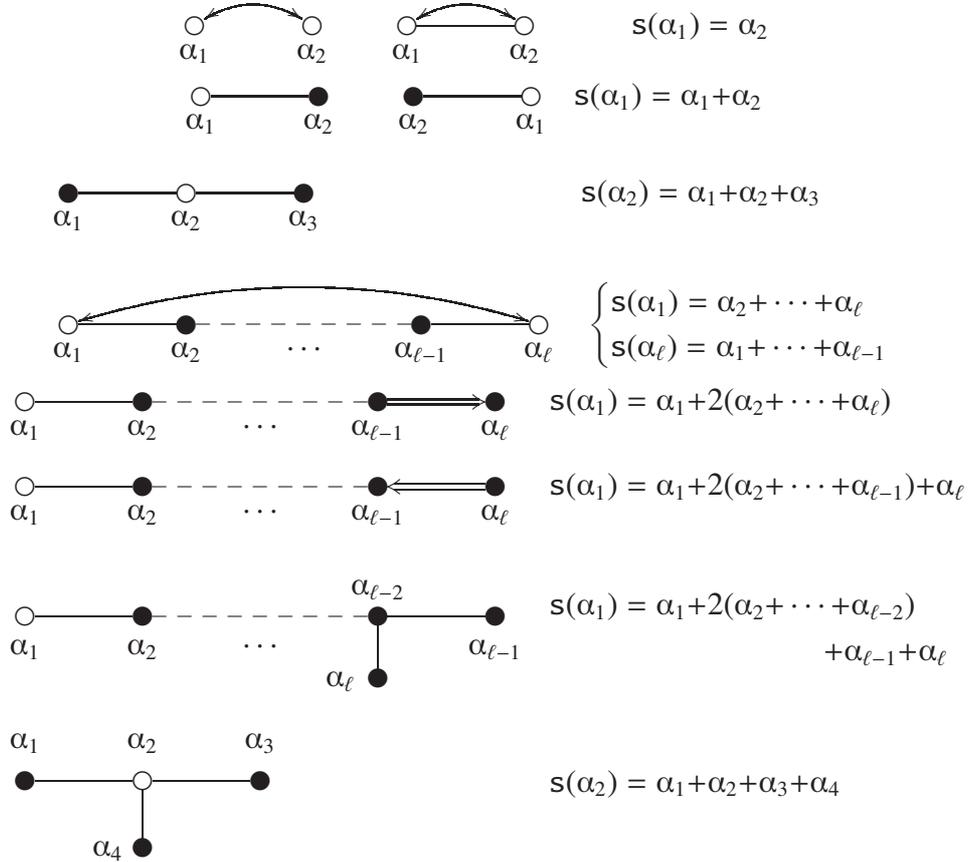
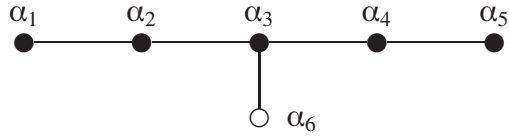
\begin{equation*}\begin{aligned}
 & \xymatrix @M=0pt @R=2pt @!C=20pt{
\medcirc \ar@/^8pt/@{<->}[r] & \medcirc\\
\alphaup_1&\alphaup_2
} \qquad  \xymatrix @M=0pt @R=2pt @!C=20pt{
\medcirc \ar@/^8pt/@{<->}[r]\ar@{-}[r] & \medcirc\\
\alphaup_1&\alphaup_2
}\qquad && \stt(\alphaup_{1})=\alphaup_{2}
\end{aligned}
\end{equation*}
\begin{equation*}
\begin{aligned}
&  \xymatrix @M=0pt @R=2pt @!C=20pt{
\medcirc \ar@{-}[r] & \medbullet\\
\alphaup_1&\alphaup_2
}\qquad   \xymatrix @M=0pt @R=2pt @!C=20pt{
\medbullet \ar@{-}[r] & \medcirc\\
\alphaup_2&\alphaup_1
} && \stt(\alphaup_{1})=\alphaup_{1}{+}\alphaup_{2}
\end{aligned}
\end{equation*}

\begin{equation*}
\begin{aligned}
&  \xymatrix @M=0pt @R=2pt @!C=20pt{
\medbullet \ar@{-}[r] &\medcirc\ar@{-}[r]& \medbullet\\
\alphaup_1&\alphaup_2&\alphaup_{3}
} && \stt(\alphaup_{2})=\alphaup_{1}{+}\alphaup_{2}{+}\alphaup_{3}
\\[14pt]
&  \xymatrix @M=0pt @R=2pt @!C=20pt{
\medcirc \ar@/^14pt/@{<->}[rrrr] \ar@{-}[r]
&\medbullet\ar@{--}[rr]
&
&\medbullet\ar@{-}[r]
&\medcirc\\
\alphaup_1&\alphaup_2&\cdots&\alphaup_{\ell-1}&\alphaup_{\ell}
} && 
\begin{cases}
 \stt(\alphaup_{1})=\alphaup_{2}{+}\cdots{+}\alphaup_{\ell}\\
 \stt(\alphaup_{\ell})=\alphaup_{1}{+}\cdots{+}\alphaup_{\ell-1}
\end{cases}
\end{aligned}
\end{equation*}
\begin{equation*}
\begin{aligned}
&  \xymatrix @M=0pt @R=2pt @!C=20pt{
\medcirc  \ar@{-}[r]
&\medbullet\ar@{--}[rr]
&
&\medbullet\ar@{=>}[r]
&\medbullet\\
\alphaup_1&\alphaup_2&\cdots&\alphaup_{\ell-1}&\alphaup_{\ell}
} && \stt(\alphaup_{1})=\alphaup_{1}{+}2(\alphaup_{2}{+}\cdots{+}\alphaup_{\ell})
\\[8pt]
&  \xymatrix @M=0pt @R=2pt @!C=20pt{
\medcirc  \ar@{-}[r]
&\medbullet\ar@{--}[rr]
&
&\medbullet\ar@{<=}[r]
&\medbullet\\
\alphaup_1&\alphaup_2&\cdots&\alphaup_{\ell-1}&\alphaup_{\ell}
} && \stt(\alphaup_{1})=\alphaup_{1}{+}2(\alphaup_{2}{+}\cdots{+}\alphaup_{\ell-1}){+}\alphaup_{\ell}
\\[-14pt]
&  \xymatrix @M=0pt @R=4pt @!C=20pt{
&&&\alphaup_{\ell-2}\\
\medcirc  \ar@{-}[r]
&\medbullet\ar@{--}[rr]
&
&\medbullet\ar@{-}[r]\ar@{-}[dd]
&\medbullet\\
\alphaup_1&\alphaup_2&\cdots&&\alphaup_{\ell-1}\\
&&\qquad\alphaup_{\ell}\!\!\!\!\!\!\!\!\!\!\!\!\!\!\!\!\!\!&\medbullet}
 && \begin{aligned}\quad
 \\
 \quad \\
 \stt(\alphaup_{1})=\alphaup_{1}{+}2(\alphaup_{2}{+}\cdots{+}\alphaup_{\ell-2})\quad\\
 {+}
 \alphaup_{\ell-1}{+}\alphaup_{\ell}\end{aligned}\\[-10pt]
 &  \xymatrix @M=0pt @R=8pt @!C=20pt{
 \alphaup_{1}&\alphaup_{2}&\alphaup_{3}\\
 \medbullet \ar@{-}[r] & \medcirc \ar@{-}[r]\ar@{-}[dd] &\medbullet\\
 \\
 \qquad\alphaup_{4}\!\!\!\!\!\!\!\!\!\!\!\!\!\!\!\!\!\!\!&\medbullet}
 &&\begin{aligned}
 \quad \\ \quad \\
 \stt(\alphaup_{2})=\alphaup_{1}{+}\alphaup_{2}{+}\alphaup_{3}{+}\alphaup_{4}
 \end{aligned}
\end{aligned}
\end{equation*}\vspace{8pt}
\begin{equation*}\begin{gathered}
  \xymatrix @M=0pt @R=4pt @!C=20pt{
  \alphaup_{1}&\alphaup_{2}&\alphaup_{3}&\alphaup_{4}&\alphaup_{5}\\
  \medbullet \ar@{-}[r] 
  &\medbullet\ar@{-}[r]
  &\medbullet\ar@{-}[r]\ar@{-}[dddd]
  &\medbullet\ar@{-}[r]
  &\medbullet\\
  \\
  \\
  \\
  &&\medcirc &\!\!\!\!\!\!\!\!\!\!\!\!\!\!\!\!\!\!\!\!\!\!\!\!\!\!\!\!\alphaup_{6}
  }
  \\[7pt]
  \stt(\alphaup_{6})=\alphaup_{1}{+}2\alphaup_{2}{+}3\alphaup_{3}{+}2\alphaup_{4}{+}\alphaup_{5}{+}\alphaup_{6}
  \end{gathered}
\end{equation*}
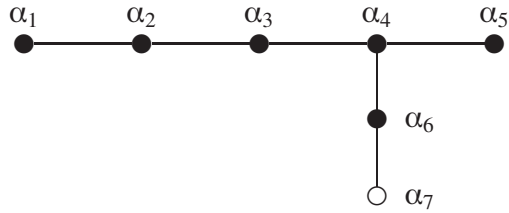
\vspace{8pt}
\par
 \begin{equation*}\begin{gathered}
  \xymatrix @M=0pt @R=4pt @!C=20pt{
  \alphaup_{1}&\alphaup_{2}&\alphaup_{3}&\alphaup_{4}&\alphaup_{5}\\
   \medbullet\ar@{-}[r]
  &\medbullet\ar@{-}[r]
  &\medbullet\ar@{-}[r]
  &\medbullet\ar@{-}[r]\ar@{-}[dddd]
  &\medbullet \\
  \\
  \\
  \\
  &&&\medbullet\ar@{-}[dddd] &\!\!\!\!\!\!\!\!\!\!\!\!\!\!\!\!\!\!\!\!\!\!\!\!\!\!\!\!\alphaup_{6}
  \\
  \\
  \\
  \\
  &&&\medcirc&\!\!\!\!\!\!\!\!\!\!\!\!\!\!\!\!\!\!\!\!\!\!\!\!\!\!\!\!\alphaup_{7}
  }\\[7pt] 
  \stt(\alphaup_{7})=\alphaup_{1}{+}2\alphaup_{2}{+}3\alphaup_{3}{+}4\alphaup_{4}{+}2\alphaup_{5}{+}3\alphaup_{6}
  {+}\alphaup_{7}  
  \end{gathered}
\end{equation*}\par\medskip
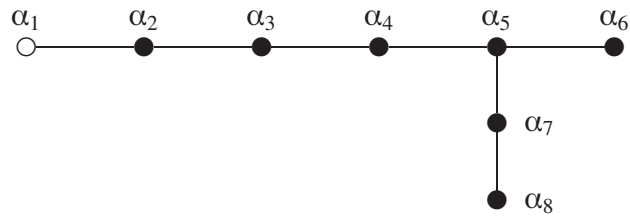
\begin{equation*}\begin{gathered}
  \xymatrix @M=0pt @R=4pt @!C=20pt{
  \alphaup_{1}&\alphaup_{2}&\alphaup_{3}&\alphaup_{4}&\alphaup_{5}&\alphaup_{6}\\
  \medcirc \ar@{-}[r] 
  &\medbullet\ar@{-}[r]
  &\medbullet\ar@{-}[r]
  &\medbullet\ar@{-}[r]
  &\medbullet\ar@{-}[r]\ar@{-}[dddd]
  &\medbullet \\ 
  \\
  \\
  \\
  &&&&\medbullet\ar@{-}[dddd] &\!\!\!\!\!\!\!\!\!\!\!\!\!\!\!\!\!\!\!\!\!\!\!\!\!\!\!\!\alphaup_{7}
  \\
  \\
  \\
  \\
  &&&&\medbullet&\!\!\!\!\!\!\!\!\!\!\!\!\!\!\!\!\!\!\!\!\!\!\!\!\!\!\!\!\alphaup_{8}
  }
  \\[7pt]
  \stt(\alphaup_{1})=\alphaup_{1}{+}3\alphaup_{2}{+}4\alphaup_{3}{+}5\alphaup_{4}{+}6\alphaup_{5}{+}3\alphaup_{6}{+}
  4\alphaup_{7}{+}2\alphaup_{8}
  \end{gathered}
\end{equation*}
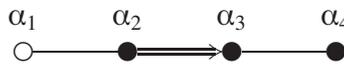
\begin{equation*} \begin{gathered}
 \xymatrix @M=0pt @R=1pt @!C=15pt{ \alphaup_{1}&\alphaup_{2}&\alphaup_{3}&\alphaup_{4}\\
  \qquad\quad 
 \\
\\
 \medcirc\ar@{-}[r]
 &\medbullet\ar@{=>}[r]
 &\medbullet\ar@{-}[r]
 &\medbullet} 
 \\[9pt]
 \stt(\alphaup_{1})=\alphaup_{1}{+}3\alphaup_{2}{+}4\alphaup_{3}{+}2\alphaup_{4}.
 \end{gathered}
 \end{equation*}
 \par\bigskip
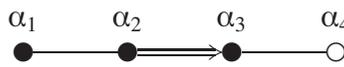
\begin{equation*}\begin{gathered}
 \xymatrix @M=0pt @R=2pt @!C=15pt{ \alphaup_{1}&\alphaup_{2}&\alphaup_{3}&\alphaup_{4}\\
  \qquad\quad 
 \\
 \medbullet\ar@{-}[r]
 &\medbullet\ar@{=>}[r]
 &\medbullet\ar@{-}[r]
 &\medcirc}\\[9pt]
 \stt(\alphaup_{4})=\alphaup_{1}{+}2\alphaup_{2}{+}3\alphaup_{3}{+}\alphaup_{4}.
 \end{gathered}
 \end{equation*}
  \par\bigskip
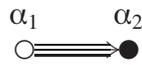
\begin{equation*}\begin{gathered}
 \xymatrix @M=0pt @R=2pt @!C=15pt{ \alphaup_{1}&\alphaup_{2}\\
  \qquad\quad 
 \\
 \medcirc\ar@3{->}[r]
 &\medbullet}\\[9pt]
\qquad \stt(\alphaup_{1})=\alphaup_{1}{+}3\alphaup_{2},
 \end{gathered}
\end{equation*}

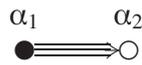
\qquad\qquad
\begin{equation*}
\begin{gathered}
 \xymatrix @M=0pt @R=2pt @!C=15pt{ \alphaup_{1}&\alphaup_{2}\\
  \qquad\quad 
 \\
 \medbullet\ar@3{->}[r]
 &\medcirc}\\[9pt]
 \qquad\stt(\alphaup_{2})=\alphaup_{1}{+}\alphaup_{2}.
 \end{gathered}
 \end{equation*}
\par\smallskip
\par
\smallskip

\begin{exam} For types $\textsc{B}_{\ell}$, $\textsc{C}_{\ell}$, $\textsc{D}_{\ell}$  
and an involution with parameters $r_{1},r_{2}$, the roots spaces $\Rd{\bullet}{\stt}$
consists of $r_{1}$ disjoint
copies of $\textsc{A}_{1}$ and a $\textsc{B}_{r'_{2}}$,  $\textsc{C}_{r_{2}}$, 
$\textsc{D}_{r_{2}}$,
where $r'_{2}\,{=}\,\left[\tfrac{r_{2}{+}1}{2}\right]$, respectively. 
We agree that $\textsc{B}_{0}$, $\textsc{C}_{0}$, $\textsc{D}_{0}$,
$\textsc{D}_{1}$ are empty
while $\textsc{B}_{1}\,{\simeq}\,\textsc{C}_{1}\,{\simeq}\,\textsc{A}_{1}$, $\textsc{D}_{2}$ is two disjoint copies of
$\textsc{A}_{1}$ and $\textsc{D}_{3}{\simeq}\textsc{A}_{3}$.
\par
The corresponding $S\!$-diagram has, in the case of a $\textsc{D}_{\ell\geq 4}$,  the last two nodes white and joined
by a curved arrow,  otherwise 
are obtained by  
blackening the last
$r_{2}$ or $r_{2}'$ roots. When $r_{1}$ is positive,
$r_{1}$ of the preceding roots
are blackened, with the unique restriction that the root to its right, if any, stays white. 
\par 
For instance,  for the involution class on $\textsc{D}_{4}$ with $r_{1}{=}r_{2}{=}1$ there are two admissible
$S\!$-diagrams  
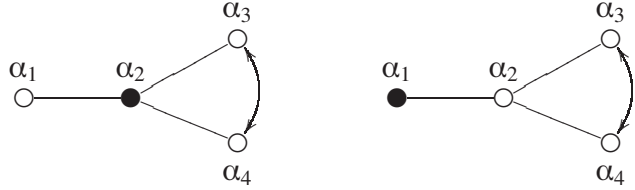
\begin{equation*}
 \xymatrix @M=0pt @R=4pt @!C=16pt{
 &&\alphaup_{3}\\
 && \medcirc\ar@/^8pt/@{<->}[dddd]\\
 \alphaup_{1}&\alphaup_{2} \\
 \medcirc \ar@{-}[r] & \medbullet \ar@{-}[uur]  \ar@{-}[ddr]\\
 \\
 &&\medcirc\\
 &&\alphaup_{4}}\quad\qquad\quad
  \xymatrix @M=0pt @R=4pt @!C=16pt{
 &&\alphaup_{3}\\
 && \medcirc\ar@/^8pt/@{<->}[dddd]\\
 \alphaup_{1}&\alphaup_{2} \\
 \medbullet \ar@{-}[r] & \medcirc \ar@{-}[uur]  \ar@{-}[ddr]\\
 \\
 &&\medcirc\\
 &&\alphaup_{4}}
\end{equation*}
and we have, according to the second and fourth lines above, 
\begin{equation*} 
\begin{cases}
 \stt(\alphaup_{1})=\alphaup_{1}{+}\alphaup_{2},\\
 \stt(\alphaup_{3})=\alphaup_{2}{+}\alphaup_{4},\\
 \stt(\alphaup_{4})=\alphaup_{2}{+}\alphaup_{3},
\end{cases}\quad\qquad\quad\quad 
\begin{cases}
 \stt(\alphaup_{2})=\alphaup_{1}+\alphaup_{2},\\
 \stt(\alphaup_{3})=\alphaup_{4},\\
 \stt(\alphaup_{4})=\alphaup_{3},
\end{cases}
\end{equation*}
respectively.
\par\smallskip For the involutions discussed in Example\,\ref{e2.33} we have for $\stt'$ 
an $S$-diagram in the basis $\e_{2}{-}\e_{4},\,\e_{4}{-}\e_{1},\,\e_{1}{-}\e_{3},\e_{3}$: 
  
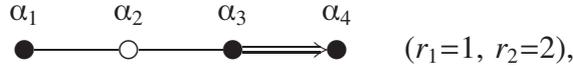
\begin{equation*} 
 \xymatrix @M=0pt @R=2pt @!C=15pt{ \alphaup_{1}&\alphaup_{2}&\alphaup_{3}&\alphaup_{4}\\
  \qquad\quad 
 \\
 \medbullet\ar@{-}[r]
 &\medcirc\ar@{-}[r]
 &\medbullet\ar@{=>}[r]
 &\medbullet &\quad\qquad (r_{1}{=}1,\,r_{2}{=}2),}
 \end{equation*}
 for $\stt''$ an $S$-diagram in the basis $\e_{1}{-}\e_{4},\,\e_{4}{-}\e_{2},\,\e_{2}{-}\e_{3},\,\e_{3}$:
   
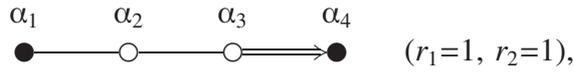
\begin{equation*}
 \xymatrix @M=0pt @R=2pt @!C=15pt{ \alphaup_{1}&\alphaup_{2}&\alphaup_{3}&\alphaup_{4}\\
  \qquad\quad 
 \\
 \medbullet\ar@{-}[r]
 &\medcirc\ar@{-}[r]
 &\medcirc\ar@{=>}[r]
 &\medbullet &\quad\qquad (r_{1}{=}1,\, r_{2}{=}1),}
 \end{equation*}
 for $\stt'''$ in the basis $\{\e_{1}{+}\e_{3},\, \e_{2}{-}\e_{3},\, \e_{4}{-}\e_{2},\,{-}\e_{4}\}$ 
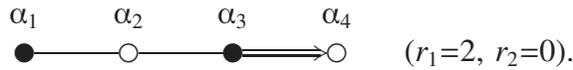
    \begin{equation*}
 \xymatrix @M=0pt @R=2pt @!C=15pt{ \alphaup_{1}&\alphaup_{2}&\alphaup_{3}&\alphaup_{4}\\
  \qquad\quad 
 \\
 \medbullet\ar@{-}[r]
 &\medcirc\ar@{-}[r]
 &\medbullet\ar@{=>}[r]
 &\medcirc &\quad\qquad(r_{1}{=}2,\, r_{2}{=}0).}
 \end{equation*}
\end{exam}
\begin{exam} While computing the involution for an \textit{inner}
complex root of the basis, we need to take care of each subdiagram of which
the root is an edge.   Take for instance
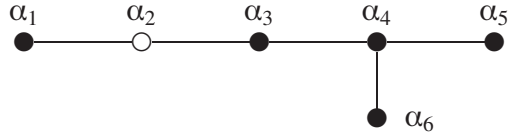
\begin{equation*}\begin{gathered}
  \xymatrix @M=0pt @R=4pt @!C=20pt{
  \alphaup_{1}&\alphaup_{2}&\alphaup_{3}&\alphaup_{4}&\alphaup_{5}\\
   \medbullet\ar@{-}[r]
  &\medcirc\ar@{-}[r]
  &\medbullet\ar@{-}[r]
  &\medbullet\ar@{-}[r]\ar@{-}[dddd]
  &\medbullet \\
  \\
  \\
  \\
  &&&\medbullet
  &\!\!\!\!\!\!\!\!\!\!\!\!\!\!\!\!\!\!\!\!\!\!\!\!\!\!\!\!\alphaup_{6}
  \\
  }
  \end{gathered}
\end{equation*} 
We consider $\alphaup_{2}$ as an \textit{edge} root of the subdiagrams $\{\alphaup_{1},\alphaup_{2}\}$, of type 
$\textsc{A}_{2}$, and $\{\alphaup_{2},\alphaup_{3},\alphaup_{4},\alphaup_{5},\alphaup_{6}\}$,
of type $\textsc{D}_{5}$. Using the corresponding involution rules we get 
\begin{equation*} 
\begin{cases}
 \alphaup_{2}\to\alphaup_{1}{+}\alphaup_{2},&(\textsc{A}_{2}),\\
 \alphaup_{2}\to \alphaup_{2}{+}2\alphaup_{3}{+}2\alphaup_{4}{+}\alphaup_{5}{+}\alphaup_{6},
 &(\textsc{D}_{5}),\\
 \stt(\alphaup_{2})=\alphaup_{1}{+}\alphaup_{2}{+}2\alphaup_{3}{+}2\alphaup_{4}
 {+}\alphaup_{5}{+}\alphaup_{6},
\end{cases}
\end{equation*}
i.e. to compute $\stt(\alphaup_{2})$ we add to $\alphaup_{2}$ the 
linear combinations of imaginary roots which are added to $\alphaup_{2}$ according with the conjugation
rules of the two subdiagrams. 
\end{exam}
We saw that 
an $S\!$-diagram is a graph 
obtained from a Dynkin diagram by painting black some of its nodes and by joining by
a double arrow some pairs of the remaining white nodes.
The fact that 
the simple roots corresponding to  black nodes generate a subsystem of roots corresponding to
a Lie subalgebra of a quasi-split real form\footnote{see Definiton\,\ref{d2.1} below.} 
rules out some graphs (see e.g. \cite{heck83, osh06}).
\begin{exam}
 \begin{equation*}\begin{gathered}
  \xymatrix @M=0pt @R=4pt @!C=20pt{
  \alphaup_{1}&\alphaup_{2}&\alphaup_{3}&\alphaup_{4}&\alphaup_{5}\\
  \medcirc \ar@{-}[r] 
  &\medbullet\ar@{-}[r]
  &\medbullet\ar@{-}[r]\ar@{-}[dddd]
  &\medbullet\ar@{-}[r]
  &\medbullet\\
  \\
  \\
  \\
  &&\medbullet &\!\!\!\!\!\!\!\!\!\!\!\!\!\!\!\!\!\!\!\!\!\!\!\!\!\!\!\!\alphaup_{6}
  }
  \end{gathered}
\vspace{6pt}\end{equation*}%
is not the $S\!$-diagram of a root system of type $\textsc{E}_{6}$. In fact, modulo equivalence,  
there is a single  involution having a $\Rd{\bullet}{\stt}$ of
type $\textsc{D}_{5}$ and 
whose $S\!$-diagram was given above. 
\end{exam}

We will now list  by type the admissible $S$-dia\-grams.

\subsubsection*{Type $\textsc{A}_{\ell}$} 
There are two graphs with no black nodes, one equal to the Dynkin diagram, the other
having nodes at the same distance from the edges connected with a double arrow.
In case there are black nodes, 
either there are no pairs of adjacent black nodes, or 
the black nodes form a connected subgraph in the center of the diagram and white nodes at the
same distance from this subgraph are connected by a double arrow.
\subsubsection*{Type $\textsc{B}_{\ell}$, $\textsc{C}_{\ell}$}  An $S\!$-diagram of type $\textsc{B}_{\ell}$ 
or $\textsc{C}_{\ell}$ is admissible
if and only if  the complement of the connected
component of $\Bz_{\bullet}^{\stt}(C)$ containing\footnote{This connected component is the empty set if the node of 
$\alphaup_{\ell}$ is white.} 
$\alphaup_{\ell}$ 
has no
adjacent black nodes. 
\subsubsection*{Type $\textsc{D}_{\ell}$} The  conditions below characterize admissible $S$-dia\-grams: 
\begin{itemize}
\item all diagrams  which do not contain adjacent black nodes are admissible;
 \item if $\ell{>}4$ and $\alphaup_{\ell-1}$ and $\alphaup_{\ell}$ are white nodes, 
 then they may or may not 
be joined by
 a double arrow; if $\ell\,{=}\,4$ only one single couple of white nodes  $\alphaup_{1},\alphaup_{3},\alphaup_{4}$
 may be joined by a curved arrow;
\item if not empty, the set of 
adjacent black nodes  contains at least three elements and
forms a connected subdiagram with a connected complement.
\end{itemize}
\subsubsection*{Type $\textsc{E}_{6}$} The involution $\epi$ 
yields the involution of the Dynkin diagram 
which keeps $\alphaup_{6}$ fixed and maps $\alphaup_{i}$ to $\alphaup_{6-i}$ for $1{\leq}i{\leq}5.$\par
The admissible $S\!$-diagrams are those which 
\begin{itemize}
 \item do not contain neither double arrows nor adjacent black nodes;
 \item do
not contain double arrows and the black nodes form a $\textsc{D}_{4}$ subdiagram;
\item all nodes are black; 
 \item $\alphaup_{6}$ is white, 
 the black nodes form an $\epi$-invariant connected subdiagram and white nodes different from
 $\alphaup_{6}$ are  
 joined by  double arrows. 
\end{itemize}
The last point yields the 
four $S\!$-diagrams below: 
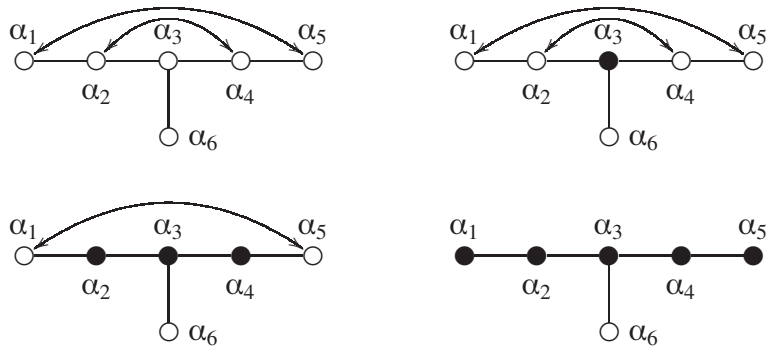
\vspace{4pt}
 \begin{equation*} \vspace{12pt}
  \xymatrix @M=0pt @R=4pt @!C=3pt{
&\alphaup_{1}&&\alphaup_{3}&&\alphaup_{5}\\
 &\medcirc \ar@{-}[r]\ar@/^20pt/@{<->}[rrrr] 
 &\medcirc\ar@{-}[r] \ar@/^16pt/@{<->}[rr]
 &\medcirc\ar@{-}[r]\ar@{-}[ddd]
 &\medcirc\ar@{-}[r] 
 &\medcirc 
\\
\qquad\qquad
&&\alphaup_{2}
 &&\alphaup_{4}
\\
\\
&&&\quad\;\medcirc\;\alphaup_{6}}
 \xymatrix @M=0pt @R=4pt @!C=3pt{
&\alphaup_{1}&&\alphaup_{3}&&\alphaup_{5}\\
 &\medcirc \ar@{-}[r]\ar@/^20pt/@{<->}[rrrr] 
 &\medcirc\ar@{-}[r] \ar@/^16pt/@{<->}[rr]
 &\medbullet\ar@{-}[r]\ar@{-}[ddd]
 &\medcirc\ar@{-}[r] 
 &\medcirc 
\\
\qquad\qquad
&&\alphaup_{2}
 &&\alphaup_{4}
\\
\\
&&&\quad\;\medcirc\;\alphaup_{6}}
\end{equation*}
\begin{equation*}
 \xymatrix @M=0pt @R=4pt @!C=3pt{
&\alphaup_{1}&&\alphaup_{3}&&\alphaup_{5}\\
 &\medcirc \ar@{-}[r]\ar@/^20pt/@{<->}[rrrr] 
 &\medbullet\ar@{-}[r] 
 &\medbullet\ar@{-}[r]\ar@{-}[ddd]
 &\medbullet\ar@{-}[r] 
 &\medcirc 
\\
\qquad\qquad
&&\alphaup_{2}
 &&\alphaup_{4}
\\
\\
&&&\quad\;\medcirc\;\alphaup_{6}}
  \xymatrix @M=0pt @R=4pt @!C=3pt{
&\alphaup_{1}&&\alphaup_{3}&&\alphaup_{5}\\
 &\medbullet \ar@{-}[r] 
 &\medbullet\ar@{-}[r] 
 &\medbullet\ar@{-}[r]\ar@{-}[ddd]
 &\medbullet\ar@{-}[r] 
 &\medbullet
\\
\qquad\qquad
&&\alphaup_{2}
 &&\alphaup_{4}
\\
\\
&&&\quad\;\medcirc\;\alphaup_{6}}
\end{equation*} 
\par\medskip
Other examples of $S\!$-diagrams for $\textsc{E}_{6}$ are: 
 
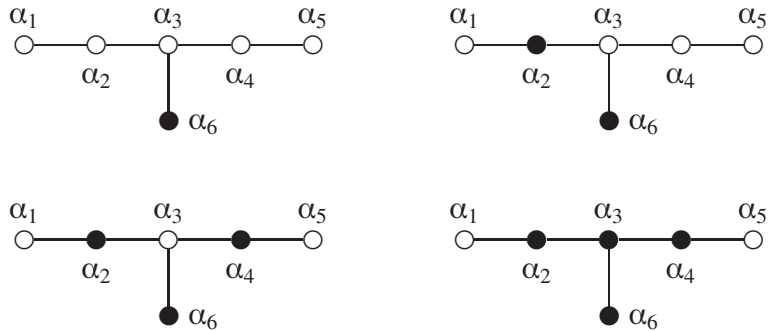
\begin{equation*} \vspace{12pt}
  \xymatrix @M=0pt @R=4pt @!C=3pt{
&\alphaup_{1}&&\alphaup_{3}&&\alphaup_{5}\\
 &\medcirc \ar@{-}[r] 
 &\medcirc\ar@{-}[r] 
 &\medcirc\ar@{-}[r]\ar@{-}[ddd]
 &\medcirc\ar@{-}[r] 
 &\medcirc 
\\
\qquad\qquad
&&\alphaup_{2}
 &&\alphaup_{4}
\\
\\
&&&\quad\;\medbullet\;\alphaup_{6}}
 \xymatrix @M=0pt @R=4pt @!C=3pt{
&\alphaup_{1}&&\alphaup_{3}&&\alphaup_{5}\\
 &\medcirc \ar@{-}[r] 
 &\medbullet\ar@{-}[r] 
 &\medcirc\ar@{-}[r]\ar@{-}[ddd]
 &\medcirc\ar@{-}[r] 
 &\medcirc 
\\
\qquad\qquad
&&\alphaup_{2}
 &&\alphaup_{4}
\\
\\
&&&\quad\;\medbullet\;\alphaup_{6}}
\end{equation*}
\begin{equation*}
 \xymatrix @M=0pt @R=4pt @!C=3pt{
&\alphaup_{1}&&\alphaup_{3}&&\alphaup_{5}\\
 &\medcirc \ar@{-}[r]
 &\medbullet\ar@{-}[r] 
 &\medcirc\ar@{-}[r]\ar@{-}[ddd]
 &\medbullet\ar@{-}[r] 
 &\medcirc 
\\
\qquad\qquad
&&\alphaup_{2}
 &&\alphaup_{4}
\\
\\
&&&\quad\;\medbullet\;\alphaup_{6}}
  \xymatrix @M=0pt @R=4pt @!C=3pt{
&\alphaup_{1}&&\alphaup_{3}&&\alphaup_{5}\\
 &\medcirc\ar@{-}[r] 
 &\medbullet\ar@{-}[r] 
 &\medbullet\ar@{-}[r]\ar@{-}[ddd]
 &\medbullet\ar@{-}[r] 
 &\medcirc
\\
\qquad\qquad
&&\alphaup_{2}
 &&\alphaup_{4}
\\
\\
&&&\quad\;\medbullet\;\alphaup_{6}}
\end{equation*} 
\par\smallskip
In the first three diagrams above black nodes may be moved arbitrarily, with the unique constraint of 
not allowing 
pairs of adjacent black nodes.
\subsubsection*{Types $\textsc{E}_{7}$, $\textsc{E}_{8}$} The admissible $S\!$-diagram for $\textsc{E}_{7}$, 
$\textsc{E}_{8}$, are those
in which the black nodes which are not isolated form a subdiagram which is either of type $\textsc{D}_{4}$,
$\textsc{D}_{6}$, $\textsc{E}_{7}$. 
\subsubsection*{Type $\textsc{F}_{4}$} The non admissible $S\!$-diagrams for $\textsc{F}_{4}$ are 
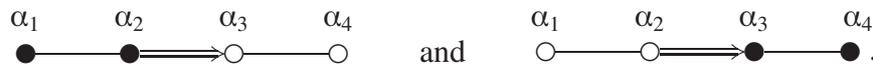
\begin{equation*}
  \xymatrix @M=0pt @R=2pt @!C=15pt{ \alphaup_{1}&\alphaup_{2}&\alphaup_{3}&\alphaup_{4}\\
  \qquad\quad 
 \\
 \medbullet\ar@{-}[r]
 &\medbullet\ar@{=>}[r]
 &\medcirc\ar@{-}[r]
 &\medcirc & \text{and}}\quad
  \xymatrix @M=0pt @R=2pt @!C=15pt{ \alphaup_{1}&\alphaup_{2}&\alphaup_{3}&\alphaup_{4}\\
  \qquad\quad 
 \\
 \medcirc\ar@{-}[r]
 &\medcirc\ar@{=>}[r]
 &\medbullet\ar@{-}[r]
 &\medbullet \; .}
\end{equation*}
\subsubsection*{Type $\textsc{G}_{2}$} All $S\!$-diagrams for $\textsc{G}_{2}$ are admissible.
\section{Involutions of $\gt$}\label{s2} 
Here  and in the next section, 
given a complex semisimple Lie algebra $\gt$, 
we discuss the conjugacy classes of pairs 
$(\go,\ho)$ consisting of
its real forms and their
Cartan subalgebras. First we show that, modulo conjugation, we can reduce to the special 
pairs $(\gs,\hs)$ of \eqref{e1.28}, \eqref{e1.29} below, 
which are consistent with the real split and compact forms associated to
a Chevalley system (see
Thm.\ref{thmchev}). 
\begin{ntz} Let $\gt$ be a complex semisimple Lie algebra, $\go$ a real form of $\gt$,
$\ho$ a Cartan subalgebra of $\go$ and $\hg$ its complexification. We will  show that every
involution of the root system lifts to an anti-involution of 
a quasi-split real form.\footnote{see Definiton\,\ref{d2.1} below.}  General real forms will be studied in the next section. Set
 \begin{itemize}
\item $\Aut_{\R}(\gt)$ for the $\R$-linear automorphisms of $\gt$; \index{$\Aut_{\R}(\gt)$}
\smallskip
\item $\Aut_{\C}(\gt)$ for the $\C$-linear automorphisms of $\gt$; \index{$\Aut_{\C}(\gt)$}
\smallskip
\item $\Aut_{\C}^{\!\vee}(\gt)$ for the anti-$\C$-linear automorphisms of $\gt$; \index{$\Aut_{\C}^{V}(\gt)$}
\smallskip
\item $\Aut(\go)$ for the $\R$-linear automorphisms of $\go$; \index{$\Aut(\go)$}
\smallskip
\item 
$\Inv_{\R}(\gt)$ for the $\R$-linear involutions of $\gt$; \index{$\Inv_{\R}(\gt)$}
\smallskip
\item $\Inv_{\C}(\gt)$ for the $\C$-linear involutions of $\gt$; \index{$\Inv_{\C}(\gt)$}
\smallskip
\item $\Inv_{\C}^{\!\vee}(\gt)$ for the anti-$\C$-linear involutions of $\gt$; \index{$\Inv_{\C}^{V}(\gt)$}
\smallskip
\item $\Inv(\go)$ for the $\R$-linear involutions of $\go$; \index{$\Ib(\go)$}
\smallskip
 \item $\Aut_{\R}(\gt,\hg)$ for the $\R$-linear automorphisms of $\gt$ leaving $\hg$ invariant; \index{$\Aut_{\R}(\gt,\hg)$}
 \smallskip
 \item $\Aut_{\C}(\gt,\hg)$ for the $\C$-linear elements of $\Aut_{\R}(\gt,\hg)$; \index{$\Aut_{\C}(\gt,\hg)$}
  \smallskip
 \item $\Aut_{\C}^{\!\vee}(\gt,\hg)$ for the anti-$\C$-linear elements of $\Aut_{\R}(\gt,\hg)$; \index{$\Aut_{\C}^{V}(\gt,\hg)$}
  \smallskip
   \item $\Aut(\go,\ho)$ for the  automorphisms of $\go$ leaving \index{$\Aut(\go,\ho)$}
  $\ho$ invariant;
  \smallskip
   \item $\Aut_{0}(\go,\ho)$ for the  automorphisms of $\go$ pointwise fixing \index{$\Aut_{0}(\go,\ho)$}
  $\ho$;
 \smallskip
 \item $\Inv(\gt,\hg)$ for the involutions in $\Aut_{\R}(\gt,\hg)$; \index{$\Ib(\gt,\hg)$}
   \smallskip
  \item $\Inv_{0}(\gt,\hg)$ for the involutions in $\Inv(\gt,\hg)$ pointwise fixing $\hg$; \index{$\Inv_{0}(\gt,\hg)$}
  \smallskip
  \item $\Inv_{\C}(\gt,\hg)$ for the involutions in $\Aut_{\C}(\gt,\hg)$; \index{$\Inv_{\C}(\gt,\hg)$}
   \smallskip
  \item $\Inv_{\C}^{\!\vee}(\gt,\hg)$ for the  involutions in $\Aut_{\C}^{\!\vee}(\gt,\hg)$;\index{$\Inv_{\C}^{V}(\gt,\hg)$}
  \smallskip
  \item $\Inv(\go,\ho)$ for the involutions in $\Aut(\go,\ho)$; \smallskip
  \item $\Inv_{0}(\go,\ho)$ for the involutions in $\Inv(\go,\ho)$ pointwise fixing $\ho$.\index{$\Inv_{0}(\go,\ho)$}
\end{itemize}\par\vspace{4pt}\noindent
For a $\sigmaup$ in $\Aut(\go)$, we will use  
the same symbol $\sigmaup$
for its extension to an element of
$\Aut_{\C}^{\!\vee}(\gt)$ and $\Hat{\sigmaup}$ for its extension to an element of $\Aut_{\C}(\gt).$  
\par\medskip
Fixing a Chevalley system $\{(X_{\alphaup},H_{\alphaup})\}_{\alphaup\in\Rad}$
for $(\gt,\hg)$ yields the split real form $\gr$, with Cartan subalgebra $\hr$ of \eqref{e1.4}
and the compact real form $\gu$, with Cartan subalgebra $\hu,$ of \eqref{e1.5}. We set
\begin{equation}\label{e2.1}\begin{cases}
\Aut^{\tauup}(\gr,\hr)=\{\ftt\in\Aut(\gr,\hr)\mid \ftt(\gu)=\gu\},\\
\Inv^{\tauup}(\gr,\hr)=\{\sigmaup\in\Inv(\gr,\hr)\mid \sigmaup(\gu)=\gu\}.
 \end{cases}
\end{equation} \index{$\Inv^{\tauup}(\gr,\hr)$}
\par 
A 
$\sigmaup\,{\in}\,\Inv(\gr,\hr)$ defines  $\Z_{2}$-gradations of $\gr$ and $\hr$: 
\begin{equation}\label{e1.28}
\begin{cases}
 \gr\,{=}\,\gs^{+}\oplus\gs^{-},\\
 \hr\,{=}\,\hs^{+}\oplus\hs^{-},
 \end{cases} \text{with}\;
\begin{cases}
 \gs^{+}\,{=}\,\{X\in\gr\mid \sigmaup(X)\,{=}\,X\},&\hs^{+}\,{=}\,\hr\cap\gs^{+},\\
\gs^{-}\,{=}\,\{X\in\gr\mid \sigmaup(X)\,{=}\,{-}X\},&\hs^{-}\,{=}\,\hr\cap\gs^{-},
\end{cases}
\end{equation}
and yields a pair $(\gs,\hs)$, with   
\begin{equation}\label{e1.29} 
\gs=\gs^{+}\oplus {i}\gs^{-},\quad\hs=\hs^{+}\oplus {i}\hs^{-},
\end{equation}
consisting of a real form $\gs$ of $\gt$, and its Cartan subalgebra $\hs$.
\end{ntz}
\begin{thm}\label{t3.1}
Any pair $(\go,\ho)$ consisting of a real form $\go$ of $\gt$ and its Cartan subalgebra $\ho$
is conjugated in $\Aut_{\C}(\gt)$ with a pair $(\gs,\hs)$, for some $\sigmaup\,{\in}\,\Inv^{\tauup}(\gr,\hr).$ 
 \end{thm} 
\begin{proof} 
Fix a Cartan involution $\thetaup_{0}$ of $\gt_{0}$ leaving $\hg_{0}$ invariant 
and
 let 
\begin{equation*}
 \gt_{0}{=}\kt_{0}{\oplus}\pt_{0},\quad \text{with}\quad \kt_{0}{=}\,\{X\,{\mid}\,\thetaup_{0}(X)\,{=}\,X\},\;\;
\pt_{0}{=}\,\{X\,{\mid}\,\thetaup_{0}(X)\,{=}{-}X\} 
\end{equation*}
 be the corresponding Cartan decomposition.
Then $\hg_{0}{=}(\hg_{0}\,{\cap}\,\kt_{0})\,{\oplus}\,(\hg_{0}\,{\cap}\,\pt_{0})$ and the 
$\ad(H)$'s with $H\,{\in}\,\hg_{0}\,{\cap}\,\kt_{0}$ (resp. with $H\,{\in}\,\hg_{0}{\cap}\,\pt_{0}$)
have purely imaginary (resp. real)
eigenvalues.
Indeed 
\begin{equation*}
 \bil_{0}(X,Y)={-}\kll(X,\thetaup_{0}(Y)),\;\;\forall X,Y\in\gt_{0}
\end{equation*}
is a scalar product on $\gt_{0}$ and, if $H\,{\in}\,\kt_{0}$, then
$\ad(H)$ is \mbox{$\bil_{0}$-antisymmetric} on $\gt_{0}$, and thus has purely
imaginary eigenvalues,  while, if $H\,{\in}\,\pt_{0}$, 
then $\ad(H)$  is 
\mbox{$\bil_{0}$-symmetric} on $\gt_{0}$, and thus has real
eigenvalues.
\par
The complexification $\hg$ of $\hg_{0}$ is a Cartan subalgebra of $\gt$. We consider
the corresponding root system
$\Rad\,{=}\,\Rad(\gt,\hg)$ and the corresponding real form
\begin{equation*}
 \hr=\{H\in\hg\mid \alphaup(H)\in\R,\;\forall\alphaup\in\Rad\}
\end{equation*} 
of $\hg$. Then 
\begin{equation*}
 \hr=(\hg_{0}{\cap}\,\pt_{0})\oplus {i}(\hg_{0}{\cap}\,\kt_{0})
\end{equation*}
and $i\hr$ is a Cartan subalgebra of 
the compact form
\begin{equation*}\tag{$*$}
 \gu=\kt_{0}\oplus\,i\pt_{0},
\end{equation*}
which 
is invariant with respect to the anti-$\C$-linear involution $\sigmaup$ pointwise fixing $\gt_{0}$.
 By \cite[Ch.IX, \S{3}, Prop.3]{Bou82}, we can find a Chevalley system
$\{(X_{\alphaup},H_{\alphaup})\}_{\alphaup\in\Rad}$ such that 
\begin{equation*}
 \gu\,{=}\,i\hr\,{\oplus}\,{\sum}_{\alpha\in\Rad}
\{\lambdaup{X}_{\alpha}{+}\bar{\,\lambdaup}X_{-\alpha}\,{\mid}\,\lambdaup\,{\in}\,\C\}
\end{equation*}
is the compact form described in  
(5) of Thm.\ref{thmchev}. \par
The involution $\sigmaup$ restricts to
an involution of $\hr,$ yielding by duality 
an $\stt\,{\in}\,\Ib(\Rad)$, so that, 
for suitable coefficients $k_{\alphaup}{\in}\,\C,$ 
\begin{equation*} \begin{cases}
\sigmaup(H_{\alphaup})=H_{\stt(\alphaup)},\\
 \sigmaup(X_{\alphaup})=k_{\alphaup}X_{\stt(\alphaup)},
 \end{cases} \;\;\forall \alphaup\in\Rad.
\end{equation*}
From ${\sigmaup}^{2}{=}\,\id$\, we get 
\begin{equation*}\tag{$**$}
 k_{\stt(\alphaup)}{\cdot}\,\bar{k}_{\alphaup}=1,\;\;\forall\alphaup\,{\in}\,\Rad.
\end{equation*} \par
Since $\sigmaup$ is an involution of $\gt,$ we have 
\begin{equation*}\tag{$\begin{array}{c} * \\[-7pt] **
\end{array}$}\begin{cases}
 k_{\alphaup}k_{-\alphaup}=1, &\forall\alphaup\in\Rad,\\
 k_{\alphaup+\betaup}={\pm}\,k_{\alphaup} k_{\betaup}, &\text{if}\;\; \alphaup,\betaup,\alphaup{+}\betaup\in\Rad.
 \end{cases}
\end{equation*}
Being
\begin{equation*} 
\begin{cases}
 \sigmaup(X_{\alphaup}{+}X_{-\alphaup})=k_{\alphaup}X_{\stt(\alphaup)}{+}\,k_{-\alphaup}X_{-\stt(\alphaup)}\in\gu,\\
  \sigmaup(i (X_{\alphaup}{-}X_{-\alphaup}))=
  -i k_{\alphaup}X_{\stt(\alphaup)}{+} i\,k_{-\alphaup}X_{-\stt(\alphaup)}\in\gu,
\end{cases}
\end{equation*}
the condition that $\sigmaup(\gu)\,{=}\gu$ is equivalent to $k_{-\alphaup}\,{=}\,\bar{k}_{\alpha}$, which,
by $\big(\begin{array}{c} * \\[-7pt] **
\end{array}\big)$, yields $|k_{\alpha}|\,{=}\,1$,
 for all $\alphaup\,{\in}\,\Rad.$\par
Let us fix a Weyl chamber $C$ and, for each $\alphaup\,{\in}\,\Bz(C)$,  choose a square root
$\lambdaup_{\alphaup}$ of $k_{\alphaup}$. This assignment uniquely extends to an automorphism $\ftt$
in $\Aut_{\C}(\gt)$ keeping $\hg$ 
pointwise fixed, defining coefficients $\lambdaup_{\alphaup}\,{\in}\,\C$
with 
\begin{equation*}
 \ftt(X_{\alphaup})=\lambdaup_{\alphaup}X_{\alphaup},\;\;\forall\alphaup\,{\in}\,\Rad.
\end{equation*}
Then $\{(Y_{\alphaup}{=}\lambdaup_{\alphaup}X_{\alphaup},H_{\alphaup})\}_{\alphaup\in\Rad}$  
satisfies conditions $(1),(2),(3),(4)$
in Thm.\ref{thmchev} and 
\begin{equation*}
 \sigmaup(Y_{\alphaup})=\bar{\,\lambdaup}_{\alphaup} \sigmaup(X_{\alphaup})=\bar{\,\lambdaup}_{\alphaup}
 k_{\alphaup}\cdot\lambdaup_{\stt(\alphaup)}^{-1}Y_{\stt(\alphaup)}.
 \end{equation*}
By $(**)$ and the fact that $\sigmaup^{2}{=}\,\id$, we have 
\begin{equation*}
( \bar{\,\lambdaup}_{\alphaup}
 k_{\alphaup}\cdot\lambdaup_{\stt(\alphaup)}^{-1})^{2}=\bar{k}_{\alphaup}k_{\alphaup}^{2}k_{\stt(\alphaup)}^{-1}=
|{k}_{\alphaup}|^{4}=1.
\end{equation*}
This shows that $\sigmaup(Y_{\alphaup})\,{=}\,{\pm}Y_{\stt(\alphaup)}$  \; and hence
\begin{equation*}
 \gr=\hr\oplus {\sum}_{\alphaup\in\Rad}\langle{Y}_{\alphaup}\rangle_{\R}
\end{equation*}
is a $\sigmaup$-invariant split real form of $\gt.$ From $|\lambdaup_{\alphaup}|{=}1$ and
$\lambdaup_{\alphaup}\bar{\,\lambdaup}_{-\alphaup}\,{=}\,1$ we obtain that
$\lambdaup_{-\alphaup}{=}\,\bar{\,\lambdaup}_{\alphaup}$ and hence 
$\{(Y_{\alphaup},H_{\alphaup})\}_{\alphaup\in\Rad}$ is a Chevalley system and the corresponding
compact form $(5)$ coincides with the $\gu$ of $(*)$. \par
Since different Chevalley systems are conjugated by $\Aut_{\C}(\gt)$, this yields the statement. 
\end{proof}

\begin{prop} \label{p3.2}
If $\sigmaup\,{\in}\,\Inv^{\tauup}(\gr,\hr),$ 
then the restriction of $\tauup$  
to $\gs$ is a Cartan involution, 
yielding the Cartan decomposition 
\begin{equation}\label{e1.7}
 \gs=\ks\oplus\ps,\;\;\text{with}\;\; \begin{cases}
 \ks=\gs\cap\gu=\{X\in\gs\mid\tauup(X)=X\},\\
 \ps=\{X\in\gs\mid \tauup(X)=- X\},
 \end{cases}
\end{equation}
and $\hs$ is a $\tauup$-invariant Cartan subalgebra with
\begin{equation}\vspace{-18pt}
\label{e1.8}
 \hs^{+}\subset\ps,\;\; i\hs^{-}\subset\ks.
\end{equation}
\qed
\end{prop}
\begin{rmk}
 In suitable complex
 matrix representations we can identify $\gr$ with the real and $\gu$ with the
 skew Hermitian matrices of $\gt$. Then \textit{conjugation} with respect to $\gr$
 extends to matrices the usual conjugation of complex numbers, while $\tauup(Z)\,{=}{-}Z^{*}$.
\end{rmk}
\subsection{Automorphisms and involutions pointwise fixing $\hr$}  In this subsection we  describe the
automorphisms
of $\gr$ pointwise fixing $\hr.$ 
\begin{prop}\label{p3.4}
The automorphisms of $\gr$ pointwise fixing $\hr$ form a normal subgroup $\Aut_{0}(\gr,\hr)$
of $\Aut(\gr,\hr)$.
A group isomorphism 
\begin{equation}\label{e2.6}
 \psiup:\Hom(\Z[\Rad],\R^{*}) \to \Aut_{0}(\gr,\hr)
\end{equation}
is naturally defined by associating to $\etaup\,{\in}\,\Hom(\Z[\Rad],\R^{*})$ the unique automorphism
$\psiup_{\etaup}$ of $\gr$ characterized by 
\begin{equation} \label{eq4.9a}\vspace{-18pt}
\begin{cases}
 \psiup_{\etaup}(H)=H, & \forall H\in\hr,\\
 \psiup_{\etaup}(X)=\etaup({\alphaup})X, & \forall\alphaup\in\Rad,\;\forall X\in\gr^{\alphaup}.
\end{cases}
\end{equation}\qed
\end{prop}
\begin{rmk} Let $C\,{\in}\,\Cd(\Rad)$. The corresponding set $\Bz(C)$ of simple positive roots 
 is  
a basis of 
the free $\Z$-module
$\Z[\Rad]$. Hence an $\etaup\,{\in}\,
\Hom(\Z[\Rad],\R^{*})$ is completely determined by its values
on the elements of $\Bz(C),$
that can be arbitrarily assigned. 
\end{rmk}
Denote by
$\Z_{2}^{*}$ the multiplicative subgroup $\{\pm{1}\}$ \index{$\Z_{2}^{*}$}
of $\R^{*}.$ 
\begin{lem}
 The involutions of $\Hom(\Z[\Rad],\R^{*})$ form its subgroup 
\begin{equation}
 \Hom(\Z[\Rad],\Z^{*}_{2}).
\end{equation}
We have 
\begin{equation}
 \Aut_{0}^{\tauup}(\gr,\hr)=\{\sigmaup\,{\in}\,\Aut_{0}(\gr,\hr)\,{\mid}\,\sigmaup(\gu)\,{=}\,\gu\}=\psiup( \Hom(\Z[\Rad],\Z^{*}_{2})).
\end{equation}
\par
The subgroup $ \Hom(\Z[\Rad],\R_{>0})$ is  normal 
and 
\begin{equation}\label{e1.36}
 \Hom(\Z[\Rad],\R^{*})=\Hom(\Z[\Rad],\R_{>0})\ltimes \Hom(\Z[\Rad],\Z_{2}^{*}).
\end{equation}
\end{lem} 
\begin{proof} If $\etaup\,{\in}\,\Hom(\Z[\Rad],\R^{*})$ and $\psiup_{\etaup}(\gu)\,{=}\,\gu,$ then
$\etaup(\alphaup)\,{=}\,\etaup(-\alphaup)\,{=}\,(\etaup(\alphaup))^{-1}$ for all $\alphaup\,{\in}\,\Rad$.
This is equivalent to the fact that $\etaup$ takes values in $\Z_{2}^{*}$.
\par
The last claim follows from the exactness of the  sequence 
\begin{equation*} 
\begin{CD}
 0 \to \!\Hom(\Z[\Rad],\R_{>0}) \!@>>> \!\Hom(\Z[\Rad],\R) \! @>{\sgn{\circ}}>>\!\Hom(\Z[\Rad],\Z_{2}^{*})\! \to 0,
\end{CD}
\end{equation*}
where the first map is inclusion and the second composition with the signature\, 
$\sgn:\R^{*}{\to}\,\Z_{2}^{*}$  ($\sgn(\ttt){=}\ttt/|\ttt|$). 
\end{proof}
Let $(\Z[\Rad])^{*}$ be the \emph{dual lattice} of $\Z[\Rad]$: \index{$(\Z[\Rad])^{*}$}
\begin{equation}
 (\Z[\Rad])^{*}=\{\omegaup\in\hr^{*}\mid (\alphaup\,|\,\omegaup)\in\Z,\;\forall\alphaup\in\Rad\}.
\end{equation}
\begin{ntz}
For $\omegaup\,{\in}\,(\Z[\Rad])^{*}$ and $\vq\,{\in}\,\hr^{*}$, we define the homomorphism 
\index{$\etaup_{(\omegaup,\vq)}$}
\begin{equation} \Z[\Rad]\ni\xiup \to
 \etaup_{\omegaup,\vq}(\xiup)=\exp((\xiup|\vq+i\piup\omegaup))\,{\in}\,\R^{*}.
\end{equation}
To simplify notation, we will write $\etaup_{\omegaup}$ for $\etaup_{\omegaup,0}$. 
\end{ntz}
\begin{lem} The correspondence  
\begin{equation}
 (\Z[\Rad])^{*}\oplus\hr^{*}\ni(\omegaup,\vq)\longrightarrow\etaup_{(\omegaup,\vq)}\in\Hom(\Z[\Rad],\R^{*}) 
\end{equation}
is a surjective homomorphism of abelian groups.\par \index{$\etaup_{\omegaup}$}
The elements of $\Hom(\Z[\Rad],\Z^{*}_{2})$ are of the form 
\begin{equation}\label{e2.13}\vspace{-18pt}
 \etaup_{\omegaup}(\xiup)=(-1)^{(\xiup\,|\,\omegaup)},\;\;\forall\xiup\in\Z[\Rad],\;\;\;
 \text{for some $\omegaup\,{\in}\,(\Z[\Rad])^{*}$.}
\end{equation}
\qed
\end{lem} 
\begin{prop} The isomorphism \eqref{e2.6} restricts to an isomorphism 
\begin{equation}
 \psiup:\Hom(\Z[\Rad],\Z_{2}^{*}) \to \Inv_{0}(\gr,\hr)
\end{equation}
with the subgroup $\Inv_{0}(\gr,\hr)$ of $\Aut(\gr,\hr)$  
consisting of the involutions of $\gr$ pointwise fixing $\hr$. If 
$\omegaup\,{\in}\,(\Z[\Rad])^{*}$, then 

\begin{equation} \label{e1.43}
\begin{cases}
 \phiup_{\omegaup}(H)=H,&\forall H\in\hr,\\
 \phiup_{\omegaup}(X_{\alphaup})=e^{i(\alphaup|\omegaup)}X_{\alphaup},&\forall\alphaup\in\Rad
\end{cases}
\end{equation}
defines an isomorphism $\phiup_{\omegaup}$ of $\gr$ onto $\gt_{\psiup_{\omegaup}}.$ \qed
\end{prop} 
In the next subsection we will construct for each involution $\stt$ of the root system 
an involution in 

$\Inv_{0}(\gr,\hr)$.
Involutions  are associated to  
systems of strongly orthogonal roots. The following Lemma will be useful. 
\begin{lem}\label{l1.19}
 If $\Stt\,{=}\,\{\betaup_{1},\hdots,\betaup_{r}\}$ is any system of 
 strongly orthogonal roots, then we can
 find $\omegaup\,{\in}\,(\Z[\Rad])^{*}$ such that 
\begin{equation}\label{e3.16}
 (\betaup_{i}|\omegaup)=1,\;\forall i=1,\hdots,r.
\end{equation}
\end{lem}
\begin{proof}
It suffices to consider the case where $\Rad$ is irreducible and 
$\Stt$ maximal. We employ  the notation of \S\ref{s2.4}. For the convenience of the reader, we list first the
maximal systems of strongly orthogonal roots for which we will produce a corresponding $\omegaup$. 
For type $\textsc{E}$ we chose the root systems $\Rad(\textsc{E}_{\ell})$, $\ell=6,7,8$ 
described therein.
\begin{align*}
& \Mtt(\textsc{A}_{\ell})=\{\e_{2i-1}{-}\e_{2i}\,{\mid}\, 1{\leq}i{\leq}[(\ell+1)/2]\},\\
& \Mtt(\textsc{B}_{\ell})= 
\begin{cases}
 \{\e_{2i-1}{\pm}\e_{2i}\mid 1{\leq}i{\leq}(\ell/2)\}, & \ell\,{\in}\,2\Z,\\
  \{\e_{2i-1}{\pm}\e_{2i}\mid 1{\leq}i{\leq}(\ell/2)\}\,{\cup}\,\{\e_{\ell}\}, & \ell\,{\notin}\,2\Z,
\end{cases}\\
&\Mtt_{r_{1},r_{2}}(\textsc{C}_{\ell})=\{\e_{2i-1}{+}\e_{2i}\,{\mid}\,1{\leq}i{\leq}r_{1}\}\cup\{\e_{2r_{1}+i}\,{\mid}\,1{\leq}i{\leq}r_{2}\},
& 2r_{1}{+}r_{2}{=}\ell,\\
&\Mtt(\textsc{D}_{\ell})=\{\e_{2i-1}{\pm}\e_{2i}\,{\mid}\,1{\leq}i{\leq}{[\ell/2]},\\
&\Mtt(\textsc{E}_{6})=\{\e_{1}{\pm}\e_{2},\,\e_{3}{\pm}\e_{4}\},\\
&\Mtt(\textsc{E}_{7})=\{\e_{2i-1}{\pm}\e_{2i}\,{\mid}i=1,2,3\}\cup\{\e_{7}{+}\e_{8}\},\\
&\Mtt(\textsc{E}_{8})=\{\e_{2i-1}{\pm}\e_{2i}\,{\mid}\, i=1,2,3,4\},\\
&
\begin{cases}\Mtt(\textsc{F}_{4})=
 \{\e_{1}{\pm}\e_{2},\, \e_{3}{\pm}\e_{4}\},\\
 \Mtt'(\textsc{F}_{4})=
 \{\e_{1}{\pm}\e_{2}\}\cup\{\e_{3}\},
\end{cases}\\
&\Mtt(\textsc{G}_{2})=\{\e_{1}-\e_{2}, \e_{1}{+}\e_{2}{-}2\e_{3}\}.
\end{align*}
\par 
Correspondingly, we may choose 
\begin{equation*} 
\begin{array}{| c | l |}
\hline
\text{type} & \qquad\qquad\omegaup \\[3pt]
\hline
\textsc{A}_{\ell}
& {\sum}_{i=1}^{[(\ell+1)/2}\e_{2i-1}{-}\tfrac{[(\ell+1)/2]}{\ell+1}{\sum}_{i=1}^{\ell+1}\e_{i}\\[3pt]
 \hline
\textsc{B}_{\ell}
& {\sum}_{i=1}^{[(\ell+1)/2]} \e_{2i-1} \\[3pt]
 \hline
 \textsc{C}_{\ell} 
 & \tfrac{1}{2}{\sum}_{i=1}^{\ell}\e_{i}\\[3pt]
 \hline  
 \textsc{D}_{\ell}
 & {\sum}_{i=1}^{[\ell/2]}\e_{2i-1}\\[3pt]
 \hline
\textsc{E}_{6} 
 & \e_{1}+\e_{3} \\[3pt]
 \hline
 \textsc{E}_{7}
 & \e_{1}+\e_{3}+\e_{5}+\e_{7}\\[3pt]
 \hline
  \textsc{E}_{8}
  & \e_{1}+\e_{3}+\e_{5}+\e_{7}\\[3pt]
 \hline
\textsc{F}_{4}
 & \e_{1}+\e_{3}\\[3pt]
 \hline
 \textsc{G}_{2}
 & \e_{1} \\[3pt]
 \hline
\end{array}
\end{equation*}
[For types admitting different maximal systems of strongly orthogonal roots, the same
 $\omegaup$ on the right is adapt for the different choices.]
\end{proof}

\subsection{Lifting involutions of $\Rad$ to involutions in $\Inv^{\tauup}(\gr,\hr)$}\label{s1.5}
 A map 
 $\phiup$ of $\Aut(\gr,\hr)$ 
acts on $\hr^{*}$ by $\phiup^{*}(\xiup)\,{=}\,\xiup\,{\circ}\,\phiup$; as 
$\phiup^{*}(\Rad)\,{=}\,\Rad$, its restriction to $\Rad$ defines an element of
$\Af(\Rad)$, because 
automorphisms of $\gr$ preserve the Killing form.  The map \index{$\rhoup$}
\begin{align}\label{eqrho}
 \rhoup\,{:}\,\Aut(\gr,\hr) \,{\to}\, \Af(\Rad),\;\text{with}\; \rhoup(\phiup)(\alphaup)=
 \alphaup\circ\phiup^{-1},\,\,\,\,\,\,\\
 \notag
\,\, \forall\phiup\,{\in}\,\Aut(\gr,\hr),
 \;\forall\alphaup\,{\in}\,\Rad,
\end{align}
is a group homomorphism.  \par
We show in this subsection that all  
involutions 
in $\Ib(\Rad)$  lift to involutive automorphisms of $(\gr,\hr)$, yielding pairs
$(\gs,\hs)$ with a quasi-split~$\gs$. 
\begin{ntz}
 Having fixed $\stt\,{\in}\,\Ib(\Rad)$, we will denote by $\Invs$ the set \index{$\Invs$}
 of involutions $\sigmaup$ in $\Inv^{\tauup}(\gr,\hr)$ with $\rhoup(\sigmaup)\,{=}\,\stt$.
\end{ntz}

\begin{exam}
  Let us explain a way of lifting 
an involution in $\Ib_{\Wf}(\Rad)$ to an involution of $\Inv^{\tauup}(\gr,\hr)$ 
in the simplest case where $\gr{=}\,\slt_{2}(\R)$ and $\hr$ consists
of its diagonal matrices. We have $\Rad\,{=}\,\{{\pm}(\e_{1}{-}\,\e_{2})\}$ and 
$\Af(\Rad)$ contains the identity and of the involution $\stt\,{=}\,{-}\id.$
\par 
The inner involution
\begin{equation*}
\sigmaup:\begin{pmatrix}
 a & b\\
 c & -a\;
\end{pmatrix}\longrightarrow 
\begin{pmatrix}
 0\; & 1\\
 {-}1& 0 
\end{pmatrix} 
\begin{pmatrix}
 a & b\\
 c & -a\;
\end{pmatrix}
 \begin{pmatrix}
0 & -1 \\
1 & \; 0
\end{pmatrix} =
 \begin{pmatrix}
 {-}a & c\\
 b\; & a\;
\end{pmatrix} \end{equation*}
lifts $\stt$ and yields the real form $(\gs,i\hr)$ with 
\begin{equation*}
 \gs=\left.\left\{ 
\begin{pmatrix}
 ia & z\\
 \bar{z} & {-}ia
\end{pmatrix}\,\right|\, a\in\R,\; z\in\C\right\}.
\end{equation*}
Setting $\lambdaup\,{=}\,\tfrac{1{+}i}{\sqrt{2}}$, we obtain 
\begin{equation*}
\begin{pmatrix}
 \lambdaup & {-}\bar{\,\lambdaup}\\
 \lambdaup & \;\; \bar{\,\lambdaup}
\end{pmatrix}
\begin{pmatrix}
 a & b\\
 c & -a\;
\end{pmatrix}
\begin{pmatrix}
 \lambdaup & {-}\bar{\,\lambdaup}\\
 \lambdaup & \;\; \bar{\,\lambdaup}
\end{pmatrix}= 
\begin{pmatrix}
 i(c{-}b) & 2a{+}i(b{+}c)\\
 2a{-}i(b{+}c) & i(b{-}c)
\end{pmatrix}
\end{equation*}
and thus  the map 
\begin{equation*}
 X \longrightarrow A\,X\,A^{-1},\;\;\;\text{with}\;\;
  A = 
\begin{pmatrix}
 \lambdaup & {-}\bar{\lambdaup}\;\\
\lambdaup & \; {\bar{\lambdaup}}
\end{pmatrix}
\end{equation*}
yields an isomorphism of $\gr$ onto $\gs$.
\end{exam}
We begin by considering involutions belonging to the Weyl group.
\begin{thm} \label{t3.13}
If $\stt\,{\in}\,\Ib_{\Wf}(\Rad),$ 
then we can find a 
$\sigmaup\,{\in}\,\Invs$ with 
$\gs$ isomorphic to $\gr$.
\end{thm} 
\begin{proof} 
By Prop.\ref{p2.7}, we have $\stt\,{=}\,\sq_{\,\betaup_{1}}{\circ}\,\cdots\,{\circ}\,\sq_{\,\betaup_{r}}$ 
for a system $\betaup_{1},\hdots,\betaup_{r}$ of strongly orthogonal
roots. \par
By Lemma\,\ref{l1.19} there is  $\omegaup\,{\in}\,(\Z[\Rad])^{*}$ with
$(\betaup_{i}{|}\,\omegaup)\,{=}\,1$ for $1\,{\leq}\,i\,{\leq}r.$ Define an involution of $\gr$ by
\begin{equation*}\begin{cases}
 \psiup_{\omegaup}(H)=H, &\forall H\,{\in}\,\hr,\\
 \psiup_{\omegaup}(X)=(-1)^{(\alphaup|\omegaup)}X,&\forall X\in\gr^{\alphaup},\;\forall\alphaup\in\Rad.
 \end{cases}
\end{equation*}
Using the notation
of Lemma\,\ref{l1.2}, we set $K_{\betaup_{i}}{=}X_{\betaup_{i}}{+}X_{-\betaup_{i}}$ and
$T_{\betaup_{i}}{=}X_{\betaup_{i}}{-}X_{-\betaup_{i}}$. With
\begin{equation*}
 K=K_{\betaup_{1}}+\cdots+K_{\betaup_{r}},
\end{equation*}
we define
\begin{equation}\label{e1.64}
 \sigmaup=\exp(\tfrac{\piup}{4}\ad(K))\circ\psiup_{\omegaup}\circ\exp(-\tfrac{\piup}{4}\ad(K)).
\end{equation}
Being a conjugate of the involution $\psiup_{\omegaup},$ also this map is an involution of $\gr$,
which, since $K\,{\in}\,\gu,$ leaves $\gu$ invariant. 
Moreover we have 
\begin{equation*} 
\begin{cases}
 \sigmaup(H)=H,\;\;\forall H\in\hr \;\;\text{satisfying}\;\;\betaup_{i}(H)=0\;\;\forall i\,{=}\,1,\hdots,r,\\
 \sigmaup(H_{\betaup_{i}})=\exp(\tfrac{\piup}{4}\ad(K_{\betaup_{i}}))({-}T_{\betaup_{i}})=-H_{\betaup_{i}},
 \;\;\forall i{=}1,\hdots,r,\\
 \sigmaup(T_{\betaup_{i}})=\exp(\tfrac{\piup}{4}\ad(K_{\betaup_{i}}))({-}H_{\betaup_{i}})=T_{\betaup_{i}},
 \;\;\forall i{=}1,\hdots,r,\\
  \sigmaup(K_{\betaup_{i}})=(-1)^{(\betaup_{i}|\omegaup)}K_{\betaup_{i}}={-}K_{\betaup_{i}},
 \;\;\forall i{=}1,\hdots,r.
\end{cases}
\end{equation*} 
The first two lines show that $\sigmaup(\hr)\,{=}\,\hr$ and that 
$\sigmaup\,{\in}\,\Invs$. \par
Let $\phiup_{\omegaup}\,{\in}\,\Aut_{\C}(\gt,\hg)$ be characterized by
\begin{equation*}\begin{cases}
 \phiup_{\omegaup}(H)=H, & \forall H\in\hr,\\
 \phiup_{\omegaup}(X)=i^{(\alphaup|\omegaup)}X, &\forall X\in\gr^{\alphaup},\; \forall \alphaup\in\Rad.
 \end{cases}
\end{equation*}
The composition 
\begin{equation*}
 \phiup=\exp(\tfrac{\piup}{4}\ad(K))\,{\circ}\,\phiup_{\omegaup}\circ\exp(-\tfrac{\piup}{4}\ad(K))
\end{equation*}
belongs to 
$\Aut_{\C}(\gt,\hg)$ and restricts to an isomorphism of  $\gr$ onto $\gs.$
\end{proof}
\begin{ntz} \label{n3.4}
If $\Btt\,{=}\,\{\betaup_{1},\hdots,\betaup_{r}\}$ 
is a system of strongly orthogonal roots, we set 
\begin{equation*}
\begin{cases}
 \sq_{\,\Btt}\,{=}\,\sq_{\,\betaup_{1}}\,{\circ}\,\cdots\,{\circ}\,\sq_{\,\betaup_{r}},\\
 K_{\Btt}\,{=}\,K_{\betaup_{1}}\,{+}\,\cdots\,{+}\,K_{\betaup_{r}},\\
 \quad\text{and, 
having fixed an $\omegaup\,{\in}\, (\Z[\Rad])^{*}$ with $(\betaup_{i}|\,\omegaup)\,{=}\,1$
for $1{\leq}i{\leq}r$,}\\
  \sq_{\,\Btt}^{\sharp}\,{=}\,
  \exp(\tfrac{\pi}{4}\ad(K_{\Btt})\,{\circ}\,\psiup_{\omegaup}\,{\circ}\,
   \exp(-\tfrac{\pi}{4}\ad(K_{\Btt})
   \end{cases}
\end{equation*}
\end{ntz}\par\vspace{2pt}
\begin{thm}\label{t3.14} For $C\,{\in}\,\Cd(\Rad)$ and 
$\varepsilon\,{\in}\,\Af(\Rad,C)$ 
there is a unique automorphism 
${\epi^{\sharp}}$ in $\Aut^{\tauup}(\gr,\hr)$ with 
\begin{equation*}
 {\epi^{\sharp}}(\{X_{\alphaup}\mid \alphaup\in\Bz(C)\})=\{X_{\alphaup}\mid \alphaup\in\Bz(C)\}\quad
 \text{and $\rhoup({\epi^{\sharp}})\,{=}\,\varepsilon.$}
\end{equation*} \par 
 If  $\varepsilon$ is an involution of $\Rad\,$,
then 
${\epi^{\sharp}}$ is an involution in $\Inv^{\tauup}_{\!\epi}(\gr,\hr).$  \par
 If $\sigmaup_{1},\sigmaup_{2}\,{\in}\,\Inv^{\tauup}_{\!\epi}(\gr,\hr)$,
 then
 $\gt_{\sigmaup_{1}}$ and $\gt_{\sigmaup_{2}}$ are isomorphic. 
\end{thm} 
\begin{proof}
 By \cite[Ch.VIII,\S{5},n.$^{\mathrm{o}}$1,Prop.1]{Bou82}, for  $\epi\,{\in}\,\Af(\Rad,C)$
 there is  
 a unique ${\epi^{\sharp}}\,{\in}\,\Aut(\gr,\hr)$ such that 
\begin{equation} \label{e1.65}
\begin{cases}
 {\epi^{\sharp}}(H_{\alphaup})=H_{\varepsilon(\alphaup)},\\
 {\epi^{\sharp}}(X_{\alphaup})=X_{\varepsilon(\alphaup)},
\end{cases}\quad \forall\alphaup\,{\in}\,\Bz(C).
\end{equation}
Clearly ${\epi^{\sharp}}{\in}\,\Aut^{\tauup}(\gr,\hr)$ and, if $\epi\,{\in}\,\Ib(\Rad,C),$ then  
this ${\epi^{\sharp}}$ is an involution. Since 
$\Rad_{\,\;\bullet}^{\varepsilon}{=}\,\emptyset$,
the last claim is a special case  of Lemma\,\ref{l3.3},
that will be independently proved later~on.
\end{proof} 
\begin{defn}\label{d2.1}
 A real form $\gt_{0}$ of $\gt$ is called \emph{quasi-split} if there is a Borel subalgebra $\bt$ of $\gt$
 such that  
 $\sigmaup(\bt)\,{=}\,\bt,$ 
 for the $\sigmaup\,{\in}\,\Inv_{\C}^{\!\vee}(\gt)$ pointwise fixing $\gt_{0}$.
\end{defn}
\begin{prop}\label{p1.23}
 A real form of $\gt$ is 
{quasi-split} iff it is
isomorphic to
 $\gs$, for a $\sigmaup\,{\in}\,\Invs$ 
 with $\stt\,{\in}\,\Ib^{*}(\Rad)$.\qed\end{prop}
 
\begin{exam} Split Lie algebras are also quasi-split.
 The Lie algebras $\so(n,n)$ and $\so(n,n{+}1)$ are split, while $\su(n,n)$ and $\su(n,n{+}1)$ are
 quasi-split, but not split. 
 Any semisimple complex Lie algebra, thought as a \textit{real Lie algebra}, is
 quasi-split,
 but not split: e.g. $\slt_{n}(\R)\,{\oplus}\,\slt_{n}(\R)$ and $\slt_{n}(\C)$ are real forms, 
 the first split and the second only quasi-split, 
 of $\slt_{n}(\C)\,{\oplus}\,\slt_{n}(\C)$.
\end{exam}
\begin{thm}\label{t3.15}
 For every $\stt\,{\in}\,\Ib(\Rad)$ we can find $\sigmaup\,{\in}\,\Invs$ with 
$\gs$ quasi-split.
\end{thm} 
\begin{proof} 
By Thm.\ref{t3.13} and Thm.\ref{t3.14}
we can restrain to the case where $\Rad$ is irreducible and  where, according to Prop.\ref{p2.7} and Prop.\ref{p2.34},
\begin{equation*}
 \stt=\varepsilon\circ\sq_{\,\betaup_{1}}\circ\cdots\circ\sq_{\,\betaup_{r}}=\varepsilon\circ\sq_{\,\Btt},
\end{equation*}
with $\id{\neq}\epi\,{\in}\,\Ib(\Rad,C),$ for a Weyl chamber $C$ and a system 
$\Btt\,{=}\,\{\betaup_{1},\hdots,\betaup_{r}\}$, with $r{\geq}1$, 
of strongly orthogonal roots in $\Rd{\circ}{\varepsilon}\,{\cap}\,\Rd{\bullet}{\stt}.$
Since we assumed that $\varepsilon\,{\neq}\,\id$,\, the irreducible 
root system $\Rad$ is of one of the types $\textsc{A}_{\ell},\textsc{D}_{\ell},\textsc{E}_{6}$. \par 
With the ${\epi^{\sharp}}$ of \eqref{e1.65} 
we obtain ${\epi^{\sharp}}(X_{\betaup}){=}{\pm}X_{\betaup}$ for all
$\betaup\,{\in}\,\Rad_{\,\;\circ}^{\epi}.$ Let us check the signs in the different cases.

\smallskip\noindent
\textsc{Type $\textsc{A}_{\ell}$.}\quad We can assume that 
\begin{equation*}
 \Bz(C)=\{\alphaup_{i}{=}\,\e_{i}{-}\e_{i+1}\mid 1\leq{i}\leq\ell\}
\end{equation*}
and 
$\epi$ is defined by $\epi(\e_{i}){=}{-}\e_{\ell{+}2-i}$ ($1{\leq}i{\leq}\ell{+}1$),
 so that $\epi(\alphaup_{i}){=}\alphaup_{\ell{+}1{-}i},$
 for $1{\leq}i{\leq}\ell$
and 
\begin{equation*}\Rad_{\,\;\circ}^{\epi}{=}\{{\pm}(\e_{i}{-}\e_{\ell{+}2-i})\,{\mid}\,1{\leq}i{\leq}\ell{+}1\}.
\end{equation*}
We claim that
\begin{equation*} \tag{$*$}
\begin{cases}
 {\epi^{\sharp}}(X_{\betaup})=X_{\betaup},&\forall\betaup\,{\in}\,\Rad_{\,\;\circ}^{\epi},\;\;
 \text{if $\ell$ is odd,}\\
{\epi^{\sharp}}(X_{\betaup})=-X_{\betaup},&\forall\betaup\,{\in}\,\Rad_{\,\;\circ}^{\epi},\;\;
 \text{if $\ell$ is even.} 
\end{cases}
\end{equation*} 
We prove $(*)$ by writing each positive root in $\Rad_{\,\;\circ}^{\epi}$
as a sum of roots in $\Bz(C)$ and arguing by recurrence 
on the number of summands. \par
For $2i{<}\ell,$  
$\betaup_{i}{\coloneqq}\,\e_{i}{-}\e_{\ell+2-i}{=}\,\alphaup_{i}{+}\cdots{+}\alphaup_{\ell{+}1{-}i}$ and\footnote{We set by
recurrence $[Y_{1}]=Y_{1}$, $\hdots$,
$[Y_{1},Y_{2},\hdots,Y_{h}]=[[Y_{1},[Y_{2},\hdots,Y_{h}]]$ for
$Y_{i}{\in}\gt.$} 
$X_{\betaup_{i}}$ is a multiple of 
$$[X_{\alphaup_{i}},X_{\alphaup_{i+1}},{\hdots},X_{\alphaup_{\ell-i}},X_{\alphaup_{\ell+1-i}}].$$
\par If $\ell\,{=}\,2m{+}1$ is odd, then $\betaup_{m+1}{=}\,\alphaup_{m+1}$ and
$X_{\betaup_{m+1}}{=}\,X_{\alphaup_{m+1}}$ is fixed by $\epi$.\par
If $\ell=2m$ is even, then $\betaup_{m}{=}\,\alphaup_{m}{+}\,\alphaup_{m+1},$ 
$X_{\betaup_{m}}{=}{\pm}[X_{\alphaup_{m}},X_{\alphaup_{m+1}}]$ 
and therefore
$\epi(X_{\betaup_{m}})\,{=}{-}X_{\betaup_{m}}$ follows from
$\epi([X_{\alphaup_{m}},X_{\alphaup_{m+1}}]){=}[X_{\alphaup_{m+1}},X_{\alphaup_{m}}]
{=}{-}[X_{\alphaup_{m}},X_{\alphaup_{m+1}}]$. \par
For smaller $i$, we get
\begin{equation*}
 {\epi^{\sharp}}([X_{\alphaup_{i}},X_{\alphaup_{i+1}},{\hdots},X_{\alphaup_{\ell-i}},X_{\alphaup_{\ell+1-i}}]
){=}[X_{\alphaup_{\ell+1-i}},X_{\alphaup_{\ell-i}},{\hdots},X_{\alphaup_{i+1}},X_{\alphaup_{i}}].
\end{equation*}
Since $[X_{\alphaup_{j}},X_{\alphaup_{h}}]\,{=}\,0$ for $|j{-}h|\,{\neq}\,1,$ we obtain 
\begin{align*}
 [X_{\alphaup_{i}},X_{\alphaup_{i+1}},{\hdots},X_{\alphaup_{\ell-i}},X_{\alphaup_{\ell+1-i}}]=
 [X_{\alphaup_{i}},[[X_{\alphaup_{i+1}},{\hdots},X_{\alphaup_{\ell-i}}],X_{\alphaup_{\ell+1-i}}]] \\
 =[X_{\alphaup_{\ell+1-i}},[[X_{\alphaup_{i+1}},{\hdots},X_{\alphaup_{\ell-i}}],X_{\alphaup_{i}}]]
\end{align*}
by Jacobi's identity. This allows to argue recursively, noting that 
when $\ell$ is odd (resp. even)
the smaller number of $X_{\alphaup_{j}}$'s in the inner bracket 
is odd (resp. even).\par\smallskip\noindent
\textsc{Type $\textsc{D}_{\ell}$.}\quad We can assume that 
\begin{equation*}
 \Bz(C)=\{\alphaup_{i}{=}\e_{i}{-}\e_{i+1}\,{\mid}\,1{\leq}{i}{<}\ell\}\cup\{\alphaup_{\ell}{=}\e_{\ell-1}{+}\e_{\ell}\}
\end{equation*}
and that $\epi(\e_{i}){=}\e_{i}$ for $1{\leq}i{<}\ell$ and $\epi(\e_{\ell}){=}{-}\e_{\ell}$,
so that $\epi(\alphaup_{i}){=}\alphaup_{i}$ for $1{\leq}i{\leq}(\ell{-}2)$, $\epi(\alphaup_{\ell-1}){=}\alphaup_{\ell},$
$\epi(\alphaup_{\ell}){=}\alphaup_{\ell-1}.$ Then 
\begin{equation*}
 \Rad_{\,\;\circ}^{\epi}=\{{\pm}(\e_{i}{\pm}\e_{j})\,{\mid}\, 1{\leq}i{<}j{<}\ell\}.
\end{equation*} 
We obtain 
\begin{equation*} \tag{$*$}
{\epi^{\sharp}}(X_{\betaup})=X_{\betaup},\;\;\forall\betaup\,{\in}\,\Rad^{\epi}_{\,\;\circ}.
\end{equation*}
This is trivial for the roots of the form $\e_{i}{-}\e_{j}$, with $1{\leq}i{<}j{<}\ell$.
We have 
\begin{equation*}
 [[X_{\alphaup_{\ell-1}},X_{\alphaup_{\ell-2}}],X_{\alphaup_{\ell}}]=N_{\alphaup_{\ell-1},\alphaup_{\ell-{2}}}
 N_{\alphaup_{\ell-1}+\alphaup_{\ell-{2}},\alphaup_{\ell}}X_{\e_{\ell-2}+\e_{\ell-1}},
\end{equation*}
with $N_{\alphaup_{\ell-1},\alphaup_{\ell-{2}}}
 N_{\alphaup_{\ell-1}+\alphaup_{\ell-{2}},\alphaup_{\ell}}\,{=}\,\pm{1}\,{\neq}\,{0}$ 
and then 
\begin{equation*}
 {\epi^{\sharp}}( [[X_{\alphaup_{\ell-1}},X_{\alphaup_{\ell-2}}],X_{\alphaup_{\ell}}]=
  [[X_{\alphaup_{\ell}},X_{\alphaup_{\ell-2}}],X_{\alphaup_{\ell-1}}]= [[X_{\alphaup_{\ell-1}},X_{\alphaup_{\ell-2}}],X_{\alphaup_{\ell}}],
\end{equation*}
the last equality being a consequence of the Jacobi identity, shows that 
\begin{equation*}
 {\epi^{\sharp}}(X_{\e_{\ell-2}+\e_{\ell-1}})=X_{\e_{\ell-2}+\e_{\ell-1}}.
\end{equation*}
For $1{\leq}i{<}j{\leq}{\ell{-}2}$, 
the Lie product $[[X_{\e_{i}-\e_{\ell-2}},X_{\e_{\ell-2}+\e_{\ell-1}}],X_{\e_{\ell_{j}}-\e_{\ell-1}}]$
is a nonzero multiple of $X_{\e_{i}{+}\e_{j}}$. Thus 
\begin{equation*}
 {\epi^{\sharp}}([[X_{\e_{i}-\e_{\ell-2}},X_{\e_{\ell-2}+\e_{\ell-1}}],X_{\e_{\ell_{j}}-\e_{\ell-1}}])=
 [[X_{\e_{i}-\e_{\ell-2}},X_{\e_{\ell-2}+\e_{\ell-1}}],X_{\e_{\ell_{j}}-\e_{\ell-1}}]
\end{equation*}
shows that ${\epi^{\sharp}}(X_{\e_{i}+\e_{j}})\,{=}\,X_{\e_{i}+\e_{j}}$ also for $1{\leq}i{<}j{\leq}{\ell{-}2}$,
completing the proof of~$(*)$.
\par\smallskip\noindent
\textsc{Type $\textsc{E}_{6}$}\quad Let us take the root system $\Rad'(\textsc{E}_{6})$ 
and the canonical basis of simple roots $\Bz'(\textsc{E})$ of \eqref{rE6a}, \eqref{bE6a}.
Then we can take $\epi$ defined by $\epi(\e_{i}){=}{-}\e_{7-i}$ for $1{\leq}i{\leq}6$
and $\epi(\e_{7}){=}{-}\e_{8},$ $\epi(\e_{8})\,{=}\,{-}\e_{7},$ so that 
\begin{equation*} \epi(\alphaup_{i})=
\begin{cases}
 \alphaup_{6-i}, &\text{if $i=1,2,4,5$},\\
 \alphaup_{i}, & \text{if $i=3,6$.}
\end{cases}
\end{equation*}
We have 
\begin{gather*} 
 \Rad_{\,\;\circ}^{\epi}=\Rad_{\,\;\circ}^{\epi \prime}\cup\Rad_{\,\;\circ}^{\epi\prime\prime}\cup
\{{\pm}(\e_{7}{-}\e_{8})\}\\
 \text{with}\;\begin{cases}
\Rad_{\,\;\circ}^{\epi\prime}=\{{\pm}(\e_{1}{-}\e_{6}),\;{\pm}(\e_{2}{-}\e_{5}),\;{\pm}(\e_{3}{-}\e_{4})\},\\
\Rad_{\,\;\circ}^{\epi\prime\prime}
 \{\tfrac{1}{2}({\pm}(\e_{1}{-}\e_{6}){\pm}(\e_{2}{-}\e_{5}){\pm}(\e_{3}{-}\e_{4}){\pm}(\e_{7}{-}\e_{8}))\}.
 \end{cases}
\end{gather*}
We claim that also in this case 
\begin{equation*}\tag{$*$}
 {\epi^{\sharp}}(X_{\betaup})=X_{\betaup},\;\;\forall\betaup\in\Rad_{\,\;\circ}^{\epi}.
\end{equation*}

For the roots of $\Rad_{\,\;\circ}^{\epi\prime}$ we reduce to the case of $A_{5}$. \par
Next we note that $X_{(\e_{7}{-}\e_{8})}$ is a multiple of
$[X_{\alphaup_{6}},X_{\e_{1}-\e_{6}},X_{\e_{2}-\e_{5}},X_{\e_{3}-\e_{4}},X_{\alphaup_{6}}],$
 which is left invariant by ${\epi^{\sharp}}$. \par
 All $X_{\alphaup}$ with $\alphaup\,{\in}\,\Rad_{\,\;\circ}^{\epi\prime\prime}$
  are nonzero multiples of Lie products $[Y_{1},\hdots,Y_{h}]$ with
 $Y_{i}{\in}\,\{X_{\alphaup}\,{\mid}\,\alphaup\,{\in}\,\Rad_{\,\;\circ}^{\epi\prime}\,{\cup}\,{\{\pm}\alphaup_{6}\}
 \,{\cup}\,\{{\pm}(\e_{7}{-}\e_{8})\}\}$, 
 which are invariant
 by ${\epi^{\sharp}}$.  This proves $(*)$ also in the case of $\textsc{E}_{6}$.
\par \smallskip
In the cases $\textsc{A}_{\ell}$ with odd $\ell$,  $\textsc{D}_{\ell}$ for all $\ell$ and  $\textsc{E}_{6}$, 
the map ${\epi^{\sharp}}$ of \eqref{e1.65} and $\sq_{\,\Btt}^{\sharp}$ of Notation\,\ref{n3.4} commute and hence 
\begin{equation*}
 \sigmaup={\epi^{\sharp}}\circ\sq_{\,\Btt}^{\sharp}
 \in\Inv(\gr,\hr) \;\;\text{satisfies}\;\; \rhoup(\sigmaup)=\stt.
\end{equation*}
\par
If $\Rad$ is of type $\textsc{A}_{\ell}$ with an odd $\ell\,{\geq}\,3$, we define the involution ${\epi^{\sharp}}'$
by 
\begin{equation*} 
\begin{cases}
 {\epi^{\sharp}}'(H_{\alphaup_{i}})=H_{\alphaup_{\ell+1-i}}, &\text{for $i=1,\hdots,\ell$},\\
 {\epi^{\sharp}}'(X_{\alphaup_{i}})=(-1)^{(\alphaup_{i}|\omegaup)}X_{\alphaup_{\ell+1-i}},&
 \text{for ${i}=1,\hdots,\ell$},
\end{cases}
\end{equation*}
where $\omegaup{\in}(\Z[\Rad])^{*}$ and $(\betaup_{i}|\omegaup) {=} 1$
for $i=1,\hdots,r$. Then ${\epi^{\sharp}}'$ and $\sq_{\,\Btt}^{\sharp}$  commute and 
\begin{equation*}
  \sigmaup={\epi^{\sharp}}'\circ\sq_{\,\Btt}^{\sharp}\in\Inv^{\tauup}(\gr,\hr) \;\;\text{satisfies}\;\; \rhoup(\sigmaup)=\stt.
\end{equation*}
\par
Let $\phiup\,{\in}\,\Aut_{\C}(\gt,\hg)$ be characterized by
\begin{equation*} 
\begin{cases}
 \phiup(H)=H, & \forall H\in\hg,\\
 \phiup(X_{\alphaup})=i^{(\alphaup|\omegaup)}X_{\alphaup}, & \forall\alphaup\in\Rad.
\end{cases}
\end{equation*}
Define $\gt_{0}$ equal to $\gt_{{\epi^{\sharp}}}$ in the case $\textsc{A}_{\ell}$ with $\ell$ even
and to $\phiup(\gt_{{\epi^{\sharp}}})$ in the other cases.
Using the notation in the proof of Thm.\ref{t3.13}, we find that 
\begin{equation*}
 \exp(\tfrac{\pi}{4}\ad(K_{\Btt}))(\gt_{0})=\gs.
\end{equation*}
By construction, $\gs$ is isomorphic to~$\gt_{{\epi^{\sharp}}}$ and hence
quasi-split.
\end{proof} 

\begin{thm}
 Let $\sigmaup\,{\in}\,\Inv^{\tau}$, $\stt\,{=}\,\rhoup(\sigmaup)\,{\in}\,\Ib(\Rad)$. \par
 If $\Btt\,{=}\,\{\betaup_{1},\hdots,\betaup_{r}\}$ is  a system of strongly orthogonal roots in $\Rd{\circ}{\stt}$,
 then we can find $\sigmaup'\,{\in}\,\Invs$, with 
\begin{equation*}\begin{cases}
 \rhoup(\sigmaup')=\sq_{\,\Btt}\,{\circ}\,\stt,\\
 \gt_{\sigmaup'}\;{\simeq}\; \gs.
 \end{cases}
\end{equation*}
\end{thm} 
\begin{proof}
Using Notation\,\ref{n3.4}, we set $\sigmaup'\,{=}\,\sq_{\,\Btt}^{\sharp}{\circ}\,
\sigmaup$. The argument in the proof of Thm.\ref{t3.13} yields the statement. 
\end{proof}

\subsection{Structure of the groups $\Aut(\gr,\hr)$ and $\Aut^{\tauup}(\gr,\hr)$}\quad\par
We have (see e.g. \cite[VIII,\S{5.2},Prop.2]{Bou82}):
\begin{prop} \label{propaut}
With the maps described above,
we get exact sequences 
\begin{equation} \label{e1.24}
\begin{CD}
 0 \to 
 \Hom(\Z[\Rad],\R^{*})
 @>{\psiup}>>\Aut(\gr,\hr) @>{\rhoup}>> \Af(\Rad) \to 0.
\end{CD}
\end{equation} 
\begin{equation}
 \begin{CD}
 0 \to 
 \Hom(\Z[\Rad],\Z_{2}^{*})
 @>{\psiup}>>\Aut^{\tauup}(\gr,\hr) @>{\rhoup}>> \Af(\Rad) \to 0.
\end{CD}
\end{equation}
\end{prop} \vspace{5pt}
\begin{proof}
By Thm.\ref{t3.15}, $\rhoup$ is onto.  
It is clear from the definition that in both cases 
$\psiup$ is injective and 
that its image is the kernel of $\rhoup$. 
\end{proof}
\subsection{Cartan subalgebras of quasi-split semisimple real Lie algebras}
We keep the notation and the assumptions above.
From Thm.\ref{t3.1}, Thm.\ref{t3.13}, Thm.\ref{t3.14}, Thm.\ref{t3.15} we obtain 
\begin{thm} \label{t3.20}
Let $\gs$ be a quasi-split real form of $\gt,$ with $\sigmaup\,{\in}\,\Inv(\gr,\hr)$
and $\epi\,{=}\,\rhoup(\sigmaup)$. 
Then
the conjugacy classes of Cartan subalgebras of $\gs$
 are in one-to-one correspondence with the conjugacy classes of sets of strongly orthogonal
 roots  in $\Rad_{\,\;\circ}^{\epi}$.
 \end{thm}
 \begin{proof}
Theorem\,\ref{t3.1} established a correspondence between conjugacy classes
 of Cartan subalgebras and involutions $\sigmaup$ in $\Inv^{\tauup}(\gr,\hr)$.
 By Theorem,\ref{t3.15} every involution $\stt$ of the root system $\Rad$ 
 lifts to an involution 
$\sigmaup\,{\in}\,\Invs$ for which
$\gs$ is quasi-split. By using Theorems\,\ref{t3.13} and \ref{t3.14} , and  
Propositions\,\ref{p2.7} and \ref{p2.34}, we get 
\begin{equation*}
 \stt=\varepsilon\circ\sq_{\,\betaup_{1}}\circ\cdots\circ\sq_{\,\betaup_{r}},
\end{equation*}
 with $\epi\,{\in}\,\Ib(\Rad,C)$ and a system $\{\betaup_{1},\hdots,\betaup_{r}\}$ of strongly orthogonal
 roots in $\Rad^{\epi}_{\,\;\circ}$. 
 Since real forms of $\hg$ are in one-to-one correspondence with the involutions of $\Rad$,
 this correspondence yields the statement of the Theorem.
 \par

 \end{proof}
If $\epi$ is the identity, then $\gt_{{\epi^{\sharp}}}{=}\,\gr$; thus the theorem above characterizes
in particular
the conjugacy classes of Cartan subalgebras of the split real form.

\section{Real forms and their Cartan subalgebras} \label{s3}
Let $\gt$ be a semisimple complex Lie algebra. 
By Thm.\ref{t3.1}
any 
pair $(\gt_{0},\hg_{0})$, consisting of a real form $\gt_{0}$
of 
$\gt$ and its Cartan subalgebra $\hg_{0}$, is conjugate with a pair
 $(\gs,\hs)$, for some $\sigmaup\,{\in}\,\Inv^{\tauup}(\gr,\hr)$. \par
 By Thm.\ref{t3.15}, for every $\stt\,{\in}\,\Ib(\Rad)$ the set $\Invs$
 is not empty and contains a $\sigmaup_{0}$ for which $\gt_{\sigmaup_{0}}$
 is quasi-split.
 \par
All Cartan subalgebras $\hs$ with $\sigmaup\,{\in}\,\Inv^{\tau}_{\!\stt}(\gr,\hr)$ coincide 
as real abelian Lie subalgebras of $\gt$  and (cf. \eqref{e2.7}) are equal~to 
\begin{equation}\label{eq4.1}
\hsi=\hsi^{+}\oplus{i}\hsi^{-},\;\; \text{with}\;\; 
\begin{cases}
 \hsi^{+}=\{H\,{\in}\,\hr\,{\mid}\, \stt^{*}(H)\,{=}\,H\},\\
  \hsi^{-}=\{H\,{\in}\,\hr\,{\mid}\, \stt^{*}(H)\,{=}\,{-}H\}.
\end{cases}
\end{equation} 
\par
 \subsection{Involutions pointwise fixing $\hst$}  
Fix a Chevalley system $\{(X_{\alphaup},H_{\alphaup})\}_{\alphaup\in\Rad}$, and,  
for $\alphaup,\betaup,\alphaup{+}\betaup\,{\in}\Rad$, let 
$N_{\alphaup,\betaup}$
be  the coefficients defined in \eqref{e1.6}. Given an involution $\stt\,{\in}\,\Ib(\Rad)$,
we introduce the notation
\begin{equation}\label{eq4.2} 
 \Fq_{\stt}(\Rad)=\left\{f:\Rad\,{\to}\,\Z_{2}^{*}\left| 
\begin{aligned}
 f(\alphaup){\cdot}f(\stt(\alphaup))=1,\quad\forall\alphaup\,{\in}\,\Rad,\quad\;\\
 N_{\alphaup,\betaup}f(\alphaup{+}\betaup)\,{=}\,N_{\stt(\alphaup),\stt(\betaup)}f(\alphaup)f(\betaup),\\
 \text{if $\alphaup,\betaup,\alphaup{+}\betaup\,{\in}\,\Rad$}
\end{aligned}\right\},\right.
 \end{equation}

\begin{rmk}
 Since $\stt$ is an automorphism of $\Rad,$ we have $N_{\stt(\alphaup),\stt(\betaup)}{=}{\pm}N_{\alphaup,\betaup}$
 whenever  $\alphaup,\betaup,\alphaup{+}\betaup\,{\in}\,\Rad.$ 
\end{rmk}
In the Proposition below we show that  
\index{$\Homs$}
\begin{equation}\label{eq4.3}
\left\{
 \begin{aligned}\Homs=\{\etaup\,{\in}\,\Homz\,{\mid}\,\psiup_{\etaup}{\circ}\,\sigmaup\,{\in}\,\Inv^{\tau}_{\!\stt}(\gr,\hr),\\
 \;\forall \sigmaup\,{\in}\,\Inv^{\tau}_{\!\stt}(\gr,\hr)\}\end{aligned}\right.
\end{equation} 
is a subgroup of $\Homz$, acting 
transitively on $\Invs$. 
\begin{prop} \label{p4.2} Given  $\stt$ in $\Ib(\Rad)$, 
there is a one-to-one correspondence $\sigmaup\,{\leftrightarrow}\,{f}_{\sigmaup}$ 
between involutions in
$\Invs$ 
and elements of $\Fq_{\stt}(\Rad)$, via  
\begin{equation} \label{eq4.4}
\begin{cases}
 \sigmaup(H_{\alphaup})=H_{\stt(\alphaup)},\\
  \sigmaup(X_{\alphaup})=f_{\sigmaup}(\alphaup)X_{\stt(\alphaup)},
\end{cases}\quad \forall\alphaup\in\Rad.
\end{equation}\par
We have 
\begin{align}\label{eq4.5} 
&\Hom_{\stt}(\Z[\Rad],\Z_{2}^{*}){=}\big\{\etaup\,{\in}\,\Hom(\Z[\Rad],\Z_{2}^{*}){\mid} \,
 \etaup(\alphaup){\cdot}\etaup(\stt(\alphaup)){=}1,\;\forall\alphaup{\in}\Rad\big\},
 \intertext{and, if $\sigmaup_{0}$ is any element of $\Invs$, then}
 \label{eq4.6}
&\Invs
 =\{\psiup_{\etaup}\circ\sigmaup_{\sigmaup_{0}}\mid \etaup\in \Hom_{\stt}(\Z[\Rad],\Z_{2}^{*})\}.
\end{align}
\end{prop}\begin{proof} If $\sigmaup\,{\in}\,\Invs,$ then 
$\sigmaup(H_{\alphaup})\,{=}\,H_{\stt(\alphaup)}$  and 
$\sigmaup(X_{\alphaup})\,{=}\,{\pm}X_{\stt(\alphaup)},$ uniquely determining 
a map $f_{\sigmaup}{:}\Rad\,{\to}\,\Z_{2}^{*}$ for which 
\eqref{eq4.4} holds. We have
\begin{equation*}\tag{$*$}
 X_{\alphaup}=\sigmaup^{2}(X_{\alphaup})=
f_{\sigmaup}(\alphaup)f_{\sigmaup}(\stt(\alphaup))X_{\alphaup},\;\;\forall\alphaup\in\Rad,
\end{equation*}
and, for  $\alphaup,\betaup,\alphaup{+}\betaup\in\Rad$, we obtain
\begin{align*}\tag{$**$}
 \sigmaup([X_{\alphaup},X_{\betaup}])&=N_{\alphaup,\betaup}\sigmaup(X_{\alphaup\,{+}\,\betaup})
 =N_{\alphaup,\betaup} f_{\sigmaup}(\alphaup\,{+}\,\betaup) X_{\stt(\alphaup\,{+}\,\betaup)}\\
 &\begin{aligned}
 =[\sigmaup(X_{\alphaup}),\sigmaup(X_{\betaup})]=f_{\sigmaup}(\alphaup)f_{\sigmaup}({\betaup})
 [X_{\stt(\alphaup)},X_{\stt(\betaup)}]  \qquad\\
 =N_{\stt(\alphaup),\stt(\betaup)}f_{\sigmaup}(\alphaup)f_{\sigmaup}({\betaup})X_{\stt(\alphaup+\betaup)}.
 \end{aligned}
\end{align*}
This shows that $f_{\sigmaup}{\in}\,\Fq_{\stt}(\Rad)$. 
Vice versa, if $f_{\sigmaup}{\in}\,\Fq_{\stt}(\Rad),$ then 
$(*)$ and $(**)$ hold true, showing that 
\eqref{eq4.4} defines an element of
$\Invs$.   \par
Finally, \eqref{eq4.5} follows because the elements of $\Hom_{\stt}(\Z[\Rad],\R^{*})$
are quotients of pairs of functions in $\Fq_{\stt}(\Rad).$ 
 \end{proof}
 \begin{ntz}
 If $\sigmaup_{1},\sigmaup_{2}{\in}\,\Invs,$  then we set 
\begin{equation}
 (\sigmaup_{2}/\sigmaup_{1})=\etaup
\end{equation}
 for the unique
 $\etaup\,{\in}\,\Homs$ such that 
 $\sigmaup_{2}{=}\psiup_{\etaup}\,{\circ}\,\sigmaup_{1}$.
 \end{ntz}
 \begin{lem}\label{l3.3} Let $\sigmaup_{1},\sigmaup_{2}{\in}\,\Invs$
 and $C$ be an $S$-chamber for $\stt$. If 
\begin{equation}\label{e2.11}
(\sigmaup_{2}/\sigmaup_{1})(\alphaup)=1, \;\;\forall\alphaup\in\Bz_{\bullet}^{\stt}(C),
\end{equation}
then 
$(\gt_{\sigmaup_{1}},\hst)$ and $(\gt_{\sigmaup_{2}},\hst)$
are 
isomorphic in $\gt$. 
\end{lem}
\begin{proof} 
Define $\omegaup\,{\in}\,(\Z[\Rad])^{*}$ by requiring that 
\begin{equation*}
 (\alphaup|\omegaup)= 
\begin{cases}
 0, & \text{if $\alphaup\in\Bz(C)$ and $\sigmaup_{1}(X_{\alphaup})=\sigmaup_{2}(X_{\alphaup})$},\\
 1, & \text{if $\alphaup\in\Bz(C)$ and $\sigmaup_{1}(X_{\alphaup})=-\sigmaup_{2}(X_{\alphaup})$}.
\end{cases}
\end{equation*}
Since $\sigmaup_{1}(X_{\alphaup})\,{=}\,{-}\sigmaup_{2}(X_{\alphaup})$ iff 
$\sigmaup_{1}(X_{\stt(\alphaup)})\,{=}\,{-}\sigmaup_{2}(X_{\stt(\alphaup)})$, because $\sigmaup_{1}$ and $\sigmaup_{2}$
are involutions of $\Invs$,
the condition that $(\alphaup|\omegaup)\,{=}\,0$ for all $\alphaup\,{\in}\,\Bz_{\bullet}^{\stt}(C)$
implies, by \eqref{eq1.21},  
that 
\begin{equation*}
 (\alphaup|\omegaup)=(\stt(\alphaup)|\omegaup),\quad\forall\alphaup\in\Rad.
\end{equation*}
This is a sufficient condition in order that  the $\C$-linear $\phiup$ defined by
\begin{equation*} 
\begin{cases}
 \phiup(H)=H, &\forall H\in\hg,\\
 \phiup(X_{\alphaup})=i^{(\alphaup|\omegaup)}X_{\alphaup},&\forall\alphaup\in\Rad,
\end{cases}
\end{equation*}
be an element of $\Aut_{\C}(\gt,\hg)$, yielding  an
isomorphism of $\gt_{\sigmaup_{2}}$ onto $\gt_{\sigmaup_{1}}$. 
\end{proof}
 \begin{prop} \label{p4.4}
 Let $\stt\,{\in}\,\Ib(\Rad)$, 
$\omegaup\,{\in}\,(\Z[\Rad])^{*}$. Then 
\begin{itemize}
 \item[(1)] $\etaup_{\omegaup}\,{\in}\,\Homs$
if and only if there is $C\,{\in}\,\Cd(\Rad)$ such that
\begin{equation}\label{eq4.9}
 (\alphaup\,{-}\,\stt(\alphaup)\mid\omegaup)\in{2}\Z,\;\;\forall\alphaup\in\Bz^{\stt}_{\,\star}(C);
\end{equation}
\item[(2)] $\etaup_{\omegaup}\,{\in}\,\Homs$
if and only if 
\begin{equation}\label{eq4.9'}\tag{\ref{eq4.9}$'$}
 (\alphaup\,{-}\,\stt(\alphaup)\mid\omegaup)\in{2}\Z,\;\;\forall\alphaup\in\Rad ;
 \end{equation}
 \item[(3)] if $\omegaup\,{\in}\,(\Z[\Rad])^{*}$ and $\stt(\omegaup)\,{=}\,{-}\omegaup$, then
$\etaup_{\omegaup}{\in}\,\Hom_{\stt}(\Z[\Rad],\Z_{2}^{*})$;
\item[(4)] If $\etaup_{\omegaup}\,{\in}\,\Hom_{\stt}(\Z[\Rad],\Z_{2}^{*})$, then also
\begin{equation}\label{e3.9}
 \omegaup'=\tfrac{1}{2}(\omegaup-\stt(\omegaup)).
\end{equation}
belongs to $\Hom_{\stt}(\Z[\Rad],\Z_{2}^{*})$ 
and, for every $\sigmaup\,{\in}\,\Invs$, the real forms 
$\gt_{\etaup_{\omegaup}\circ\sigmaup}$ and $\gt_{\etaup_{\omegaup'}\circ\sigmaup}$
are isomorphic. 
 \end{itemize}
\end{prop} 
\begin{proof}
Condition \eqref{eq4.9'} 
translates the one characterizing  $\Homs$ in \eqref{eq4.5} and
is equivalent to 
\begin{equation*}
  (\alphaup\,{-}\,\stt(\alphaup)\mid\omegaup)\in{2}\Z,\;\;\forall\alphaup\in\Rad^{\stt}_{\,\;\star},
\end{equation*}
since \eqref{eq4.9'} is trivially satisfied by the roots in
 $\Rad_{\,\;\circ}^{\stt}{\cup}\,\Rad_{\,\;\bullet}^{\stt}$ .
 This yields (2) and also (1), because \eqref{eq4.9'} is valid by linearity if it holds true for the
 element of a basis $\Bz(C)$ of positive simple roots. \par 
If $\omegaup\,{\in}\,\Hom(\Z[\Rad],\Z_{2}^{*})$ and
$\stt(\omegaup)\,{=}\,{-}\omegaup,$ then condtion \eqref{eq4.9'} is trivially satisfied 
and hence $\omegaup\,{\in}\,\Hom_{\stt}(\Z[\Rad],\Z_{2}^{*}).$\par 
\par
Assume next that $\omegaup\,{\in}\,\Hom_{\stt}(\Z[\Rad],\Z_{2}^{*}).$
If $\omegaup'$ is defined by \eqref{e3.9}, then 
\begin{equation*}
 (\alphaup|\omegaup)=(\alphaup|\omegaup'),\;\;\forall \alphaup\in\Rad_{\,\;\bullet}^{\stt},
\end{equation*}
and thus 
$\gt_{\etaup_{\omegaup}\circ\sigmaup}$ and $\gt_{\etaup_{\omegaup'}\circ\sigmaup}$
are isomorphic by Lemma\,\ref{l3.3}.
\end{proof}
\begin{cor}\label{c4.5}
Let $\att$ be the involution $\att(\alphaup){=}{-}\alphaup$ of $\Rad$. Then\par
\centerline{$\etaup_{\omegaup}\,{\in}\,\Homs$  
for every $\omegaup\,{\in}\,(\Z[\Rad])^{*}$.}
\end{cor} 
\begin{proof}
 Indeed condition \eqref{eq4.9} is trivially satisfied in this case.
\end{proof}
\begin{ntz}
 Given $\stt\,{\in}\,\Ib(\Rad),$ we denote by $(\Z[\Rad])^{*}_{\stt}$ the set of $\omegaup$ in
 $(\Z[\Rad])^{*}$ satisfying the equivalent conditions 
 \eqref{eq4.9}, \eqref{eq4.9'} of Prop.\ref{p4.4}. \index{$(\Z[\Rad])^{*}_{\stt}$}
Note that $(\Z[\Rad])^{*}_{\stt}$ is a subgroup of the abelian group $(\Z[\Rad])^{*}$.
\end{ntz}
\subsection{Imaginary compact and noncompact roots}
We recall from Prop.\ref{p3.2} that, if $\sigmaup\,{\in}\,\Invs,$ then the restriction  of
$\tauup$ to $\gs$ is a Cartan involution, yielding the Cartan decomposition $\gs\,{=}\,\ks\,{\oplus}\,\ps$. 
\begin{lem}
 Let $\stt\,{\in}\,\Ib(\Rad)$,  $\sigmaup\,{\in}\,\Invs$ and $\alphaup\,{\in}\,\Rad_{\,\;\bullet}^{\stt}.$ Then $ X_{\alphaup}{+}\sigmaup(X_{\alphaup})$, and $i(X_{\alphaup}{-}\sigmaup(X_{\alphaup}))$
 belong either both to $\ks$ or both to $\ps.$ 
The intersection of $\gs$ with $\lt_{\alphaup}{=}\langle{H}_{\alphaup},{X}_{\alphaup},X_{-\alphaup}\rangle_{\C}$
is isomorphic in the first case to $\su(2),$ in the second to $\slt_{2}(\R).$ 
\end{lem} 
\begin{proof} If $\alphaup\,{\in}\,\Rad_{\,\;\bullet}^{\stt}$,
then  $\sigmaup(X_{\alphaup}){=}{\pm}X_{-\alphaup}.$ Then 
\begin{align*}
\sigmaup(X_{\alphaup})\,{=}\;X_{{-}\alphaup} \;\;\Longrightarrow X_{\alphaup}{+}\sigmaup(X_{\alphaup})=
X_{\alphaup}
{+}X_{{-}\alphaup},\,i(X_{\alphaup}{-}\,\sigmaup(X_{\alphaup}))=i(X_{\alphaup}{-}X_{{-}\alphaup})
\in \ks,\\
\sigmaup(X_{\alphaup})\,{=}-X_{{-}\alphaup} \Longrightarrow 
X_{\alphaup}{+}\sigmaup(X_{\alphaup})=X_{\alphaup}{-}X_{-\alphaup},\,
i(X_{\alphaup}{-}\,\sigmaup(X_{\alphaup}))=i(X_{\alphaup}{+}X_{-\alphaup})
\in \ps.
\end{align*}
\par
Setting $K_{\alphaup}{=}X_{\alphaup}{+}\,X_{-\alphaup},$\;
$T_{\alphaup}{=}X_{\alphaup}{-}\,X_{-\alphaup}$, we have 
\begin{equation*} 
 [H_{\alphaup},K_{\alphaup}]=2T_{\alphaup},\;\;
  [H_{\alphaup},T_{\alphaup}]=2K_{\alphaup},\;\;
  [T_{\alphaup},K_{\alphaup}]=-2H_{\alphaup}.
\end{equation*}

When $\sigmaup(X_{\alphaup})\,{=}\,X_{-\alphaup}$, we have 
$\lt_{\alphaup}{\cap}\gs\,{=}\, \langle{}K_{\alphaup},\, iH_{\alphaup},\, iT_{\alphaup}\rangle_{\R}$ 
and in this basis 
the adjoint action is represented by the matrices 
\begin{equation*}
 \ad(K_{\alphaup})= 
\begin{pmatrix}
 0 & 0 & 0\\
 0 & 0 & -2\\
 0 & 2 & 0
\end{pmatrix},\;\;\ad(iH_{\alphaup})= 
\begin{pmatrix}
 0 & 0 & -2\\
 0 & 0 & 0\\
 2 & 0 & 0
\end{pmatrix},\;\; \ad(iT_{\alphaup}) = 
\begin{pmatrix}
 0 & 2 & 0\\
 -2 & 0 & 0\\
 0 & 0 & 0
\end{pmatrix},
\end{equation*}
showing that
$\lt_{\alphaup}{\cap}\gs\,{\simeq}\,\so(3)\,{\simeq}\,\su(2)$. 
When $\sigmaup(X_{\alphaup})\,{=}{-}X_{-\alphaup}$, setting
$Z_{\alphaup}{=}\,\tfrac{i}{2}(K_{\alphaup}{+}\,H_{\alphaup}),$  
$Z_{-\alphaup}{=}\,\tfrac{i}{2}(K_{\alphaup}{-}\,H_{\alphaup}),$ we obtain
\begin{equation*}
 [T_{\alphaup},Z_{\alphaup}]=2Z_{\alphaup},\;\;  [T_{\alphaup},Z_{-\alphaup}]=-2Z_{-\alphaup},\;\;
 [Z_{-\alphaup},Z_{\alphaup}]=T_{\alphaup},
\end{equation*}
proving that $\lt_{\alphaup}{\cap}\gs\,{=}\,\langle Z_{-\alphaup},\,Z_{\alphaup},\, T_{\alphaup}\rangle_{\R}$ 
is isomorphic to $\slt_{2}(\R)$.
 \end{proof}
  \begin{defn}
Having fixed $\stt$ in $\Ib(\Rad),$ the choice of a $\sigmaup\,{\in}\,\Invs$ 
yields a 
partition of $\Rad_{\,\;\bullet}^{\stt}$: \index{$\Rad_{\,\;\bullet}^{\sigmaup}$} \index{$\Rad_{\,\;\oast}^{\sigmaup}$}
\begin{equation}
 \Rad_{\,\;\bullet}^{\stt}=\Rad_{\,\;\bullet}^{\sigmaup}\cup \Rad_{\,\;\oast}^{\sigmaup},\;\;
 \text{with}\;\;
\begin{cases}
 \Rad_{\;\;\bullet}^{\sigmaup}=\{\alphaup\in\Rad_{\,\;\bullet}^{\stt}\mid \sigmaup(X_{\alphaup})
 =X_{-\alphaup}\},\\
\Rad_{\;\;\oast}^{\sigmaup}=\{\alphaup\in\Rad_{\,\;\bullet}^{\stt}\mid \sigmaup(X_{\alphaup})
 =-X_{-\alphaup}\}. 
\end{cases}
\end{equation}
We call \emph{compact} the roots in $\Rad_{\,\;\bullet}^{\sigmaup}$ and
\emph{noncompact} those in $\Rad_{\,\;\oast}^{\sigmaup}$.
 \end{defn}
 We also introduce, for $C\,{\in}\,\Cd(\Rad),$ the notation 
\begin{equation}
 \Bz^{\sigmaup}_{\bullet}(C)=\Bz(C)\,{\cap}\,\Rad^{\sigmaup}_{\,\;\bullet},\;\;\; 
  \Bz^{\sigmaup}_{\oast}(C)=\Bz(C)\,{\cap}\,\Rad^{\sigmaup}_{\,\;\oast}.
 \end{equation}
\begin{rmk}\label{rm4.7}
 For the $\sigmaup$ obtained in Thms.\ref{t3.13},\ref{t3.15}
 the real form $\gs$ contains $T_{\betaup_{i}}{=}X_{\betaup_{i}}{-}
 X_{-\betaup_{i}}$ and $iK_{\betaup_{i}}{=}i(X_{\betaup_{i}}{+}
 X_{-\betaup_{i}}).$ Therefore 
\begin{equation*}
 \sigmaup(X_{\betaup_{i}}{-}X_{-\betaup_{i}})=X_{\betaup_{i}}{-}X_{-\betaup_{i}},\;\;
  \sigmaup(X_{\betaup_{i}}{+}X_{-\betaup_{i}})={-}X_{\betaup_{i}}{-}X_{-\betaup_{i}}
  \Rightarrow \sigmaup(X_{\betaup_{i}})=-X_{-\betaup_{i}},
\end{equation*}
showing that $\betaup_{1},\hdots,\betaup_{r}\,{\in}\,\Rad_{\,\;\oast}^{\sigmaup}$. \par 
Note that the actual composition of $\Rd{\bullet}{\sigmaup}$ and $\Rd{\oast}{\sigmaup}$
depends on the choice of the Chevalley basis, although, as we show in the subsection below,
their cardinality only depends on $\sigmaup$.
\end{rmk} 
\begin{exam} \label{ex4.8}
In the root system $\Rad(\textsc{B}_{2})$, to the real form corresponding to
the involution $\att\,{=}\,\sq_{\e_{1}}{\circ}\sq_{\e_{2}}$ and isomorphic to the split real form we
can associate two anti-involutions $\sigmaup_{1},\sigmaup_{2},$ with 
\begin{equation*} \begin{cases}
 \Rd{\bullet}{\sigmaup_{1}}\,{=}\,\{{\pm}\e_{1}\},\\
 \Rd{\oast}{\sigmaup_{1}}\,{=}\,\{{\pm}\e_{2}\}\,{\cup}\,\{{\pm}\e_{1}{\pm}\e_{2}\},
 \end{cases}\,\;\;
  \begin{cases}
 \Rd{\bullet}{\sigmaup_{2}}\,{=}\,\{{\pm}\e_{2}\},\\
 \Rd{\oast}{\sigmaup_{2}}\,{=}\,\{{\pm}\e_{1}\}\,{\cup}\,\{{\pm}\e_{1}{\pm}\e_{2}\}.
 \end{cases}
\end{equation*}
 
\end{exam}
\subsection{Imaginary roots and Cartan decompositions}
Let $\sigmaup\,{\in}\,\Invs$. 
We associate to $\stt$ and $\sigmaup$ the nonnegative integers:
\begin{equation} \quad\left\{
\begin{aligned}
 n_{1}&{=}\tfrac{1}{2}\#\Rad_{\,\;\circ}^{\stt},\\
 n_{2}&{=}\tfrac{1}{4}
 \#\Rad_{\,\;\star}^{\stt},\\
 n_{3}&{=}\tfrac{1}{2}\#\Rad_{\,\;\bullet}^{\stt},\end{aligned}\right.
 \qquad 
\left\{  \begin{aligned}
& \ell{=}\dim_{\R}\hr,\\
 &\ell^{\kt}{=}\dim_{\R}\hsi^{-},\\
 &\ell^{\pt}{=}\dim_{\R}\hsi^{+}.
\end{aligned}\right.\qquad
\left\{ \begin{aligned}
 n_{3}^{\kt}&{=}\tfrac{1}{2}\#\Rad_{\,\;\bullet}^{\sigmaup},\\
 n_{3}^{\pt}&{=}\tfrac{1}{2}\#\Rad_{\,\;\oast}^{\sigmaup},\end{aligned}\right.
\end{equation}
 \begin{prop} With the notation above we have 
\begin{equation} \label{e4.14}
\begin{cases}
 \dim_{\R}\ks=\ell^{\kt}+n_{1}+2n_{2}+2n_{3}^{\kt},\\
 \dim_{\R}\ps=\ell^{\pt}+n_{1}+2n_{2}+2n_{3}^{\pt}.
\end{cases}
\end{equation}
 \end{prop} 
\begin{proof}
 If $\alphaup\,{\in}\,\Rad_{\,\;\circ}^{\stt},$ then $\sigmaup(X_{\alphaup})\,{=}\,\ttt_{\alphaup}X_{\alphaup}$
 and $\sigmaup(X_{-\alphaup})\,{=}\,\ttt_{\alphaup}X_{-\alphaup}$ for some $\ttt_{\alphaup}{\in}\Z_{2}^{*}$
 and $\gs$ contains either $X_{\alphaup},\,X_{-\alphaup}$, if $\ttt_{\alphaup}\,{>}\,0$, 
 or $iX_{\alphaup},\,iX_{-\alphaup}$
if $\ttt_{\alphaup}\,{<}\,0.$ 
If $\ttt_{\alphaup}\,{=}\,1$, then $X_{\alphaup}{+}X_{-\alphaup}{\in}\,\ks$,
$X_{\alphaup}{-}X_{-\alphaup}{\in}\,\ps$; if $\ttt_{\alphaup}\,{=}\,{-}1$, 
then $i(X_{\alphaup}{-}X_{-\alphaup})\,{\in}\,\ks$,
$i(X_{\alphaup}{+}X_{-\alphaup})\,{\in}\,\ps,$ respectively. Therefore every pair of roots ${\pm}\alphaup$
in $\Rad_{\,\;\circ}^{\stt}$ contribute by one unit to the dimension of both $\ks$ and $\ps.$ 
If $\alphaup\,{\in}\,\Rad_{\,\;\star}^{\stt},$ then $X_{\alphaup},\, X_{\stt(\alphaup)},\, X_{-\alphaup},\, X_{-\stt(\alphaup)}$
are linearly independent and, for some $\ttt_{\alphaup}{\in}\,\Z_{2}^{*}$ we have 
$\sigmaup(X_{\pm\alphaup})\,{=}\,\ttt_{\alphaup}X_{\pm\stt(\alphaup)}$.
Then 
\begin{equation*} 
\begin{cases}
 X_{\alphaup}{+}X_{-\alphaup}{+}\ttt_{\alphaup}(X_{\stt(\alphaup)}{+}X_{-\stt(\alphaup)}),\;
 i( X_{\alphaup}{-}X_{-\alphaup}){+}i\ttt_{\alphaup}(X_{\stt(\alphaup)}{-}X_{-\stt(\alphaup)})\,\in\ks,\\
  X_{\alphaup}{-}X_{-\alphaup}{+}\ttt_{\alphaup}(X_{\stt(\alphaup)}{-}X_{-\stt(\alphaup)}),\;
 i( X_{\alphaup}{+}X_{-\alphaup}){+}i\ttt_{\alphaup}(X_{\stt(\alphaup)}{+}X_{-\stt(\alphaup)})\,\in\ps
\end{cases}
\end{equation*}
shows that every quadruple ${\pm}\alphaup,\,{\pm}\stt(\alphaup)$ contained in
$\Rad_{\,\;\star}^{\stt}$ adds two units both to the dimension of $\ks$ and to that of $\ps.$
We already noted that any pair ${\pm}\alphaup$ contained in $\Rad_{\,\;\bullet}^{\stt}$ either adds
two to the dimension of $\ks,$ or adds two to the dimension of $\ps.$ \par
Finally, we already noted that 
the elements of $\hs^{+}$ belong to $\ps,$ those in $i\hs^{-}$ to $\ks.$ 
\end{proof}
 
\subsection{Cayley transforms} \label{s3.2.2}
The construction  in  Prop.\ref{p4.10} below is also called 
\emph{Cayley transform} (see  \cite{HC56}, or \cite[Ch.VI,\S{7}]{Kn:2002}).
\begin{prop}\label{p4.10}
 Let $\stt\,{\in}\,\Ib(\Rad)$, $\betaup\,{\in}\,\Rd{\bullet}{\stt}$, $\stt'\,{=}\,\stt{\circ}\sq_{\,\betaup}$. Then 
 $\stt'\,{\in}\,\Ib(\Rad)$ and
 \begin{equation}\label{e4.15} \left\{
 \begin{aligned}
 & \Rd{\bullet}{\stt'}\,{=}\,\Rd{\bullet}{\stt}\,{\cap}\,\betaup^{\perp},\;\; \Rd{\circ}{\stt}\,{=}\,\Rd{\circ}{\stt'}
 \,{\cap}\,\betaup^{\perp},
 \\
& \hg_{\stt'}=\hg_{\stt'}^{+}\oplus{i}\hg_{\stt'}^{-},\;\;\;\text{with}\;\;
\begin{cases}
 \hg_{\stt'}^{+}=\hst^{+}\oplus\langle{H}_{\betaup}\rangle_{\R},\\
 \hg_{\stt'}^{-}=\{H\in\hst^{-}\mid \betaup(H)=0\}.
\end{cases}
\end{aligned}
\right.
\end{equation}

Assume that
 $\sigmaup\,{\in}\,\Invs$, and  $\betaup\,{\in}\,\Rad_{\,\;\oast}^{\sigmaup}$.
 Set $K_{\betaup}{=}\,X_{\betaup}{+}X_{-\betaup}$. 
Then,  for a $\sigmaup'{\in}\,\Inv^{\tauup}_{\!\stt'}(\gr,\hr)$,   $\exp(\tfrac{\pi}{4}\ad(K_{\betaup}))(\gs)\,{=}\,\gt_{\sigmaup'}$ 
and  
\begin{equation}\label{e4.16}\begin{cases}
 \Rd{\bullet}{\sigmaup'}{=} \big(\Rd{\bullet}{\sigmaup}\cap\betaup^{\Perp}\big)
 \,{\cup}\,\big(\Rd{\oast}{\sigmaup}\cap\betaup^{\perp}\big),\\
  \Rd{\oast}{\sigmaup'}{=} \big(\Rd{\oast}{\sigmaup}\cap\betaup^{\Perp}\big)
  \,{\cup}\,\big(\Rd{\bullet}{\sigmaup}\cap\betaup^{\perp}\big),
  \end{cases}
\end{equation}
where \,\textquotedblleft{$\Perp$}\textquotedblright \, means strong 
orthogonality (cf. Notation\,\ref{perp}).
\end{prop} 
We will say that \textsl{$\sigmaup'$ was obtained from $\sigmaup$ by the Cayley transform
with respect to $\betaup$.}
\begin{proof} If $\alphaup\,{\in}\,\Rd{\bullet}{\stt'}$, then 
\begin{align*}\big(
  {-}\alphaup\,{=}\,\stt'(\alphaup)\,{=}\,\stt(\alphaup\,{-}\,\langle\alphaup\,|\,\betaup\rangle\betaup)=
 \stt(\alphaup)\,{+}\,\langle\alphaup\,|\,\betaup\rangle\betaup\,\big)\;\Rightarrow\; (\alphaup{+}\stt(\alphaup)){+}
 \langle\alphaup\,|\,\betaup\rangle\betaup\,{=}\,0.
\end{align*}
Since $(\alphaup{+}\stt(\alphaup))$ and $\betaup$ are orthogonal (the first is an eigenvector with eigenvalue $1$
and the second with eigenvalue $({-}1)$ of $\stt$) we obtain that $\alphaup\,{\in}\,\Rd{\bullet}{\stt}\,{\cap}\,\betaup^{\perp}$.
The inclusion $\Rd{\bullet}{\stt}\cap\betaup^{\perp}\,{\subseteq}\Rd{\bullet}{\stt'}$ is trivial and thus we have
equality.
Analogously, if $\alphaup\,{\in}\,\Rd{\circ}{\stt}$, we have 
\begin{equation*}\big(
 \alphaup\,{=}\,\stt(\alphaup)\,{=}\,\stt'\,{\circ}\,\sq_{\,\betaup}(\alphaup)\,{=}\,\stt'(\alphaup)-\langle\alphaup\,|\,
 \betaup\rangle
 \betaup\,\big)\;{\Rightarrow}\; \stt'(\alphaup){-}\alphaup\,{=}\,\langle\alphaup\,|\,\betaup\rangle\betaup
\end{equation*}
and this implies that $\alphaup\,{\in}\,\Rd{\circ}{\stt'}$ and $\langle\alphaup\,|\,\betaup\rangle\,{=}\,0$
because $\stt'(\alphaup){-}\alphaup$ and $\betaup$ are orthogonal, being eigenvectors of $\stt'$
for $({-}1)$ and $1$, respectively.  For every $\xiup\,{\in}\,\hr^{*}$ the identity
\begin{align*}
 \xiup{-}\tfrac{1}{2}\langle{\xiup}\,|\,{\betaup}\rangle\betaup\,{=}\,\stt'(\xiup{-}\tfrac{1}{2}\langle\xiup\,|\betaup
 \rangle\betaup)=\stt(\xiup{-}\tfrac{1}{2}\langle\xiup\,|\betaup
 \rangle\betaup)
\end{align*}
yields by duality the equality $\hg_{\stt'}^{+}\,{=}\,\hst^{+}{\oplus}\langle{H}_{\betaup}\rangle_{\R}$ and
by passing to the orthogonal, also $\hg_{\stt'}^{-}\,{=}\,\hst^{-}\,{\cap}\,H_{\betaup}^{\perp}\,{=}\,\hst^{-}\,{\cap}\,
\ker\betaup$.
This completes the proof of \eqref{e4.15}.
Assume now that, for a given $\sigmaup\,{\in}\,\Invs$, $\betaup\,{\in}\,\Rd{\oast}{\sigmaup}.$ 
Since $K_{\betaup}\,{\in}\,\gu\,{\cap}\,\gr$, then    
 $\exp(\tfrac{\pi}{4}\ad(K_{\betaup})(\gs)$  
is invariant under conjugation and under $\tauup$ and
therefore equal to $\gt_{\sigmaup'}$
for some $\sigmaup'{\in}\,\Inv^{\tauup}(\gr,\hr)$. \par
Being $\betaup\,{\in}\,\Rad_{\,\;\oast}^{\sigmaup},$ we have
$T_{\betaup}{=}X_{\alphaup}{-}
 \,X_{-\betaup}\in\gs$ and therefore
\begin{align*}\begin{cases}
 \exp(\tfrac{\pi}{4}\ad(K_{\betaup}))(T_{\betaup})=H_{\betaup},\\ 
  \exp(\tfrac{\pi}{4}\ad(K_{\betaup}))(H)=H,\;\; \text{if $H\in\hr$ and $\betaup(H)=0$.}
  \end{cases}
\end{align*}
Thus $\sigmaup'\,{\in}\,\Inv^{\tauup}_{\stt'}(\gr,\hr)$ and 
\eqref{e4.15} is a decomposition of 
the canonical Cartan subalgebra $\hg_{\sigmaup'}\,{=}\,\hg_{\stt'}$~of~$\gt_{\sigmaup'}$.
It is clear that $\Rad^{\sigmaup}_{\,\;\bullet}\,{\cap}\,\betaup^{\Perp}\,{\subseteq}\,\Rad_{\,\;\bullet}^{\sigmaup'}$
and $\Rad^{\sigmaup}_{\,\;\oast}\,{\cap}\,\betaup^{\Perp}\,{\subseteq}\,\Rad_{\,\;\oast}^{\sigmaup'}$.
If $\alphaup$ is orthogonal, but not strongly orthogonal to $\betaup,$ then the $\betaup$-string through
$\alphaup$ is $(\alphaup{-}\betaup,\,\alphaup,\,\alphaup{+}\betaup).$ 
If $\alphaup\,{\in}\,\Rad_{\,\;\bullet}^{\sigmaup},$
then $\alphaup{\pm}\betaup\,{\in}\,\Rad_{\,\;\oast}^{\sigmaup}$ and hence $iK_{\alphaup{\pm}\betaup}\,{\in}\,\gs$. \par
 The matrix representing the action of 
$\ad(K_{\betaup})$ on 
the invariant subspace $\langle{K}_{\alphaup{-}\betaup},K_{\alphaup},K_{\alphaup{+}\betaup}\rangle_{\C}$\, is 
\begin{equation*} 
\begin{pmatrix}
 0 & N_{\betaup,{-}\alphaup}&0\\
 N_{\betaup,\alphaup{-}\betaup} & 0 & N_{\betaup,-(\alphaup{+}\betaup)}\\
 0 & N_{\betaup,\alphaup} & 0
\end{pmatrix},
\end{equation*}
which equals one of the matrices 
\begin{equation*} 
\pm\begin{pmatrix}
0 & 2 & 0 \\
{-}1\; & 0 & {-}1\;\\
0 & 2 & 0 
\end{pmatrix},\quad\text{or}\quad \pm\begin{pmatrix}
0 & {-}2\; & 0 \\
1 & 0 & {-}1\;\\
0 & 2 & 0 
\end{pmatrix}.
\end{equation*}
In the first case $K_{\alphaup{-}\betaup}{-}K_{\alphaup{+}\betaup}$ is in the kernel, while the restriction to 
the subspace generated by the basis $K_{\alphaup{-}\betaup}{+}K_{\alphaup{+}\betaup},K_{\alphaup}$
has matrix 
\begin{equation*}
 A=\pm 
\begin{pmatrix}
 0 & 2\\
 {-}2\; & 0
\end{pmatrix};
\end{equation*}
in the second case the kernel is generated by $K_{\alphaup{-}\betaup}{+}K_{\alphaup{+}\betaup}$,
while the restriction to 
the subspace generated by the basis $K_{\alphaup{-}\betaup}{-}K_{\alphaup{+}\betaup},K_{\alphaup}$
has still matrix~$A$.  Then $\exp(\tfrac{\piup}{4}\ad(K_{\betaup}))$ has matrix 
\begin{equation*}
 {\pm} 
\begin{pmatrix}
 0 & 1\\
 {-}1\;&0
\end{pmatrix}
\end{equation*}
on a subspace $\langle{K}_{\alphaup{\pm}\betaup},K_{\alphaup}\rangle.$ By applying
$\exp(\tfrac{\piup}{4}\ad(K_{\betaup})$ to $i(K_{\alphaup+\betaup}{\pm}K_{\alphaup-\betaup})$,
we obtain 
that $iK_{\alphaup}\,{\in}\,\gt_{\sigmaup'}$, proving that $\alphaup\,{\in}\,\Rad_{\,\;\oast}^{\sigmaup'}.$ \par
Analogously, starting from an $\alphaup\,{\in}\,\Rad_{\,\;\oast}^{\sigmaup}$ which is orthogonal
but not strongly orthogonal to $\betaup$, we show  that $K_{\alphaup}$ belongs to the image by
$\exp(\tfrac{\piup}{4}\ad(K_{\betaup}))$ of $\langle{K}_{\alphaup+\betaup}{\pm}K_{\alphaup-\betaup}\rangle_{\R}\,
{\subseteq}\,\gs$, and hence to
$\gt_{\sigmaup'}$: this implies that $\alphaup\,{\in}\,\Rad_{\,\;\bullet}^{\sigmaup'}.$ \par
The proof is complete.
\end{proof} 
\begin{rmk}
 If $\sigmaup\,{\in}\,\Invs$ and 
 $\Rd{\bullet}{\sigmaup}\,{=}\,\Rd{\bullet}{\stt},$
  then $\hst$ is a maximally vectorial Cartan subalgebra of $\gs,$ i.e. with
 $\hst^{+}\,{=}\,\hst\,{\cap}\,\ps$, $i\hst^{-}\,{=}\,\hst\,{\cap}\ks$.
\end{rmk} 
\begin{thm}\label{t4.12}
 Let $\sigmaup\,{\in}\,\Invs$ and $\Mtt_{\,\oast}^{\,\sigmaup}$ a maximal system of strongly orthogonal
 roots in $\Rd{\oast}{\sigmaup}$, containing the largest number of long roots. 
 Then 
 we can find 
 $\sigmaup'\,{\in}\,\Inv^{\tauup}(\gr,\hr)$
 such that 
\begin{equation}\label{eq4.18}
 \gt_{\sigmaup'}\simeq\gs\;\;\text{and}\;\; \Rd{\bullet}{\sigmaup'}\,{=}\,\Rd{\bullet}{\stt}\,{\cap}\,
(\Mtt^{\,\sigmaup}_{\,\oast})^{\perp},\; \Rd{\oast}{\sigmaup'}\,{=}\,\emptyset.
\end{equation}
\end{thm} 
\begin{proof} If $\Rd{\oast}{\sigmaup}\,{=}\,\emptyset$, then there is nothing to prove.\par
Assume that $\emptyset\,{\neq}\,\Mtt^{\,\sigmaup}_{\,\oast}\,{=}\,\{\betaup_{1},\hdots,\betaup_{r}\}$,
with $\|\betaup_{i}\|{\geq}\|\betaup_{i+1}\|$ if $1{\leq}i{\leq}r{-}1$. We consider the sequence 
$\sigmaup_{0},\sigmaup_{1},\hdots,\sigmaup_{r}$ in $\Inv^{\tauup}(\gr,\hr)$ in which $\sigmaup_{0}\,{=}\,
\sigmaup$ and, 
 for $1{\leq}i{\leq}{r{-}1}$,  $\sigmaup_{i}$ is obtained from $\sigmaup_{i-1}$ by the Cayley transform with
respect to~$\betaup_{i}$. Note that 
\begin{equation*}\tag{$*$}
 \Rd{\bullet}{\rhoup(\sigmaup_{i})}=\Rd{\bullet}{\stt}\cap{\bigcap}_{1{\leq}j{\leq}i}\betaup_{j}^{\perp},\;\;
 \text{for $i=1,\hdots,r$}.
\end{equation*}
We claim that $\Rd{\bullet}{\sigmaup_{i-1}}\,{\cap}\,\betaup_{i}^{\perp}\,{=}\,\emptyset$.
We argue by contradiction. In fact, if there is an $\alphaup$ in $\Rd{\bullet}{\sigmaup_{i-1}}$ which is 
orthogonal to $\betaup_{i}$, then $\betaup_{i}$ is short and $\betaup'_{i}\,{=}\,\betaup_{i}{+}\alphaup$,
by $(*)$, 
is a long root in $\Rd{\oast}{\sigmaup}$ which is orthogonal to all $\betaup_{j}$ with $1{\leq}j{\leq}i{-}1$ 
and we could complete $\{\betaup_{1},\hdots,\betaup_{i-1},\betaup'_{i}\}$ to a strongly orthogonal system
in $\Rd{\oast}{\sigmaup}$ containing a larger number of long roots than $\Mtt_{\,\oast}^{\,\sigmaup}.$ 
\par
Since $\Rd{\bullet}{\sigmaup_{i-1}}\,{\cap}\,\betaup_{i}^{\perp}\,{=}\,\emptyset$, we obtain by \eqref{e4.16}
\begin{equation*}\tag{$**$}
 \Rd{\oast}{\sigmaup_{i}}=\Rd{\oast}{\sigmaup}\cap{\bigcap}_{1{\leq}j{\leq}i}\betaup_{j}^{\Perp},
 \;\;\text{for $i=1,\hdots,r$}.
\end{equation*}
Since $\Mtt_{\,\oast}^{\,\sigmaup}$ was maximal, $\sigmaup'{=}\sigmaup_{r}$ satisfies \eqref{eq4.18}. 
\end{proof}

\subsection{An isomorphism theorem}
We give now a refinement of Lemma\,\ref{l3.3}, yielding a criterion to classify pairs $(\gt_{0},\hg_{0})$
modulo the action of  $\Aut_{\C}(\gt,\hg).$ 
\begin{thm}\label{t4.13}
 Let $\stt\,{\in}\,\Ib(\Rad)$ and $\sigmaup_{1},\sigmaup_{2}\,{\in}\,\Inv^{\tau}_{\!\stt}(\gr,\hr)$. If 
\begin{equation}\label{e4.18}
 \Rad_{\,\;\bullet}^{\sigmaup_{1}}= \Rad_{\,\;\bullet}^{\sigmaup_{2}}
\end{equation}
then $(\gt_{\sigmaup_{1}},\hst)$ and $(\gt_{\sigmaup_{2}},\hst)$ are conjugated by
an automorphism of
$\Aut_{\C}(\gt,\hg)$.
\end{thm} 
\begin{proof} By \eqref{e4.18} we also have $ \Rad_{\,\;\oast}^{\sigmaup_{1}}= \Rad_{\,\;\oast}^{\sigmaup_{2}}$.
By the assumption, we can find a $C\,{\in}\,\Cd_{\stt}(\Rad)$ with 
\begin{equation*}
 \Bz_{\bullet}^{\sigmaup_{1}}(C)= \Bz_{\bullet}^{\sigmaup_{2}}(C),\;\;\;
 \Bz_{\oast}^{\sigmaup_{1}}(C)=\Bz_{\oast}^{\sigmaup_{2}}(C).
\end{equation*} 
Then $(\sigmaup_{1}{/}\sigmaup_{2})$ is positive on $ \Bz_{\bullet}^{\stt}$ and the statement 
follows by Lemma\,\ref{l3.3}.
\end{proof}
As a corollaries, we obtain
\begin{thm}\label{t4.14}
 Let $\sigmaup_{1},\sigmaup_{2}\,{\in}\,\Inv^{\tau}(\gr,\hr)$. A necessary and sufficient condition for
 $(\gt_{\sigmaup_{1}},\hg_{\sigmaup_{1}})$ and $(\gt_{\sigmaup_{2}},\hg_{\sigmaup_{2}})$ being
 conjugated in $\Aut_{\C}(\gt,\hg)$ is that the pairs 
 $(\Rad_{\,\;\bullet}^{\sigmaup_{1}},\Rad_{\,\;\oast}^{\sigmaup_{1}})$
 and $(\Rad_{\,\;\bullet}^{\sigmaup_{2}},\Rad_{\,\;\oast}^{\sigmaup_{2}})$
 are conjugated in $\Wf(\Rad)$.
\end{thm}
\begin{proof} Necessity follows from the induced symmetry on the root space
and sufficiency from Theorem\,\ref{t4.13}.\par
\end{proof}
\begin{thm}
 Let $\sigmaup\,{\in}\,\Invs$. A necessary and sufficient condition for $\gs$ being quasi-split
 is that $\Rd{\oast}{\sigmaup}$ contains a maximal system of strongly orthogonal roots
 of $\Rd{\bullet}{\stt}$. 
\end{thm}
\begin{proof}
Using the previous Theorem and  Prop.\ref{p1.23} the thesis follows.
\end{proof}
\subsection{$\Sigma$-diagrams}
Let $\stt\,{\in}\,\Ib(\Rad)$ and $C\,{\in}\,\Cd_{\stt}(\Rad)$. \par \index{$\Sigma$-diagram}
We can attach to a $\sigmaup\,{\in}\,\Invs$ a new diagram, which is obtained from an 
$S\!$-diagram for $\stt$ 
by \begin{itemize}
\item
changing from ``$\medbullet$'' to ``$\circledast$'' the 
nodes corresponding
to roots in $\Bz_{\oast}^{\sigmaup}(C)$.
\end{itemize}
By Cor.\ref{c4.5}, when $\Rd{\bullet}{\stt}\,{=}\,\Rad$, 
all $\Sigma$-diagrams obtained by arbitrarily choosing  either $\medbullet$ or $\circledast$ for 
the nodes of any basis of simple roots are admissible.   
The subdiagrams containing only nodes corresponding to roots in $\Bz_{\bullet}^{\stt}(C)$ are special
\textit{Vogan diagrams}. 
By a theorem of Borel and de Siebenthal (see \cite[Thm.6.96]{Kn:2002}, or Prop.\ref{p4.17} below)
we may add the requirement 
\begin{itemize}
 \item Each connected component of $\Bz_{\bullet}^{\stt}(C)$ in the $\Sigma$-diagram contains at most
 one ``$\circledast$''.
\end{itemize}
 A \emph{restricted} $\Sigma$-diagram is one satisfying this extra condition.
 By  Lemmas \ref{l4.15}, \ref{l4.16} below,  we obtain:
\begin{prop}\label{p4.17}
 Every $\sigmaup\,{\in}\Invs$ admits a restricted $\Sigma$-diagram. \qed
\end{prop}
 
\begin{lem}\label{l4.15}
 Let $C$ be a Weyl chamber of an irreducible root system $\Rad$. Then 
\begin{equation}
 (H_{1}\,|\, H_{2})>0, \; \; \forall H_{1},H_{2}\,{\in}\,\overline{C}{\setminus}\{0\}.
\end{equation}
\end{lem} 
\begin{proof}
For each $\alphaup\,{\in}\,\Rad$ let  
$\htt_{\alphaup}$ be the element of $\hr$ with  $(\htt_{\alphaup}\,|\,H)\,{=}\,\alphaup(H)$ for
all $H\,{\in}\,\hr$ (cf. 
\S\ref{s1.1}).
Let $\Bz(C)\,{=}\,\{\alphaup_{1},\hdots,\alphaup_{\ell}\}$ be the basis of simple positive roots for $C.$
Then $\htt_{\alphaup_{1}},\hdots,\htt_{\alphaup_{\ell}}$ is a basis of $\hr$.
Fix $H\,{\in}\,\overline{C}$.  
We claim that the real coefficients $t_{i}$ of the decomposition 
 $H\,{=}\,{\sum}_{i=1}^{\ell}t_{i}\,\htt_{\alphaup_{i}}$ are all positive. 
Let indeed $H_{+}\,{=}\,{\sum}_{t_{i}\geq{0}}t_{i}\htt_{\alphaup_{i}}$ and
$H_{-}\,{=}\,H{-}H_{+}\,{=}\,{\sum}_{t_{i}<0}t_{i}\htt_{\alphaup_{i}}$. Then 
\begin{align*}
 \|H_{-}\|^{2}=(H\,|\,H_{-})-(H_{+}\,|\,H_{-})\,{=}\,{\sum}_{t_{j}<0} t_{j}(H\,|\,\htt_{\alphaup_{j}})\,{-}\,
 {\sum}_{\begin{smallmatrix}
 t_{i}\geq{0}\\
 t_{j}<{0}
\end{smallmatrix}
} t_{i}t_{j} (\htt_{\alphaup_{i}}\,|\,\htt_{\alphaup_{j}})\qquad\\
= 2{\sum}_{t_{j}<0} t_{j}\alphaup_{j}(H)-{\sum}_{
\begin{smallmatrix}
 t_{i}{\geq}0\\
 t_{j}<0
\end{smallmatrix}
} t_{i}t_{j}(\alphaup_{i}|\alphaup_{j})\leq 0.
\end{align*}
Indeed, $H\,{\in}\,C\,{\Rightarrow}\,\alphaup_{j}(H){\leq}0$ and $(\alphaup_{i}\,|\,\alphaup_{j}){\leq}0$ for $i{\neq}j$.
This shows that  $H_{-}{=}0,$ i.e. that $t_{i}{\geq}0$ for all $i\,{=}\,1,\hdots,\ell$. 
Assume by contradiction that $t_{i}{=}0$ for some~$i$. Then  
\begin{equation*}
 0\,{\leq}\, \alphaup_{i}(H)\,{=}\,{\sum}_{j\neq{i}}t_{j}\alphaup_{i}(\htt_{\alphaup_{j}})=\,2{\sum}_{j\neq{i}}t_{j}
 (\alphaup_{i}\,|\,\alphaup_{j})
\end{equation*}
implies that $t_{j}\,{=}\,0$ also for all $j$'s with $(\alphaup_{i}\,|\,\alphaup_{j})\,{\neq}\,0$.
By the assumption that $\Rad$ is irreducible, this would imply that $H\,{=}\,0$.\par
Being $H\,{=}\,{\sum}_{i=1}^{\ell}t_{i}\htt_{\alphaup_{i}}$
with all $t_{i}$ positive, if $H'$ is 
another element of $\overline{C}$, then 
\begin{equation*}
 (H\,|\,H')\,{=}\,{\sum}_{i=1}^{\ell}t_{i}(\htt_{\alphaup_{i}}\,|\,H')\,{=}\,2{\sum}_{i=1}^{\ell}t_{i}\alphaup_{i}(H')
\end{equation*}
implies that $(H\,|\,H')\,{>}\,0$ 
if $H'\,{\neq}\,0$, as  at least one $\alphaup_{i}(H')$ is positive.
\end{proof}
\begin{lem}\label{l4.16}
 Let $\Rad$ be irreducible,
 $\att$ the antipodal involution of $\Rad$ and $\sigmaup\,{\in}\,\Inv_{\att}^{\tauup}(\gr,\hr)$.
 Then we can find a Weyl chamber $C$ such that $\#(\Bz_{\oast}^{\sigmaup}(C)){\leq}1$.
\end{lem} 
\begin{proof} If $\Rad\,{=}\,\Rd{\bullet}{\sigmaup}$, then there is nothing to prove. \par
Suppose that $\Rd{\oast}{\sigmaup}\,{\neq}\,\emptyset$ and 
 let $\Lambda$ be the subset of $\hr$ consisting of the elements $H$ for which 
 $\alphaup(H)$ is an integer for all $\alphaup\,{\in}\,\Rad$ and is even when $\alphaup\,{\in}\,\Rd{\bullet}{\sigmaup}$,
 odd when $\alphaup\,{\in}\,\Rd{\oast}{\sigmaup}$. This set is not empty: if $\Bz(C')$ is any basis of simple
 roots, $\Bz_{\oast}^{\sigmaup}(C')\,{\neq}\,\emptyset$ and 
the unique $H\,{\in}\,\hr$ defined by 
\begin{equation*}
 \alphaup(H)=
\begin{cases}
 0, &\text{if $\alphaup\in\Bz_{\bullet}^{\sigmaup}(C')$},\\
  1, &\text{if $\alphaup\in\Bz_{\oast}^{\sigmaup}(C')$},
\end{cases}
\end{equation*}
belongs to $\Lambda$. 
Take $H_{0}$ in $\Lambda$ of minimal norm and fix a Weyl chamber $C$ with $H_{0}\,{\in}\,\overline{C}$. 
Since $H_{0}\,{\neq}\,0$, there is at least a root $\alphaup$ in $\Bz(C)$ for which $\alphaup(H_{0})\,{>}\,0$.
Let $L$ be the element of $\hr$ with $\alphaup(L)\,{=}\,1,$ $\betaup(L)\,{=}\,0$ for $\betaup\,{\in}\,\Bz(C){\setminus}\{\alphaup\}$.
Then either $L\,{=}\,H_{0}$, or $H_{0}\,{-}\,L$ is a nonzero element of $\overline{C}$. Assume by contradiction that
$L\,{\neq}\,H_{0}$. By Lemma\,\ref{l4.15}
we have $(H_{0}{-}L|L){>0}$ and, since
$H_{0}{-}2L$ belongs to $\Lambda$, 
\begin{equation*}
 \| H_{0}{-}2L\|^{2}=\| H_{0}\|^{2}-4(H_{0}{-}L\,|\,L)<\| H_{0}\|^{2}
\end{equation*}
yields a contradiction. This shows that $H_{0}\,{=}\,L$ and therefore that 
$\Bz_{\oast}^{\sigmaup}(C)\,{=}\,\{\alphaup\},$ $\Bz_{\bullet}^{\sigmaup}(C)\,{=}\,\Bz(C){\setminus}\{\alphaup\}.$
\end{proof}

\begin{rmk}  
 A $\Sigma$-diagram with $\Bz_{\oast}^{\sigmaup}(C){=}\emptyset$ is a \emph{Satake diagram}
 (see e.g. \cite{Ara62, Hel78, OV93}) and vice versa.
\par
By Thm.\ref{t4.12} Satake diagrams classify real forms.
\end{rmk}

\begin{exam} Consider on $\Rad(\textsc{A}_{5})$ the involution $\stt\,{=}\,\sq_{\e_{1}{+}\e_{6}}{\circ}\sq_{\e_{2}}{\circ}
\sq_{\e_{3}}{\circ}\sq_{\e_{4}}{\circ}\sq_{\e_{5}}$. Then
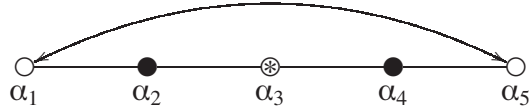
\vspace{12pt}
\begin{gather*}
 \xymatrix @M=0pt @R=2pt @!C=22pt{
\medcirc \ar@/^24pt/@{<->}[rrrr] \ar@{-}[r]
&\medbullet\ar@{-}[r]
&\circledast \ar@{-}[r]
&\medbullet\ar@{-}[r]
&\medcirc\\
\alphaup_1&\alphaup_2&\alphaup_{3}&\alphaup_{4}&\alphaup_{5}
}
\end{gather*} 
is a restricted $\Sigma$-diagram
associated to $\gs\,{\simeq}\,\su(3,3)$, on which a Cartan subalgebra with $\dim_{\R}\hsi^{-}\,{=}\,3$
has been chosen. 
\par
Likewise
 
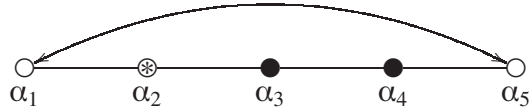
\vspace{30pt}
\begin{gather*}
 \xymatrix @M=0pt @R=2pt @!C=22pt{
\medcirc \ar@/^24pt/@{<->}[rrrr] \ar@{-}[r]
&\circledast\ar@{-}[r]
&\medbullet \ar@{-}[r]
&\medbullet\ar@{-}[r]
&\medcirc\\
\alphaup_1&\alphaup_2&\alphaup_{3}&\alphaup_{4}&\alphaup_{5}
}
\end{gather*} 
is a restricted $\Sigma$-diagram
associated to $\gs\,{\simeq}\,\su(2,4)$, on which a Cartan subalgebra with $\dim_{\R}\hsi^{-}\,{=}\,3$
has been chosen. \par
In both cases $\Rad_{\,\;\bullet}^{\stt}\,{=}\,\{{\pm}(\e_{i}{-}\e_{j})\,{\mid}\,2{\leq}i{<}j{\leq}5\}$.\par
In the first, 
\begin{equation*} 
\begin{cases}
\Rad_{\,\;\bullet}^{\sigmaup}\,{=}\,\{{\pm}(\e_{2}{-}\e_{3}),\,{\pm}(\e_{4}{-}\e_{5})\},\\
\Rad_{\,\;\oast}^{\sigmaup}\,{=}\,\{{\pm}(\e_{2}{-}\e_{4}),\,{\pm}(\e_{2}{-}\e_{5}),\,
 {\pm}(\e_{3}{-}\e_{4}),\,{\pm}(\e_{3}{-}\e_{5})\}.
 \end{cases}
 \end{equation*}
 \par 
 The set of noncompact imaginary roots contains the system $$\Btt\,{=}\,\{\e_{2}{-}\e_{4},\, \e_{3}{-}\e_{5}\}$$
 of strongly orthogonal roots and the successive Cayley transforms with respect to the roots of $\Btt$
 yields a $\sigmaup'\,{\in}\,\Inv_{\stt}^{\tauup}(\gr,\hr)$ with $\Rad_{\,\;\oast}^{\sigmaup'}\,{=}\,\Rad_{\,\;\bullet}^{\sigmaup'}\,{=}\,
 \emptyset$ and therefore $\gt_{\sigmaup'}$ is isomorphic to the quasi-split real form $\su(3,3)$ of
 $\gl_{6}(\C)$.
\par\smallskip
In the second case, 
\begin{equation*} 
\begin{cases}
\Rad_{\,\;\bullet}^{\sigmaup}\,{=}\,\{{\pm}(\e_{3}{-}\e_{4}),\,{\pm}(\e_{3}{-}\e_{5}),\,\,{\pm}(\e_{4}{-}\e_{5})\},\\
\Rad_{\,\;\oast}^{\sigmaup}\,{=}\,\{{\pm}(\e_{2}{-}\e_{3}),\,{\pm}(\e_{2}{-}\e_{4}),\,
{\pm}(\e_{2}{-}\e_{5})\},
 \end{cases}
 \end{equation*}
  and
 the Cayley transform with respect to $\betaup\,{=}\,(\e_{2}{-}\e_{5})$ yields a $\sigmaup'$ with 
 $\Rad_{\,\;\oast}^{\sigmaup'}\,{=}\,\emptyset$,  $\Rad_{\,\;\bullet}^{\sigmaup'}\,{=}\,\{{\pm}\,(\e_{3}{-}\e_{4})\}.$
 \end{exam}
\subsection{Involutions and $\Sigma$-chambers} By \eqref{eq4.1} and
Thm.\ref{t3.1}
the  Cartan subalgebras of the real forms of $\gt$ are classified, modulo conjugation,
by the equivalence classes of  orthogonal involutions $\stt$ of $\Rad$.
By Thm.\ref{t4.13}, the real forms containing a given $\hst$ 
are described by
the group
$\Hom_{\stt}(\Z[\Rad],\Z_{2}^{*})$ or, 
by Prop.\ref{p4.4}, the subgroup
$(\Z[\Rad])^{*}_{\stt}$ of $(\Z[\Rad])^{*}$. Moreover, by Lemma\,\ref{l3.3},
pairs $(\gt_{\sigmaup_{i}},\hst)$, obtained from a given $\sigmaup\,{\in}\,\Invs$
by $\etaup_{\omegaup_{1}},\etaup_{\omegaup_{2}}$, are \textit{equivalent  }
modulo conjugations 
pointwise fixing $\hst$, 
if 
\begin{equation}\label{e4.19}
 (\omegaup_{1}{-}\omegaup_{2}|\alphaup)\in{2}\Z,\;\;\forall \alphaup\,{\in}\,\Rad_{\,\;\bullet}^{\stt}. 
\end{equation}
This reduces to the simpler condition
\begin{equation}\label{e4.19'} \tag{\ref{e4.19}$'$}
 (\omegaup_{1}{-}\omegaup_{2}|\alphaup)\in{2}\Z,\;\;\forall \alphaup\,{\in}\,\Bz_{\bullet}^{\stt}(C),\;\;
 \text{for an $S\!$-chamber $C$ for $\stt$}.
\end{equation}
\par
By Thm.\ref{t3.15}  
every $\stt\,{\in}\,\Ib(\Rad)$ lifts to an $\stt^{\sharp}\,{\in}\,\Invs$
for which $\gt_{\stt^{\sharp}}$ is quasi-split.  All $\sigmaup$ in
$\Invs$ can be obtained from this $\stt^{\sharp}$ 
by the action of $\Homs$. Thus equivalence classes of $(\gs,\hst)$ can
be conveniently described
by the action of $\Homs$, modulo \eqref{e4.19'},
on the basis of simple roots of a $\Sigma$-chamber for $\stt^{\sharp}$.
\par\smallskip
\par
We can restrain our consideration to irreducible root systems.
Thus in the following subsections we will, 
for each irreducible root system $\Rad$ and $\stt\,{\in}\,\Ib(\Rad)$,
\begin{enumerate}
\item  characterize a 
$\Sigma$-diagram of its quasi-split lift $\stt^{\sharp}$,
\item describe the group $\Homs$.
\end{enumerate}
By Thm.\ref{t4.12} we obtain a classification of the real forms 
by finding  
\begin{enumerate}
\item[(3)] the $\stt\,{\in}\,\Ib(\Rad)$ admitting
a lift $\sigmaup\,{\in}\,\Invs$ with $\Rd{\oast}{\sigmaup}\,{=}\,\emptyset$.
\end{enumerate}
\par

Involutions $\stt$ 
in $(3)$ are those for which we can find 
$\etaup\,{\in}\,\Homs$  
which equals $1$ on $\Rd{\bullet}{\stt^{\sharp}}$
and $({-}1)$ on $\Rd{\oast}{\stt^{\sharp}}$. 
\par\medskip
\subsection{Real forms admitting compact Cartan subalgebras}\label{s4.5.1}
Finding $\Sigma$-di\-a\-grams representing pairs $(\gs,\hst)$ with $\gs$ quasi-split
and $\hst$ compact 
is preliminary to the discussion of general pairs, since $\Rd{\bullet}{\stt}$ is a root system 
and, by Prop.\ref{p2.32},  
we can take a $C$ for which the imaginary simple roots 
form any required basis of simple roots for $\Rd{\bullet}{\stt}$.
\par
\begin{prop} There is a
$\sigmaup\,{\in}\,\Inv^{\tauup}_{\att}(\gr,\hr)$ for which $\gs$ is quasi-split and
$\Bz(C)\,{\subset}\,\Rd{\oast}{\sigmaup}$.
\end{prop} 
\begin{proof}  By Proposition\,\ref{p4.4}, having fixed any Weyl chamber $C$,  we can find
a $\sigmaup\,{\in}\,\Hom_{\att}(\Z[\Rad],\Z_{2}^{*})$ such that $\Bz_{\oast}^{\sigmaup}(C)\,{=}\,
\Bz(C)$. Indeed condition \eqref{eq4.9} is trivially satisfied because there are no complex roots for $\att$.
\par 
Thus we need only to prove that, if 
$\Bz(C)\,{\subset}\,\Rd{\oast}{\sigmaup}$ for some $C$, then $\gs$ is quasi-split. \par
In the proof, we can and will restrain
to the case where $\Rad$ is irreducible. Assuming, for each root system, 
 that a given canonical basis
consists of roots of $\Rd{\oast}{\sigmaup}$, we 
show that,
 by a sequence of Cayley transforms, we pass from $\gs$ to its quasi-split equivalent,
 with maximally vector Cartan subalgebra. \smallskip\par
 \noindent 
 $\textsc{A}_{\ell}$. The canonical basis is 
\begin{equation}
 \label{bsA} \Bz(\textsc{A}_{\ell})\,{=}\,\{\alphaup_{i}{=}\e_{i}{-}\e_{i+1}\,{\mid} 
\begin{smallmatrix}
 1{\leq}i{\leq}\ell
\end{smallmatrix}\;\}.
\end{equation}
 The set $\Mtt\,{=}\,\{\betaup_{i}\,{=}\,\e_{2i-1}{-}\e_{2i}\,{\mid}\,1{\leq}i{\leq}r\}$, with $r\,{=}\,[(\ell{+}1)/2]$, 
 is a maximal system of strongly orthogonal roots of $\Rad(\textsc{A}_{\ell})$ and the successive
 Cayley transforms with respect to $\betaup_{1},\hdots,\betaup_{[(\ell{+}1)/2]}$ reduces to
 a $\sigmaup'\,{\in}\,\Inv^{\tauup}(\gr,\hr)$ with $\Rd{\bullet}{\rhoup(\sigmaup')}\,{=}\,\emptyset$.\par

 \smallskip
\par
 \noindent 
 $\textsc{B}_{\ell}$. The canonical basis is 
\begin{equation}
 \label{bsB} \Bz(\textsc{B}_{\ell})\,{=}\,\{\alphaup_{i}{=}\e_{i}{-}\e_{i+1}\,{\mid} 
\begin{smallmatrix}
 1{\leq}i{\leq}\ell{-}1
\end{smallmatrix}\;\}\,{\cup}\,\{\alphaup_{\ell}{=}\e_{\ell}\}.
\end{equation}
 Let $\Mtt\,{=}\,\{\betaup_{2i-1}\,{=}\,\e_{2i-1}{-}\e_{2i},\,\betaup_{2i}\,{=}\,\e_{2i-1}{+}\e_{2i}\,{\mid}
 \, 1{\leq}i{\leq}[\ell/2]\}$. Since \begin{equation*}
 \betaup_{2i}\,{=}\,\alphaup_{2i-1}{+}2\left({\sum}_{j=2i}^{\ell}\alphaup_{j}\right),\;\;\forall j=1,\hdots,[\ell/2],\end{equation*}
 we obtain that $\Mtt\,{\subset}\,\Rd{\oast}{\sigmaup}$. Let $\sigmaup'\,{\in}\,\Inv^{\tauup}(\gr,\hr)$
 be obtained by
the successive Cayley transforms with
 respect to $\betaup_{1},\hdots,\betaup_{\ell}$.  It $\ell$ is even, then 
$\Rd{\bullet}{\rhoup(\sigmaup')}=\emptyset$ and $\gt_{\sigmaup'}$ is the split form.
 If $\ell$ is odd, then $\e_{1}\,{=}\,{\sum}_{i=1}^{\ell}\alphaup_{i}$ is an element of $\Rd{\oast}{\sigmaup}$
 and hence $\Rd{\oast}{\sigmaup'}\,{=}\,\{{\pm}\e_{\ell}\}$ and by the Cayley transform with respect to $\e_{\ell}$
we get from $\sigmaup'$ a $\sigmaup''\,{\in}\,\Inv^{\tauup}(\gr,\hr)$ with 
 $\Rd{\bullet}{\rhoup(\sigmaup'')}=\emptyset$ and $\gt_{\sigmaup''}$ is the split form.
 \smallskip
 \par
 \noindent
 $\textsc{C}_{\ell}$. The canonical basis is 
\begin{equation}
 \label{bsC} \Bz(\textsc{C}_{\ell})\,{=}\,\{\alphaup_{i}{=}\e_{i}{-}\e_{i+1}\,{\mid} 
\begin{smallmatrix}
 1{\leq}i{\leq}\ell{-}1
\end{smallmatrix}\;\}\,{\cup}\,\{\alphaup_{\ell}{=}2\e_{\ell}\}.
\end{equation}
Let $\Mtt\,{=}\,\{\betaup_{i}{=}\,2\e_{i}\,{\mid}\, 1{\leq}i{\leq}\ell\}$. Since 
\begin{equation*}\betaup_{i}\,{=}\,
 2\e_{i}\,{=}\,2\left({\sum}_{j=i}^{\ell-1}\alphaup_{i}\right)\,{+}\,\alphaup_{\ell},\;\;\forall\,1{\leq}i{\leq}\ell{-}1,
\end{equation*}
we have $\Mtt\,{\subset}\,\Rd{\oast}{\sigmaup}$ and the successive Cayley transforms with respect
to the $\betaup_{i}$'s yield a split $\gt_{\sigmaup'}$.
\smallskip
 \par
 \noindent
 $\textsc{D}_{\ell}$. The canonical basis is 
\begin{equation}
 \label{bsD} \Bz(\textsc{C}_{\ell})\,{=}\,\{\alphaup_{i}{=}\e_{i}{-}\e_{i+1}\,{\mid} 
\begin{smallmatrix}
 1{\leq}i{\leq}\ell{-}1
\end{smallmatrix}\;\}\,{\cup}\,\{\alphaup_{\ell}{=}\e_{\ell-1}{+}\e_{\ell}\}.
\end{equation}
Let $\Mtt\,{=}\,\{\betaup_{2i-1}{=}\e_{2i-1}{-}\e_{2i},\, \betaup_{2i}{=}\e_{2i-1}{+}\e_{2i}\,{\mid}\,
1{\leq}i{\leq}\ell/2\}$. Since 
\begin{equation*}
 \betaup_{2i}=\alphaup_{i}{+}2\left({\sum}_{j=i+1}^{\ell-2}\alphaup_{i}\right)+\alphaup_{\ell-1}{+}\alphaup_{\ell},
 \;\; \text{if $i{\geq}1$, $2i{\leq}\ell{-1}$},
\end{equation*}
we have $\Mtt\,{\subset}\,\Rd{\oast}{\sigmaup}$ and hence the successive Cayley transforms with
respect to the $\betaup_{i}$'s yield a $\sigmaup'$ with $\Rd{\bullet}{\rhoup(\sigmaup')}{=}\,\emptyset$,
which is split for $\ell$ even and quasi-split for odd $\ell$.
\smallskip
 \par
 \noindent
 $\textsc{E}_{6}$. We take the canonical basis of $\Rad(\textsc{E}_{6})$ 
\begin{equation}\label{bsE6}
 \Bz(\textsc{E}_{6})\,{=}\,\{\alphaup_{1}{=}\zetaup_{\emptyset},\, \alphaup_{2}{=}\e_{1}{+}\e_{2}\}\,{\cup}\,
 \{\alphaup_{i}=\e_{i-1}{-}\e_{i-2}\,{\mid} 
\begin{smallmatrix}
 3{\leq}i{\leq}6
\end{smallmatrix}\!\}.
\end{equation}
Then $\Mtt\,{=}\,\{\betaup_{2i-1}{=}\e_{2i}{-}\e_{2i-1},\,\betaup_{2i}{=}\e_{2i-1}{+}\e_{2i}\,{\mid}\,i{=}1,2\}$ 
is a maximal system of orthogonal roots 
\begin{equation*}
 \betaup_{4}\,{=}\,\e_{3}{+}\e_{4}\,{=}\, \alphaup_{2}{+}\alphaup_{3}{+}2\alphaup_{4}{+}\alphaup_{5}
\end{equation*}
shows that $\Mtt\,{\subset}\,\Rd{\oast}{\sigmaup}$. Then the successive applications of the 
Cayley transforms with respect to the $\betaup_{i}$'s yield a $\sigmaup'\,{\in}\,\Inv^{\tauup}(\gr,\hr)$,
which is quasi-split because
$\Rd{\bullet}{\rhoup(\sigmaup')}\,{=}\,\emptyset$.
\smallskip
 \par
 \noindent
 $\textsc{E}_{7}$. We take the canonical basis of $\Rad(\textsc{E}_{7})$ 
\begin{equation}\label{bsE7}
 \Bz(\textsc{E}_{6})\,{=}\,\{\alphaup_{1}{=}\zetaup_{\emptyset},\, \alphaup_{2}{=}\e_{1}{+}\e_{2}\}\,{\cup}\,
 \{\alphaup_{i}=\e_{i-1}{-}\e_{i-2}\,{\mid} 
\begin{smallmatrix}
 3{\leq}i{\leq}7
\end{smallmatrix}\!\}.
\end{equation}
Then $\Mtt\,{=}\,\{\betaup_{2i-1}{=}\e_{2i}{-}\e_{2i-1},\,\betaup_{2i}{=}\e_{2i-1}{+}\e_{2i}\,{\mid}\,i=1,2,3\}
\,{\cup}\,\{\betaup_{7}{=}{-}\e_{7}{-}\e_{8}\}$ 
is a maximal system of orthogonal roots of $\Rad(\textsc{E}_{7})$ and 
\begin{equation*}\begin{cases}
 \betaup_{4}\,{=}\,\e_{3}{+}\e_{4}\,{=}\, \alphaup_{2}{+}\alphaup_{3}{+}2\alphaup_{4}{+}\alphaup_{5},\\
 \betaup_{6}\,{=}\,\e_{5}{+}\e_{6}\,{=}\, \alphaup_{2}{+}\alphaup_{3}{+}2(\alphaup_{4}{+}\alphaup_{5}{+}\alphaup_{6})
 {+}\alphaup_{7},\\
 \betaup_{7}\,{=}\,{-}\e_{7}{-}\e_{8}=2\zetaup_{\emptyset}{+}\betaup_{2}{+}\betaup_{4}{+}\betaup_{6}
 \end{cases}
\end{equation*}
shows that $\Mtt\,{\subset}\,\Rd{\oast}{\sigmaup}$. Then the successive  
Cayley transforms with respect to the $\betaup_{i}$'s yield a $\sigmaup'\,{\in}\,\Inv^{\tauup}(\gr,\hr)$,
which is split because
$\Rd{\bullet}{\rhoup(\sigmaup')}\,{=}\,\emptyset$.
\smallskip
 \par
 \noindent
 $\textsc{E}_{8}$. We take the canonical basis of $\Rad(\textsc{E}_{8})$ 
\begin{equation}\label{bsE8}
 \Bz(\textsc{E}_{6})\,{=}\,\{\alphaup_{1}{=}\zetaup_{\emptyset},\, \alphaup_{2}{=}\e_{1}{+}\e_{2}\}\,{\cup}\,
 \{\alphaup_{i}=\e_{i-1}{-}\e_{i-2}\,{\mid} 
\begin{smallmatrix}
 3{\leq}i{\leq}8
\end{smallmatrix}\!\}.
\end{equation}
Then $\Mtt\,{=}\,\{\betaup_{2i-1}{=}\e_{2i}{-}\e_{2i-1},\,\betaup_{2i}{=}\e_{2i-1}{+}\e_{2i}\,{\mid}\,i=1,2,3.4\}$ 
is a maximal system of orthogonal roots of $\Rad(\textsc{E}_{8})$ and 
\begin{equation*}\begin{cases}
 \betaup_{4}\,{=}\,\e_{3}{+}\e_{4}\,{=}\, \alphaup_{2}{+}\alphaup_{3}{+}2\alphaup_{4}{+}\alphaup_{5},\\
 \betaup_{6}\,{=}\,\e_{5}{+}\e_{6}\,{=}\, \alphaup_{2}{+}\alphaup_{3}{+}2(\alphaup_{4}{+}\alphaup_{5}{+}\alphaup_{6})
 {+}\alphaup_{7},\\
 \betaup_{8}\,{=}\e_{6}{+}\e_{7}=\alphaup_{2}{+}\alphaup_{3}{+}2(\alphaup_{4}{+}\alphaup_{5}{+}\alphaup_{6}{+}
 \alphaup_{7}){+}\alphaup_{8}
 \end{cases}
\end{equation*}
shows that $\Mtt\,{\subset}\,\Rd{\oast}{\sigmaup}$. Then the successive  
Cayley transforms with respect to the $\betaup_{i}$'s yield a $\sigmaup'\,{\in}\,\Inv^{\tauup}(\gr,\hr)$,
which is split because
$\Rd{\bullet}{\rhoup(\sigmaup')}\,{=}\,\emptyset$.
\smallskip
 \par
 \noindent
 $\textsc{F}_{4}$. We take the canonical basis of $\Rad(\textsc{F}_{4})$ 
\begin{equation}\label{bsF4}
 \Bz(\textsc{F}_{4})\,{=}\,\{\alphaup_{1}{=}\e_{1}{-}\e_{2},\, \alphaup_{2}{=}\e_{2}{-}\e_{3},\,
 \alphaup_{3}{=}\e_{3},\,\alphaup_{4}{=}\tfrac{1}{2}(\e_{4}{-}\e_{1}{-}\e_{2}{-}\e_{3})\}.
\end{equation}
Then $\Mtt\,{=}\,\{\betaup_{2i-1}{=}\e_{2i-1}{-}\e_{2i},\,\betaup_{2i}{=}\e_{2i-1}{+}\e_{2i}\,{\mid}\,i=1,2\}$ 
is a maximal system of strongly orthogonal roots of $\textsc{F}_{4}$. Since 
\begin{equation*} 
\begin{cases}
 \betaup_{2}\,{=}\,\e_{1}{+}\e_{2}=\alphaup_{1}+2(\alphaup_{2}{+}\alphaup_{3}),\\
 \betaup_{3}\,{=}\,\e_{3}{-}\e_{4}=-\betaup_{2}-2\alphaup_{4},\\
 \betaup_{4}\,{=}\, \betaup_{2}{+}2\alphaup_{3}{+}2\alphaup_{4},
\end{cases}
\end{equation*}
shows that $\Mtt\,{\subset}\,\Rd{\oast}{\sigmaup}$, the successive 
Cayley transforms with respect to the $\betaup_{i}$'s yield a $\sigmaup'\,{\in}\,\Inv^{\tauup}(\gr,\hr)$,
which is split because
$\Rd{\bullet}{\rhoup(\sigmaup')}\,{=}\,\emptyset$.
\smallskip
 \par
 \noindent
 $\textsc{G}_{2}$. We take the canonical basis of $\Rad(\textsc{G}_{2})$ 
\begin{equation}\label{bsG2}
 \Bz(\textsc{G}_{2})\,{=}\,\{\alphaup_{1}{=}\e_{1}{-}\e_{2},\, \alphaup_{2}{=}\e_{2}{+}\e_{3}{-}2\e_{1}\}.
\end{equation}
Then $\Mtt\,{=}\,\{\betaup_{1}{=}\e_{2}{+}\e_{3}{-}2\e_{1},\,\betaup_{2}{=}\e_{3}{-}\e_{2}\}$ 
 is a maximal system of strongly orthogonal roots of $\textsc{G}_{2}$. Since 
\begin{equation*} 
\betaup_{2}=\betaup_{1}+2\alphaup_{1},
\end{equation*}
$\Mtt\,{\subset}\,\Rd{\oast}{\sigmaup}$, and the successive 
Cayley transforms with respect to  $\betaup_{1}$
and $\betaup_{2}$ yield a $\sigmaup'\,{\in}\,\Inv^{\tauup}(\gr,\hr)$,
which is split because
$\Rd{\bullet}{\rhoup(\sigmaup')}\,{=}\,\emptyset$.
\end{proof}
We have also a vice versa: 
\begin{thm}
 If $\sigmaup\,{\in}\,\Inv_{\att}^{\tauup}(\gr,\hr)$ and $\gs$ is quasi-split, then we can find
 a Weyl chamber $C$ with $\Bz_{\oast}^{\sigmaup}(C)\,{=}\,\Bz(C)$.
\end{thm} 
\begin{proof}
 When $\sigmaup\,{\in}\,\Inv_{\att}^{\tauup}(\gr,\hr)$ and $\gs$ is quasi-split, 
 by the previous theorem we know that $\Rd{\oast}{\sigmaup}$ contains a maximal system
 of strongly orthogonal roots of $\Rad$. In the proof we can and will reduce to the case where
 $\Rad$ is irreducible. \par\smallskip
 $\textsc{A}_{\ell}$. Let $r\,{=}\,[(\ell{+}1)/2]$ and $\Mtt\,{=}\,\{\betaup_{i}\,{=}\,\e_{2i-1}{-}\e_{2i}\,{\mid}\,
 1{\leq}i{\leq}r\}$ be a maximal orthogonal system contained in $\Rd{\oast}{\sigmaup}$. 
 We consider the family $\mathpzc{A}$ of Weyl chambers $C$ such that, for
$\Bz(C)\,{=}\,\{\alphaup_{1},\hdots,\alphaup_{\ell}\}$,  the set $\Mtt$ contains either $\alphaup_{2i-1}$
or $({-}\,\alphaup_{2i-1})$, 
 for $1{\leq}i{\leq}r$.
  Let $\muup(C)\,{=}\,\#\Bz_{\oast}^{\sigmaup}(C)$.
 We claim that a $C$ with $\muup(C)$ maximal has $\Bz_{\oast}^{\sigmaup}(C)\,{=}\,\Bz(C)$.
 Indeed, if this is not the case, there is a smallest $q$ such that $2q{\leq}\ell$ and $\alphaup_{2q}{\in}
 \Rd{\bullet}{\sigmaup}$. But then 
\begin{equation*} \alphaup'_{i}= 
\begin{cases}
 {-}\alphaup_{2q-i}, & \text{if $1{\leq}i{\leq}2q{-}1$},\\
 {\sum}_{j=1}^{2q}\alphaup_{j}, & \text{if $i=2q$},\\
 \alphaup_{i}, & \text{if $2q{+}1{\leq}i{\leq}\ell$},
\end{cases}
\end{equation*}
is the set of simple roots of a Weyl chamber $C'$ of $\mathpzc{A}$ with $\muup(C')\,{=}\,\muup(C){+}1.$ 
The contradiction proves the claim.
\par\smallskip
\noindent
$\textsc{B}_{\ell}.$ We distinguish the cases where $\ell$ is even or odd.
If $\ell\,{=}\,2m$ is even, then we can assume that $\Rd{\oast}{\sigmaup}$ contains the
maximal set $\Mtt\,{=}\,\{\e_{2i-1}{\pm}\e_{2i}\,{\mid}\,1{\leq}i{\leq}m\}$ of strongly orthogonal roots.
Then, for each $1{\leq}i{\leq}m$,
either $\e_{2i-1}{\in}\,\Rd{\oast}{\sigmaup},\, \e_{2i}{\in}\Rd{\bullet}{\sigmaup}$,
or  $\e_{2i-1}{\in}\,\Rd{\bullet}{\sigmaup},\, \e_{2i}{\in}\Rd{\oast}{\sigmaup}$.
By choosing a permutation $(i_{1},\hdots,i_{2m})$ such that $\e_{i_{2h-1}}\,{\in}\,\Rd{\bullet}{\sigmaup},\;
\e_{i_{2h}}\,{\in}\,\Rd{\oast}{\sigmaup}$, we obtain that
\begin{equation*}
 \{\e_{i_{h}}{-}\e_{i_{h+1}}\mid 
\begin{smallmatrix}
 1{\leq}i{\leq}2m{-}1
\end{smallmatrix}\,\}\,{\cup}\,\{\e_{2m}\}
\end{equation*}
is the basis of simple roots for a Weyl chamber $C'$ with $\Bz(C')\,{\subset}\,\Rd{\oast}{\sigmaup}.$ \par
When $\ell{=}2m{+}1$ is odd, we can assume that $\Mtt{=}\{\e_{2i-1}{\pm}\e_{2i}\,{\mid}\,1{\leq}i{\leq}m\}
{\cup}\{\e_{2m+1}\}$ is contained in $\Rd{\oast}{\sigmaup}$ and, by repeating the argument above, find
a permutation of indices $(i_{1},\hdots,i_{2m},\ell)$ 
such that $\e_{i_{2h-1}}\,{\in}\,\Rd{\oast}{\sigmaup}$ and $\e_{i_{2h}}\,{\in}\,\Rd{\bullet}{\sigmaup}$
for $1{\leq}h{\leq}m$. Then 
\begin{equation*}
 \{\e_{i_{h}}{-}\e_{i_{h+1}}\mid 
\begin{smallmatrix}
 1{\leq}i{\leq}2m
\end{smallmatrix}\,\}\,{\cup}\,\{\e_{2m+1}\}
\end{equation*}
is the basis of simple roots for a Weyl chamber $C'$ with $\Bz(C')\,{\subset}\,\Rd{\oast}{\sigmaup}.$ \par
\smallskip
\noindent
$\textsc{C}_{\ell}$. We can assume that $\Mtt\,{=}\,\{2\e_{i}\,{\mid}\,1{\leq}i{\leq}\ell\}\,{\subset}\,\Rd{\oast}{\sigmaup}$.\par
From $(\e_{i}{-}\e_{j}){+}(\e_{i}{+}\e_{j})\,{=}\,2\e_{i}\,{\in}\,\Rd{\oast}{\sigmaup}$ it follows that
one of the roots $(\e_{i}{-}\e_{j}),$ $(\e_{i}{+}\e_{j})$ belongs to $\Rd{\oast}{\sigmaup}$, the other to 
$\Rd{\bullet}{\sigmaup}$. Hence, by setting $\epi_{1}\,{=}\,\e_{1}$, we can recursively choose
$\epi_{i}\,{\in}\,\{{\pm}\e_{i} \}$, for $2{\leq}i{\leq}\ell$, in such a way that $\epi_{i}{-}\epi_{i+1}\,{\in}\,\Rd{\oast}{\sigmaup}$
for all $1{\leq}i{\leq}\ell{-}1$. Then 
\begin{equation*}
 \{\epi_{i}{-}\epi_{i+1}\,{\mid}\,
\begin{smallmatrix}
 1{\leq}i{\leq}\ell{-}1\end{smallmatrix}\!\}\cup\{2\epi_{\ell}\}
\end{equation*}
are the simple roots for a Weyl chamber $C'$ with $\Bz(C')\,{\subset}\,\Rd{\oast}{\sigmaup}$.\par\smallskip
\noindent
$\textsc{D}_{\ell}$. For $1{\leq}i{<}j{\leq}\ell$, the roots $\e_{i}{\pm}\e_{j}$ cannot 
both belong to $\Rd{\bullet}{\sigmaup}$
because $\Rd{\oast}{\sigmaup}$ contains a maximal system of orthogonal roots. 
We can assume that $\e_{\ell-1}{\pm}\e_{\ell}$ both belong to $\Rd{\oast}{\sigmaup}$. 
We set $\epi_{j}{=}\e_{j}$ for $j{=}\ell{-}1,\ell$. By recurrence we can take $\epi_{\ell-j}$, for $2{\leq}j{\leq}{\ell{-}1}$,
in such a way that $\epi_{i}{-}\epi_{i+1}\,{\in}\,\Rd{\oast}{\sigmaup}$
for all $i{=}1,\hdots,\ell{-}2$. Then 
\begin{equation*}
 \{\epi_{i}{-}\epi_{i+1}\,{\mid} 
\begin{smallmatrix}
 1{\leq}i{\leq}\ell{-}1
\end{smallmatrix}\!\}\cup\{\epi_{\ell-1}{+}\epi_{\ell}\}
\end{equation*}
are the simple roots for a Weyl chamber $C'$ with $\Bz(C')\,{\subset}\,\Rd{\oast}{\sigmaup}$. 
\par\smallskip
\noindent
$\textsc{E}_{\ell}$, $\ell=6,7,8$. We can take a maximal system of orthogonal roots in $\Rad(\textsc{E}_{\ell})$
which is contained in the root subsystem $\Rad(\textsc{D}_{\ell-1})$. By the previous part of the proof,
we can assume then that, for the canonical basis $\Bz(C)$, defined by \eqref{bsE6}, \eqref{bsE7},
\eqref{bsE8}, we have $\{\alphaup_{i}\,{\mid}\,2{\leq}i{\leq}\ell\}\,{\subset}\,\Rd{\oast}{\sigmaup}$. 
Then  
\begin{equation*}\begin{aligned}
 \Bz(C')=\{\alphaup'_{1}{=}\zetaup_{1,2},\,\alphaup'_{2}{=}{-}(\e_{1}{+}\e_{2}),\,\alphaup'_{3}{=}\e_{1}{-}\e_{2},\,
 \alphaup'_{4}{=}\e_{2}{+}\e_{3}\}\qquad \\
 \cup\,\{\alphaup'_{i}{=}\alphaup_{i}{=}\e_{i-1}{-}\e_{i-2}\,{\mid}
\begin{smallmatrix}
 4{\leq}i{\leq}\ell
\end{smallmatrix}\!\}
\end{aligned}
\end{equation*}
are the simple roots for a Weyl chamber $C'$.
We note that $\alphaup'_{2}{=}{-}\alphaup_{2}$, $\alphaup'_{3}{=}{-}\alphaup_{3}$,
$\alphaup'_{4}{=}\alphaup_{2}{+}\alphaup_{3}{+}\alphaup_{4}$ all belong to $\Rd{\oast}{\sigmaup}$,
while either $\zetaup_{\emptyset}$ or $\zetaup_{1,2}{=}\zetaup_{\emptyset}{+}\alphaup_{2}$
belongs to~$\Rd{\oast}{\sigmaup}$. Hence
either $\Bz(C)$ or $\Bz(C')$ is contained in $\Rd{\oast}{\sigmaup}.$ \par\smallskip
\noindent
$\textsc{F}_{4}$. We can assume that $\Mtt\,{=}\,\{\e_{1}{\pm}\e_{2},\,\e_{3}{\pm}\e_{4}\}\,{\subset}\,
\Rd{\oast}{\sigmaup}$. Then, for each $i{=}1,2$, on of the short roots $\e_{2i-1},\,\e_{2i}$ belongs
to $\Rd{\oast}{\sigmaup}$, the other to $\Rd{\bullet}{\sigmaup}$. By changing the names of indices,
we can assume that $\e_{1},\e_{3}\,{\in}\,\Rd{\oast}{\sigmaup},$ $\e_{2},\e_{4}\,{\in}\,\Rd{\bullet}{\sigmaup}.$
Let $\Bz(C)$ be the canonical basis \eqref{bsF4}. Then either $\alphaup_{4}{=}\tfrac{1}{2}(\e_{4}{-}\e_{1}{-}\e_{2}{-}\e_{3})$
or $\alphaup'_{4}{=}\alphaup_{4}{+}(\e_{1}{+}\e_{2})$ belongs to $\Rd{\oast}{\sigmaup}$.
Hence either the canonical basis $\Bz(C)$ or 
\begin{equation*}
 \Bz(C')=\{\alphaup_{1}'=\e_{2}{-}\e_{1},\,\alphaup_{2}'{=}{-}\e_{2}{-}\e_{3},\,\alphaup'_{3}{=}\e_{3},\,
 \alphaup'_{4}{=}\tfrac{1}{2}(\e_{1}{+}\e_{2}{+}\e_{4}{-}\e_{3})\}
\end{equation*}
is contained in $\Rd{\oast}{\sigmaup}$.\par\smallskip
\noindent
$\textsc{G}_{2}$. We can assume that $\e_{1}{-}\e_{2}\,{\in}\,
\Rd{\oast}{\sigmaup}$. \par
The canonical basis is $\Bz(C)\,{=}\,\{\alphaup_{1}{=}\e_{1}{-}\e_{2},\,\alphaup_{2}
{=}\e_{2}{+}\e_{3}{-}2\e_{1}\}$.  
Either $\alphaup_{2}$ or $\alphaup'_{2}{=}2\e_{1}{+}\e_{3}{-}2\e_{2}\,{=}\, \alphaup_{2}{+}3\alphaup_{1}$
belongs to $\Rd{\oast}{\sigmaup}$ and therefore either the canonical basis $\Bz(C)$ or
$\Bz(C')\,{=}\,\{\alphaup_{1}'{=}{-}\alphaup_{1},\,\alphaup_{2}'\}$ is contained in $\Rd{\oast}{\sigmaup}.$
\end{proof}

\subsection{Irreducibe real forms with an assigned Cartan subalgebra}
\subsubsection*{Type $\textsc{A}_{\ell}$} We take the root system $\Rad(\textsc{A}_{\ell})$ of \eqref{rA}
and choose representatives of $\stt\,{\in}\,\Ib(\Rad(\textsc{A}_{\ell}))$ for which the canonical basis
\begin{equation}\label{bA}
 \Bz(\textsc{A}_{\ell})=\{\e_{i}{-}\e_{i+1}\,{\mid}\, 
\begin{smallmatrix}
 1{\leq}i{\leq}\ell
\end{smallmatrix}\!\}
\end{equation}
consists of the simple positive roots of an $S$-chamber $C$ for $\stt$.

\begin{lem}[Type $\textsc{A}_{\ell}$] \label{l2.12} 
 Let $\Rad$ be an irreducible root system of type $\textsc{A}_{\ell}$ and $\stt\,{\in}\,\Ib(\Rad).$ Then:
\begin{itemize}
 \item[1.] If $\stt\,{\in}\,\Ib_{\Wf}(\Rad)$ and either $\ell$ is even or $\stt$ has not maximal lenght,
 then $\etaup(\betaup)\,{=}\,1$ for all $\etaup\,{\in}\,\Hom_{\stt}(\Z[\Rad],\Z_{2}^{*})$ and all 
 $\betaup\,{\in}\,\Rad_{\,\;\bullet}^{\stt}$. 
\item[2.] If $\ell$ is odd, $\stt$ has maximal lenght and $\etaup\,{\in}\Homs$, then 
 either $\etaup(\betaup)\,{=}\,1$ for all $\betaup\,{\in}\,\Rad_{\,\;\bullet}^{\stt}$,
 or 
 $\etaup(\betaup)\,{=}\,{-}1$ for all $\betaup\,{\in}\,\Rad_{\,\;\bullet}^{\stt}$ 
 and both cases occur.

 \item[3.] If $\stt\,{\in}\,\Ib(\Rad){\setminus}\Ib_{\Wf}(\Rad)$ and $C$ is an $S$-chamber for $\stt,$
 then for every subset $\Phi$ of $\Bz_{\bullet}^{\stt}(C)$ we can find $\etaup_{\Phi}\,{\in}\,\Hom_{\stt}(\Z[\Rad],\Z_{2}^{*})$
 such that 
\begin{equation}\label{e2.19}
 \etaup_{\Phi}(\alphaup)= 
\begin{cases}
 {-}1,&\forall\alphaup\in\Phi,\\
 {+}1, & \forall \alphaup\,{\in}\,\Bz_{\bullet}^{\stt}(C){\setminus}\Phi.
\end{cases}
\end{equation}
\end{itemize}
\end{lem} 
\begin{proof} {[1]-[2]} \;\; If $\stt\,{\in}\,\Ib_{\Wf}(\Rad),$ then 
$\stt\,{=}\,\sq_{\,\Btt}$
for a system $\Btt\,{=}\,\{\betaup_{1},\hdots,\betaup_{r}\}$ of strongly orthogonal roots of $\Rad(\textsc{A}_{\ell})$.
 Then $r{\leq}\tfrac{\ell{+}1}{2}$ and,
by reordering the indices we can assume that
$\betaup_{i}{=}\e_{2i-1}{-}\e_{2i},$ for $1{\leq}i{\leq}r$. 
The canonical basis 
\begin{equation}\label{bA}
 \Bz(\textsc{A}_{\ell})\,{=}\,\{\alphaup_{i}{=}\e_{i}{-}\e_{i+1}\,{\mid}\,1{\leq}i{<}j{\leq}\ell{+}1\}
\end{equation}
is the basis of simple positive roots of an $S$-chamber $C$
for $\stt$, with 
\begin{equation*}
 \Bz_{\bullet}^{\stt}(C)=\big\{\alphaup_{2i-1}\,{\mid}\,1{\leq}i{\leq}r\big\},\;\; \Bz_{\star}^{\stt}(C)=
 \big\{\alphaup_{2i}\,{\mid}\,
 1{\leq}i{\leq}\min\big(r,\tfrac{\ell{+}1}{2}\big)\big\}
\end{equation*}
 We have (drop the second line when $\ell$ is odd and $r{=}\tfrac{\ell{+}1}{2}$)
\begin{equation*} \tag{$*$} \begin{cases}
 \stt(\alphaup_{2i})\,{=}\,\alphaup_{2i-1}{+}\alphaup_{2i}{+}\alphaup_{2i{+}1}, &\text{if $1{\leq}i{\leq}{r{-}1}$},\\
 \stt(\alphaup_{2r})\,{=}\,\alphaup_{2r-1}{+}\alphaup_{2r}, &\text{if $r{<}\tfrac{\ell{+}1}{2}$.}
 \end{cases}
\end{equation*}
If $\omegaup\,{\in}\,(\Z[\Rad])^{*}_{\stt}$, then, by applying \eqref{eq4.9} to the first line of $(*)$ we obtain
that 
\begin{equation*}
 (\alphaup_{2i-1}{+}\alphaup_{2i{+}1}|\omegaup)\in 2\Z,\;\forall 1{\leq}i{\leq}r{-}1.
\end{equation*}
If $r{<}\tfrac{\ell{+}1}{2}$, by applying \eqref{eq4.9} to the second line of $(*)$ we obtain that 
$(\alphaup_{2r-1}|\omegaup)$ is even
and hence $(\alphaup_{2i-1}|\omegaup)\,{\in}\,2\Z$ 
for $1{\leq}i{\leq}r.$ If $\ell$ is odd and $r{=}\tfrac{\ell{+}1}{2}$, then we can take e.g. 
$\omegaup\,{=}\,{\sum}_{i=1}^{{(\ell{+}1)}/{2}}\e_{2i-1}\,{\in}\,(\Z[\Rad])^{*}_{\stt}.$ 
 \par\smallskip
[3]\quad  If $\stt\,{\in}\,\Ib^{*}(\Rad)$, then $\Rad_{\,\;\bullet}^{\stt}{=}\emptyset$
and
there is nothing to prove. Thus we will discuss in the rest of the proof the case
$\stt\,{=}\,\epi\,{\circ}\,\sq_{\,\Btt}$, where
$\id{\neq}\epi\,{\in}\Ib^{*}(\Rad)$ and $\Btt\,{=}\,\{\betaup_{1},\hdots,\betaup_{r}\}\,{\subset}\,\Rad_{\,\;\circ}^{\epi}$
is a system of orthogonal roots. 
We can assume that $\epi$ is defined by 
\begin{equation*}
 \epi(\e_{i})={-}\e_{\ell+2-i},\;\;\text{for $i\,{=}\,1,\hdots,\ell,\ell{+}1$}
\end{equation*}
and, for a pair of positive integers $p,q$ with $p{+}q{=}\ell{+}1$, $r\,{=}\,[(q{-}p)/2],$ 
\begin{equation*}
 \Btt=\{\betaup_{i}\,{=}\,\e_{p+i}{-}\e_{q-i}\mid 1{\leq}i{\leq}r\}.
\end{equation*}
Then $\Bz(\textsc{A}_{\ell})$ is the basis of an $S$-chamber $C$ for $\stt$ with 
\begin{gather*} \Bz_{\;\circ}^{\stt}(C)\,{=}\,\emptyset, \;\; 
 \Bz_{\star}^{\stt}(C)=\{\alphaup_{i}\mid (1{\leq}i{\leq}p)\,{\vee}\,( q{\leq}i{\leq}\ell)\},\;\;
 \Bz_{\bullet}^{\stt}(C)=\{\alphaup_{i}\,{\mid}\,
 p{<}i{<}q\},\\
 \stt(\alphaup_{i})= 
\begin{cases}
 \alphaup_{\ell{+}1-i}, & 1{\leq}i{<}p,\; q{<}i{\leq}\ell,\\
 \alphaup_{p+1}+\cdots+\alphaup_{q}, & i=p,\\
 \alphaup_{p}+\cdots+\alphaup_{q-1}, & i=q,\\
 {-}\alphaup_{i}, & p<i<q.
\end{cases}
\end{gather*}
Then $\omegaup\,{\in}\,(\Z[\Rad])^{*}$ belongs to $(\Z[\Rad])_{\;\stt}^{*}$
if and only if 
\begin{equation*} \tag{$*$}
\begin{cases}
 (\alphaup_{i}{-}\alphaup_{\ell{+}1{-}i}|\omegaup)\in 2\Z, & \forall 1\,{\leq}\,i\,{<}p,\\
 (\alphaup_{p}{+}\alphaup_{q}|\omegaup)+\#\Phi\in 2\Z, \\
 (\alphaup_{i}|\omegaup)\notin 2\Z, & \forall \alphaup_{i}\in\Phi,\\
 (\alphaup_{i}|\omegaup)\in 2\Z, & \forall \alphaup_{i}\,{\in}\,\Bz_{\bullet}^{\stt}(C){\setminus}\Phi.
\end{cases}
\end{equation*}
This is possible for any choice of $\Phi\,{\subseteq}\,\Bz_{\bullet}^{\stt}(C).$ The proof is complete.
 \end{proof} 
\begin{cor}
 The non-compact and non-split real forms of type $\textsl{A}_{\ell}$ are 
\par\noindent
 \item[] \textsc{[A\,II]} \; $\ell{=}2m-1$ odd, 
 $\Btt\,{=}\,\{\e_{2i-1}{-}\e_{2i}\,{\mid}\, 
\begin{smallmatrix}
 1{\leq}i{\leq}m
\end{smallmatrix}\!\}$
 $\stt\,{=}\sq_{\Btt}$\,,  
 $\Rd{\bullet}{\sigmaup}\,{=}\,\{{\pm}(\e_{2i-1}{-}\e_{2i}\,{\mid}\,\begin{smallmatrix}
 1{\leq}i{\leq}m
\end{smallmatrix}\!\}$, 
$\gs\,{\simeq}\,\slt_{m}(\Hb)\,{\simeq}\,\su^{*}(2m)$;
\par\noindent \textsc{[A\,III,IV]} \; for a pair of integers $1{\leq}p{\leq}q$ with $p{+}q\,{=}\,\ell{+}1$ 
\begin{equation*} 
\begin{array}{| c | c | c | c |} 
\hline 
p,q & \stt & \Rd{\bullet}{\sigmaup} & \gs \\
\hline 
 q{=}p &  \sq_{\e_{1}{+}\e_{2p}}{\circ}\cdots {\circ}\sq_{\e_{p} +\e_{p+1}} & \emptyset & \su(p,p)\\
 \hline
 q{=}p{+}1 & \sq_{\e_{1}{+}\e_{2q-1}}{\circ}\cdots {\circ}\sq_{\e_{q-1} +\e_{q+1}}
  {\circ}\sq_{\e_{q}} & \emptyset & \su(p,p{+}1)\\
  \hline 
  q{-}p{\geq}2 & \sq_{\e_{1}{+}\e_{\ell+1}}{\circ}\cdots {\circ}\sq_{\e_{p} +\e_{q+1}}{\circ}\sq_{p+1}{\circ}\cdots{\circ}\sq_{q} &
  \{{\pm}(\e_{i}{-}\e_{j})\,{\mid}\! 
\begin{smallmatrix}
 p{<}i{<}j{\leq}q
\end{smallmatrix}\!\}& \su(p,q) \\
\hline
\end{array}
\end{equation*}
\end{cor}
 \subsubsection*{Type $\textsc{B}_{\ell}$}
 \begin{lem}[Type $\textsc{B}_{\ell}$]
 Assume that $\Rad$ is an irreducible root system of type $\textsc{B}_{\ell}$ and let
 $\stt$ be a nontrivial involution in $\Ib(\Rad).$ Then 
\begin{itemize}
 \item[1.]  there is at most one 
 irreducible component $\Rad_{\,\;\bullet}^{\stt,\sharp}$ of $\Rad_{\,\;\bullet}^{\stt}$ containing short  roots;
 \item[2.] its complement 
 $\Rad_{\,\;\bullet}^{\stt,\flat}{\coloneqq}(\Rad_{\,\;\bullet}^{\stt}{\setminus}\Rad_{\,\;\bullet}^{\stt,\sharp})$
 in $\Rad_{\,\;\bullet}^{\stt}$ 
 is a union of strongly orthogonal root subsystems of type $\textsc{A}_{1}$;
  \item[3.] 
 $\etaup(\betaup)\,{=}\,1,$ for all $\betaup\,{\in}\,\Rad_{\,\;\bullet}^{\stt,\flat}$ and all $\etaup\,{\in}\,\Homs$;
\end{itemize}\par
 Let $C$ be an $S$-chamber for $\stt$ and set \par\centerline{
 $ \Bz_{\bullet}^{\stt,\sharp}(C)\,{=}\,\Bz(C)\,{\cap}\,\Rad_{\,\;\bullet}^{\stt,\sharp}$,\;\;
  $ \Bz_{\bullet}^{\stt,\flat}(C)\,{=}\,\Bz(C)\,{\cap}\,\Rad_{\,\;\bullet}^{\stt,\flat}$.} Then 
\begin{itemize}
\item[4.] for every $\Phi\,{\subseteq}\, \Bz_{\bullet}^{\stt,\sharp}(C)$ we can find $\etaup\,{\in}\,\Homs$
 such that 
\begin{equation*} 
\begin{cases}
 \etaup(\alphaup)={-}1, &\text{if $\alphaup\,{\in}\,\Phi$,}\\
 \etaup(\alphaup)={+}1, &\text{if $\alphaup\,{\in}\, \Bz_{\bullet}(C){\setminus}\Phi.$}
\end{cases}
\end{equation*}
\end{itemize}
\end{lem} 
\begin{proof} Roots belonging to distinct irreducible components of $\Rad_{\,\;\bullet}^{\stt}$ are
strongly orthogonal. Since distinct short roots of $\Rad$ are not strongly orthogonal, this yields [1] and
[2]. 
\par
We can assume that  the canonical basis 
\begin{equation}\label{bB}
 \Bz(\textsc{B}_{\ell})\,{=}\,\{\alphaup_{i}{=}\e_{i}{-}\e_{i+1}\,{\mid}\,1{\leq}i{<}\ell\}\,{\cup}\,\{\alphaup_{\ell}{=}\e_{\ell}\}
\end{equation}
consists of the simple positive roots of an $S$-chamber $C$ for $\stt$ 
and that, for integers $h,k$ with $0{\leq}h{<}2h{\leq}k{-}1{\leq}\ell,$ 
\begin{equation*}
 \Bz_{\bullet}^{\stt,\flat}(C)\,{=}\,\{\alphaup_{2i-1}\,{\mid}\,1{\leq}i{\leq}h\},\quad
 \Bz_{\bullet}^{\stt,\sharp}(C)\,{=}\,\{\alphaup_{i}\,{\mid}\, k{\leq}{i}{\leq}\ell\}
\end{equation*}
[if $h{=}0$ the first, if $k{=}\ell{+}1$ the second set are empty.]
Then 
\begin{equation*}
 \Bz_{\star}^{\stt}(C)=\{\alphaup_{2i}\,{\mid}\,1{\leq}i{\leq}h\}\,{\cup}\,\{\alphaup_{k-1}\}
\end{equation*}
and we get 
\begin{equation*} 
\begin{cases}
 \stt(\alphaup_{2i})=\alphaup_{2i-1}{+}\alphaup_{2i}{+}\alphaup_{2i{+}1}, \;\;\text{if $1{\leq}i{<}h$,}\\
 \stt(\alphaup_{2h})= 
\begin{cases}
 \alphaup_{2h-1}{+}\alphaup_{2h},\; \text{if $2h{<}k{-}1$,}\\
 \alphaup_{2h-1}{+}\alphaup_{2h}{+}2{\sum}_{i=k}^{\ell}\alphaup_{i},\;\text{if $2h{=}k{-}1$},
\end{cases}\\
\stt(\alphaup_{k-1}){=}\alphaup_{k-1}{+}2{\sum}_{i=k}^{\ell}\alphaup_{i},\;\text{if $2h{<}k{-}1$.}
\end{cases}
\end{equation*}
The statement is then a consequence of \eqref{eq4.9}.
\end{proof}
\begin{cor}
 The non-compact and non-split real forms of type $\textsl{B}_{\ell}$ ($\ell{\geq}2$)
 are parametrized by an integer $p$, with $1{\leq}p{<}\ell$ and correspond to \par\noindent
 $\stt\,{=}\,\sq_{\e_{p}}{\circ}\cdots{\circ}\sq_{\e_{\ell}}$, with $\Rd{\bullet}{\sigmaup}\,{=}\,
 \{{\pm}\e_{i}\,{\mid}\! \begin{smallmatrix}
 p{<}i{\leq}\ell\end{smallmatrix}\!\}\,{\cup}\,\{{\pm}\e_{i}{\pm}\e_{j}\,{\mid}\!
\begin{smallmatrix}
 p{<}i{<}j{\leq}\ell
\end{smallmatrix}\!\}$, $\gs\,{\simeq}\,\so(p,2\ell{+}1{-}p)$. 
\end{cor}
\subsubsection*{Type $\textsc{C}_{\ell}$}
\begin{lem}[type $\textsc{C}_{\ell}$] \label{l3.20}
Let $\Rad$ be an irreducible root system of type $\textsc{C}_{\ell}$ and let
$\stt\,{=}\,\sq_{\Btt}$ 
for a system $\Btt\,{=}\,\{\betaup_{1},\hdots,\betaup_{r}\}$ of strongly orthogonal roots. Then
\begin{itemize}
 \item[1.]
 there is at most one irreducible subsystem $\Rad_{\,\;\bullet}^{\stt,\sharp}$ of
$\Rad_{\,\;\bullet}^{\stt}$ containing a long root;
\item[2.] its complement 
 $\Rad_{\,\;\bullet}^{\stt,\flat}{\coloneqq}(\Rad_{\,\;\bullet}^{\stt}{\setminus}\Rad_{\,\;\bullet}^{\stt,\sharp})$
 in $\Rad_{\,\;\bullet}^{\stt}$ 
 is a union of strongly orthogonal root subsystems of type $\textsc{A}_{1}$.
 \end{itemize}
 Denote by $\Rad_{\,\;\bullet}^{\stt,\natural}$ the union of $\Rad_{\,\;\bullet}^{\stt,\flat}$ and of
 the set of long roots of $\Rad_{\,\;\bullet}^{\stt,\sharp}$.
 \begin{itemize}
  \item[3.] if $\Btt$ is not maximal, then 
 $\etaup(\betaup)\,{=}\,1$,  $\forall\betaup\,{\in}\,\Rad_{\,\;\bullet}^{\stt,\natural}$ 
 and $\forall\etaup\,{\in}\,\Homs$;
\end{itemize}\par
 Let $C$ be an $S$-chamber for $\stt$ and set 
\begin{equation*}
 \Bz_{\bullet}^{\stt,\sharp}(C)\,{=}\,\Bz(C)\,{\cap}\,\Rad_{\,\;\bullet}^{\stt,\sharp},\;\;
 \Bz_{\bullet}^{\stt,\flat}(C)\,{=}\,\Bz(C)\,{\cap}\,\Rad_{\,\;\bullet}^{\stt,\flat},\;\;
 \Bz_{\bullet}^{\stt,\natural}\,{=}\,\Bz(C)\,{\cap}\,\Rad_{\,\;\bullet}^{\stt,\natural}
\end{equation*}

\begin{itemize}
\item[4.] If $\Btt$ is maximal, then 
\begin{equation*}
\etaup(\betaup)=\etaup(\betaup'),\;\;\forall\betaup,\betaup'\,{\in}\, \Bz_{\bullet}^{\stt,\natural}(C),\;
\forall\etaup\,{\in}\Homs.
\end{equation*}
If $\lambdaup\,{=}\,{\pm}1$ and 
$\Phi\,{\subseteq}\, \Bz_{\bullet}^{\stt,\sharp}(C){\setminus} \Bz_{\bullet}^{\stt,\natural}$, 
then we can find $\etaup\,{\in}\,\Homs$
 such that 
\begin{equation*} 
\begin{cases}
\etaup(\alphaup)=\lambdaup,&\text{if $\alphaup\,{\in}\, \Bz_{\bullet}^{\stt,\natural}$},\\
 \etaup(\alphaup)={-}1, &\text{if $\alphaup\,{\in}\,\Phi$,}\\
 \etaup(\alphaup)={+}1, &\text{if $\alphaup\,{\in}\, \Bz_{\bullet}^{\stt,\sharp}(C){\setminus}(\Phi\,{\cup}
 \Rad_{\,\medbullet}^{\stt,\natural})$}.
\end{cases}
\end{equation*}
\end{itemize}
 \end{lem} 
\begin{proof} If $\Rad_{\,\;\bullet}^{\stt}$ contains the roots $2\e_{i},2\e_{j}$ with $1{\leq}i{<}j{\leq}\ell,$
then it contains also ${\pm}\e_{i}{\pm}\e_{j}$. This implies [1] and [2]. \par
We can assume that the canonical basis
\begin{equation}\label{bC}
 \Bz(\textsc{C}_{\ell})=\{\e_{i}{-}\e_{i+1}\,{\mid}\,1{\leq}i{<}\ell\}\,{\cup}\,\{\alphaup_{\ell}{=}2\ell\}
\end{equation}
is the basis of simple positive roots of an
$S$-chamber $C$ for $\stt$ and that, with a pair of indices $h,k$ with $0{\leq}h{\leq}2h{\leq}k{-}1{\leq}\ell,$ 
we have
\begin{equation*}
 \Bz_{\bullet}^{\stt,\flat}(C)\,{=}\,\{\alphaup_{2i-1}\,{\mid}\,1{\leq}i{\leq}h\},\quad
 \Bz_{\bullet}^{\stt.\sharp}(C)\,{=}\,\{\alphaup_{i}\,{\mid}\,k{\leq}i{\leq}\ell\}
\end{equation*}
[if $h{=}0$ the first, if $k{=}\ell{+}1$ the second set are empty.]
Then 
\begin{equation*}
 \Bz_{\star}^{\stt}(C)=\{\alphaup_{2i}\,{\mid}\,1{\leq}i{\leq}h\}\,{\cup}\,\{\alphaup_{k-1}\}
\end{equation*}
and we get 
\begin{equation*} 
\begin{cases}
 \stt(\alphaup_{2i})=\alphaup_{2i-1}{+}\alphaup_{2i}{+}\alphaup_{2i{+}1}, \;\;\text{if $1{\leq}i{<}h$,}\\
 \stt(\alphaup_{2h})= 
\begin{cases}
 \alphaup_{2h-1}{+}\alphaup_{2h},\; \text{if $2h{<}k{-}1$,}\\
 \alphaup_{2h-1}{+}\alphaup_{2h}{+}\alphaup_{\ell}{+}2{\sum}_{i=k}^{\ell-1}\alphaup_{i},\;\text{if $2h{=}k{-}1$},
\end{cases}\\
\stt(\alphaup_{k-1}){=}\alphaup_{k-1}{+}\alphaup_{\ell}{+}2{\sum}_{i=k}^{\ell-1}\alphaup_{i},\;\text{if $2h{<}k{-}1$.}
\end{cases}
\end{equation*}
Note that $2h{=}k{-}1$ if and only if $\Btt$ is maximal.
The statement is then a consequence of \eqref{eq4.9}. 
\end{proof} 
\begin{cor}
 The non-compact and non-split real forms of type $\textsc{C}_{\ell}$ ($\ell{\geq}3$)
 are parametrized by the equivalence classes of the maximal systems of strongly orthogonal roots
 containing short roots 
\begin{equation*} 
 \Mtt_{p}=\begin{cases}\{\e_{2i-1}{-}\e_{2i}\,{\mid}\! 
\begin{smallmatrix}
 i=1,\hdots,p
\end{smallmatrix}\!\}\,{\cup}\,\{2\e_{i}\,{\mid}\! 
\begin{smallmatrix}
 i=2p{+}1,\hdots,\ell
\end{smallmatrix}\!\}, & 1{\leq}p{\leq}(\ell{-}1)/2,\\
\{\e_{2i-1}{-}\e_{2i}\,{\mid}\! 
\begin{smallmatrix}
 i=1,\hdots,p
\end{smallmatrix}\!\}, & \ell{=}2p.
\end{cases}
\end{equation*}
Then we have $\stt\,{=}\,\sq_{\Mtt_{p}}$ and 
\begin{equation*} \Rd{\bullet}{\sigmaup}
=\begin{cases}\begin{aligned}
\{{\pm}(\e_{2i-1}{-}\e_{2i})\,{\mid}\! 
\begin{smallmatrix}
 i=1,\hdots,p
\end{smallmatrix}\!\}\,{\cup}\,\{{\pm}2\e_{i}\,{\mid}\! 
\begin{smallmatrix}
 i=2p{+}1,\hdots,\ell
\end{smallmatrix}\!\}\,{\cup}\,
\{{\pm}\e_{i}{\pm}\e_{j}\,{\mid}\! 
\begin{smallmatrix}
 2p{<}i{<}j{\leq}\ell
\end{smallmatrix}\!\}, \\
 1{\leq}p{\leq}(\ell{-}1)/2,
 \end{aligned}\\
\{{\pm}(\e_{2i-1}{-}\e_{2i})\,{\mid}\! 
\begin{smallmatrix}
 i=1,\hdots,p
\end{smallmatrix}\!\}, \qquad\qquad\qquad\qquad\qquad\quad \ell{=}2p,
\end{cases}
\end{equation*}
with $\gs\,{\simeq}\,\spt(p,\ell{-}p)$.
\end{cor}
\subsubsection*{Type $\textsc{D}_{\ell}$}
In the following Lemma we will consider the
irreducible  system $\Rad(\textsc{D}_{\ell})$
of type $\textsc{D}_{\ell}$.
Let us fix some notation.
Set 
\begin{equation*}
 \Btt'_{r}= 
\begin{cases}
 \emptyset, & r=0,\\
 \{\e_{2i-1}{-}\e_{2i}\,{\mid}\, 
\begin{smallmatrix}
 1{\leq}i{\leq}r
\end{smallmatrix}\!\}, & 1{\leq}r{\leq}[\ell/2],
\end{cases}\quad \Btt''_{r}= 
\begin{cases}
 \emptyset, & r=0,\\
 \{\e_{i}\,{\mid} 
\begin{smallmatrix}
 \ell{-}r{+}1\leq{i}\leq\ell
\end{smallmatrix}\!\}, & 1{\leq}r{\leq}\ell.
\end{cases}
\end{equation*}
By Thm.\ref{t2.26} the involutions 
 of $\Rad(\textsc{D}_{\ell})$ are pa\-ram\-e\-tri\-zed by
a pair of nonnegative integers $r_{1},r_{2}$ with $2r_{1}{+}r_{2}{\leq}\ell$.
We can take  representatives 
\begin{equation} \label{e4.24}
\sq_{r_{1},r_{2}}=\sq_{\vq_{1}}{\circ}\cdots{\circ}\sq_{\vq_{r_{1}+r_{2}}},
\;\;\text{with}\;\;\{\vq_{1},\hdots,\vq_{r_{1}+r_{2}}\}\,{=}\, 
\Btt_{r_{1},r_{2}}=\Btt'_{r_{1}}\cup\Btt''_{r_{2}}.
\end{equation}
The canonical basis 
\begin{equation}\label{bD}
 \Bz(\Rad(\textsc{D}_{\ell}))
 \,{=}\,\{\alphaup_{i}{=}\e_{i}{-}\e_{i+1}\mid 1{\leq}i<\ell\}\cup\{\alphaup_{\ell}{=}\e_{\ell-1}{+}\e_{\ell}\}, \end{equation}
consists  of the simple positive roots of an $S$-chamber $C$ for $\stt_{{r_{1},r_{2}}}$. Set 
{\small
\begin{align*}
 &\Bz'_{\bullet}(r)= 
\begin{cases}
 \emptyset, & r=0,\\
 \{\alphaup_{2i-1}\,{\mid}\, 
\begin{smallmatrix}
 1{\leq}i{\leq}r
\end{smallmatrix}\!\}, & 1{\leq}r{\leq}[\ell/2],
\end{cases} && \Bz''_{\bullet}(r)= 
\begin{cases}
 \emptyset, & r=0,1,\\
 \{\e_{i}\,{\mid} 
\begin{smallmatrix}
 \ell{-}r{+}1\leq{i}\leq\ell
\end{smallmatrix}\!\}, & 2{\leq}r{\leq}\ell,
\end{cases}
\\[6pt]
& \Bz'_{\star}(r)= 
\begin{cases}
 \emptyset, & r=0,\\
 \{\alphaup_{2i}\,{\mid}\, 
\begin{smallmatrix}
 1{\leq}i{\leq}r
\end{smallmatrix}\!\}, & 1{\leq}r{\leq}\tfrac{\ell-2}{2},\\
\{\alphaup_{2i}\,{\mid} 
\begin{smallmatrix}
 1{\leq}i{\leq}\tfrac{\ell-1}{2}
\end{smallmatrix}\!\}\,{\cup}\,\{\alphaup_{\ell}\}, & r=\tfrac{\ell-1}{2},\\
\{\alphaup_{2i}\,{\mid} 
\begin{smallmatrix}
 1{\leq}i{\leq}\tfrac{\ell-2}{2}
\end{smallmatrix}\!\}, & r=\tfrac{\ell}{2},
\end{cases} && \Bz''_{\star}(r)= 
\begin{cases}
 \emptyset, & r=0,\ell\\
 \{\alphaup_{\ell-1},\alphaup_{\ell}\}, & r=1,\\
 \{\alphaup_{\ell-r}\}, & 2{\leq}r{<}\ell.
\end{cases}
\end{align*}}
Then 
\begin{equation}\label{eq4.26}
 \Bz_{\bullet}^{\sq_{r_{1},r_{2}}}(C)=\Bz_{\bullet}'(r_{1})\cup\Bz''_{\bullet}(r_{2})\;\;\text{and}\;\;
 \Bz_{\star}^{\sq_{r_{1},r_{2}}}(C)=\Bz'_{\star}(r_{1})\cup\Bz''_{\star}(r_{2}).
\end{equation}
\begin{lem}[Type $\textsc{D}_{\ell}$] Let $\sq_{r_{1},r_{2}}$ be defined by \eqref{e4.24}. 
\begin{itemize}
\item If $2r_{1}{+}r_{2}{<}\ell$, then 
\begin{equation*}
 \etaup(\alphaup_{2i-1})=1,\;\;\forall \, 1{\leq}i{\leq}r_{1}.
\end{equation*}
\item If $2r_{1}{+}r_{2}{<}\ell$ and $r_{2}{=}2$, then 
\begin{equation*}
 \etaup(\alphaup_{\ell-1})=\etaup(\alphaup_{\ell})=\lambdaup\in\{{\pm}1\},\;\;\forall \etaup\,{\in}\,\Homs
\end{equation*}
and both cases occur. 
\item If $2r_{1}{+}r_{2}{<}\ell$ and 
$r_{2}{\geq}3,$ then 
\begin{equation*} 
 \etaup(\alphaup_{\ell-1})=\etaup(\alphaup_{\ell})
\;\; \forall \etaup\in\Homs
\end{equation*}
 and 
for every choice of $\lambdaup_{i}\,{=}\,{\pm}1$ we can find $\etaup\,{\in}\,\Homs$ such that 
\begin{equation*}  \etaup(\alphaup_{i})=\lambdaup_{i}, \; \forall\,i\;\text{with}\; \ell{-}r_{2}{+}1{\leq}i{\leq}\ell{-}1.
\end{equation*}
\item If $2r_{1}{=}\ell$, then 
\begin{equation*}
 \etaup(\alphaup_{2i-1})=\lambdaup\,{\in}\,\{{\pm}1\},\;\;\forall 1{\leq}i{\leq}r_{1},\;\; \forall\etaup\,{\in}\,\Homs
\end{equation*}
and both cases may occur.
\item If $2r_{1}{+}1{=}\ell$ and $r_{2}{=}1$, then 
\begin{equation*}
 \etaup(\alphaup_{2i-1})=\lambdaup\,{\in}\,\{{\pm}1\},\;\;\forall\etaup\,{\in}\,\Homs
\end{equation*}
and both cases may occur.
\item If $2r_{1}{+}r_{2}{=}\ell$ and $2{\leq}r_{2}{<}\ell$, then 
\begin{equation*} 
\begin{cases}
 \etaup(\alphaup_{2i-1})=\lambdaup\,{\in}\,\{{\pm}1\},\;\;\forall 1{\leq}i{\leq}r_{1},\\
 \etaup(\alphaup_{\ell-1}){\cdot}\etaup(\alphaup_{\ell})=\lambdaup, 
\end{cases}
\end{equation*}
and both cases of $\lambdaup$ may occur. 
\end{itemize}
\end{lem} 
\begin{proof} We will use \eqref{eq4.26} and \eqref{eq4.9}.
\par
If $2r_{1}{+}r_{2}{<}\ell$, 
 then 
\begin{equation*} 
\begin{cases}
 \sq_{r_{1},r_{2}}(\alphaup_{2i})=\alphaup_{2i-1}{+}\alphaup_{2i}{+}\alphaup_{2i+1}, & 1{\leq}i{<}r_{1},\\
 \sq_{r_{1},r_{2}}(\alphaup_{2r_{1}})=\alphaup_{2r_{1}-1}{+}\alphaup_{2r_{1}},
\end{cases}
\end{equation*}
implies 
that $\etaup(\alphaup_{2i-1})\,{=}\,1$ for $1{\leq}i{\leq}r_{1}$. \par 
If $r_{2}\,{=}\,0,1$ there are no further conditions on the $\etaup\,{\in}\,\Homs$.
When $r_{2}{\geq}2$, then from  
\begin{equation*}
 \sq_{r_{1},r_{2}}(\alphaup_{\ell-r_{2}})=\alphaup_{\ell-r_{2}}+2(\alphaup_{\ell-r_{2}+1}{+}
 \cdots{+}\alphaup_{\ell-2})+\alphaup_{\ell-1}{+}\alphaup_{\ell},
\end{equation*}
we obtain that all $\etaup\,{\in}\,\Homs$ have $\etaup(\alphaup_{\ell-1}){=}\etaup(\alphaup_{\ell})$
and, assigning arbitrarily 
 $\lambdaup_{\ell-r_{2}+1},\hdots,\lambdaup_{\ell-2},\lambdaup\,{\in}\,\{{\pm}1\}$ 
there is an $\etaup\,{\in}\,\Homs$ with 
\begin{equation*}
 \etaup(\alphaup_{i})=\lambdaup_{i},\;\;\text{for $\ell{-}r_{2}{+}1{\leq}i{\leq}\ell{-}2$},\;\;
 \etaup(\alphaup_{\ell-1})\,{=}\,\etaup(\alphaup_{\ell})\,{=}\,\lambdaup.
\end{equation*}
\par\smallskip 
Let us consider next the cases in which $2r_{1}{+}r_{2}{=}\ell$. \par
If $\ell$ is even and equals $2r_{1}$, then from 
\begin{equation*}
 \sq_{(\ell/2),0}(\alphaup_{2i})=\alphaup_{2i-1}{+}\alphaup_{2i}{+}\alphaup_{2i+1},\;\;\forall \,
 1{\leq}i{\leq}(\ell/2)-1
\end{equation*} 
we obtain that, for all $\etaup\,{\in}\,\Homs$, 
\begin{equation*}\tag{$*$}
 \etaup(\alphaup_{1})\,{=}\,\cdots\,{=}\,\etaup(\alphaup_{2i-1})\,{=}\,\cdots\,{=}\,
 \etaup(\alphaup_{2r_{1}-1})\,{=}\,\lambdaup\,{\in}\{{\pm}1\}
\end{equation*}
 and both values of $\lambdaup$ are attained.\par
 If $\ell$ is odd and equal to $2r_{1}{+}1$, then we still have $(*)$ and the condition that 
\begin{equation*}
 \sq_{(\ell-1)/2,1}(\alphaup_{\ell-1})=\alphaup_{\ell-2}{+}\alphaup_{\ell},\;\;
 \sq_{(\ell-1)/2,1}(\alphaup_{\ell})=\alphaup_{\ell-2}{+}\alphaup_{\ell-1},
\end{equation*}
shows that both values ${\pm}1$ of $\lambdaup$ may be attained. 
When $r_{2}{\geq}2,$ we have  
\begin{equation*}
 \begin{cases}
 \sq_{r_{1},r_{2}}(\alphaup_{2i})=\alphaup_{2i-1}{+}\alphaup_{2i}{+}\alphaup_{2i+1}, & 1{\leq}i{<}r_{1},\\
 \sq_{r_{1},r_{2}}(\alphaup_{2r_{1}})=\alphaup_{2r_{1}-1}{+}\alphaup_{2r_{1}}{+}2(\alphaup_{2r_{1}+1}{+}
 \cdots{+}\alphaup_{\ell-2})+\alphaup_{\ell_{1}}+\alphaup_{\ell}.
\end{cases}
\end{equation*}
This yields for $\etaup\,{\in}\,\Homs$ 
the conditions 
\begin{equation*} \vspace{-18pt}
\begin{cases}
 \etaup(\alphaup_{2i-1})=\lambdaup\,{\in}\,\{{\pm}1\}, & 1{\leq}i{\leq}r_{1}, \\
 \etaup(\alphaup_{\ell-1}){\cdot}\etaup(\alphaup_{\ell})=\lambdaup.
\end{cases}
\end{equation*}
\end{proof}
\begin{cor} The non-compact and non-split real forms of type $\textsc{D}_{\ell}$ ($\ell{\geq}4$) 
are parametrized by the equivalence classes of the involutions $\sq_{\Btt_{r_{1},r_{2}}}$ with 
\begin{equation*} \left\{
\begin{aligned}
 &\Btt_{0,r}\,{=}\,\{\e_{i}\,{\mid}\,r{\leq}i{\leq}\,\ell\}, && 1{\leq}r{\leq}\ell, &&\gs\,{\simeq}\,\so(\ell{-}r,\ell{+}r)\\
 &\Btt_{r,0}\,{=}\,\{\e_{2i-1}{-}\e_{2i}\,{\mid}\, 1{\leq}i{\leq}r\}, && \ell{=}2r,&&\gs\,{\simeq}\,\su^{*}_{2r}(\Hb)\\
 &\Btt_{r,1}\,{=}\,\{\e_{2i-1}{-}\e_{2i}\,{\mid}\, 1{\leq}i{\leq}r\}\,{\cup}\,\{\e_{\ell}\}, && \ell{=}2r{+}1, &&\gs
 \,{\simeq}\, su^{*}_{2r+1}(\Hb). \qed
\end{aligned}\right.
\end{equation*}
 \end{cor}
\subsubsection*{Type $\textsc{E}_{6}$} 
In the following Lemma
we will discuss  
the exceptional Lie algebras of type $\mathrm{\textsc{E}_{6}}$. We keep the notation of \S\ref{s2.4} and
use the canonical basis 
\eqref{bE6} and \eqref{bE6a}:
\begin{equation*} \begin{aligned}
 &\Bz(\textsc{E}_{6})=\{\alphaup_{1}{=}\zetaup_{\emptyset},\,\alphaup_{2}{=}\e_{1}{+}\e_{2}\}\cup
 \{\alphaup_{i}{=}\e_{i-1}{-}\e_{i-2}\,{\mid}\, i=3,4,5,6\},\\
 & \Bz'(\textsc{E}_{6})=\{\alphaup'_{i}{=}\e_{i}{-}\e_{i+1}\,{\mid}\,1{\leq}i{\leq}5\}\cup\{\alphaup'_{6}{=}\zetaup_{4,5,6,7}\},
 \end{aligned}
\end{equation*}
the first to discuss involutions in the Weyl group, the second for those not belonging to the Weyl group. 
\begin{lem}[Type $\textsc{E}_{6}$] Let $\Rad$ be the irreducible root system
$\Rad(\textsc{E}_{6})$ of type $\textsc{E}_{6}$.  Modulo equivalence, the involutions belonging to the Weyl
group and the corresponding $\etaup$ in $\Homs$ are described by $\stt\,{=}\,\sq_{\Btt}$, with 
\begin{equation*}
\begin{array}{| c | c | c | c | c |}
\hline
\# & \Btt & \Bz_{\bullet}^{\stt}(C) & \Bz_{\star}^{\stt}(C) & \text{conditions} \\
\hline
(1) &\{\alphaup_{3}\} &\{\alphaup_{3}\} & \{\alphaup_{4}\} & \etaup(\alphaup_{3})=1\\
\hline
(2) & \{\alphaup_{3},\alphaup_{6}\} & \{\alphaup_{3},\alphaup_{6}\} & \{\alphaup_{4},\alphaup_{5}\}
& \etaup(\alphaup_{3})=\etaup(\alphaup_{6})=1\\
\hline
(3) & \{\alphaup_{2},\alphaup_{3},\alphaup_{6}\} &  \{\alphaup_{2},\alphaup_{3},\alphaup_{6}\} & 
\{\alphaup_{1},\alphaup_{4},\alphaup_{5}\} & \!\!\etaup(\alphaup_{2}){=}\etaup(\alphaup_{3}){=}\etaup(\alphaup_{6}){=}1\!\!\\
\hline
(4) & \{\e_{1}{\pm}\e_{2},\,\e_{3}{\pm}\e_{4}\} & \{\alphaup_{2},\alphaup_{3},\alphaup_{4},\alphaup_{5}\} & 
\{\alphaup_{1},\alphaup_{6}\} & 
\begin{gathered}
 \etaup(\alphaup_{2}({\cdot}\etaup(\alphaup_{3})=1,\\
  \etaup(\alphaup_{3}({\cdot}\etaup(\alphaup_{5})=1.
\end{gathered}\\
\hline
 \end{array}
\end{equation*}
\par
Let $\Rad$ be the irreducible system $\Rad'(\textsc{E}_{6})$ of type $\textsc{E}_{6}$. 
Modulo equivalence, the involutions not belonging to the Weyl
group  are described by $\stt\,{=}\,\epi{\circ}\sq_{\Btt}$, with
$\epi(\alphaup'_{i})=\alphaup'_{6-i}$ for $1{\leq}i{\leq}5$ and $\epi(\alphaup'_{6})=\alphaup'_{6}$. 
The corresponding $\etaup$ in $\Homs$ are described by the conditions 
in the table below:
\begin{equation*}
\begin{array}{| c | c | c | c | c |}
\hline
\# & \Btt & \Bz_{\bullet}^{\stt}(C) & \Bz_{\star}^{\stt}(C) & \text{conditions} \\
\hline
(5) & \emptyset & \emptyset & \!\!\{\alphaup'_{1},\alphaup'_{2},\alphaup'_{4},\alphaup'_{5}\}\! \!& \emptyset \\
\hline
(6) &\{\alphaup'_{3}\} &\{\alphaup'_{3}\} & \{\alphaup'_{2},\alphaup'_{4},\alphaup'_{6}\} & \etaup(\alphaup'_{3}){=}1 \\
\hline
(7) &\!\! \{\e_{2},\e_{3},\e_{4},\e_{5}\} \!\! & \{\alphaup'_{2},\alphaup'_{3},\alphaup'_{4}\} & \{\alphaup'_{1},\alphaup'_{5},\alphaup'_{6}\}
& \etaup(\alphaup'_{2}){\cdot}\etaup(\alphaup'_{4}){=}1\\
\hline
(8) & \begin{aligned}
\{\e_{1},\e_{2},\e_{3},\;\;\\[-4pt]
\e_{4},\e_{5},\e_{6}\} 
\end{aligned}&  \{\alphaup'_{1},\alphaup'_{2},\alphaup'_{3},\alphaup'_{4},\alphaup'_{5}\} & 
\{\alphaup'_{6}\} & \!\!\etaup(\alphaup'_{1}){\cdot}\etaup(\alphaup'_{3}){\cdot}\etaup(\alphaup'_{5}){=}1\!\!
\\
\hline
(9) & \!\!\begin{aligned}
\{\e_{1},\e_{2},\e_{3},\e_{4},\;\\[-4pt]
\e_{4},\e_{6},\e_{7},\e_{8}\} 
\end{aligned}\!\!& \begin{aligned} \{\alphaup'_{1},\alphaup'_{2},\alphaup'_{3},\alphaup'_{4},\;\;\\[-4pt]
\alphaup'_{5},\alphaup'_{6},\alphaup'_{7},\alphaup'_{8}\}\end{aligned} & 
\emptyset & \emptyset
\\
\hline
 \end{array}
\end{equation*}
\end{lem} 
\begin{proof} The statement follows from \eqref{eq4.9}, as we get, fo the roots in $\Bz_{\star}^{\stt}(C)$, 
\begin{align*}
 \tag{1} &\stt(\alphaup_{4})=\alphaup_{3}{+}\alphaup_{4},\\
 \tag{2} & 
\begin{cases}
 \stt(\alphaup_{4})=\alphaup_{3}{+}\alphaup_{4},\\
 \stt(\alphaup_{5})=\alphaup_{5}{+}\alphaup_{6}
\end{cases} \\
\tag{3} & 
\begin{cases}
 \stt(\alphaup_{1})=\alphaup_{1}{+}\alphaup_{2},\\
 \stt(\alphaup_{4})=\alphaup_{3}{+}\alphaup_{4},\\
 \stt(\alphaup_{5})=\alphaup_{5}{+}\alphaup_{6},
\end{cases}
\\
\tag{4} & 
\begin{cases}
 \stt(\alphaup_{1})=\alphaup_{1}{+}2(\alphaup_{2}{+}\alphaup_{4}){+}\alphaup_{3}{+}\alphaup_{5},\\
 \stt(\alphaup_{6})=\alphaup_{2}{+}\alphaup_{3}{+}2(\alphaup_{4}{+}\alphaup_{5}){+}\alphaup_{6},
\end{cases}
\end{align*} 
\begin{align*}
 \tag{5} & 
\begin{cases}
 \stt(\alphaup'_{1})=\alphaup'_{5},\\
 \stt(\alphaup'_{2})=\alphaup'_{4},
\end{cases}\\
\tag{6} & 
\begin{cases}
 \stt(\alphaup'_{2})=\alphaup'_{3}{+}\alphaup'_{4},\\
 \stt(\alphaup'_{4})=\alphaup'_{2}{+}\alphaup'_{3},\\
 \stt(\alphaup'_{6})=\alphaup'_{3}{+}\alphaup'_{6},
\end{cases}\\
\tag{7} & 
\begin{cases}
 \stt(\alphaup'_{1})=\alphaup'_{2}{+}\alphaup'_{3}{+}\alphaup'_{4}{+}\alphaup'_{5},\\
  \stt(\alphaup'_{5})=\alphaup'_{1}{+}\alphaup'_{2}{+}\alphaup'_{3}{+}\alphaup'_{4},\\
   \stt(\alphaup'_{6})=\alphaup'_{2}{+}\alphaup'_{4}{+}2\alphaup'_{3}{+}\alphaup'_{6},
\end{cases}
\\
\tag{8} & \stt(\alphaup'_{6})=\alphaup'_{1}{+}2\alphaup'_{2}{+}3\alphaup'_{3}{+}2\alphaup'_{4}{+}\alphaup'_{5}{+}\alphaup'_{6}.
\end{align*}
 
\end{proof}
\begin{cor} The non-compact and non-split or quasi-split 
real forms of type $\textsc{E}_{6}$  correspond to types (4) and (7) above, yielding $\textsc{E\,IV}$ and
$\textsc{E\,III}$, respectively (cf. \cite[Ch.X]{Hel78}).
 \end{cor}
\par\medskip
\par\medskip
\subsubsection*{Type $\textsc{E}_{7}$} We consider the root system $\Rad\,{=}\,\Rad(\textsc{E}_{7})$ 
of type $\textsc{E}_{7}$, described in \S\ref{s2.4}. There we showed that all systems of orthogonal roots
containing $k$ elements are equivalent if $0{\leq}k{\leq}7$ and $k\,{\neq}\,3,4$
and that there are two equivalence classes for $k\,{=}\,3,4$. We recall that the standard basis is 
\begin{equation*}
 \Bz(\textsc{E}_{7})=\{\alphaup_{1}{=}\zetaup_{\emptyset},\alphaup_{2}{=}\e_{1}{+}\e_{2}\}\cup
 \{\alphaup_{i}=\e_{i-1}{-}\e_{i-2}\,{\mid}\, 3{\leq}i{\leq}7\}.
\end{equation*}

\par\bigskip
\begin{lem}[Type $\textsc{E}_{7}$]
 Let $\Rad\,{=}\,\Rad(\textsc{E}_{7})$ be an irreducible system of type $\textsc{E}_{7}$, $\stt\,{\in}\,\Ib(\Rad)$.
 We assume, as we can, that the standard basis $\Bz(\textsc{E}_{7})$ of \ref{bE7} 
 consists of the simple positive roots of an $S$-chamber for $\stt\,{=}\,\sq_{\Btt}$.
Modulo conjugation, we reduce to the following cases: 
\begin{equation*} 
\begin{array}{| c | c | c | c | c |}
\hline
\# & \Btt & \Bz_{\bullet}^{\stt}(C) & \Bz_{\star}^{\stt}(C) & \text{conditions} \\
\hline 
(1) & \{\alphaup_{3}\} & \{\alphaup_{3}\} & \alphaup_{4} & \etaup(\alphaup_{3}){=}1\\
 \hline
 (2) & \{\alphaup_{3},\alphaup_{7}\} & \{\alphaup_{3},\alphaup_{7}\} & \{\alphaup_{4},\alphaup_{6}\} &
 \etaup(\alphaup_{3}){=}\etaup(\alphaup_{7}){=}1\\
 \hline
 (3) & \{\alphaup_{3},\alphaup_{5},\alphaup_{7}\} & \{\alphaup_{3},\alphaup_{7}\} & \{\alphaup_{4},\alphaup_{6}\} &
 \etaup(\alphaup_{3}){=}\etaup(\alphaup_{5}){=}\etaup(\alphaup_{7})\\
 \hline
  (4) & \{\alphaup_{2},\alphaup_{3},\alphaup_{5}\} & \{\alphaup_{2},\alphaup_{3},\alphaup_{5}\} & \{\alphaup_{1},\alphaup_{4},\alphaup_{6}\} & \!\!
 \etaup(\alphaup_{2}){=}\etaup(\alphaup_{3}){=}\etaup(\alphaup_{5}){=}1\!\!\\
 \hline
  (5) & \{\alphaup_{1},\alphaup_{3},\alphaup_{5},\alphaup_{7}\} & \!\!\{\alphaup_{1},\alphaup_{3},\alphaup_{5},\alphaup_{7}\} 
 \!\!& \{\alphaup_{4},\alphaup_{6}\} & \begin{aligned}
 \etaup(\alphaup_{1}){=}
 \etaup(\alphaup_{3}){=}\etaup(\alphaup_{5})\;\;\\[-4pt]
 {=}\etaup(\alphaup_{7}){=}1\end{aligned}\\
 \hline
 (6) & \{\e_{1}{\pm}\e_{2},\e_{3}{\pm}\e_{4}\} &\!\!\{\alphaup_{2},\alphaup_{3},\alphaup_{4},\alphaup_{5}\} 
 & \{\alphaup_{1},\alphaup_{6}\} \!\!& 
 \etaup(\alphaup_{2}){=}\etaup(\alphaup_{3}){=}\etaup(\alphaup_{5})\\
 \hline
  (7) &\!\!\! \begin{aligned}
  \{\e_{1}{\pm}\e_{2},\e_{3}{\pm}\e_{4},\;\;
  \\[-4pt]
  \e_{5}{-}\e_{6}\} \end{aligned}
  & \begin{aligned}
  \{\alphaup_{2},\alphaup_{3},\alphaup_{4},\;\; \\[-4pt]
  \alphaup_{5},\alphaup_{7}\} \end{aligned}\!\!
 & \{\alphaup_{1},\alphaup_{6}\} & \begin{gathered}
\etaup(\alphaup_{3}){\cdot}\etaup(\alphaup_{5}){=}1,\\[-4pt]
\etaup_{2}{\cdot}\etaup(\alphaup_{3}){\cdot}\etaup(\alphaup_{7}){=}1
\end{gathered}\\
 \hline
   (8) &\!\! \begin{aligned}
  \{\e_{1}{\pm}\e_{2},\e_{3}{\pm}\e_{4},\;\;
  \\[-4pt]
  \e_{5}{\pm}\e_{6}\} \end{aligned}\!\!\!
  & \begin{aligned}
  \{\alphaup_{2},\alphaup_{3},\alphaup_{4},\;\; \\[-4pt]
  \alphaup_{5},\alphaup_{6},\alphaup_{7}\} \end{aligned}
 & \{\alphaup_{1}\} & \!\!
\etaup(\alphaup_{2}){\cdot}\etaup(\alphaup_{5}){\cdot}\etaup(\alphaup_{7}){=}1\!\!
\\
 \hline
  (9) &\!\! \begin{aligned}
  \{\e_{1}{\pm}\e_{2},\e_{3}{\pm}\e_{4},\;\;
  \\[-4pt]
  \e_{5}{\pm}\e_{6},\e_{7}{-}\e_{8}\} \end{aligned}\!\!\!
  & \!\!\begin{aligned}
  \{\alphaup_{1},\alphaup_{2},\alphaup_{3},\alphaup_{4},\;\; \\[-4pt]
  \alphaup_{5},\alphaup_{6},\alphaup_{7}\} \end{aligned}\!\!
 & \emptyset & \emptyset
\\
 \hline
\end{array}
\end{equation*}
\end{lem}
\begin{proof}
The statement is a consequence of \eqref{eq4.9} and the equalities:
\begin{align*}
\tag{1} 
& \stt(\alphaup_{4})\,{=}\,\alphaup_{3}{+}\alphaup_{4},
\\
\tag{2} &\begin{cases}
 \stt(\alphaup_{4}){=}\alphaup_{3}{+}\alphaup_{4},\\
 \stt(\alphaup_{6}){=}\alphaup_{6}{+}\alphaup_{7},
\end{cases}\\
\tag{3} & \begin{cases}
 \stt(\alphaup_{4}){=}\alphaup_{3}{+}\alphaup_{4}{+}\alphaup_{5},\\
 \stt(\alphaup_{6}){=}\alphaup_{5}{+}\alphaup_{6}{+}\alphaup_{7},
\end{cases}\\
\tag{4} & \begin{cases}
 \stt(\alphaup_{1}){=}\alphaup_{1}{+}\alphaup_{2},\\
 \stt(\alphaup_{4}){=}\alphaup_{2}{+}\alphaup_{3}{+}\alphaup_{4}{+}\alphaup_{5},\\
 \stt(\alphaup_{6}){=}\alphaup_{5}{+}\alphaup_{6},
\end{cases}\\
\tag{5} &\begin{cases}
 \stt(\alphaup_{2}){=}\alphaup_{1}{+}\alphaup_{2},\\
 \stt(\alphaup_{4}){=}\alphaup_{3}{+}\alphaup_{4}{+}\alphaup_{5},\\
 \stt(\alphaup_{6}){=}\alphaup_{5}{+}\alphaup_{6}{+}\alphaup_{7},
\end{cases}\\
\tag{6} &\begin{cases}
 \stt(\alphaup_{1}){=}\alphaup_{1}{+}2(\alphaup_{2}{+}\alphaup_{4}){+}2\alphaup_{3}{+}\alphaup_{5},\\
 \stt(\alphaup_{6}){=}\alphaup_{2}{+}\alphaup_{3}{+}2(\alphaup_{4}{+}\alphaup_{5}){+}\alphaup_{6},
\end{cases}\\
\tag{7} 
&\begin{cases}
 \stt(\alphaup_{1}){=}\alphaup_{1}{+}2(\alphaup_{2}{+}\alphaup_{4}){+}\alphaup_{3}{+}\alphaup_{5},\\
 \stt(\alphaup_{6}){=}\alphaup_{2}{+}\alphaup_{3}{+}2(\alphaup_{4}{+}\alphaup_{5}){+}\alphaup_{6}{+}\alphaup_{7},
\end{cases}\\
 \tag{8} & \stt(\alphaup_{1})=\alphaup_{1}{+}3\alphaup_{2}{+}2\alphaup_{3}
{+}4\alphaup_{4}{+}3\alphaup_{5}{+}2\alphaup_{6}{+}\alphaup_{7}.
\end{align*}
\end{proof}
\begin{cor}
The non-compact and non-split or quasi-split 
real forms of type $\textsc{E}_{7}$  correspond to types (3) and (6) above, yielding $\textsc{E\,VI}$ and
$\textsc{E\,VII}$, respectively (cf. \cite[Ch.X]{Hel78}). 
\end{cor}
\subsubsection*{Type $\textsc{E}_{8}$} 
We consider the root system $\Rad\,{=}\,\Rad(\textsc{E}_{8})$ 
of type $\textsc{E}_{8}$, described in \S\ref{s2.4}. There we showed that all systems of orthogonal roots
containing $k$ elements are equivalent if $0{\leq}k{\leq}8$ and $k\,{\neq}\,4$
and that there are two equivalence classes for $k\,{=}\,4$. We recall that the standard basis is 
\begin{equation*}
 \Bz(\textsc{E}_{8})=\{\alphaup_{1}{=}\zetaup_{\emptyset},\alphaup_{2}{=}\e_{1}{+}\e_{2}\}\cup
 \{\alphaup_{i}=\e_{i-1}{-}\e_{i-2}\,{\mid}\, 3{\leq}i{\leq}8\}.
\end{equation*}
\begin{lem}[Type $\textsc{E}_{8}$]
 Let $\Rad\,{=}\,\Rad(\textsc{E}_{8})$ be an irreducible system of type $\textsc{E}_{8}$, $\stt\,{\in}\,\Ib(\Rad)$.
 We assume, as we can, that the standard basis $\Bz(\textsc{E}_{8})$ of \ref{bE8} 
 consists of the simple positive roots of an $S$-chamber for $\stt$.
Modulo conjugation, we reduce to the following cases: 
 \begin{align*}
 \tag{1}
 &\Bz_{\bullet}^{\stt}(C)\,{=}\,\{\alphaup_{1}\}, && \Bz_{\star}^{\stt}(C)\,{=}\,\{\alphaup_{2}\}, \\ 
\tag{2}
 &\Bz_{\bullet}^{\stt}(C)\,{=}\,\{\alphaup_{1},\alphaup_{3}\}, && \Bz_{\star}^{\stt}(C)\,{=}\,\{\alphaup_{2},\alphaup_{4}\}, 
 \\
\tag{3}
 &\Bz_{\bullet}^{\stt}(C)\,{=}\,\{\alphaup_{3},\alphaup_{5},\alphaup_{7}\}, 
 && \Bz_{\star}^{\stt}(C)\,{=}\,\{\alphaup_{4},\alphaup_{6},\alphaup_{8}\}, \\ 
\tag{4}
 &\Bz_{\bullet}^{\stt}(C)\,{=}\,\{\alphaup_{1},\alphaup_{3},\alphaup_{5},\alphaup_{7}\}, 
 && \Bz_{\star}^{\stt}(C)\,{=}\,\{\alphaup_{2},\alphaup_{4},\alphaup_{6},\alphaup_{8}\}, \\
\tag{5}
 &\Bz_{\bullet}^{\stt}(C)\,{=}\,\{\alphaup_{2},\alphaup_{3},\alphaup_{4},\alphaup_{5}\}, 
 && \Bz_{\star}^{\stt}(C)\,{=}\,\{\alphaup_{1},\alphaup_{6}\}, \\
\tag{6}
 &\Bz_{\bullet}^{\stt}(C)\,{=}\,\{\alphaup_{2},\alphaup_{3},\alphaup_{4},\alphaup_{5},\alphaup_{7}\}, 
 && \Bz_{\star}^{\stt}(C)\,{=}\,\{\alphaup_{1},\alphaup_{6},\alphaup_{8}\}, \\
\tag{7}
&\Bz_{\bullet}^{\stt}(C){=}\{\alphaup_{2},\alphaup_{3},\alphaup_{4},\alphaup_{5},\alphaup_{6},\alphaup_{7}\}, 
&&\Bz_{\star}(C){=}\{\alphaup_{1},\alphaup_{8}\}, \\
\tag{8}
&\Bz_{\bullet}^{\stt}(C){=}\{\alphaup_{1},\alphaup_{2},\alphaup_{3},\alphaup_{4},\alphaup_{5},\alphaup_{6},\alphaup_{7}\}, 
&&\Bz_{\star}(C){=}\{\alphaup_{8}\}
\end{align*}
Accordingly, the following conditions characterize  $\etaup\,{\in}\,\Homs$: 
\begin{align*}
 \tag{1} & \etaup(\alphaup_{1})\,{=}\,1,\\
 \tag{2} & \etaup(\alphaup_{1}){=}\etaup(\alphaup_{3})\,{=}\,1,\\
 \tag{3} &  \etaup(\alphaup_{3}){=}\etaup(\alphaup_{5}){=}\etaup(\alphaup_{7})=1,\\ 
 \tag{4} &  \etaup(\alphaup_{1}){=}\etaup(\alphaup_{3})=\etaup(\alphaup_{5}){=}\etaup(\alphaup_{7})=1,\\ 
 \tag{5} &  \etaup(\alphaup_{2}){=}\etaup(\alphaup_{3}){=}\etaup(\alphaup_{5})\,{=}\,{\pm}1,\\ 
 \tag{6} &  \etaup(\alphaup_{2}){=}\etaup(\alphaup_{3}){=}\etaup(\alphaup_{5})\,{=}\,{\pm}1,\;
 \etaup(\alphaup_{7}){=}1, \\ 
  \tag{7} &  \etaup(\alphaup_{2}){=}\etaup(\alphaup_{3})\,{=}\,{\pm}1,\; \etaup(\alphaup_{2}){\cdot}
  \etaup(\alphaup_{5}){\cdot}\etaup(\alphaup_{7})\,{=}\,1,
  \\ 
   \tag{8} &  \etaup(\alphaup_{3}){\cdot}\etaup(\alphaup_{5}){\cdot}\etaup(\alphaup_{7})=1.\\
\end{align*}
\end{lem}
[Notice, for case $(6)$, that the diagram $\textsc{D}_{4}$ in which the three extreme roots are compact
and the one in the middle is not compact represents the corresponding split Lie algebra.]
\begin{proof}
The statement is a consequence of \eqref{eq4.9} and the equalities:
\begin{align*}
\tag{1} & \begin{cases}
 \stt(\alphaup_{2})\,{=}\,\alphaup_{1}{+}\alphaup_{2},
\end{cases}\\
\tag{2} &\begin{cases}
 \stt(\alphaup_{2}){=}\alphaup_{1}{+}\alphaup_{2},\\
 \stt(\alphaup_{4}){=}\alphaup_{3}{+}\alphaup_{4},
\end{cases}\\
\tag{3} & \begin{cases}
 \stt(\alphaup_{4}){=}\alphaup_{3}{+}\alphaup_{4}{+}\alphaup_{5},\\
 \stt(\alphaup_{6}){=}\alphaup_{5}{+}\alphaup_{6}{+}\alphaup_{7},\\
 \stt(\alphaup_{8}){=}\alphaup_{7}{+}\alphaup_{8},
\end{cases}\\
\tag{4} & \begin{cases}
 \stt(\alphaup_{2}){=}\alphaup_{1}{+}\alphaup_{2},\\
 \stt(\alphaup_{4}){=}\alphaup_{3}{+}\alphaup_{4}{+}\alphaup_{5},\\
 \stt(\alphaup_{6}){=}\alphaup_{5}{+}\alphaup_{6}{+}\alphaup_{7},\\
 \stt(\alphaup_{8}){=}\alphaup_{7}{+}\alphaup_{8},
\end{cases}\\
\tag{5} &\begin{cases}
 \stt(\alphaup_{1}){=}\alphaup_{1}{+}2(\alphaup_{2}{+}\alphaup_{4}){+}\alphaup_{3}{+}\alphaup_{5},\\
 \stt(\alphaup_{6}){=}\alphaup_{2}{+}\alphaup_{3}{+}2(\alphaup_{4}{+}\alphaup_{5}){+}\alphaup_{6},\\
\end{cases}\\
\tag{6} &\begin{cases}
 \stt(\alphaup_{1}){=}\alphaup_{1}{+}2\alphaup_{2}{+}\alphaup_{3}{+}2\alphaup_{4}{+}\alphaup_{5},\\
 \stt(\alphaup_{6}){=}\alphaup_{2}{+}\alphaup_{3}{+}2\alphaup_{4}{+}2\alphaup_{5}{+}\alphaup_{6}{+}\alphaup_{7},\\
 \stt(\alphaup_{8}){=}\alphaup_{7}{+}\alphaup_{8}
\end{cases}\\
\tag{7}
&\begin{cases}
 \stt(\alphaup_{1}){=}\alphaup_{1}{+}3\alphaup_{2}{+}2\alphaup_{3}{+}4\alphaup_{4}
 {+}3\alphaup_{5}{+}2\alphaup_{6}{+}\alphaup_{7},\\
 \stt(\alphaup_{8}){=}\alphaup_{2}{+}\alphaup_{3}{+}2(\alphaup_{4}{+}\alphaup_{5}{+}\alphaup_{6}{+}\alphaup_{7}){+}
 \alphaup_{8},
\end{cases}\\
 \tag{8} & \begin{cases}\stt(\alphaup_{8})=2\alphaup_{1}{+}4\alphaup_{2}{+}3\alphaup_{3}
{+}6\alphaup_{4}{+}5\alphaup_{5}{+}4\alphaup_{6}{+}3\alphaup_{7}{+}\alphaup_{8}.
\end{cases}
\end{align*}
\end{proof}
\par\medskip
\subsubsection*{Type $\textsc{F}_{4}$}
We consider the root system $\Rad\,{=}\,\Rad(\textsc{F}_{4})$ 
of type $\textsc{F}_{4}$, described in \S\ref{s2.4}. There we showed that all systems of 
strongly orthogonal roots
with the same number of short and long roots are equivalent. 
 We recall that the standard basis is 
\begin{equation}\label{bF4}
 \Bz(\textsc{F}_{4})=\big\{\alphaup_{1}{=}\e_{1}{-}\e_{2},\,\alphaup_{2}{=}\e_{2}{-}\e_{3},\, \alphaup_{3}{=}\e_{3},\,
 \alphaup_{4}{=}\tfrac{1}{2}(\e_{4}{-}\e_{1}{-}\e_{2}{-}\e_{3}\big\}.
\end{equation}
\begin{lem}[Type $\textsc{F}_{4}$]
 Let $\Rad\,{=}\,\Rad(\textsc{F}_{4})$ be an irreducible system of type $\textsc{F}_{4}$.
 Modulo equivalence, we reduce to the following cases, 
where $\stt\,{=}\,\sq_{\Btt}$
 for a system of orthogonal roots $\Btt\,{\subset}\,\Rad(\textsc{F}_{4})$,
and 
 the standard basis $\Bz(\textsc{F}_{4})$ of \eqref{bF4} 
 consists of the simple positive roots of an $S$-chamber for $\stt$. 
\begin{equation*} 
\begin{array}{| c | c | c | c | c |} \hline
\#  & \Btt & \Bz_{\bullet}^{\stt}(C) & \Bz_{\star}^{\stt}(C) & \text{conditions} \\
\hline
(1) & \{\alphaup_{4}\} &
\{\alphaup_{4}\} & \{\alphaup_{3}\} & \etaup(\alphaup_{4})=1 \\
 \hline
 (2) &\{\alphaup_{1}\} & \{\alphaup_{1}\} & \{\alphaup_{2}\} & \etaup(\alphaup_{1})=1 \\
 \hline
 (3) & \{\alphaup_{1},\alphaup_{3}\} & \{\alphaup_{1},\alphaup_{3}\} & \{\alphaup_{2},\alphaup_{4}\} & \etaup(\alphaup_{1})=1,\,
\etaup( \alphaup_{3})=1\\ \hline
(4) & \{\e_{2}{\pm}\e_{3}\} & \{\alphaup_{2},\alphaup_{3}\} & \{\alphaup_{1},\alphaup_{4}\} & \etaup(\alphaup_{2})=1 \\
\hline
(5) & \{\e_{1}{\pm}\e_{2},\e_{3}\} & \{\alphaup_{1},\alphaup_{2},\alphaup_{3}\} & \{\alphaup_{4}\} & \etaup(\alphaup_{1}){\cdot}
\etaup(\alphaup_{3})=1\\
\hline
(6) & \{\e_{1}{-}\e_{4},\,\e_{2}{\pm}\e_{3}\} & 
\{\alphaup_{2},\alphaup_{3},\alphaup_{4}\} & \{\alphaup_{1}\} &\etaup(\alphaup_{2})=1\\
\hline
\end{array}
\end{equation*}
\end{lem}
\begin{proof}
The statement is a consequence of \eqref{eq4.9} and the equalities:
\begin{align*}
(1) &\quad
 \stt(\alphaup_{3})\,{=}\,\alphaup_{3}{+}\alphaup_{4},
\\
(2) &\quad
 \stt(\alphaup_{2}){=}\alphaup_{1}{+}\alphaup_{2},
\\
(3) &\quad \begin{cases}
 \stt(\alphaup_{2}){=}\alphaup_{1}{+}\alphaup_{2}{+}2\alphaup_{3},\\
 \stt(\alphaup_{4}){=}\alphaup_{3}{+}\alphaup_{4},
\end{cases}\\
(4) &\quad \begin{cases}
\stt(\alphaup_{1}){=}\alphaup_{1}{+}2\alphaup_{2}{+}2\alphaup_{3},\\
 \stt(\alphaup_{4}){=}\alphaup_{2}{+}2\alphaup_{3}{+}\alphaup_{4},
\end{cases}\\ 
(5) &\quad
 \stt(\alphaup_{4}){=}\alphaup_{1}{+}2\alphaup_{2}{+}3\alphaup_{3}{+}\alphaup_{4}
\\
(6) &\quad
 \stt(\alphaup_{1}){=}\alphaup_{1}{+}3\alphaup_{2}{+}3\alphaup_{3}{+}2\alphaup_{4}.
\end{align*}
\end{proof}
\begin{rmk} The $S\!$-diagrams in the different cases are 
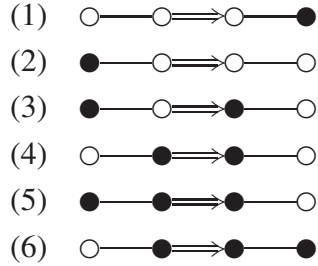
\begin{gather*}
(1)\quad \xymatrix @M=0pt @R=7pt @!C=3pt{
 \medcirc \ar@{-}[r] 
 &\medcirc\ar@{=>}[r]
 &\medcirc\ar@{-}[r]
 &\medbullet}\\
 (2)\quad \xymatrix @M=0pt @R=7pt @!C=3pt{
 \medbullet \ar@{-}[r] 
 &\medcirc\ar@{=>}[r]
 &\medcirc\ar@{-}[r]
 &\medcirc }\\
  (3)\quad \xymatrix @M=0pt @R=7pt @!C=3pt{
 \medbullet \ar@{-}[r] 
 &\medcirc\ar@{=>}[r]
 &\medbullet\ar@{-}[r]
 &\medcirc }\\
  (4)\quad \xymatrix @M=0pt @R=7pt @!C=3pt{
 \medcirc \ar@{-}[r] 
 &\medbullet\ar@{=>}[r]
 &\medbullet\ar@{-}[r]
 &\medcirc }\\
(5)\quad \xymatrix @M=0pt @R=7pt @!C=3pt{
 \medbullet \ar@{-}[r] 
 &\medbullet\ar@{=>}[r]
 &\medbullet\ar@{-}[r]
 &\medcirc} \\
(6)\quad \xymatrix @M=0pt @R=7pt @!C=3pt{
 \medcirc \ar@{-}[r] 
 &\medbullet\ar@{=>}[r]
 &\medbullet\ar@{-}[r]
 &\medbullet} \\
\end{gather*}
 
\end{rmk}

\subsubsection*{Type $\textsc{G}_{2}$}
A root system of type $\textsc{G}_{2}$ and its unique, modulo conjugation, maximal system of
strongly orthogonal roots are 
\begin{align*}
 &\Rad\,{=}\,\{{\pm}(\e_{i}{-}\e_{j})\,{\mid}\,(i,j){\in}\Sb_{3;2}\}\,{\cup}\,\{{\pm}(2\e_{i}{-}\e_{j}{-}\e_{k}\,{\mid}\,
 (i,j,k){\in}\Sb_{3}\},\\
& \Btt\,{=}\,\{\e_{1}{-}\e_{2},\,2\e_{3}{-}\e_{1}{-}\e_{2}\}.
\end{align*}
\paragraph{1} With $\Btt\,{=}\,\{\e_{1}{-}\e_{2}\}$ we obtain $\Rad_{\,\;\bullet}^{\stt}\,{=}\,\{{\pm}(\e_{1}{-}\e_{2})\}$,
with $S\!$-diagram 
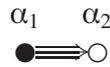
\begin{equation*}
 \xymatrix @M=0pt @R=7pt @!C=3pt{
 \alphaup_{1}&\alphaup_{2}\\
 \medbullet\ar@3{->}[r]&\medcirc}
 \end{equation*}
 and therefore 
\begin{equation*}
 \Bz_{\bullet}^{\stt}(C)\,{=}\,\{\alphaup_{1}\},\;\; \Bz_{\star}^{\stt}(C)\,{=}\,\{\alphaup_{2}),\;\;
 \stt(\alphaup_{2})\,{=}\, 3\alphaup_{1}{+}\alphaup_{2}.
\end{equation*}
Then,  by \eqref{eq4.9}, \par\centerline{
$\etaup\in\Homs$ iff $\etaup\,{\in}\,\Homz$ and $\etaup(\alphaup_{1})\,{=}\,1$.} 
\par\medskip 
 \paragraph{2} With $\Btt\,{=}\,\{2\e_{3}{-}\e_{1}{-}\e_{2}\}$ we obtain 
 $\Rad_{\,\;\bullet}^{\stt}\,{=}\,\{{\pm}(2\e_{3}{-}\e_{1}{-}\e_{2})\}$,
with $S\!$-diagram 
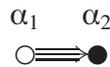
\begin{equation*}
 \xymatrix @M=0pt @R=7pt @!C=3pt{
 \alphaup_{1}&\alphaup_{2}\\
 \medcirc\ar@3{->}[r]&\medbullet}
 \end{equation*}
 and therefore 
\begin{equation*}
 \Bz_{\bullet}^{\stt}(C)\,{=}\,\{\alphaup_{2}\},\;\; \Bz_{\star}^{\stt}(C)\,{=}\,\{\alphaup_{1}),\;\;
 \stt(\alphaup_{1})\,{=}\, \alphaup_{1}{+}\alphaup_{2}.
\end{equation*}
Then,  by \eqref{eq4.9}, \par\centerline{
$\etaup\in\Homs$ iff $\etaup\,{\in}\,\Homz$ and $\etaup(\alphaup_{2})\,{=}\,1$.}
 
\subsection{Cartan subalgebras of a real semisimple Lie algebra}
Assume now that a real form $\go$ of $\gt$ has been fixed. 
By using Thm.\ref{t3.1} and 
the Cayley transfom (see e.g. \S\ref{s3.2.2}),
we can find $\sigmaup_{0}\,{\in}\,\Invs$ such that, for $\stt\,{=}\,\rhoup(\sigmaup_{0})$, 
we have $\go\,{\simeq}\,\gt_{\sigmaup_{0}}$ and $\Rad^{\stt}_{\bullet}\,{=}\,\Rad^{\sigmaup_{0}}_{\bullet}$.
In this case $\hg_{\stt}$ is a Cartan subalgebra of $\gt_{\sigmaup_{0}}$, with maximal vector part. 
From the discussion of the previous subsections we obtain 
\begin{thm}
 Let $\go$ be a real semisimple Lie algebra and $\sigmaup_{0}\,{\in}\,\Inv^{\tauup}(\gr,\hr)$ an involution
 such that $\gt_{\sigmaup_{0}}\,{\simeq}\go$. Let 
\begin{equation*}
 \sigmaup_{0}\,{=}\,\epi\circ\sq_{\,\betaup_{1}}{\circ}\cdots{\circ}\sq_{\,\betaup_{r}}
\end{equation*}
for an involution $\epi\,{\in}\,\Ib^{*}(\Rad)$ and a set $\betaup_{1},\hdots,\betaup_{r}$
of strongly orthogonal roots in $\Rd{\circ}{\epi}$. Then the 
equivalence classes of Cartan subalgebras of $\go$
are in a one to one correspondence with the equivalence classes $\Btt$ of strongly orthogonal roots
in $\Rd{\circ}{\epi}$ containing $\betaup_{1},\hdots,\betaup_{r}$. \qed
\end{thm}


%
\end{document}